\documentclass{article}
\usepackage {amsmath, amsfonts, amssymb, mathrsfs, amsthm,stmaryrd, sectsty,hyphenat,url,accents,graphicx,comment}
\usepackage[usenames]{color}
\usepackage{palatino}
\usepackage[all]{xy}
\usepackage[toc,page]{appendix}
\usepackage{tocloft}
\usepackage{enumitem}
\usepackage{dynkin-diagrams}

\def\mf#1{\mathfrak{#1}}
\def\mc#1{\mathcal{#1}}
\def\mb#1{\mathbb{#1}}
\def\tx#1{\textrm{#1}}
\def\tb#1{\textbf{#1}}

\def\R{\mathbb{R}}

\def\C{\mathbb{C}}
\def\Q{\mathbb{Q}}

\def\Z{\mathbb{Z}}

\def\lmod{\backslash}

\def\ol#1{\overline{#1}}

\def\hat{\widehat}

\def\rw{\rightarrow}
\def\lw{\leftarrow}
\def\from{\leftarrow}

\def\lrw{\longrightarrow}
\def\llw{\longleftarrow}

\def\rrw{\rightrightarrows}
\def\lw{\leftarrow}
\def\sm{\smallsetminus}

\def\<{\langle}
\def\>{\rangle}

\newenvironment{mytitle}
{\begin{center}\large\sc}
{\end{center}}

\newtheorem{thm}{Theorem}[subsection]
\newtheorem{lem}[thm]{Lemma}
\newtheorem{pro}[thm]{Proposition}
\newtheorem{cor}[thm]{Corollary}
\newtheorem{fct}[thm]{Fact}

\newtheorem{cnj}[thm]{Conjecture}
\newtheorem{cnd}[thm]{Condition}

\theoremstyle{definition}
\newtheorem{rem}[thm]{Remark}
\newtheorem{dfn}[thm]{Definition}
\newtheorem{exa}[thm]{Example}

\sectionfont{\center\sc\normalsize}
\subsectionfont{\bf\normalsize}

\numberwithin{equation}{subsection}

\newlength{\sumcorr}
\def\ssum#1{\setlength{\sumcorr}{(\widthof{$\displaystyle\sum_{#1}$}-\widthof{$\displaystyle\sum$})/2} \hspace{-\sumcorr}\sum_{#1}\hspace{-\sumcorr} }

\begin{document}

\begin{mytitle} On the local Langlands conjectures for disconnected groups \end{mytitle}

\begin{center} Tasho Kaletha \end{center}

\begin{abstract}
We extend the local Langlands conjectures to a certain class of disconnected groups, allowing non-abelian component groups, and recast in this language some aspects of twisted endoscopy. We further introduce normalized twisted transfer factors and a normalized correspondence between an $L$-packet for a disconnected group and the set of representations of the centralizer groups of its Langlands parameter. We prove the first instance of this conjecture, in which the identity component of the (possibly non-abelian) disconnected group is a torus.
\end{abstract}
{\let\thefootnote\relax\footnotetext{T.K. was partially supported by NSF grants DMS-1801687 and DMS-2301507 as well as a Simons Fellowship.}}

\section{Introduction}

Let $F$ be a local field of characteristic zero. The goal of this paper is to extend the refined local Langlands conjecture to the case of disconnected groups. We recall briefly the statement of this conjecture, referring to \cite{KalSimons} for details. Given a connected reductive $F$-group $G'$ there should be a bijection between the set of (equivalence classes of) Langlands parameters $\varphi : L_F \to {^LG'}$ and the set of $L$-packets $\Pi_{\varphi}(G')$. An $L$-packet is a finite set of irreducible admissible representations of $G'(F)$. It is empty if and only if $\varphi$ is non-relevant for $G'$. The $L$-packets are disjoint and exhaust the set of isomorphism classes of irreducible admissible representations of $G'(F)$. To enumerate the constituents of $\Pi_\varphi(G')$ one fixes an inner twisting $\xi : G \to G'$ with $G$ quasi-split and enriches it to a rigid inner form datum $(\xi,z)$. One further fixes a Whittaker datum $\mf{w}$ for $G$. The inner twisting provides an identification of dual groups $\hat G' = \hat G$ and of $L$-groups $^LG'={^LG}$. Let $Z \subset G$ be a finite central subgroup that is sufficiently large to realize $z$. Let $\bar G=G/Z$. The natural quotient map $G \to \bar G$ is an isogeny. Let $\hat{\bar G} \to \hat G$ be the dual isogeny and let $Z(\hat{\bar G})^+$ be the preimage of $Z(\hat G)^\Gamma$. The element $z$ provides a character $[z] : \pi_0(Z(\hat{\bar G})^+) \to \C^\times$. When $F$ is $p$-adic the character $[z]$ determines the equivalence class of the rigid inner twist $(G',\xi,z)$ uniquely. When $F=\R$ multiple equivalence classes of rigid inner twists may lead to the same character $[z]$, and they are related by $H^1(\R,G'_\tx{sc})$. Let $S_\varphi \subset \hat G$  be the centralizer of the image of $\varphi$ and let $S_\varphi^+$ be its preimage in $\hat{\bar G}$. Let $\tx{Irr}(\pi_0(S_\varphi^+),[z])$ be the set of isomorphism classes of those irreducible representations of the finite group $\pi_0(S_\varphi^+)$ whose restriction to $\pi_0(Z(\hat{\bar G})^+)$ is $[z]$-isotypic. There should be a map $\Pi_\varphi(G') \to \tx{Irr}(\pi_0(S_\varphi^+),[z])$. In the $p$-adic case it should be bijective. In the case $F=\R$ the map should become bijective if one replaces $\Pi_\varphi(G')$ with the disjoint union over all rigid inner twists giving rise to the same character $[z]$. In all cases, this map depends on the choice of $\mf{w}$ and the rigid inner twist data. That same data provides a normalization $\Delta[\mf{\dot e},\mf{z},\mf{w},(\xi,z)]$ of the Langlands--Shelstad transfer factor for each refined endoscopic datum $\mf{\dot e}$ for $G$ and $z$-pair $\mf{z}$ for it. The map $\Pi_\varphi(G') \to \tx{Irr}(\pi_0(S_\varphi^+),[z])$ is expected to satisfy the endoscopic character identities with respect to this normalization of the transfer factor when the parameter $\varphi$ is tempered. More precisely, if $\pi \mapsto \rho_\pi$ is the above map, then a semi-simple element $\dot s \in S_\varphi^+$ leads to the virtual character $\Theta_\varphi^{\dot s}=e(G')\sum_{\pi \in \Pi_\varphi(G')} \tx{tr}\rho_\pi(\dot s)\Theta_\pi$ of $G'(F)$. At the same time the connected centralizer $\hat H$ of the image $s \in \hat G$ of $\dot s$ and the parameter $\varphi$ lead to a quasi-split group $H$ and a parameter $\varphi^\mf{z}$ for its cover $H^\mf{z}$ that is part of the $z$-pair, hence to a similar virtual character $S\Theta_{\varphi^\mf{z}}=\sum_{\pi^\mf{z} \in \Pi_{\varphi^\mf{z}}(H^\mf{z})} \tx{dim}\rho_{\pi^\mf{z}}\Theta_{\pi^\mf{z}}$ on $H^\mf{z}(F)$. The transfer factor $\Delta[\mf{\dot e},\mf{z},\mf{w},(\xi,z)]$ gives rise to a correspondence of functions $f \leftrightarrow f^\mf{z}$ between functions on $G'(F)$ and functions of $H^\mf{z}(F)$ and the expected character identity is $\Theta_\varphi^{\dot s}(f)=S\Theta_{\varphi^\mf{z}}(f^\mf{z})$. A suitable generalization is supposed to hold in the non-tempered case once $\varphi$ has been replaced by an Arthur parameter. 

In this paper we extend these conjectures to certain disconnected algebraic groups whose identity component is reductive. Motivation for this comes on the one hand from the natural occurrence of disconnected groups in number theoretic contexts, most notably the orthogonal groups, and on the other hand from the natural occurrence of disconnected groups in representation theoretic contexts, for example by taking centralizers of semi-simple elements. In fact, disconnected groups appear in the classification of tempered representations of connected reductive groups. If $M'$ is a Levi subgroup of the connected reductive group $G'$ and $\sigma$ is a square-integrable representation of $M'(F)$, the subgroup of $G'(F)$ that normalizes $M'$ and stabilizes the isomorphism class of $\sigma$ plays an important role in the decomposition into irreducible pieces of the parabolic induction of $\sigma$. In order to properly normalize the intertwining operators needed to decompose this parabolic induction one is led to study the representation theory of disconnected groups of this form. This leads to a normalized version of Arthur's local intertwining relation \cite[\S7]{ArtUARC}, \cite[\S2.4]{Art13}. We will present this in a forthcoming paper as an application of the results of the current paper.

The class of disconnected groups we consider in this paper are those affine algebraic $F$-groups $\tilde G'$ whose identity component $G'$ is reductive, and for which there exists an isomorphism $\tilde G'_{\bar F} \cong G'_{\bar F} \rtimes A$ over the algebraic closure $\bar F$ of $F$, where $A$ is a finite (possibly non-abelian) group of automorphisms of $G'$ that preserves a $\bar F$-pinning. The second condition is automatically fulfilled if $G'$ is adjoint, but in general it does restrict the class of disconnected groups we are considering. 

The possible forms of such groups $\tilde G'$ can be classified cohomologically in a manner similar to the connected case. In the connected case the classification has two steps -- one first classifies quasi-split connected reductive groups by means of based root data, and then inner forms in terms of Galois cohomology. In the disconnected case the classification has three steps -- one first classifies quasi-split disconnected groups, then inner forms, and then (what we have called) ``translation forms''. 

A quasi-split disconnected group is of the form $G \rtimes A$, where $G$ is a quasi-split connected reductive $F$-group, and $A$ is a subgroup of its automorphism group that fixes an $F$-pinning. We have $[G \rtimes A](F) = G(F) \rtimes A(F)$. If we are only interested in the group of $F$-rational points, we may replace $A$ with the constant finite $F$-group scheme $A(F)$. 

An inner form of $G \rtimes A$ is obtained by twisting the entire disconnected group $G \rtimes A$ via elements of $Z^1(F,G/Z(G)^A)$. A translation form is obtained by twisting via elements of $Z^1(F,Z^1(A,Z(G)))$. These two twisting steps can be performed in either order. While in the quasi-split case the split exact sequence
\[ 1 \to G \to G \rtimes A \to A \to 1 \]
remains exact on $F$-points and retains a canonical splitting, after inner twisting or translational twisting neither of these statements is true in general. More precisely, given $\bar z \in Z^1(F,G/Z(G)^A)$ there is a natural subgroup $A^{[\bar z]} \subset A$ so that if $\tilde G_{\bar z}$ is the corresponding inner form of the quasi-split group $\tilde G=G \rtimes A$, then the sequence
\[ 1 \to G_{\bar z}(F) \to \tilde G_{\bar z}(F) \to A^{[\bar z]} \to 1 \]
is exact, but it is not equipped with a natural splitting even if it is split. A similar remark applies to translation forms. In this paper we extend the formulation of the refined local Langlands correspondence to inner forms of quasi-split disconnected groups, leaving the treatment of translation forms, as well as the removal of the condition $(\tilde G')_{\bar F} \cong (G' \rtimes A)_{\bar F}$, to a future paper.

There are multiple questions one must answer when attempting to extend the Langlands conjectures to the disconnected setting: What will be the dual group, or the $L$-group, of a disconnected group? What will be the concept of a Langlands parameter and of its centralizer? What are endoscopic groups? What are transfer factors and how does one normalize them?

We hasten to say that we do not perform any non-trivial harmonic analysis in this paper. Instead, we use the already established framework of twisted endoscopy and the fundamental results of Langlands, Shelstad, Kottwitz, Arthur, Waldspurger, Ngo, and others. Part of this paper consists of introducing a slightly different language for this theory. We hope that this language will be beneficial for some applications. One advantage it provides is that the statements of the conjectures for disconnected groups become formally very similar to the statements for connected groups. Another advantage it provides is in organizing multiple automorphisms of a connected reductive group and keeping track of their interaction. This allows for more natural normalizations. We note further that our language does not encompass the full generality of twisted endoscopy, even if we restrict attention to a cyclic component group. Indeed, we do not consider a character $\omega : G(F) \to \C^\times$, and the automorphisms of $G_{\bar z}$ we obtain from elements of $\tilde G_{\bar z}(F)$ are not as general as the theory of twisted endoscopy allows.

With this in mind, the answers we give to the above questions are the following. Consider a quasi-split disconnected reductive group $\tilde G = G \rtimes A$ and an inner form $\tilde G_{\bar z}$ corresponding to an element $\bar z \in Z^1(F,G/Z(G)^A)$. The groups $\tilde G$ and $\tilde G_{\bar z}$ will have the same set of Langlands parameters, and this is the set of Langlands parameters for the connected group $G$ (and its inner form $G_{\bar z}$). Thus, the notion of Langlands parameters remains unchanged when we pass from the connected group $G$ to the disconnected group $\tilde G$. What changes is the notion of equivalence. Two parameters for $G_{\bar z}$ are considered equivalent if they are conjugate under $\hat G$. We declare two parameters for $\tilde G_{\bar z}$ to be equivalent if they are conjugate under $\hat G \rtimes A^{[\bar z]}$. Note that the notion of equivalence depends on the inner form $\bar z$ being considered, in contrast to the case of connected groups.

The new notion of equivalence leads to a new notion of the centralizer group $S_{\varphi}$ of a Langlands parameter $\varphi : L_F \to {^LG}$. Indeed, $S_\varphi$ can be viewed as the group of self-equivalences of $\varphi$. In the disconnected case we now obtain $\tilde S_\varphi^{[\bar z]}$ as the group of self-equivalences of $\varphi$ in the new sense of equivalence. In other words, $\tilde S_\varphi^{[\bar z]}$ is the centralizer of $\varphi$ in the group $\hat G \rtimes A^{[\bar z]}$. We obtain the exact sequence
\[ 1 \to S_\varphi \to \tilde S_\varphi^{[\bar z]} \to A^{[\varphi],[\bar z]} \to 1, \]
where $A^{[\varphi],[\bar z]} = A^{[\varphi]} \cap A^{[\bar z]}$ is the stabilizer in $A^{[\bar z]}$ of the $\hat G$-conjugacy class of $\varphi$. 

Let $\Pi_{\varphi}(\tilde G_{\bar z})$ denote the set of irreducible admissible representations of $\tilde G_{\bar z}(F)$ whose restriction to $G_{\bar z}(F)$ intersects $\Pi_\varphi(G_{\bar z})$. We think of $\Pi_{\varphi}(\tilde G_{\bar z})$ as the $L$-packet for the disconnected group $\tilde G_{\bar z}(F)$ associated to the parameter $\varphi$. To enumerate its members, choose a lift $z \in Z^1(u \to W,Z \to G)$ of $\bar z$, where $Z \subset Z(G)^A$ is a sufficiently large finite subgroup, thereby realizing $\tilde G_{\bar z}$ as a rigid inner form of $\tilde G$. The abstractly described subgroup $A^{[\bar z]}$ of $A$ can now be concretely described as the stabilizer $A^{[z]}$ of the cohomology class of $z$ for the action of $A$. Choose an $A$-admissible Whittaker datum for $G$ (see \S\ref{sub:awhit}). As above we obtain from $z$ a character $[z] : \pi_0(Z(\hat{\bar G})^+) \to \C^\times$. Let $\tilde S_\varphi^{+,[z]}$ be the preimage of $\tilde S_\varphi^{[z]}$ in $\hat{\bar G} \rtimes A^{[z]}$. This group surjects onto $A^{[\varphi],[z]}$ and we have the exact sequence 
\[ 1 \to S_\varphi^+ \to \tilde S_\varphi^{+,[z]} \to A^{[\varphi],[z]} \to 1. \]
Then there should be a map
\[ \Pi_{\varphi}(\tilde G_{\bar z}) \to \tx{Irr}(\pi_0(\tilde S_\varphi^{+,[z]}),[z]), \]
which is again expected to be bijective in the $p$-adic case, and become bijective in the real case once its target has been replaced by a suitable disjoint union. As in the connected case, this map should lead to character identities with respect to a normalized transfer factor $\Delta[\mf{\dot e},\mf{z},\mf{w},(\xi,z)]$. More precisely, a semi-simple element $\dot{\tilde s} \in \tilde S_\varphi^+$ leads to the virtual character $\Theta_\varphi^{\dot{\tilde s}}=e(G_{\bar z})\sum_{\tilde\pi \in \Pi_\varphi(\tilde G_{\bar z})} \tx{tr}\rho_{\tilde\pi}(\dot{\tilde s})\Theta_{\tilde \pi}$ of $\tilde G_{\bar z}(F)$. The connected centralizer $\hat H$ of the image $\tilde s \in \hat G \rtimes A$ of $\dot{\tilde s}$ and the parameter $\varphi$ lead to a quasi-split connected reductive group $H$ and a parameter $\varphi^\mf{z}$ for its cover $H^\mf{z}$, hence to a similar virtual character $S\Theta_{\varphi^\mf{z}}=\sum_{\pi^\mf{z} \in \Pi_{\varphi^\mf{z}}(H^\mf{z})} \tx{dim}\rho_{\pi^\mf{z}}\Theta_{\pi^\mf{z}}$ on $H^\mf{z}(F)$. The transfer factor $\Delta[\mf{\dot e},\mf{z},\mf{w},(\xi,z)]$ gives rise to a correspondence of functions $\tilde f \leftrightarrow \tilde f^\mf{z}$ between functions on $\tilde G_{\bar z}(F)$ and functions of $H^\mf{z}(F)$ and the expected character identity is $\Theta_\varphi^{\dot{\tilde s}}(\tilde f)=S\Theta_{\varphi^\mf{z}}(\tilde f^\mf{z})$. Note the strong similarity with the connected case.

The normalization of the transfer factor is one of the main results of this paper. The relative transfer factor for twisted endoscopy was introduced in \cite{KS99} and some adjustments were later made in \cite{KS12}. It is a function that assigns a complex number to two pairs of elements $(\gamma,\tilde\delta)$, where $\gamma$ is a sufficiently regular semi-simple element of $H^\mf{z}(F)$, and $\tilde\delta$ is a strongly regular semi-simple element of $\tilde G_{\bar z}(F)$ that lies in a fixed coset determined by the image of $\tilde s$ in $A$. If one fixes arbitrarily a pair $(\gamma,\tilde\delta)$, then one obtains from this a function of just one such pair, called an absolute transfer factor. But the arbitrary choice means that this function is well-defined only up to multiplication by a non-zero complex scalar. A specific normalization useful for applications was given in \cite[\S5.3]{KS99} for quasi-split twisted groups. In this paper we provide a normalization for all (rigid) inner forms of quasi-split twisted groups. We call this factor $\Delta_{KS}$. By a simple averaging procedure we obtain from it the transfer factor $\Delta[\mf{\dot e},\mf{z},\mf{w},(\xi,z)]$ used in the above paragraph, which may now be supported on multiple cosets of $G_{\bar z}(F)$ in $\tilde G_{\bar z}(F)$.

The normalization of $\Delta_{KS}$ comes down to a definition of an absolute term $\Delta_{III}^\tx{new}$ that replaces the relative term $\Delta_{III}$ constructed in \cite[\S4.4]{KS99}. The construction we offer here is shorter and simpler than the one of loc. cit. for two reasons. First, our setting ensures that the class $\mathbf{z}$ of \cite[Lemma 3.1.A(3)]{KS99} is trivial. This implies that the transfer of twisted classes between the twisted group and its quasi-split form is defined over $F$, and that the rational structure of the endoscopic group $H$ does not need a shift. Second, we define an absolute invariant $\tx{inv}(\gamma^\mf{z},\tilde\delta)$ that measures the relative position of a related pair $(\gamma^\mf{z},\tilde\delta)$, thus avoiding the complications caused by dealing with two related pairs simultaneously. The construction of the invariant involves a blend of the techniques from \cite{KS99} and \cite{KalRI}, and an interpretation of conjugacy classes in inner form of disconnected groups in terms of a certain non-abelian cohomology set $H^1(F,G \rrw G)$ studied here. 

Besides stating the conjectures, we prove a number of reduction results in this paper. We show how the conjectures for disconnected groups can be reduced to the conjectures for connected groups, a conjecture on the compatibility of the conjectures for connected groups with automorphisms, and a certain amplification of the endoscopic character identity conjecture in twisted endoscopy. We also discuss various functorial constructions, such as restriction and induction of component groups.

Finally, we prove our conjectures in the special case when the identity component is a torus. In this case, the representation theory of the identity component essentially disappears and one can clearly see the additional information present in the consideration of disconnectedness. The core of the proof consists of showing that two group extensions, produced from the same data but one in terms of $G$ and the other in terms of $\hat G$, are canonically isomorphic. 

To illustrate the point let us discuss the case of pure inner forms. Consider a torus $T$ and a finite group of $F$-automorphisms $A$. Then $\tilde T=T \rtimes A$ is a quasi-split disconnected group in our sense. Let $z \in Z^1(F,T)$ and let $\tilde T_z$ be the corresponding pure inner form. Of course the identity components of $\tilde T$ and $\tilde T_z$ are canonically identified, but the disconnected groups $\tilde T(F)=T(F) \rtimes A$ and $\tilde T_z(F)$ are not. Let $\varphi : W_F \to {^LT}$ be a Langlands parameter and let $[\varphi] : T(F) \to \C^\times$ be the corresponding character. Let $A^{[\varphi],[z]}$ be the subgroup of $A$ that fixes both the $\hat T$-conjugacy class of $\varphi$ and the cohomology class of $z$. Simple arguments reduce the problem to the case $A=A^{[\varphi],[z]}$. The $L$-packet $\Pi_\varphi(\tilde T_z)$ consists of the irreducible representations of (the usually non-abelian) group $\tilde T_z(F)$ whose restriction to $T(F)$ is $[\varphi]$-isotypic; a set we can also call $\tx{Irr}(\tilde T_z(F),[\varphi])$. The set $\tx{Irr}(\tilde S_{\varphi},[z])$ consists of the irreducible representations of $\tilde S_\varphi$ whose restriction to $\hat T^\Gamma$ is $[z]$-isotypic (note we do not need the covers $\tilde S_\varphi^+$ and $[\hat{\bar T}]^+$ since we are using a pure inner form). Therefore we are led to consider the following two push-out diagrams
\[ \xymatrix{
	1\ar[r]&T(F)\ar[r]\ar[d]^{[\varphi]}&\tilde T_z(F)\ar[r]&A\ar[r]&1\\
	&\C^\times
}
\]
and
\[ \xymatrix{
	1\ar[r]&\hat T^\Gamma\ar[r]\ar[d]^{[z]}&\tilde S_\varphi\ar[r]&A\ar[r]&1\\
	&\C^\times
}
\]
Both push-outs are central extensions of $A$ by $\C^\times$. The $\tx{id}$-isotypic irreducible representations of the first extension are in canonical bijection with $\Pi_\varphi(\tilde T_z)$, while those of the second extension are in canonical bijection with $\tx{Irr}(\tilde S_{\varphi},[z])$. The conjecture about the internal structure of $L$-packets requires us to show that these two extensions are canonically isomorphic. There appears to be no a-priori reason why these extensions should even be isomorphic, let alone canonically. But we are able to produce a canonical isomorphism. We then show that this isomorphism satisfies the endoscopic character identities with respect to the normalized transfer factor.

Further evidence for the conjectures put forth here appears in \cite{KM}, where these conjectures are proved in the case $F=\R$ and $\varphi$ discrete.

We now describe the contents of the paper. In \S\ref{sec:disc_grps} we discuss basic results about disconnected groups, such as the classification of their forms in \S\ref{sub:disc_forms}, focusing on inner forms in \S\ref{sub:inner}. In \S\ref{sub:norms} we recall facts about twisted conjugacy classes and norms from \cite{KS99} and adapt them to our present language. In \S\ref{sub:awhit} we discuss Whittaker data invariant under $A$. 

The next two sections -- \S\ref{sec:pure} and \S\ref{sec:rigid} -- contain the statements of the refined local Langlands conjecture in the settings of pure respectively rigid inner forms. We have decided to present these cases separately, rather than only dealing with the general case of rigid inner forms, because we feel that the setting of pure inner forms illustrates more clearly the ideas behind conjugacy classes, relative positions, and invariants, as well as the structure of the conjecture. The more general case of rigid inner forms follows the same structure and ideas, but combines them with a technical cohomological discussion. In \S\ref{sec:pure} we first discuss the concept of rational conjugacy classes across pure inner forms, their norms, and the associated invariants. These are based on the non-abelian cohomology set $H^1(F,G \rrw G)$. The constructions are ultimately the same as those of \cite{KS99}, but our language is slightly different and our situation is more specialized, which makes the arguments simpler and shorter. For this reason we have given them in full detail in the hope that this would be helpful to the reader. In \S\ref{sub:pure_llc} we state the first part of the refined local Langlands conjecture -- the correspondence between parameters and packets and the internal structure of packets. We then recall the notion of twisted endoscopic data from \cite{KS99}. Our definitions are in fact slightly different, both for data and for their isomorphisms. This difference is very mild; it ensures that absolute transfer factors are invariant under isomorphisms. We then turn to the normalization of transfer factors in the setting of pure inner forms. In \S\ref{sub:pure_tf} we explain how the factor $\Delta[\mf{\dot e},\mf{z},\mf{w},(\xi,z)]$ is related to the twisted factor $\Delta_{KS}$. The normalization of $\Delta_{KS}$ is done in \S\ref{sub:pure_tf1} and \S\ref{sub:pure_tf2}. We have again split the exposition in the hope that this will make the construction most transparent, by first treating the less technical set-up when a $z$-pair is not needed, and then the more general case when it is. In \S\ref{sub:trans} we summarize the fundamental results of Langlands, Shelstad, Kottwitz, Arthur, Waldspurger, and Ngo, on endoscopic transfer of functions. These results allow us to state the second part of the refined local Langlands conjecture -- the character identities -- in \S\ref{sub:pure_charid}.

The treatment of rigid inner forms in \S\ref{sec:rigid} requires the blending of the hypercohomology techniques of \cite[A.3]{KS99} and the Galois gerbes of \cite{KalRI}. This is done in the first two subsections. The rest of the section consists of slight generalizations of material of \S\ref{sec:pure}, and we allow ourselves to be more brief. The refined local Langlands conjecture for rigid inner forms is stated in \S\ref{sub:llc_rigid}.

In \S\ref{sec:change_whit} we discuss how the parameterization of the internal structure of $L$-packets depends on the choice of Whittaker datum. This is the disconnected analog of the corresponding results from \cite{KalGen}. We introduce the concept of an $A$-admissible Whittaker datum, which is stronger than that of an $A$-stable Whittaker datum.

In \S\ref{sec:change_comp} we discuss how the conjectures change when we change the component group of the disconnected group. The simplest possible change is passing to a subgroup of the component group. It is discussed in \S\ref{sub:comp_rest}. In \S\ref{sub:comp_coset} we discuss how the conjectures for disconnected groups can be related to those for connected groups and twisted endoscopy. This clarifies the information that the disconnected case caries beyond the twisted case. In \S\ref{sub:ind} we study an operation dual to restriction, which may be called induction. If $G$ is a connected reductive group on which a finite group of automorphisms $A$ operates, and $B$ is a group containing $A$, one can form the connected reductive group $H=\tx{Ind}_A^B G$, on which $B$ operates. We discuss how the conjectures for inner forms of $G \rtimes A$ imply those for inner forms of $H \rtimes B$. The discussion here is elementary, but unfortunately rather long and technical. We have included it because the process of induction appears quite often in applications.

In \S\ref{sec:tori} we prove the conjectures made here in the special case of tori. We close the paper is three short appendices. In \S\ref{app:func} we formulate a conjecture about the compatibility of the refined local Langlands correspondence for connected groups with automorphisms. This conjecture is undoubtedly well-known to experts, but we have not been able to locate a reference. In \S\ref{app:weil} we discuss automorphisms of reductive groups that arise via Weil-restriction. In \S\ref{app:projchar} we review orthogonality relations for irreducible projective representations of finite groups, which are used in the proofs in \S\ref{sec:tori}. In \S\ref{app:fold} we review the basic tool of folding of root systems, which we use in \S\ref{app:g-ga} to discuss the relationship between a quasi-split connected reductive group $G$ and the subgroup $G^A$ of points fixed under a finite group $A$ of automorphisms of $G$ that preserve a rational pinning.

\ \\

\noindent\textbf{Acknowledgements:} This paper presents the development of ideas communicated by Robert Kottwitz to the author over 15 years ago, while the author was still a graduate student. The ideas of what a ``quasi-split disconnected group'' should be, the construction of the centralizer $\tilde S_\varphi$, and the usage of relative non-abelian cohomology $H^1(F,G \rrw G)$ to describe conjugacy classes, are due to him. The author expresses his gratitude and admiration to Kottwitz for sharing these beautiful ideas. The author also thanks Wee Teck Gan for correcting an inaccuracy in an earlier draft.

\tableofcontents

\section{Notation}

Throughout the paper, $F$ will denote a local field of characteristic zero, $\Gamma$ the absolute Galois group with respect to a fixed algebraic closure $\bar F$ of $F$, and $W_F$ the Weil group. We will write $Z^1(\Gamma,G)$ for the set of continuous 1-cocycles of $\Gamma$ valued in the discrete group $G(\bar F)$, $H^1(\Gamma,G)$ for the set of cohomology classes of such cocycles, $\tilde Z^1(\Gamma,G)$ for the set of continuous sections $\tilde z : \Gamma \to G(\bar F) \rtimes \Gamma$ of the natural projection, and $\tilde H^1(\Gamma,G)$ for the set of $G(\bar F)$-conjugacy classes of such sections. The assignment $z \mapsto \tilde z$ defined by $\tilde z(\sigma)=z(\sigma)\rtimes\sigma$ is a bijection $Z^1(\Gamma,G) \to \tilde Z^1(\Gamma,G)$ that descends to a bijection $H^1(\Gamma,G) \to \tilde H^1(\Gamma,G)$. We will switch freely in the notation between $\tilde z$ and $z$.

Given an automorphism $a$ of $G$ and $\delta \in G$ we have the element $\tilde\delta=\delta \rtimes a \in G \rtimes a \subset G \rtimes \<a\>$. The assignment $\tilde\delta \mapsto \delta$ translates the action of $G$ on $G \rtimes a$ by conjugation to the action of $G$ on itself by $a$-twisted conjugation. When $a$ is understood from the context, we will switch freely between $\tilde\delta$ and $\delta$.

\section{Disconnected groups} \label{sec:disc_grps}

\subsection{Split disconnected groups and their forms} \label{sub:disc_forms}

Let $F$ be a local field of characteristic zero. We denote by $W_F$ the Weil group of $F$ and by $\Gamma$ the absolute Galois $\tx{Gal}(\ol{F}/F)$. In this paper we will study affine algebraic groups defined over $F$ whose connected component is reductive. We will call such groups ``disconnected reductive'' for short. We will however restrict our attention to those disconnected reductive groups $\tilde G$ that satisfy the following condition:

\begin{cnd} \label{cnd:main}
	There exists an isomorphism defined over $\ol{F}$
	\[ \tilde G \rw G \rtimes A \]
	where $G$ is a connected reductive group, $A$ is a finite group, and $A$ acts on $G$ by automorphisms which preserve a fixed $\bar F$-pinning.
\end{cnd}

Not all disconnected reductive groups satisfy this condition. The most basic counterexample is the normalizer of the torus in $\tx{SL}_2$. On the other hand, this condition is satisfied by many naturally occurring disconnected reductive groups, including the orthogonal groups as well as the groups involved in the classification of tempered representations of connected reductive groups. The latter are among the main motivations for our study.

Just as in the connected case, one can classify the possible $\tilde G$ that satisfy the above condition in terms of root data and Galois cohomology. First, one can consider a split connected reductive group $G$ defined over $F$ and a finite group $A$, interpreted as a constant groups scheme over $F$, and let $A$ act on $G$ and preserve a fixed $F$-pinning. Then $G \rtimes A$ is a special case of a disconnected reductive group defined over $F$ and we will call it ``split''. This adjective carries for us a double meaning -- not only is the connected component $G$ split, but the extension $G \rtimes A$ is also split. It is clear that the split disconnected reductive group $G \rtimes A$ is classified by the root datum of $G$ and the action of $A$ on this root datum.

Now fix an isomorphism $\iota : \tilde G \rw G \rtimes A$ as in Condition \ref{cnd:main}. We may assume without loss of generality that $G$ is split and that $A$ preserves an $F$-pinning. Then
\[ \Gamma \rw \tx{Aut}(G \rtimes A),\quad \sigma \mapsto \iota^{-1}\sigma(\iota) \]
is a 1-cocycle and the isomorphism class of $\tilde G$ is determined by the split form $G \rtimes A$ and the cohomology class of this 1-cocycle.

In order to understand this cohomology better, we look more closely at $\tx{Aut}(G \rtimes A)$. The action of $G$ on $G \rtimes A$ by conjugation factors through $G/Z(G)^A$, where $Z(G)^A$ is the group of fixed points for the action of $A$ on the center of $G$. Thus $G/Z(G)^A$ is a subgroup of $\tx{Aut}(G \rtimes A)$, and in fact this subgroup is normal, because $G$ is a characteristic subgroup of $G \rtimes A$ (being the neutral connected component).

Consider now the group $Z^1(A,Z(G))$ of 1-cocycles of $A$ valued in $Z(G)$. The map sending $z \in Z^1(A,Z(G))$ to the automorphism $g \rtimes a \mapsto gz(a) \rtimes a$ of $G \rtimes A$ embeds $Z^1(A,Z(G))$ as a normal subgroup of $\tx{Aut}(G \rtimes A)$. The two normal subgroups $G/Z(G)^A$ and $Z^1(A,Z(G))$ of $\tx{Aut}(G \rtimes A)$ commute. Their intersection can be described as
the subgroup $Z(G)/Z(G)^A$ of $G/Z(G)^A$, or equivalently its isomorphic image $B^1(A,Z(G)) \subset Z^1(A,Z(G))$ under the differential $z \mapsto z \cdot a(z)^{-1}$.

We thus have the normal subgroup $G/Z(G)^A \cdot Z^1(A,Z(G))$ of $\tx{Aut}(G \rtimes A)$. It has a complement. In order to specify it, we use the pinning of $G$ defined over $F$ and preserved by the action of $A$ and let $\tx{Aut}_\tx{pin}(G \rtimes A)$ be those automorphisms of $G \rtimes A$ whose restriction to $G$ preserves the pinning \emph{and} which preserve the subgroup $1 \rtimes A$ of $G \rtimes A$. We conclude
\[ \tx{Aut}(G \rtimes A) = (G/Z(G)^A \cdot Z^1(A,Z(G))) \rtimes \tx{Aut}_\tx{pin}(G \rtimes A). \]
This means that any form of $G \rtimes A$ can be obtained by a 3-step process: First, using an element of $Z^1(\Gamma,\tx{Aut}_\tx{pin}(G \rtimes A))$ one twists the rational structure of $G \rtimes A$. The result is again a group of the form $G \rtimes A$, where now $G$ is a quasi-split connected reductive group, $A$ is a (not necessarily constant) finite group scheme over $F$, and $A$ acts on $G$ again by automorphisms that preserve a fixed $F$-pinning. We shall call such disconnected reductive groups ``quasi-split''. Second, using an element $z \in Z^1(\Gamma,G/Z(G)^A)$ we twist the quasi-split group $\tilde G = G \rtimes A$ and obtain an ``inner form'' $\tilde G_z$ of it. Finally, we twist $\tilde G_z$ by an element of $Z^1(\Gamma,Z^1(A,Z(G)))$ to obtain a ``translation form'' of $\tilde G_z$.

\subsection{Inner forms} \label{sub:inner}

Let $G$ be a connected reductive group, defined and quasi-split over $F$. Let $(T,B,\{X_\alpha\})$ be an $F$-pinning of $G$ and let $A$ be a finite group that acts on $G$ by pinned automorphisms. Assume given an action of $\Gamma$ on $A$ so that for $\sigma \in \Gamma$ we have $\sigma(a(g)) = \sigma(a)(\sigma(g))$. Thus $\tilde G = G \rtimes A$ is a quasi-split disconnected group in the sense of the previous subsection.

We have an exact sequence of algebraic groups
\[ 1 \rw G \rw \tilde G \rw A \rw 1 \]
which leads to an exact sequence of topological groups
\[ 1 \rw G(F) \rw \tilde G(F) \rw A^\Gamma \rw 1. \]
Both of these extensions are split and come equipped with a splitting.

The group $G$ acts on $\tilde G$ by conjugation and this action preserves the decomposition $\tilde G = \bigsqcup_{a \in A} G \rtimes a$ of $\tilde G$ into left $G$-cosets. In addition, the group $\tilde G$ acts on itself by conjugation, and this action preserves the group $G$.

Writing $\bar G=G/Z(G)^A$, the group $\bar G \rtimes \Gamma$ acts on the group $\tilde G$, with $\bar g \rtimes \sigma$ acting as the automorphism $\tx{Ad}(\bar g)\circ\sigma$. Given $\bar z \in Z^1(\Gamma,\bar G)$, we denote by $\tilde G_{\bar z}$ the algebraic group defined over $F$ which satisfies $\tilde G_{\bar z}(\ol{F}) = \tilde G(\ol{F})$ and where $\Gamma$ acts on $\tilde G_{\bar z}(\ol{F})$ via the homomorphism $\tilde{\bar z} : \Gamma \rw \bar G \rtimes \Gamma$ and the action of $\bar G \rtimes \Gamma$ on $\tilde G(\ol{F})$. We call $\tilde G_{\bar z}$ the \emph{inner form} of $\tilde G$ corresponding to $\bar z$.

We still have the exact sequence of algebraic groups
\[ 1 \rw G_{\bar z} \rw \tilde G_{\bar z} \rw A \rw 1 \]
but the sequence of $F$-points
\[ 1 \rw G_{\bar z}(F) \rw \tilde G_{\bar z}(F) \rw A, \]
need not be exact. The image of the last map lies in $A^\Gamma$ and we denote it by $A^{[\bar z]} \subset A^\Gamma$. We obtain an extension
\[ 1 \rw G_{\bar z}(F) \rw \tilde G_{\bar z}(F) \rw A^{[\bar z]} \to 1. \]
Unlike the case of the quasi-split group $\tilde G$, this extension, even when it is split, does not come equipped with a distinguished splitting.

The action of $G_{\bar z}(F)$ on $\tilde G_{\bar z}$ preserves the subset $\tilde G_{\bar z}(F)$. The action of $\tilde G_{\bar z}(F)$ on $G_{\bar z}$ is realized by automorphisms defined over $F$ and in particular preserves the subset $G_{\bar z}(F)$.

If we replace $A$ by $A^\Gamma$ the group $\tilde G_{\bar z}(F)$ remains unchanged. Since we shall ultimately be interested in the topological group $\tilde G_{\bar z}(F)$ and its representations, we will assume from now on that the action of $\Gamma$ on $A$ is trivial. In other words, we will treat the group $A$ as a constant group scheme.

\subsection{Strongly regular semi-simple elements and norms} \label{sub:norms}

We recall some material from \cite{KS99}. An automorphism $\theta$ of $G$ is called quasi-semi-simple if it preserves a Borel pair. A maximal torus that is part of a $\theta$-stable Borel pair is called $\theta$-admissible. The automorphism $\theta$ is furthermore called strongly regular if the fixed point group $G^\theta$ is abelian. For such an automorphism $\theta$, there is a unique $\theta$-admissible maximal torus of $G$, namely $\tx{Cent}(G^\theta,G)$. If $S \subset G$ is a $\theta$\-invariant maximal torus we will write $S_\theta=S/(1-\theta)S$ for the quotient of $\theta$\-coinvariants.

We shall call an element of $G \rtimes A$ (strongly regular) semi-simple, if the automorphism of $G$ it induces by conjugation is (strongly regular) quasi-semi-simple. Clearly these notions are invariant under conjugacy by $G \rtimes A$.

\begin{lem} \label{lem:c1}
\begin{enumerate}
	\item Let $\tilde\delta=\delta \rtimes a \in \tilde G(\bar F)$ be semi-simple. Given an $a$-admissible maximal torus $S \subset G$ there exists $g \in G(\bar F)$ such that $\tilde\delta^*=g^{-1}\tilde\delta g$ belongs to $S(\bar F) \rtimes a$.
	\item Write $\tilde\delta^*=\delta^* \rtimes a$, so that $\delta^* \in S(\bar F)$. Write $\gamma \in S_a(\bar F)$ for the image of $\delta^*$ in the torus $S_a=S/(1-a)S$ of $a$-coinvariants. The set of pairs $(S,\gamma)$ obtained in this way for a fixed $\tilde\delta$ and varying $g$ forms a single $G^{a,\circ}(\bar F)$-conjugacy class.
    \item If $\tilde\delta$ is strongly regular and $(S_1,\gamma_1)$ and $(S_2,\gamma_2)$ are two such pairs, then all $g \in G^{a,\circ}(\bar F)$ such that $\tx{Ad}(g)(S_1,\gamma_1)=(S_2,\gamma_2)$ induce the same isomorphism $\tx{Ad}(g) : S_1 \to S_2$. 
	\item Given $\gamma \in S_a(\bar F)$, the set of $\tilde\delta \in \tilde G(\bar F)$ corresponding to the $G^{a,\circ}(\bar F)$-conjugacy class of $(S,\gamma)$ is a single $G(\bar F)$-conjugacy class.
\end{enumerate}
\end{lem}
\begin{proof}
(1) This is essentially \cite[Lemma 3.2.A]{KS99}. Let $(T_{\tilde\delta},B_{\tilde\delta})$ be a Borel pair normalized by $\tilde\delta$ and let $C$ be a Borel subgroup containing $S$ and normalized by $a$. Let $g \in G(\bar F)$ be such that $\tx{Ad}(g)(S,C)=(T_{\tilde\delta},B_{\tilde\delta})$. Set $\tilde\delta^*=g^{-1}\tilde\delta g$. Then $(S,C)$ is normalized by  both $\tilde\delta^*$ and $a$, so also by $\delta^*:=\tilde\delta^*\cdot a^{-1}$, hence $\delta^* \in S(\bar F)$.

(2) We fix for $i=1,2$ $a$-stable Borel subgroups $C_i$ of $G$ defined over $\bar F$ and containing $S_i$ and elements $g_i \in G(\bar F)$ as in (1) and set $\tilde\delta_i^* = g_i^{-1}\tilde\delta g_i \in S_i(\bar F)\rtimes a$, so that the image of $\delta_i^*$ is $\gamma_i$. Since any two $a$-stable Borel pairs are conjugate under $G^{a,\circ}(\bar F)$, we may modify $g_2$ to assume $S_1=S_2=S$ and $C_1=C_2=C$. Thus $\tilde\delta_1^*$ and $\tilde\delta_2^*$ belong to $S \rtimes a$ and are conjugate by $g:=g_2^{-1}g_1$. It follows that $S^{a,\circ}$ and $\tx{Ad}(g^{-1})S^{a,\circ}$ are maximal tori of $\tx{Cent}(\tilde\delta_1^*,G)^\circ$, hence conjugate by an element $x \in \tx{Cent}(\tilde\delta_1^*,G)^\circ$. Replacing $g_1$ by $g_1x$ we may assume $\tx{Ad}(g^{-1})$ normalizes $S^{a,\circ}$. Then it normalizes $S$ and its image in $\Omega(S,G)$ is $a$-fixed. It is thus representable by an element of $G^{a,\circ}(\bar F)$.

(3) Let $\delta_1^* \in S_1(\bar F)$ and $\delta_2^* \in S_2(\bar F)$ be elements mapping to $\gamma_1$ and $\gamma_2$ and such that $\tilde\delta_1^*$ and $\tilde\delta_2^*$ are $G(\bar F)$-conjugate to $\tilde\delta$. A given $g \in G^{a,\circ}(\bar F)$ with $\tx{Ad}(g)(S_1,\gamma_1)=(S_2,\gamma_2)$ can only be modified to $hg$ for $h \in G^{a,\circ}(\bar F)$ normalizing $S_2$ and fixing $\gamma_2$. Thus there exists $s \in S_2(\bar F)$ such that $\tx{Ad}(sh) \in \tx{Cent}(\tilde\delta_2^*,G)=S_2^a$. It follows that the isomorphism $\tx{Ad}(g) : S_1 \to S_2$ carrying $\gamma_1$ to $\gamma_2$ does not depend on the choice of $g$.

(4) follows immediately from the fact that the set of $\tilde\delta^*=\delta^* \rtimes a \in S(\bar F) \rtimes a$ such that $\delta^*$ maps to $\gamma$ forms a single $S(\bar F)$-conjugacy class.
\end{proof}

\begin{dfn} \label{dfn:norm}
Let $\tilde\delta=\delta \rtimes a \in \tilde G_{\bar z}(F)$ be strongly regular semi-simple. A \emph{norm} of $\tilde\delta$ is a pair $(S,\gamma)$ consisting of a maximal torus $S \subset G$ defined over $F$ and $a$-admissible, and an element $\gamma \in S_a(F)$, such that there exists $g \in G(\bar F)$ with the property that $\tilde\delta^* = g^{-1}\tilde\delta g \in S(\bar F) \rtimes a$ and the image of $\delta^* \in S(\bar F)$ in $S_a(\bar F)$ equals $\gamma$.
\end{dfn}

\begin{rem} \label{rem:norm}
	It is in general not possible to arrange that $\delta^* \in S(F)$. This makes the discussion of transfer factors in the twisted setting more complicated than in the untwisted setting, see Remarks \ref{rem:inv} and \ref{rem:d3}.
\end{rem}

\begin{lem} \label{lem:c2}
Let $\tilde\delta=\delta \rtimes a \in \tilde G_{\bar z}(F)$ be strongly regular semi-simple.
\begin{enumerate}
	\item There exists a norm $(S,\gamma)$ of $\tilde\delta$.
 	\item Let $T_{\tilde\delta}=\tx{Cent}(\tx{Cent}(\tilde\delta,G),G)$. Any $g \in G(\bar F)$ with the property that $\tilde\delta^* = g^{-1}\tilde\delta g \in S(\bar F) \rtimes a$ induces the same isomorphism $\tx{Ad}(g) : T_{\tilde\delta} \to S$ and this isomorphism is defined over $F$.
	\item For any two norms $(S_1,\gamma_1)$ and $(S_2,\gamma_2)$ of $\tilde\delta$ the canonical isomorphism $\tx{Ad}(g) : S_1 \to S_2$, $g \in G^{a,\circ}(\bar F)$, carrying $\gamma_1$ to $\gamma_2$ is defined over $F$.
\end{enumerate}
\end{lem}
\begin{proof}
(1) The arguments for this part are contained in the proofs of \cite[Lemmas 3.3.B,3.3.C]{KS99}. We collect them here for the convenience of the reader. By Lemma \ref{lem:c1} we may find $h \in G(\bar F)$ such that $\tilde\delta^0 := \tx{Ad}(h)^{-1}\tilde\delta \in T(\bar F) \rtimes A$, where we recall that $T$ is the maximal torus that is part of the $A$-invariant $F$-pinning of $G$. Since $\tilde\delta$ is fixed by $\tx{Ad}(z_\sigma)\rtimes\sigma$ for all $\sigma \in \Gamma$, its $G$-conjugacy class is fixed by $\sigma$, and part (4) of Lemma \ref{lem:c1} implies that the $\Omega(T,G)^a$-orbit of the image $\gamma^0 \in T_a(\bar F)$ of $\delta^0$ is $\Gamma$-invariant. Thus for every $\sigma \in \Gamma$ there exists $w_\sigma \in \Omega(T,G)^a$ such that $w_\sigma\sigma(\gamma^0)=\gamma^0$. Since $\tilde\delta$ is strongly regular, no element of $\Omega(T,G)^a$ fixes $\gamma^0$, and hence $w_\sigma$ is determined by $\sigma$. The map $\sigma \mapsto w_\sigma$ is a 1-cocycle. Since $\Omega(T,G)^a=\Omega(T_\tx{sc}^a,G_\tx{sc}^a)$, Steinberg's theorem implies the existence of $g \in G_\tx{sc}^a$ such that $g^{-1}\sigma(g)$ normalizes $T_\tx{sc}^a$ and induces $w_\sigma$. Set $\tilde\delta^*=\tx{Ad}(gh^{-1})\tilde\delta$. Then $S=\tx{Ad}(g)T$ is the unique $\tilde\delta^*$-admissible maximal torus and we have $\delta^* \in S(\bar F)$. Since $T$ is the centralizer in $G$ of $T_\tx{sc}^a$ we see that $g^{-1}\sigma(g)$ normalizes $T$ and this implies that $S$ is defined over $F$. The image $\gamma \in S_a(\bar F)$ of $\tilde\delta^*$ under the projection $S \to S_a$ coincides with the image of $\gamma^0$ under $\tx{Ad}(g) : T_a \to S_a$ and is $\Gamma$-fixed.

(2) The identity $\tx{Ad}(g)\tilde\delta^*=\tilde\delta$ implies that $\tx{Ad}(g)$ induces an isomorphism between the centralizers in $G$ of $\tilde\delta^*$ and $\tilde\delta$, and in turn the centralizers in $G$ of these centralizers, i.e. $S \to T_{\tilde\delta}$. Another $g$ is obtained by right multiplication by an element of $S^{\tilde\delta^*}=S^a \subset S$, hence the isomorphism $\tx{Ad}(g) : S \to T_{\tilde\delta}$ is unchanged. The automorphism $\tx{Ad}(\bar z_\sigma)\circ\sigma \circ \tx{Ad}(g) \circ \sigma^{-1}$ of $G$ maps $\sigma(\tilde\delta^*)$ to $\tilde\delta$. Since $\gamma \in S_a(F)$ we see $(\delta^*)^{-1}\sigma(\delta^*) \in (1-a)S$, hence there exists $s \in S$ such that $\sigma(\tilde\delta^*)=s\tilde\delta^*s^{-1}$. We conclude that $g^{-1}z_\sigma\sigma(g)s$ centralizes $\tilde\delta^*$, where $z_\sigma \in G(\bar F)$ is any lift of $\bar z_\sigma$. Hence $g^{-1}z_\sigma\sigma(g)$ lies in $S$, showing that $\tx{Ad}(g) : S \to T_{\tilde\delta}$ is defined over $F$.

(3) follows from part (3) of Lemma \ref{lem:c1}, since both $\tx{Ad}(g)$ and $\tx{Ad}(\sigma(g))$ map $(S_1,\gamma_1)$ to $(S_2,\gamma_2)$.
\end{proof}

\subsection{$A$-stable and $A$-admissible Whittaker data} \label{sub:awhit}

We continue with a connected reductive group $G$ defined and quasi-split over $F$, and $A$ a finite group of automorphisms that leaves invariant an $F$-pinning of $G$.

Recall that a Whittaker datum is a $G(F)$-conjugacy class of pairs $(B,\psi)$ consisting of a Borel $F$-subgroup $B \subset G$ and a generic unitary character $\psi : U(F) \to \C^\times$, where $U \subset B$ is the unipotent radical, and ``generic'' means that $\psi$ is non-trivial on each relative simple root subgroup of $U(F)$.

\begin{dfn} A Whittaker datum $\mf{w}$ is called
	\begin{enumerate}
		\item \emph{$A$-stable}, if the $G(F)$-conjugacy class $\mf{w}$ is stable under the action of $A$.
  		\item \emph{$A$-admissible}, if $G(F)$-conjugacy class $\mf{w}$ contains a pair $(B,\psi)$ that is stable under the action of $A$.
	\end{enumerate}
\end{dfn}

It is clear that an $A$-admissible Whittaker datum is $A$-stable. The converse need not be true, see Example \ref{exa:a-whit}. It turns out that $A$-admissible Whittaker data play an important role. In this subsection we will investigate their properties, as well as the following related notions.

\begin{dfn}
	A regular nilpotent $G(F)$-conjugacy class in $\mf{g}(F)$ is called
	\begin{enumerate}
		\item \emph{$A$-stable}, if it is stable under the action of $A$.
  		\item \emph{$A$-admissible}, if it has an $A$-fixed element, i.e. it intersects $\mf{g}(F)^A$.
	\end{enumerate}
\end{dfn}

\begin{dfn}
	A $G(F)$-conjugacy class of $F$-pinnings is called
	\begin{enumerate}
		\item \emph{$A$-stable}, if it is stable under the action of $A$.
  		\item \emph{$A$-admissible}, if it has an $A$-stable element (i.e. it contains an $A$-stable $F$-pinning).
	\end{enumerate}
\end{dfn}

We now recall two basic constructions of Whittaker data. Let $(T,B,\{X_\alpha\})$ be an $F$-pinning and let $\psi_F : F \to \C^\times$ be a non-trivial unitary character. Let $U$ be the unipotent radical of $B$. Following \cite[\S5.3]{KS99} we obtain a generic unitary character $\psi : U(F) \to \C^\times$, via the following procedure: The fixed pinning induces an isomorphism from the abelianization $U^\tx{ab}=U/[U,U]$ to $\prod_{\alpha\in\Delta}\mb{G}_a$. This isomorphism is defined over $F$ if we let $\Gamma$ act on the product in a way compatible with the action of $\Gamma$ on $\Delta$. The summation map $\prod_{\alpha\in\Delta}\mb{G}_a \to \mb{G}_a$ is then defined over $F$. The generic character $\psi$ is given by the composition
\[ U(F) \to U^\tx{ab}(F) \to \Big(\prod_{\alpha \in \Delta} \mb{G}_a\Big)(F) \to F \to \C^\times. \]

The second construction starts with a regular nilpotent $\bar X \in \tx{Lie}(G)(F)$ and a non-trivial unitary character $\psi_F : \to \C^\times$. Let $\bar B$ be the unique Borel subgroup of $G$ whose Lie algebra contains $\bar X$; it is an $F$-Borel subgroup. Choose a maximal $F$-torus $T$ in $\bar B$ and let $B$ be the $T$-opposite Borel subgroup to $\bar B$. If $\kappa$ denotes the Killing form on $\tx{Lie}(G_\tx{der})(F)$, then there exists a unique generic character $\psi : U(F) \to \C^\times$ such that
\[ \psi(\exp(X))=\psi_F(\kappa(X,\bar X)),\qquad X \in \tx{Lie}(U)(F).\]
Note that, since all choices of $T$ are conjugate under $\bar B(F)$, the $G(F)$-conjugacy class of $[B,\psi]$ depends only on the $G(F)$-conjugacy class of $\bar X$.

The following is elementary, see e.g. \cite[Lemma 3.1.2]{DPR}.
\begin{fct} \label{fct:whit}
	Fix a non-trivial unitary character $\psi_F : F \to \C^\times$. The above constructions produce bijections between 
	\begin{enumerate}
		\item The set of Whittaker data.
  		\item The set of $G(F)$-conjugacy classes of regular nilpotent elements in $\tx{Lie}(G)(F)$.
    	\item The set of $G(F)$-conjugacy classes of $F$-pinnings of $G$.
	\end{enumerate}
	These bijections are equivariant under $\tx{Aut}(G)(F)$. The bijection between (2) and (3) sends a pinning $(T,B,\{X_\alpha\})$ to $\sum X_{-\alpha}$, where $X_{-\alpha}$ is determined by $[X_\alpha,X_{-\alpha}]=H_\alpha$.
\end{fct}

If $(T,B,\{X_\alpha\})$ is stable under $A$, then so is $(B,\psi)$, and the $G(F)$-conjugacy class of $(B,\psi)$ is an $A$-admissible Whittaker datum.

\begin{cor} \label{cor:adm}
	Fix a non-trivial unitary character $\psi_F : F \to \C^\times$. Under the above bijections, the $A$-stable Whittaker data correspond to $A$-stable $G(F)$-conjugacy classes of regular nilpotent elements of $\tx{Lie}(G)(F)$ and to $A$-stable $G(F)$-conjugacy classes of $F$-pinnings of $G$, while the $A$-admissible Whittaker data correspond to $A$-admissible $G(F)$-conjugacy classes of regular nilpotent elements of $\tx{Lie}(G)(F)$ and to $A$-admissible $G(F)$-conjugacy classes of $F$-pinnings of $G$.
\end{cor}
\begin{proof}
	Compatibility with $A$-stability follows from the $A$-equivariance of the bijections. For compatibility with $A$-admissibility, it is clear that $A$-stable pinnings and $A$-stable regular nilpotent elements produce Whittaker data. The correspondence between $A$-stable pinnings are $A$-stable Whittaker data is given in Corollary \ref{cor:pinu}. Given an $A$-stable Whittaker datum $(B,\psi)$, there is a unique regular nilpotent element of $B$ that gives rise to $\psi$, which must therefore be $A$-stable.
\end{proof}

\begin{cor} Any two $A$-admissible Whittaker data are conjugate by $G_\tx{ad}(F)^A=(G^1)_\tx{ad}(F)$.
\end{cor}
\begin{proof}
This follows from Corollary \ref{cor:adm} and Proposition \ref{pro:stein}(9).
\end{proof}

\begin{lem}
	Let $(T,B,\{X_\alpha\})$ be an $A$-stable $F$-pinning and $\psi_F : F \to \C^\times$ a non-trivial unitary character. Let $\psi : U(F) \to \C^\times$ be the corresponding $A$-stable generic character. The map $t \mapsto (B,\psi_t)$, where $\psi_t=\psi\circ \tx{Ad}(t)$, is a bijection between
	\begin{enumerate}
		\item the set 
		\[ [T/Z(G)](F)/[T(F)/Z(G)(F)] \]
		and the set of Whittaker data for $G$;
		\item the set 
        \[ \{t \in [T/Z(G)](F)|\forall a \in A : t^{-1}a(t) \in T(F)/Z(G)(F)\}/[T(F)/Z(G)(F)] \]
		and the set of $A$-stable Whittaker data.
		\item the set 
		\[ [T/Z(G)]^A(F)/[T^A(F)/Z(G)^A(F)] \] 
		and the set of $G(F)$-conjugacy classes of $A$-admissible Whittaker data;
  		
	\end{enumerate}
\end{lem}
\begin{proof}
	(1) This follows from Fact \ref{fct:whit} and the conjugacy of $F$-Borel pairs under $G(F)$.

	(2) From (1) it follows that the $G(F)$-conjugacy class of $(B,\psi_t)$ is $A$-stable if and only if the coset of $t$ in $[T/Z(G)](F)/[T(F)/Z(G)(F)]$ is $A$-stable.

	(3) If $t \in [T/Z(G)]^A(F)$, then $\psi_t$ is $A$-fixed, so $(B,\psi_t)$ represents an $A$-admissible Whittaker datum. Conversely, an $A$-admissible Whittaker datum arises from an $A$-admissible $G(F)$-conjugacy class of $F$-pinnings by Corollary \ref{cor:adm}, and Proposition \ref{pro:stein}(7) allows us to chose an $A$-stable $F$-pinning in this class whose underlying Borel pair is $(T,B)$. This pinning differs from the fixed one by conjugation by a unique element of $[T/Z(G)](F)$, which must then by $A$-fixed.
\end{proof}

\begin{exa} \label{exa:a-whit}
	We will now give an example of an $A$-stable Whittaker datum that is not $A$-admissible.

	Consider $G=\tx{SL}_4$ with the standard pinning $(T,B,\{X_\alpha\})$, where $T$ is the group of diagonal matrices, $B$ is the group of upper triangular matrices, and $X_\alpha$ are the matrices that contain a single non-zero entry, equal to $1$, in the first off-diagonal. Let $A=\<a\>$ with $a$ the non-trivial automorphism preserving this pinning. It is given by the composition of transpose, inverse, and conjugation by the anti-diagonal matrix with entries alternating between $+1$ and $-1$.

	Write $T_\tx{ad}=T/Z(G)$ and $T(F)_\tx{ad}=T(F)/Z(G)(F)$. We will now examine the inclusions
	\begin{equation} \label{eq:exaadm}
		\frac{[T_\tx{ad}]^a(F)}{[T(F)_\tx{ad}]^a} \subset \frac{\{t \in [T_\tx{ad}](F)| t^{-1}a(t) \in T(F)_\tx{ad}\}}{T(F)_\tx{ad}} \subset \frac{T_\tx{ad}(F)}{T(F)_\tx{ad}}.
	\end{equation}
	We have the diagram
	\[ \xymatrix{
		1\ar[r]&Z(G)(F)\ar[r]&T(F)\ar[r]&T_{\tx{ad}}(F)\ar[d]^\cong\ar[r]^-{\tx{det}}&F^\times/F^{\times,4}\ar[r]&1\\
		&&&(F^\times)^3\ar[ur]^\eta
	}\]
	with exact top row, where $F^{\times,4}$ denotes the image of the map $(-)^4 : F^\times \to F^\times$, the vertical isomorphism is given by $(\alpha,\beta,\gamma)$, where $\alpha,\beta,\gamma$ are the simple roots, ordered so that $a(\alpha)=\gamma$ and $a(\beta)=\beta$, and $\eta(x,y,z)=xy^2z^3$. The diagram is $a$-equivariant if let $a$ act on $F^\times$ by inversion and on $(F^\times)^3$ by $a(x,y,z)=(z,y,x)$.

	The right-most group in the chain \eqref{eq:exaadm} is mapped by $\tx{det}$ isomorphically to $F^\times/F^{\times,4}$. The middle group is mapped by $\tx{det}$ to the subgroup of $F^\times/F^{\times,4}$ fixed by inversion. This group is represented by elements $x \in F^\times$ such that $x^2 \in F^{\times,4}$, which is equivalent to $x \in F^{\times,2} \cup -F^{\times,2}$. The left-most group is mapped by $\tx{det}$ to the subgroup $F^{\times,2}/F^{\times,4}$.

	We conclude that the index of the left-most group in the middle group is equal to $1$ if $-1 \in F^{\times,2}$ and equal to $2$ otherwise. Since the left-most group acts simply transitively on the set of $A$-admissible Whittaker data, and the middle group acts simply transitively on the set of $A$-stable Whittaker data, we conclude that, when $-1 \notin F^{\times,2}$, there exists an $A$-stable Whittaker datum that is not $A$-admissible.
\end{exa}

\subsection{$A$-admissible maximal tori}

\begin{dfn}
	A maximal $F$-torus $S \subset G$ is called \emph{$A$-admissible}, if it is $A$-stable and there exists an $A$-stable $\bar F$-Borel subgroup of $G$ containing $S$. 
\end{dfn}

\begin{lem}
	Any two $A$-admissible maximal $F$-tori are conjugate under $G^A(\bar F)$.
\end{lem}
\begin{proof}
	This follows from Proposition \ref{pro:stein}(6) applied to $G_{\bar F}$.
\end{proof}

Recall the definition of the toral invariant $f_{(S,G)} : R(S,G)_\tx{sym} \to \{\pm1\}$ for a maximal $F$-torus $S \subset G$, defined in \cite[\S4]{KalEpi}. Choose any non-zero $X_\alpha \in \mf{g}_\alpha(F_\alpha)$, then
\[ f_{(S,G)}(\alpha) = \kappa_\alpha\Big(\frac{[X_\alpha,\sigma_\alpha(X_\alpha)]}{H_\alpha}\Big), \]
where $F_{\pm\alpha} \subset F_\alpha$ are the fields of definition of ${\pm\alpha}$ and $\alpha$, $\sigma_\alpha$ is the non-trivial element of $\tx{Gal}(F_\alpha/F_{\pm\alpha})$, and $\kappa_\alpha$ is the isomorphism $F_{\pm\alpha}^\times/N_{F_\alpha/F_{\pm\alpha}}(F_\alpha^\times) \to \{\pm1\}$.

\begin{lem}
	Let $S \subset G$ be an $A$-admissible maximal torus and let $S^1=S \cap G^1$ be the corresponding maximal $F$-torus of $G^1$. For $\alpha \in R(S,G)_\tx{sym}$ mapping to $\alpha^1 \in R(S^1,G^1)_\tx{sym}$ we have
	\[ f_{(S,G)}(\alpha) = \kappa_\alpha(n_\alpha) \cdot f_{(S^1,G^1)}(\alpha^1)^{[F_\alpha:F_{\alpha_1}]},\]
	where $\kappa_\alpha$ is the sign character of the extension $F_\alpha/F_{\pm\alpha}$ and $n_\alpha \in \{1,2\}$ is the number of $\alpha' \in R(S,G)$ mapping to $\alpha^1$ such that $\alpha+\alpha' \in R(S,G)$.
\end{lem}
\begin{proof}
	Choose $0 \neq X_\alpha \in \mf{g}_\alpha(F_\alpha)$. By definition we have $f_{(G,S)}(\alpha)=\kappa_\alpha(\eta_\alpha)$ with $\eta_\alpha=[X_\alpha,\sigma_\alpha(X_\alpha)]/H_\alpha \in F_{\pm\alpha}^\times$. Since the actions of $A$ and $\Gamma$ commute, we have $F_{a(\alpha)}=F_\alpha$ and $F_{\pm a(\alpha)}=F_{\pm\alpha}$ for all $a \in A$. Hence, for $a \in A$, we can use $0 \neq a(X_\alpha) \in \mf{g}_{a(\alpha)}(F_\alpha)$ and obtain $\eta_{a(\alpha)}=\eta_\alpha$.
	
	According to Proposition \ref{pro:stein}(8) the restriction of $\alpha$ to $S^1$ is non-divisible, and Proposition \ref{pro:fold1} shows that no $a \in A$ that fixes $\alpha$ acts non-trivially on the component of $\alpha$. Lemma \ref{lem:linesigns} shows that any $a \in A$ that fixes $\alpha$ also fixes $X_\alpha$. Therefore we can form $X_{\alpha^1}=\sum_{a \in A/A_\alpha}a(X_\alpha)$, a non-zero element of $\mf{g}^A=\mf{g}^1$ that belongs to the $\alpha^1$-root line $\mf{g}^1_{\alpha^1}$. Any $\sigma \in \Gamma_{\alpha_1}$ permutes the fiber over $\alpha_1$ of the restriction map $R(S,G) \to R(S^1,G^1)$. According to Proposition \ref{pro:stein}(8) this fiber is precisely the $A$-orbit of $\alpha$. Thus we see that $X_{\alpha^1} \in \mf{g}^1(F_{\alpha^1})$.
	
	Since the restriction map $R(S,G) \to R(S,G)_\tx{res}$ is $\Gamma$-equivariant and $\Gamma$ preserves the subsystem of non-divisible elements in $R(S,G)_\tx{res}$, which equals $R(S^1,G^1)$ by Proposition \ref{pro:stein}(8), we see that $\Gamma_\alpha \subset \Gamma_{\alpha^1}$, $\Gamma_{\pm\alpha} \subset \Gamma_{\pm\alpha^1}$, and $\sigma_\alpha\alpha^1=-\alpha^1$, hence $\Gamma_{\pm\alpha}/\Gamma_\alpha \to \Gamma_{\pm\alpha^1}/\Gamma_{\alpha^1}$ is an isomorphism.

	By Proposition \ref{pro:fold1}(4) we have $H_{\alpha^1}=n_\alpha\sum_{a \in A/A_\alpha}a(H_\alpha)$. Using the isomorphism $\Gamma_{\pm\alpha}/\Gamma_\alpha \to \Gamma_{\pm\alpha^1}/\Gamma_{\alpha^1}$ we can replace $\sigma_{\alpha^1}$ by $\sigma_\alpha$ in the computation of $[X_{\alpha^1},\sigma_{\alpha^1}(X_{\alpha^1})]$. For any $a \in A$ we have $\sigma_\alpha(a(X_\alpha))=a(\sigma_\alpha(X_\alpha)) \in \mf{g}_{-a(\alpha)}$, and hence $[X_{\alpha^1},\sigma_\alpha(a(X_\alpha))]=[a(X_\alpha),a(\sigma_\alpha(X_\alpha))]=a[X_\alpha,\sigma_\alpha(X_\alpha)]=\eta_\alpha a(H_{\alpha})$. Hence
	\begin{eqnarray*}
		\eta_{\alpha_1}H_{\alpha_1}&=&[X_{\alpha^1},\sigma_\alpha(X_{\alpha^1})]\\
		&=&\sum_{a \in A/A_\alpha} [X_{\alpha^1},\sigma_\alpha(a(X_\alpha))]\\
		&=&\eta_\alpha \sum_{a \in A/A_\alpha} a(H_\alpha)\\
		&=&\eta_\alpha n_\alpha^{-1}H_{\alpha^1},
	\end{eqnarray*}
	concluding $\eta_\alpha=n_\alpha\eta_{\alpha^1}$, and hence 
	\[ f_{(G,S)}(\alpha) = \kappa_\alpha(\eta_\alpha) = \kappa_\alpha(n_\alpha) \kappa_{\alpha_1}(N_{F_\alpha/F_{\alpha_1}}(\eta_{\alpha_1})) = \kappa_\alpha(n_\alpha) f_{(G^1,S^1)}(\alpha^1)^{[F_\alpha:F_{\alpha_1}]}. \]

\end{proof}

\section{The conjecture for pure inner forms} \label{sec:pure}

\subsection{Pure inner forms} \label{sub:pure}

Let $z \in Z^1(\Gamma,G)$ and let $\tilde z : \Gamma \to G \rtimes \Gamma$ be the corresponding section. We have the inner form $\tilde G_z$ as in Subsection \ref{sub:inner}. We will call such inner forms pure, in analogy with the case of connected groups. In the exact sequence
\[ 1 \rw G_{z}(F) \rw \tilde G_{z}(F) \rw A^{[z]} \rw 1, \]
the group $A^{[z]}$ is the stabilizer in $A^\Gamma$ of the cohomology class $[z] \in H^1(\Gamma,G)$.

\subsection{Rational and stable conjugacy classes} \label{sub:pure_rat}

For a given $a \in A$ we want to describe those $\tilde\delta = \delta\rtimes a \in G \rtimes A$ that are rational for $\tilde G_z$, i.e. $\tilde\delta \in \tilde G_z(F)$. This is by definition equivalent to the commutativity of $\tilde z(\sigma) = z(\sigma)\rtimes \sigma$ and $\tilde\delta$ for all $\sigma \in \Gamma$. Following Vogan's suggestion \cite{Vog93} in the case of connected reductive groups, we shall consider all pure inner forms together, and are thus lead to consider the set of pairs $(\tilde z,\tilde\delta)$, where $\tilde z \in \tilde Z^1(\Gamma,G)$ and $\tilde\delta \in G(\bar F) \rtimes A$ commute. This is the set of rational elements of all pure inner forms of $\tilde G$. The group $G(\bar F)$ acts on this set by conjugation. Two elements $(\tilde z,\tilde\delta_1)$ and $(\tilde z,\tilde\delta_2)$ with the same first component lie in the same conjugacy class if and only if $\tilde\delta_1$ and $\tilde\delta_2$ are conjugate under $G_z(F)$. Thus the set of $G(\bar F)$-orbits of commuting pairs $(\tilde z,\tilde\delta)$ can be seen as the set of rational conjugacy classes of rational elements of pure inner forms of $\tilde G$.

The set of rational elements, and its quotient under rational conjugacy, have the following cohomological interpretation. Given a pair $(\tilde z,\tilde\delta)$, with $\tilde z(\sigma)=z(\sigma)\rtimes\sigma$ and $\tilde\delta=\delta\rtimes a$, the commutativity of $\tilde z$ and $\tilde\delta$ is equivalent to the equation $a(z(\sigma))=\delta^{-1}z(\sigma)\sigma(\delta)$ for all $\sigma \in \Gamma$. This equation says that $\delta$ is a coboundary between the 1-cocycles $z$ and $a(z)$. This leads us to consider the set
\[ Z^1(\Gamma,G \accentset{1}{\underaccent{\ \ a}{\rightrightarrows}} G)
\]
consisting of pairs $(z,\delta)$, where $z \in Z^1(\Gamma,G)$ and $\delta \in G$ satisfy the above equation. Slightly more generally one could consider for two group homomorphisms $(b,a) : G \rightrightarrows G$ the set of pairs $(z,\delta)$ consisting of $z \in Z^1(\Gamma,G)$ and $\delta \in G$ such that $a(z(\sigma))=\delta^{-1}b(z(\sigma))\sigma(\delta)$. For our purposes the case $b=\tx{id}$ will be sufficient. In order to ease typesetting, we shall use the notation $Z^1_a(\Gamma,G \rightrightarrows G)$ instead. As just discussed, the set $Z^1_a(\Gamma,G \rightrightarrows G)$ is identified with the disjoint union $\bigsqcup_{z \in Z^1(\Gamma,G)} [G \rtimes a]_z(F)$. Taking the union over $a \in A$ we obtain an identification between the disjoint union $\bigsqcup_{z \in Z^1(\Gamma,G)} \tilde G_z(F)$ and the disjoint union $\bigsqcup_{a \in A} Z^1_a(\Gamma,G \rrw G)$.

Fix $a \in A$. The action of $G(\bar F)$ by conjugation on the set of pairs $(\tilde z,\tilde\delta)$ is translated to the action of $G(\bar F)$ on $(z,\delta) \in Z^1_a(\Gamma,G \rrw G)$ by $g(z,\delta)=(gz(\sigma)\sigma(g^{-1}),g\delta a(g)^{-1})$. We let $H^1_a(\Gamma,G \rrw G)$ be the set of orbits of that action and thus obtain an identification of
\[ \bigsqcup_{a \in A} H^1_a(\Gamma, G \rrw G) \]
with the set of rational conjugacy classes of rational elements of pure inner forms of $\tilde G$.

Keeping in line with our notation, we shall write $\tilde Z^1_a(\Gamma,G \rrw G)$ for the set of commuting $\tilde z \in \tilde Z^1(\Gamma,G)$ and $\tilde\delta \in G \rtimes a$, and $\tilde H^1_a(\Gamma,G \rrw G)$ for their $G$-conjugacy classes, and will freely use the bijections $Z^1_a(\Gamma,G \rrw G) \to \tilde Z^1_a(\Gamma,G \rrw G)$ and $H^1_a(\Gamma,G \rrw G) \to \tilde H^1_a(\Gamma,G \rrw G)$ given by $(z,\delta) \mapsto (\tilde z,\tilde\delta)$.

\subsection{The invariant} \label{sub:pure_inv}

We are particularly interested in the $G$-conjugacy classes of pairs $(\tilde z,\tilde \delta)$ for which $\tilde\delta$ is semi-simple and strongly regular. According to Lemma \ref{lem:c2} such a conjugacy class has a norm $(S,\gamma)$, well-defined up to $G^{a,\circ}(\bar F)$-conjugacy. We shall now define an element $\tx{inv}(\gamma,(z,\delta)) \in H^1_a(\Gamma,S \rrw S)$.

\begin{lem} \label{lem:inv}
\begin{enumerate}
	\item If $(z^*,\delta^*)$ is a representative of the equivalence class of $(z,\delta)$ such that $\delta^* \in S$ and the image of $\delta^*$ in $S_a$ is $\gamma$, then $z^*(\sigma) \in S$ and hence $(z^*,\delta^*) \in Z^1_a(\Gamma,S \rrw S)$. The class $\tx{inv}(\gamma,(z,\delta)) \in H^1_a(\Gamma,S \rrw S)$ of $(z^*,\delta^*)$ is independent of the choice of $(z^*,\delta^*)$.
	\item If $(S',\gamma')$ is another norm of the same equivalence class, and $(z',\delta')$ the corresponding representative, the unique isomorphism $\tx{Ad}(g) : S \to S'$ mapping $\gamma$ to $\gamma'$ induces an isomorphism $H^1_a(\Gamma,S \rrw S) \to H^1_a(\Gamma,S' \rrw S')$ identifying the class $\tx{inv}(\gamma,(z,\delta))$ of $(z^*,\delta^*)$ with the class $\tx{inv}(\gamma',(z,\delta))$ of $(z',\delta')$.
\end{enumerate}
\end{lem}
\begin{proof}
Since the element $\gamma$ is $\Gamma$\-fixed, the $S(\bar F)$-conjugacy class of $\tilde\delta^*$ is $\Gamma$\-fixed (in fact the two statements are equivalent). For $\sigma \in \Gamma$ let $s(\sigma) \in S$ be such that $\sigma(\tilde\delta^*)=\tx{Ad}(s(\sigma))\tilde\delta^*$. Since $\tilde z^*(\sigma)=z^*(\sigma)\rtimes\sigma$ commutes with $\tilde\delta^*$ we see that $z^*(\sigma)s(\sigma)$ commutes with $\tilde\delta^*$, thus $z^*(\sigma)s(\sigma) \in S^a$, thus $z^*(\sigma) \in S$.

Let us show that the class of $(\tilde z^*,\tilde\delta^*)$ in $\tilde H^1_a(\Gamma,S \rrw S)$ is independent of the choice of $(\tilde z^*,\tilde \delta^*)$. Another such choice is of the form $\tx{Ad}(h)(\tilde z^*,\tilde \delta^*)$ for some $h \in G(\bar F)$. By assumption $h\delta^*a(h^{-1})$ maps to $\gamma$, so there exists $s \in S(\bar F)$ such that $sh\delta^*a(sh)^{-1}=\delta^*$, i.e. $sh \in \tx{Cent}(\tilde\delta,G)=S^a$, thus $h \in S$, as claimed.

Now let $(S',\gamma')$ be another norm and choose by Lemma \ref{lem:c1} an element $g \in G^{a,\circ}(\bar F)$ such that $\tx{Ad}(g) : S \to S'$ carries $\gamma$ to $\gamma'$. Then $\tx{Ad}(g)(\tilde z^*,\tilde\delta^*)$ is a representative of the conjugacy class of $(\tilde z,\tilde\delta)$ and lies in $\tilde Z^1(\Gamma,S' \rrw S')$.
\end{proof}

\begin{rem} \label{rem:inv}
	We have noted in Remark \ref{rem:norm} that the pair $(z^*,\delta^*)$ in Lemma \ref{lem:inv}(1) cannot in general be chosen to additionally satisfy $\delta^* \in S(F)$. Let us for a moment consider the situation where $\delta^* \in S(F)$. Then $\tilde\delta^*$ commutes with both $\sigma$ and $z^*(\sigma) \rtimes \sigma$. This implies that $z^*(\sigma)$ takes values in the centralizer of $\tilde\delta^*$, which is $S^a$.
	
	Thus $z^*$ provides a class in $H^1(\Gamma,S^a)$, while $\delta^*$ provides a class in $H^0(\Gamma,S)$, and we see that $\tx{inv}(\gamma,(z,\delta))$ is the image of the pair of classes $([z^*],[\delta^*])$ under the natural map $H^1(\Gamma,S^a) \times H^0(\Gamma,S) \to H^1_a(\Gamma,S \rrw S)$.
\end{rem}

\subsection{Comparison with \cite{KS99}} 
It would be informative to compare the notions of norms and invariants given here with those in \cite{KS99}. In short, the notions of norms are the same apart from cosmetics, while the notion of invariant introduced here is the twisted analog of the untwisted absolute invariant introduced in \cite[\S2.1]{KalECI}, and thus refines the relative invariant introduced in \cite[\S4.4]{KS99}, and generalizes the absolute invariant introduced in the setting of quasi-split groups in \cite[\S5.3]{KS99}.

More precisely, let $(\tilde z,\tilde\delta) \in \tilde Z^1_a(\Gamma,G \rrw G)$. Then $\theta=\tx{Ad}(\tilde\delta)$ is an automorphism of $G_z$ defined over $F$. The element $\tilde\delta$ is semi-simple and strongly-regular if and only if $1 \in G_z(F)$ is $\theta$-semi-simple and $\theta$-strongly regular. Let $\theta^*=a$. The map $m : Cl(G_z,\theta) \rw Cl(G,\theta^*)$ of \cite[\S3.1]{KS99} is given by $h \mapsto h \cdot \tilde\delta$ and in particular sends $1$ to $\tilde\delta$. The 1-cochain defined in \cite[Lemma 3.1.A]{KS99} and denote by $z_\sigma$ there (beware that this is not the same as our $z_\sigma$ here) is identically equal to $1$.

Now let $(\tilde z^*,\tilde\delta^*)$ and $\gamma \in S_a(F)$ be be as in Lemma \ref{lem:inv}. Then $\gamma$ is a norm of $1$ in the sense of \cite[\S3.3]{KS99}. Moreover, if $g \in G$ is such that $g^{-1}(\tilde z^*,\tilde \delta^*)g=(\tilde z,\tilde \delta)$, then the 1-cocycle $v(\sigma)=gu(\sigma)\sigma(g)^{-1}$ considered in \cite[Lemma 4.4.A]{KS99}, which takes values in $S_\tx{sc}$, when composed with the natural map $S_\tx{sc} \rw S$, becomes equal to $z^*$.

We have the the isomorphism
\begin{equation} \label{eq:ksiso}
 Z^1_a(\Gamma,S \rrw S) \rw Z^1(\Gamma,S\stackrel{1-a}{\lrw}S),\qquad (z,\delta) \mapsto (z^{-1},\delta)
\end{equation}
and it allows us to view $\tx{inv}(\gamma,(z,\delta))$ as an element of $H^1(\Gamma,S\stackrel{1-a}{\lrw}S)$. This element is then an absolute version of the relative invariant $\tx{inv}(\gamma,\delta;\bar\gamma,\bar\delta)$ introduced in \cite[\S4.4]{KS99}, and a generalization of the absolute invariant $\tx{inv}(\gamma,\delta)$ that was introduced in \cite[\S5.3]{KS99}.

\subsection{The dual group} \label{sub:dual}

Let $\hat G$ be a complex dual group for $G$. We rigidify it by fixing a $\Gamma$\-invariant pinning $(\hat T,\hat B,\{\hat X_\alpha\})$ and requiring it to be dual to the fixed pinning of $G$. That is, we assume given an identification $X^*(\hat T) = X_*(T)$ under which the $B$-positive coroots are identified with the $\hat B$-positive roots. We define the $L$-group of $G$ as $^LG=\hat G \rtimes W_F$, where $W_F$ acts on $\hat G$ by fixing the pinning. We also let the group $A$ act on $\hat G$ by fixing the pinning. More precisely, given $a \in A$, we have the automorphism $a_*$ of $X_*(T)$ given by $(a_*\lambda)(x)=a(\lambda(x))$ for $x \in \mb{G}_m$, and we let $a$ act on $\hat T = \tx{Hom}(X_*(T),\C^\times)$ by $[at](\lambda)=t(a_*^{-1}\lambda)$ for $t \in \hat T$ and $\lambda \in X_*(T)$. Note that the automorphism $a$ of $\hat G$ obtained in this way is related to the automorphism $\hat \theta$ of $\hat G$ obtained from $\theta^*=a$ as in \cite[\S1.2]{KS99} by $a=\hat \theta^{-1}$. This will later have the effect of $H^1(\Gamma,S \stackrel{1-a}{\lrw} S)$ being paired with $H^1(W_F,\hat S \stackrel{1-a^{-1}}{\lrw} \hat S)$.

\subsection{The local correspondence} \label{sub:pure_llc}

Given an irreducible admissible representation of the locally profinite group $\tilde G_z(F)$, its restriction to $G_z(F)$ is a finite length semi-simple admissible representation. We shall say that a representation of $\tilde G_z(F)$ is $G$-tempered respectively $G$-discrete, if its restriction to $G_z(F)$ contains (equivalently, consists of) a tempered respectively discrete representations.

We will now begin formulating the refined local Langlands conjecture for the disconnected groups $\tilde G_z$. The irreducible admissible $G$-tempered representations of $\tilde G_z$ will again be parameterized by pairs $(\phi,\rho)$. The first part of the pair, the Langlands parameter $\phi$, will remain unchanged. That is, we will use the same tempered Langlands parameters $\phi : L_F \rw {^LG}$ as for the connected group $G$. However, we will change what we mean by equivalence of parameters. Two parameter will be seen as $\tilde G$-equivalent if they are conjugate under the group $\hat G \rtimes A$. Given a parameter $\phi$, its group of $\tilde G$-self-equivalences is then $\tilde S_\phi=\tx{Cent}(\phi,\hat G \rtimes A)$. This group contains the group $S_\phi = \tx{Cent}(\phi,\hat G)$ of $G$-self-equivalences of $\phi$. We have the exact sequence
\[ 1 \rw S_\phi \rw \tilde S_\phi \rw A^{[\phi]} \rw 1, \]
where $A^{[\phi]}$ is the stabilizer in $A$ of the $G$-equivalence class of $\phi$. This exact sequence leads to the exact sequence
\[ 1 \rw \pi_0(S_\phi) \rw \pi_0(\tilde S_\phi) \rw A^{[\phi]} \rw 1. \]
Recall from \cite{Kot86} that the cohomology class $[z]$ gives a character $\pi_0(Z(\hat G)^\Gamma) \rw \C^\times$, which we will also denote by $[z]$. The stabilizer of this character in $A$ is equal to the stabilizer of the cohomology class of $z$ -- this is immediate if $F$ is $p$-adic, and can be checked if $F=\R$.
Let $A^{[\phi],[z]} = A^{[\phi]} \cap A^{[z]}$. If we pull back the above extension to $A^{[\phi],[z]}$ we obtain the extension
\[ 1 \rw \pi_0(S_\phi) \rw \pi_0(\tilde S_\phi^{[z]}) \rw A^{[\phi],[z]} \rw 1, \]
where $\tilde S_\phi^{[z]}=\tx{Cent}(\phi,\hat G \rtimes A^{[z]})$. The pull-back of an irreducible representation of $\pi_0(\tilde S_\phi^{[z]})$ to $\pi_0(Z(\hat G)^\Gamma)$ is either $[z]$-isotypic, or it does not contain $[z]$. We write $\tx{Irr}(\pi_0(\tilde S_\phi^{[z]}),{[z]})$ for the set of irreducible representations of $\pi_0(\tilde S_\phi^{[z]})$ whose pull-back to $\pi_0(Z(\hat G)^\Gamma)$ is $[z]$-isotypic.

We remark that we could alternatively consider the set $\tx{Irr}(\pi_0(\tilde S_\phi),{[z]})$ of those irreducible representations whose restriction to $\pi_0(Z(\hat G)^\Gamma)$ contains the character $[z]$. Since $Z(\hat G)^\Gamma$ is not central in $\tilde S_\phi$, this restriction will contain other characters as well. According to Clifford theory induction from $\tilde S_\phi^{[z]}$ to $\tilde S_\phi$ gives a bijection between $\tx{Irr}(\pi_0(\tilde S_\phi^{[z]}),{[z]})$ and $\tx{Irr}(\pi_0(\tilde S_\phi),{[z]})$. Indeed, any element of $\tilde S_\phi$ that normalizes $\tilde S_\phi^{[z]}$ and stabilizes $\rho \in \tx{Irr}(\pi_0(\tilde S_\phi^{[z]}),{[z]})$ also stabilizes $[z]$ and thus belongs to $\tilde S_\phi^{[z]}$, so $\tx{Ind}\rho$ is irreducible. However, working with $\tx{Irr}(\pi_0(\tilde S_\phi^{[z]}),{[z]})$ will be more convenient for us.

Consider for a moment the special case $z=1$. Choose an $A$-admissible Whittaker datum $\mf{w}$ for $G$. Any $\mf{w}$-generic representation $\pi$ of $G(F)$ has a canonical extension $\tilde\pi$ to $\tilde G(F)=G(F) \rtimes A^\Gamma$, obtained by setting $\tilde\pi(a)$ to be the unique $G(F)$-map $\pi \circ a^{-1} \to \pi$ that preserves one (hence any) $\mf{w}$-Whittaker functional. We shall say that these $\tilde\pi$ are $\mf{w}$-generic representations of $\tilde G(F)$. We can now state the first part of the local Langlands conjecture for the groups $\tilde G_z$. In the case $F=\R$ set $^K\tilde G_z$ to be the associated $K$-group, i.e. the disjoint union of $\tilde G_{z'}$ for all $z'$ in the image of $H^1(\R,G_{z,\tx{sc}}) \to H^1(\R,G_z) \to H^1(\R,G)$.

\begin{cnj} \label{cnj:llc_pure_is} The choice of an $A$-admissible Whittaker datum $\mf{w}$ on $G$ determines a bijection between the set of irreducible admissible $G$-tempered representations of $\tilde G_z(F)$ when $F/\Q_p$, or of $^K\tilde G_z(F)$ when $F=\R$, and the set of $(\hat G \rtimes A^{[z]})$-conjugacy classes of pairs $(\phi,\tilde\rho)$, where $\phi : L_F \to {^LG}$ is a tempered Langlands parameter, and $\tilde\rho \in \tx{Irr}(\pi_0(\tilde S_\phi^{[z]}),{[z]})$. When $z=1$ the representation corresponding to $(\phi,\tilde\rho)$ is $\mf{w}$-generic if and only if $\tilde\rho=1$.
\end{cnj}

Let us write $\tilde \Pi_{\phi,z}$ for the finite set of representations of $\tilde G_z(F)$ corresponding to pairs $(\phi,\tilde\rho)$ for a fixed $\phi$ and varying $\tilde\rho$. These can be called $L$-packets for the disconnected group $\tilde G_z(F)$. In the \S\ref{sub:pure_charid} we will add another piece of the conjecture, which will in particular determine uniquely the sets $\tilde\Pi_{\phi,z}$ in terms of the $L$-packets of the connected group $G_z$. The new information in Conjecture \ref{cnj:llc_pure_is} is thus contained in the bijection between $\tilde\Pi_{\phi,z}$ and $\tx{Irr}(\pi_0(\tilde S_\phi^{[z]}),{[z]})$.

\subsection{Endoscopic data} \label{sub:endo}

We shall use essentially the same notion of endoscopic data as in \cite[\S2.1]{KS99}, with one minor but important difference that affects both the definition of datum as well as of an isomorphism of data. More precisely, an endoscopic datum will be a tuple $\mf{e}=(G^\mf{e},\mc{G}^\mf{e},\tilde s^\mf{e},\xi^\mf{e})$ consisting of
\begin{enumerate}[label=(\thesubsection.\arabic*)]
	\item a quasi-split connected reductive group $G^\mf{e}$ defined over $F$;
	\item a split extension $\mc{G}^\mf{e}$ of $W_F$ by $\hat G^\mf{e}$ (but without the choice of splitting);
	\item a semi-simple element $\tilde s^\mf{e} \in \hat G \rtimes A$;
	\item a homomorphism $\xi^\mf{e} : \mc{G}^\mf{e} \to {^LG}$ of extensions;
\end{enumerate}
and satisfying
\begin{enumerate}[label=(\thesubsection.\arabic*),resume]
	\item the homomorphism $W_F \to \tx{Out}(\hat G^\mf{e})$ arising from the extension $\mc{G}^\mf{e}$ is transported under the canonical identification $\tx{Out}(\hat G^\mf{e})=\tx{Out}(G^\mf{e})$ to the one given by the rational structure of $G^\mf{e}$;
	\item $\xi^\mf{e}$ induces an isomorphism $\hat G^\mf{e} \to \tx{Cent}(\tilde s^\mf{e},\hat G)^\circ$; \label{item:e0}
	\item $\tilde s^\mf{e}$ commutes with the image of $\xi^\mf{e}$. \label{item:e1}
\end{enumerate}
This completes the description of the tuple $\mf{e}$. An isomorphism $\mf{e} \to \mf{e'}$ is an element $g \in \hat G$ satisfying
\begin{enumerate}[label=(\thesubsection.\arabic*),resume]
	\item $\xi^\mf{e'}=\tx{Ad}(g) \circ \xi^\mf{e}$;
	\item $\tilde s^\mf{e'}=\tx{Ad}(g)\tilde s^\mf{e}$ modulo $Z(\hat G)^\circ$. \label{item:e2}
\end{enumerate}
The difference between these definitions and those in \cite{KS99} is the following. First, we are only considering here the case $\omega=1$ and hence $\mathbf{a}=1$. Second, our requirement \ref{item:e1} is stricter than \cite[(2.1.4a)]{KS99}. The definition \cite[(2.1.6)]{KS99} of isomorphism however implies that every isomorphism class of endoscopic data in the sense of \cite{KS99} contains a representative that satisfies \ref{item:e1}. Third, our requirement \ref{item:e2} is stricter than \cite[(2.1.6)]{KS99}. This implies that a single isomorphism class in the sense of \cite{KS99} can consist of multiple isomorphism classes in the sense of our definition.

\subsection{Two constructions of endoscopic data} \label{sub:endocnst}

We now review two constructions of endoscopic data, one geometric and one spectral. In the case of connected groups, they are summarized in \cite[\S4.2]{She83}. In the twisted case the geometric appears in the proof of \cite[Lemma 7.2]{KS99} and the spectral one appears at the end of \cite[\S2]{KS99}.

We begin with the spectral construction, which is a little easier to describe. Let $\phi : L_F \to {^LG}$ be an $L$-parameter and $\tilde s \in \tilde S_\phi$ a semi-simple element. The pair $(\phi,\tilde s)$ leads to an endoscopic datum as follows. Set $\hat G^\mf{e}=\tx{Cent}(\tilde s,\hat G)^\circ$, $\mc{G}^\mf{e}=\hat G^\mf{e}\cdot \phi(W_F)$, and let $\xi^\mf{e}$ be the natural inclusion. Let $G^\mf{e}$ be the quasi-split group defined over $F$ that is dual to $\hat G^\mf{e}$ and whose rational structure is determined by $\Gamma \to \tx{Out}(\hat G^\mf{e})=\tx{Out}(G^\mf{e})$, the first map coming from the extension $\mc{G}^\mf{e}$. Then $(G^\mf{e},\mc{G}^\mf{e},\tilde s,\xi^\mf{e})$ is an endoscopic datum. Note that $\phi$ factors through $\xi^\mf{e}$ and thus becomes a parameter for $\mc{G}^\mf{e}$ (in order to relate it to $G^\mf{e}$, one needs to further choose a z-pair).

For the geometric construction, let $\tilde G_{\bar z}$ be an inner form of $\tilde G$ and let $\tilde \delta \in \tilde G_{\bar z}(F)$ be strongly regular semi-simple. Let $S' \subset G$ be the maximal torus $\tx{Cent}(\tx{Cent}(\tilde\delta,G),G)$. As a maximal torus of $G_{\bar z}$ it is defined over $F$, and $\tx{Ad}(\tilde\delta) : S' \to S'$ is an automorphism defined over $F$. Let $\kappa \in H^1(W_F,(1-\tilde\delta) : \hat S' \to \hat S')$. The pair $(\tilde\delta,\kappa)$ leads to an endoscopic datum $(G^\mf{e},\mc{G}^\mf{e},\tilde s^\mf{e},\xi^\mf{e})$ and a stable conjugacy class of elements $\gamma^\mf{e} \in G^\mf{e}(F)$ as follows.

Choose a norm $(S,\gamma)$ of $\tilde\delta$ and $g \in G(\bar F)$ such that $g^{-1}\tilde\delta g = \tilde\delta^* = \delta^* \rtimes a \in S(\bar F) \rtimes a$ with $\delta^* \mapsto \gamma$. Then $\tx{Ad}(g)$ provides an isomorphism $H^1(W_F,(1-\tilde\delta) : \hat S' \to \hat S') \to H^1(W_F,(1-a) : \hat S \to \hat S)$. Choose an $a$-invariant Borel pair $(\hat T,\hat B)$ of $\hat G$ and an $a$-invariant Borel subgroup $C$ containing $S$. These lead to an equivariant isomorphism $\hat S \to \hat T$ under which a 1-hypercocycle representing $\kappa$ is transported to a pair $(t_w^{-1},s)$ satisfying, for every $w \in W_F$, the relation $s^{-1}\sigma_S(s)=t_w^{-1}a(t_w)$, where $\sigma \in \Gamma$ is the image of $w$ and $\sigma_S$ is the transport of the action of $\sigma$ on $\hat S$ to $\hat T$. This transport is given by $\omega_\sigma \rtimes \sigma$, for a uniquely determined $\omega_\sigma \in \Omega(\hat T,\hat G)^a$. The map $\sigma \mapsto \omega_\sigma$ belongs to $Z^1(\Gamma,\Omega(\hat T,\hat G)^a)$. In the long exact sequence of $\Gamma$-cohomology associated to the short exact sequence
\[ 1 \to \hat T^{a,\circ} \to N(\hat T^{a,\circ},\hat G^{a,\circ}) \to \Omega(\hat T,\hat G)^a \to 1 \]
the image of the class of $\omega_\sigma$ is an element of $H^2(\Gamma,\hat S)$, whose restriction to $H^2(W_F,\hat S)$ vanishes according to \cite[Lemma 4]{Lan79}. It follows that there exist lifts $n_\sigma \in N(\hat T^{a,\circ},\hat G^{a,\circ})$ of $\omega_\sigma$ so that $w \mapsto n_w \rtimes w$ is a homomorphism $W_F \to N(\hat T^{a,\circ},\hat G^{a,\circ}) \rtimes W_F$. Then $\eta : w \mapsto t_wn_w \rtimes w$ is a group homomorphism $W_F \to N(\hat T,\hat G)$ whose image commutes with $s \rtimes a$. Define $\tilde s^\mf{e}=s \rtimes a$, $\hat G^\mf{e}=\tx{Cent}(\tilde s^\mf{e},\hat G)^\circ$, $\mc{G}^\mf{e}=\hat G^\mf{e} \cdot \eta(W_F)$, and let $\xi^\mf{e}$ be the natural inclusion. Let $G^\mf{e}$ be the quasi-split group defined over $F$, dual to $\hat G^\mf{e}$, and with rational structure determined by $\Gamma \to \tx{Out}(\hat G^\mf{e}) = \tx{Out}(G^\mf{e})$, where the first map comes from the extension $\mc{G}^\mf{e}$. It is immediately checked that $(G^\mf{e},\mc{G}^\mf{e},\tilde s^\mf{e},\xi^\mf{e})$ is an endoscopic datum.

The $a$-equivariant isomorphism $\hat S \to \hat T$ and the inclusion $\hat T^{a,\circ} \to \hat G^\mf{e}$ give a canonical $G^\mf{e}(\bar F)$-conjugacy class of embeddings $S_a \to G^\mf{e}$. Thus the element $\gamma$ gives a canonical $G^\mf{e}(\bar F)$-conjugacy class of strongly regular semi-simple elements of $G^\mf{e}(\bar F)$. This class is $\Gamma$-invariant, so by \cite[Corollary 2.2]{Kot82} gives a stable class of elements $\gamma^\mf{e} \in G^\mf{e}(F)$. This completes the geometric construction.

\subsection{Normalized transfer factor invariant under $\tilde G_z(F)$} \label{sub:pure_tf}

Fix an $A$-admissible Whittaker datum as in \S\ref{sub:awhit}. Let $\mf{e}=(G^\mf{e},\mc{G}^\mf{e},\tilde s^\mf{e},\xi^\mf{e})$ be an endoscopic datum. There may or may not exist an isomorphism $\mc{G}^\mf{e} \to  {^LG}^\mf{e}$ of extensions of $W_F$ by $\hat G^\mf{e}$. If it does we choose one such, denote it by $\xi^\mf{z}$ and write $G^\mf{z}=G^\mf{e}$. If it does not, we can choose a $z$-extension $G^\mf{z} \to G^\mf{e}$ and apply \cite[Lemma 2.2.A]{KS99} which guarantees that the inclusion $\hat G^\mf{e} \to \hat G^\mf{z}$ always extends to an $L$-embedding $\xi^\mf{z} : \mc{G}^\mf{e} \to {^LG}^\mf{z}$. We denote by $\mf{z}$ the pair $(G^\mf{z},\xi^\mf{z})$.

We will define a normalized absolute transfer factor $\Delta[\mf{w},\mf{e},\mf{z}]$ as a function that assigns complex values to pairs $(\gamma^\mf{z},\tilde\delta)$ of $G^\mf{z}(F) \times \tilde G_z(F)$, where both $\gamma^\mf{z}$ and $\tilde\delta$ are strongly regular semi-simple. As a function of $\tilde\delta$ the transfer factor $\Delta[\mf{w},\mf{e},\mf{z}]$ will be conjugation\-invariant under the full group $\tilde G_z(F)$. The definition is by the formula
\begin{equation} \label{eq:pure_tf} \Delta[\mf{w},\mf{e},\mf{z}](\gamma^\mf{z},\tilde\delta) = \sum_{c \in \tilde G_z(F)/G_z(F)} \Delta_{KS}[\mf{w},\mf{e},\mf{z}](\gamma^\mf{z},c\tilde\delta c^{-1}), \end{equation}
which in turn uses a normalized absolute Kottwitz-Shelstad transfer factor $\Delta_{KS}[\mf{w},\mf{e},\mf{z}]$ that we will define below. The latter is a function that assigns complex values to pairs $(\gamma^\mf{z},\tilde\delta)$ of $G^\mf{z}(F) \times [G \rtimes b^{-1}]_z(F)$, where $b \in A$ is the image of $\tilde s^\mf{e}$ and both $\gamma^\mf{z}$ and $\tilde\delta$ are strongly regular semi-simple. In the variable $\tilde\delta$ this function is only $G_z(F)$\-conjugation invariant.

Following \cite[\S5.5]{KS12}, the factor $\Delta_{KS}[\mf{w},\mf{e},\mf{z}]$ is defined by
\begin{equation} \label{eq:pure_tf1} \Delta_{KS}[\mf{w},\mf{e},\mf{z}]:=
\epsilon_L(V,\psi)(\Delta_I^\tx{new})^{-1}\Delta_{II}(\Delta_{III}^\tx{new})^{-1}\Delta_{IV}. \end{equation}
The terms $\epsilon_L(V,\psi)$, $\Delta_I^\tx{new}$, $\Delta_{II}$, and $\Delta_{IV}$ have already been defined, in \cite[\S5.3]{KS99}, \cite[\S3.4]{KS12}, \cite[\S4.3]{KS99}, and \cite[\S4.5]{KS99}, respectively, but we will recall them for the convenience of the reader below. They are absolute, i.e. they depend on a single pair of elements $(\gamma^\mf{z},\tilde\delta)$. The term $\Delta_{III}^\tx{new}$ will be defined in this paper. It is also absolute. A relative version of it, i.e. one depending on two pairs of related elements, was defined in \cite[\S4.4]{KS99}; in the quasi-split case $z=1$ an absolute version was defined in \cite[\S5.3]{KS99}; in the untwisted case an absolute version was defined in \cite{KalECI} for pure inner forms of $p$-adic groups, and generalized in \cite{KalRI} to arbitrary inner forms of connected groups over local fields. In this paper we will define an absolute version for arbitrary $z$ in the twisted setting. 
When $\bar z=1$ the term $\Delta_{III}^\tx{new}$ coincides with the absolute term $\Delta_{III}$ defined in \cite[\S5.3]{KS99}. The factor $\Delta_{KS}$ is an absolute transfer factor whose relative version is the factor $\Delta'$ of \cite[\S5.4]{KS12}. When $z=1$ then $\Delta_{KS}$ coincides with the absolute factor \cite[(5.5.2)]{KS12}. When $b=1$ the factor $\Delta_{KS}$ coincides with 
the factor \cite[(5.10)]{KalRI}

Before we come to $\Delta_{III}^\tx{new}$ we briefly recall the definition of the other factors. These factors depend on auxiliary data that we describe first. Let $(T,B,\{X_\alpha\})$ be an $F$-pinning of $G$ invariant under $b$. Let $\psi : F \to \C^\times$ be a non-trivial character. It is assumed that the Whittaker datum arising from the pinning and the character is the given datum $\mf{w}$. Let $R_\tx{res}(S,G)$ be the set of restrictions to $S_\tx{sc}^b$ of the absolute roots of $S$ in $G$. Fix $a$-data and $\chi$-data for $R_\tx{res}(S,G)$. 

Let $\gamma^\mf{e} \in G^\mf{e}$ be the image of $\gamma^\mf{z}$ under the natural map $G^\mf{z} \to G^\mf{e}$. The complex number $\Delta_{KS}[\mf{w},\mf{e},\mf{z}](\gamma^\mf{z},\tilde\delta)$ is zero unless $\gamma^\mf{e}$ transfers to a norm of $\tilde\delta$. More precisely, let $(S,\gamma)$ be a norm of $\tilde\delta$ in the sense of Definition \ref{dfn:norm}. It exists and is unique up to $G^{b}(\bar F)$-conjugacy according to Lemma \ref{lem:c2}. In order for $\Delta_{KS}[\mf{w},\mf{e},\mf{z}](\gamma^\mf{z},\tilde\delta)$ to not be zero, there must exists an admissible isomorphism $S^\mf{e} \to S_{b}$ carrying $\gamma^\mf{e}$ to $\gamma$, where $S^\mf{e}$ is the centralizer of $\gamma^\mf{e}$. We now assume such an isomorphism exists. It is then uniquely determined by the pair $(\gamma^\mf{e},\gamma)$ and we call it $\varphi_{\gamma^\mf{e},\gamma}$. Via this isomorphism we obtain an embedding $R(S^\mf{e},G^\mf{e}) \to R_\tx{res}(S,G)$. We can transport the chosen $a$-data and $\chi$-data to $R(S^\mf{e},G^\mf{e})$. Recall that $S$ is a $b$-admissible maximal torus of $G$, $\gamma \in S_b(F)$, and there exists $g \in G(\bar F)$ such that $g^{-1}\tilde\delta g = \tilde\delta^*=\delta^* \rtimes b^{-1}$ with $\delta^* \in S(\bar F)$ mapping to $\gamma$.

The term $\epsilon_L(V,\psi)$ is the root number of the virtual $\Gamma$-representation $V=X^*(T)^b_\C-X^*(T^\mf{e})_\C$, where $T^\mf{e}$ is the (unique up to conjugation) minimal Levi subgroup of $G^\mf{e}$. 

The term $\Delta_{II}$ is a fraction. Its numerator is a product over $\Gamma$-orbits of $\alpha_\tx{res} \in R_\tx{res}(S,G)$, where the factor corresponding to $\alpha_\tx{res}$ is $\chi_{\alpha_\tx{res}}((N\alpha(\delta^*)-1)/a_{\alpha_\tx{res}})$ when $\alpha_\tx{res}$ is of type R1 or R2, and $\chi_{\alpha_\tx{res}}(N\alpha(\delta^*)+1)$ if $\alpha_\tx{res}$ is of type R3. Here $\alpha \in R(S,G)$ is any preimage of $\alpha_\tx{res}$ and $N\alpha$ is the sum of the members of the $b$-orbit of $\alpha$, which we recall is uniquely determined by $\alpha_\tx{res}$. The denominator is a product over $\Gamma$-orbits of $\alpha_\mf{e} \in R(S^\mf{e},G^\mf{e}) \subset R_\tx{res}(S,G)$, where the factor corresponding to $\alpha_\mf{e}$ is $\chi_{\alpha_\mf{e}}((\alpha_\mf{e}(\gamma^\mf{e})-1)/a_{\alpha_\mf{e}})$.

The term $\Delta_{IV}$ is again a fraction. Its numerator is a product over $\alpha_\tx{res} \in R_\tx{res}(S,G)$, where the factor corresponding to $\alpha_\tx{res}$ is $|N\alpha(\delta^*)-1|^\frac{1}{2}$ when $\alpha_\tx{res}$ is of type R1 or R2, and $|N\alpha(\delta^*)+1|^\frac{1}{2}$ if $\alpha_\tx{res}$ is of type R3. The denominator is a product over $\alpha_\mf{e} \in R(S^\mf{e},G^\mf{e}) \subset R_\tx{res}(S,G)$, where the factor corresponding to $\alpha_\mf{e}$ is $|\alpha_\mf{e}(\gamma^\mf{e})-1|^\frac{1}{2}$.

The term $\Delta_I$ is obtained by taking the Tate-Nakayama pairing of an element $t \in H^1(\Gamma,S_\tx{sc}^b)$ with an element $s_S \in \pi_0([\hat S_\tx{ad}]_b^\Gamma)$, or equivalently of the image of $t$ in $H^1(\Gamma,S^{b,\circ})$ with an element $s_S \in \pi_0([\hat S]_b^\Gamma)$. The element $t$ is the twisted splitting invariant of $S$, obtained as follows. Let $C \subset G$ be a Borel subgroup defined over $\bar F$ containing $S$ and invariant under $b$; its existence is the definition of $b$-admissibility of $S$. Choose $h \in G_\tx{sc}^b$ such that $h(T,B)h^{-1}=(S,C)$. Let $w(\sigma) \in N(T_\tx{sc}^b,G_\tx{sc}^b)/T_\tx{sc}^b$ be the image of $h^{-1}\sigma(h)$ and let $n(\sigma) \in N(T_\tx{sc}^b,G_\tx{sc}^b)$ be the Tits lift of $w(\sigma)$ with respect to the chosen pinning of $G$. Thus $n(\sigma)(h^{-1}\sigma(h))^{-1} \in T^b$. Let $y(\sigma)=\prod \alpha^\vee(a_{\alpha_\tx{res}}) \in S_\tx{sc}^b$, where the product runs over those $\alpha \in R(S,C)$ that satisfy $-\sigma(\alpha) \in R(S,C)$. Then $t$ is the class of $y(\sigma) \cdot hn(\sigma)\sigma(h)^{-1}$. To obtain the element $s_S \in \pi_0([\hat S]_b^\Gamma)$, choose a member of the canonical $\hat G^\mf{e}$-conjugacy class of embeddings $\hat S^\mf{e} \to \hat G^\mf{e}$ and compose it with $\xi^\mf{e}$ and $\hat\varphi_{\gamma^\mf{e},\gamma}$ to obtain an embedding $\hat S^{b,\circ} \to \hat G$. The image of this embedding commutes with $\tilde s^\mf{e}$. It extends uniquely to an admissible embedding $\hat S \to \hat G$, see Lemma \ref{lem:lembex} below. Replacing $\mf{e}$ by an isomorphic datum if necessary we may arrange that the image of $\hat S$ is $b$-invariant. Writing $\tilde s^\mf{e}=s^\mf{e} \rtimes b$ we see that $s^\mf{e}$ commutes with the image of $\hat S^{b,\circ}$, hence also with the image of $\hat S$, and hence lies in that image. We transport $s^\mf{e}$ to $\hat S$ and project to $[\hat S]_b$ and obtain $s_S$.

\subsection{Normalized factor $\Delta_{KS}$ without $z$-pair} \label{sub:pure_tf1}

We turn to the construction of $\Delta_{III}^\tx{new}$. For simplicity in this subsection we shall assume that there exists an $L$-isomorphism $\xi^\mf{z} : \mc{G}^\mf{e} \to {^LG}^\mf{e}$, so that $G^\mf{z}=G^\mf{e}$ and $\gamma^\mf{z}=\gamma^\mf{e}$. The general case will be treated in the next subsection.

In Lemma \ref{lem:inv} we defined an element $\tx{inv}(\gamma,(z,\delta)) \in H^1_{b^{-1}}(\Gamma,S \rrw S)$ which we transport via the isomorphism \eqref{eq:ksiso} to $H^1(\Gamma,(1-b^{-1}):S \rw S)$. On the other hand, the constructions of \cite[\S4.4]{KS99} provide an element $A_0$ of $H^1(W_F,(1-b) : \hat S \to \hat S)$. We recall them here, as they take a particularly simple form in our set-up. Namely, transport the chosen $\chi$-data for $R_\tx{res}(S_{b},G)$ via the admissible isomorphism $S^\mf{e} \to S_{b}$ to $R_\tx{res}(S^\mf{e},G^\mf{e})$. From these $\chi$-data one obtains $L$-embeddings $\xi_S^\mf{e} : {^LS}^\mf{e} \rw {^LG^\mf{e}}$ and $\xi_S^1 : {^LS}_{b} \to {^LG^1}$, where $G^1$ is the principal endoscopic group of $G \rtimes b^{-1}$, i.e. the quasi-split connected reductive group with $L$-group $\hat G^1 \rtimes W_F$, where $\hat G^1=\hat G^{b,\circ}$. We have the natural embedding $^LG^1 \to {^LG}$ and the $L$-isomorphism ${^L\varphi}_{\gamma^\mf{e},\gamma} : {^LS_{b}} \to {^LS}^\mf{e}$ corresponding to $\varphi_{\gamma^\mf{e},\gamma}$. Thus we obtain two $L$-embeddings ${^LS_{b}} \to {^LG}$,  well-defined up to $\hat G$-conjugacy, given by $\xi_S^1$ and $\xi^\mf{z}\circ\xi_S^\mf{e}\circ{^L\varphi}_{\gamma^\mf{e},\gamma}$. By Corollary \ref{cor:lembcomp} there exists a (uniquely determined) 1-cocycle $a_S : W_F \rw \hat S$ such that $\xi^\mf{e}\circ\xi_S^\mf{e}\circ{^L\varphi}_{\gamma^\mf{e},\gamma}(t \rtimes w)=\xi_S(t a_S(w) \rtimes w)$ for all $t \rtimes w \in \hat S^b \rtimes W_F$, where $\xi_S : {^LS} \to {^LG}$ is the unique $L$-embedding extending $\xi_S^1$ by Lemma \ref{lem:lembex}. The property $\tx{Ad}(\tilde s^\mf{e})\xi^\mf{e}=\xi^\mf{e}$ and the above equation immediately imply that $\tilde s^\mf{e}$ commutes both with $\xi_S^1(\hat S_b)$ and with $w \mapsto \xi_S(a_S(w) \rtimes w)$. Commuting with $\xi_S^1(\hat S^b)$ implies $\tilde s^\mf{e} = s \rtimes b$ with $s \in \xi_S(\hat S)$. Let $s_S \in \hat S$ be the preimage of $s$ under $\xi_S$. The commuting of $\tilde s$ with $\xi_S(a_S(w) \rtimes w)$ is then equivalent to  $(1-b)a_S(w)=s_S \cdot (\sigma_S(w)s_S)^{-1}$, which says that $(a_S^{-1},s_S) \in Z^1(W_F,(1-b):\hat S \rw \hat S)$. We let $A_0$ be the class of $(a_S^{-1},s_S)$.

In \cite[\S A.3]{KS99} a pairing was defined between the cohomology groups $H^1(W_F,(1-b):\hat S \rw \hat S)$ and $H^1(\Gamma,(1-b^{-1}):S \rw S)$. We define 
\begin{equation} \label{eq:d3}
	\Delta_{III}^\tx{new}(\gamma^\mf{z},\tilde\delta)=\<\tx{inv}(\gamma,(z,\delta)),A_0\>. 	
\end{equation}

\begin{rem} \label{rem:d3}
	In ordinary endoscopy, the transfer factor constructed in \cite{LS87} contains two pieces $\Delta_{III_1}$ and $\Delta_{III_2}$. In twisted endoscopy these two pieces are glued together into the factor $\Delta_{III}$, which in general cannot be written as a product of two separate pieces. The reason for that was discussed in Remark \ref{rem:norm}: if $(S,\gamma)$ is a norm for $\tilde\delta$, then the pair $(z^*,\delta^*)$ in the equivalence class of $(z,\delta)$ with $\delta^* \in S(\bar F)$ and $\delta^* \mapsto \gamma$ cannot in general be chosen to satisfy in addition $\delta^* \in S(F)$.
	
	In the special case when this is possible, we have seen in Remark \ref{rem:inv} that then the class $\tx{inv}(\gamma,(z,\delta))$ breaks up into a product of classes $([z^*],[\delta^*]) \in H^1(\Gamma,S^b) \times H^0(\Gamma,S)$. Therefore, pairing $\tx{inv}(\gamma,(z,\delta))$ with $A_0$ is the same as pairing $([z^*],[\delta^*])$ with the image of $A_0$ under the natural map $H^1(W_F,(1-b):\hat S \rw \hat S) \to H^1(W_F,\hat S) \times H^0(W_F,\hat S_b)$. This image is the pair of classes $([a_S]^{-1},[s_b])$, where $s_b$ is the image of $s_S$ under $\hat S \to \hat S_b$. Then we obtain
	\[ \Delta_{III}^\tx{new}(\gamma^\mf{z},\tilde\delta) = \<[z^*]^{-1},[s_b]\> \cdot \<[\delta^*], [a_S]^{-1}\>, \]
	where the first pairing is the Tate-Nakayama pairing between $H^1(\Gamma,S^b)$ and $H^0(\Gamma,\hat S_b)$, while the second is the Langlands pairing between $H^0(\Gamma,S)$ and $H^1(W_F,\hat S)$. The first term is the twisted analog of the absolute term $\Delta_{III_1}^\tx{new}$ in ordinary endoscopy, defined in \cite[\S2.2]{KalECI} for pure inner twists and in \cite[\S5.3]{KalRI} in general. Indeed, by construction we have $[z^*]=\tx{inv}(\tilde\delta^*,\tilde\delta)$, where $\tilde\delta^*=\delta^* \rtimes a$. The second term is the twisted analog of the absolute term $\Delta_{III_2}$ defined in \cite[\S3.5]{LS87}. We thus have 
	\begin{equation}
		\Delta_{III}^\tx{new}(\gamma^\mf{z},\tilde\delta) = \<\tx{inv}(\tilde\delta^*,\tilde\delta)^{-1},[s_b]\> \cdot \<[\delta^*], [a_S]^{-1}\>. 
	\end{equation}
\end{rem}

\begin{lem} \label{lem:lembex}
Any $L$-embedding ${^LS_b} \to {^LG}$ extends uniquely to an $L$-embedding ${^LS} \to {^LG}$.
\end{lem}
\begin{proof}
Consider an $L$-embedding $\xi_S^1 : {^LS_b} \to {^LG}$. If an extension $\xi_S : {^LS} \to {^LG}$ were given, then for $s \rtimes w \in \hat S \rtimes W_F = {^LS}$ the equality $\xi_S(s \rtimes w)=\xi_S(s)\cdot \xi_S^1(w)$ would hold. Two different extensions of $\xi_S^1$ would thus differ by an element of $\Omega(S,G)(F)$ that induces a trivial action on $S_{b}$, equivalently on $S^{b}$, but this only holds for $1 \in \Omega(S,G)(F)$. This shows uniqueness. For existence we fix Borel pairs $(\hat T,\hat B)$ of $\hat G$ and $(T,B)$ of $G$, stable under $\Gamma$ and $b$. Then $(\hat T^b,\hat B^b)$ is a $\Gamma$-invariant Borel pair of $\hat G^{b,\circ}$. Fix a $b$-invariant Borel subgroup $C$ of $G$ defined over $\bar F$ and containing $S$, and let $g \in G^{b,\circ}$ conjugate $(T,B)$ to $(S,C)$. Composing the dual of $\tx{Ad}(g) : T \to S$ with the natural identification of the dual of $T$ with $\hat T$ given by $B$ and $\hat B$ gives an admissible isomorphism $\xi_S : \hat S \to \hat T$. Its restriction $\hat S^{b,\circ} \to \hat T^{b,\circ}$ is also an admissible isomorphism, and after conjugating $\xi_S^1$ within $\hat G^{b,\circ}$ we can arrange that this latter isomorphism coincides with $\xi_S^1$. Let $\omega_\sigma \in \Omega(T,G)^b=\Omega(\hat T,\hat G)^b$ be the image of $g^{-1}\sigma(g) \in N(S^b,G^{b,\circ})$. Then the transport via $\xi_S$ of the action of $w \in W_F$ on $\hat S$ is given by $\omega_{\sigma_w} \rtimes w$ on $\hat T$. The same is true for $\xi_S^1$ and we see that $\xi_S^1(1 \rtimes w) \in N(\hat T^{b,\circ},\hat G^{b,\circ}) \rtimes W_F$ lifts $\omega_{\sigma_w} \rtimes w$. It follows that $\tx{Ad}(\xi_S^1(1 \rtimes w))\xi_S(s) = \xi_S(wsw^{-1})$, hence $s \rtimes w \mapsto \xi_S(s) \cdot \xi_S^1(1 \rtimes w)$ is an $L$-embedding extending $\xi_S^1$.
\end{proof}

\begin{fct} \label{fct:lembimg}
Let $\mc{G}$ be an extension of $W_F$ by $\hat G$ and let $\xi : {^LS} \to \mc{G}$ be an $L$-embedding.
\begin{enumerate}
	\item The image of $\xi$ is the subgroup of $\mc{G}$ defined by
\[ \mc{S} = \{x \in \mc{G}| \forall s \in \hat S : x\xi(s \rtimes 1)x^{-1} = \xi( \sigma_x(s) \rtimes 1),\} \]
where $\sigma_x \in \Gamma$ is the image of $x$.
\item In particular, the image of $\xi$ depends only on the restriction of $\xi$ to $\hat S$.
\item $\xi$ is a homeomorphism onto its image.
\end{enumerate}
\end{fct}
\begin{proof}
Certainly the image of $\xi$ is contained in $\mc{S}$. Given $x \in \mc{S}$ let $w$ be its image in $W_F$ and consider $x' = \xi(1 \rtimes w) \in \mc{S}$. Then $x'x^{-1} \in \hat G$ commutes with $\xi(\hat S)$ and thus belong to $\xi(\hat S)$, i.e. $x=\xi(s \rtimes w)$. The second point is immediate from the first, and the third follows from the open mapping theorem and the fact that $^LS$ is locally compact, Hausdorff, and $\sigma$-compact.
\end{proof}

\begin{cor} \label{cor:lembcomp}
There exists a 1-cocycle $a_S : W_F \rw \hat S$ such that the $L$-embeddings $\xi^\mf{e}\circ\xi_S^\mf{e}\circ{^L\varphi}_{\gamma^\mf{e},\gamma}$ and $\xi_S \circ \tilde a_S$ are $\hat G$-conjugate.
\end{cor}
\begin{proof}
Let $\xi_S : {^LS} \to {^LG}$ and $\xi_S' : {^LS} \to {^LG}$ be the unique extensions of $\xi_S^1$ and $\xi^\mf{e}\circ\xi_S^\mf{e}\circ{^L\varphi}_{\gamma^\mf{e},\gamma}$ given by Lemma \ref{lem:lembex}. Their restrictions to $\hat S$ are $\hat G$-conjugate, so we assume they are equal. It follows from Fact \ref{fct:lembimg} that $\xi_S$ and $\xi_S'$ have the same image and are homeomorphisms onto it, so we may form $\xi_S' \circ \xi_S^{-1}$. This is an automorphism of the topological group $\hat S \rtimes W_F$ restricting to the identity on both $\hat S$ and $W_F$. Thus it is given by multiplication by $a_S \in Z^1(W_F,\hat S)$.
\end{proof}

\subsection{Normalized factor $\Delta_{KS}$ with $z$-pair} \label{sub:pure_tf2}

We now drop the assumption that there exists an $L$-isomorphism $^LG^\mf{e} \cong \mc{G}^\mf{e}$ and instead choose a $z$-pair $\mf{z}=(G^\mf{z},\xi^{\mf{z}})$. We denote by $S^\mf{z}$ the centralizer of $\gamma^\mf{z}$. Let $S^\mf{z}_1$ be the fiber product of $S \to S_{b} \cong S^\mf{e} \lw S^\mf{z}$. The automorphism $b \times \tx{id}$ of $S \times S^\mf{z}$ preserves $S^\mf{z}_1$ and we denote the automorphism it induces by $b_1$. It fixes the kernel of $S^\mf{z}_1 \to S$ pointwise. Hence the endomorphism $(1-b_1^{-1})$ of $S^\mf{z}_1$ induces a homomorphism $(1-b_1^{-1}) : S \to S^\mf{z}_1$. We are going to refine the invariant $\tx{inv}(\gamma,(z,\delta)) \in H^1_{b^{-1}}(\Gamma,S \rrw S) \cong H^1(\Gamma,(1-b^{-1}) : S \rw S)$ constructed above to an element $\tx{inv}(\gamma^\mf{z},(z,\delta)) \in H^1(\Gamma,(1-b_1^{-1}) : S \rw S^\mf{z}_1)$. If $(\tilde z^*,\tilde\delta^*)$ is a representative of the $G$-conjugacy class of $(\tilde z,\tilde \delta)$ as in Lemma \ref{lem:inv}, then $\delta^\mf{z} = (\delta^*,\gamma^\mf{z})$ belongs to $S^\mf{z}_1(\bar F)$ and satisfies $(b_1^{-1}-1)z^*(\sigma)=(\delta^\mf{z})^{-1}\sigma(\delta^\mf{z})$, so $(z^{*,-1},\delta^\mf{z}) \in Z^1(\Gamma,(1-b_1^{-1}) : S \rw S^\mf{z}_1)$ and its class is the invariant $\tx{inv}(\gamma^\mf{z},(z,\delta))$ we want.

This invariant will be paired with an element $A_0 \in H^1(W_F,(1-b) : \hat S_1^\mf{z} \to \hat S)$, whose construction is essentially the one given in \cite[\S4.4]{KS99}. We have as above the $L$-embeddings $\xi_S^1 : {^LS}_{b} \to {^LG^1}$ and $\xi_S^\mf{e} : {^LS}^\mf{e} \rw {^LG^\mf{e}}$. They will become part of a diagram as follows
\[ \xymatrix{
^LS\ar[drr]^{\xi_S}\\^LS_{b}\ar[r]_{\xi_S^1}\ar[dd]^{{^L\varphi_{\gamma^\mf{e},\gamma}}}\ar[u]&{^LG^1}\ar[r]&{^LG}\\
&\mc{G}^\mf{e}\ar[ur]_{\xi^\mf{e}}\ar[dr]^{\xi^\mf{z}}&&\mc{S}^\mf{e}\ar@{_(->}[ll]\ar@/_3pc/[uulll]_\beta\ar@/^3pc/[ddlll]^{\alpha_0}\\
^LS^\mf{e}\ar[r]^{\xi_S^\mf{e}}\ar[d]&{^LG^\mf{e}}\ar[r]&{^LG^\mf{z}}\\
^LS^\mf{z}\ar[urr]_{\xi_S^\mf{z}}
}
\]
All arrows are $L$-embeddings. The unnamed arrows $^LG^1 \to {^LG}$, $^LG^\mf{e} \to {^LG^\mf{z}}$, $^LS_{b} \to {^LS}$, and ${^LS^\mf{e}} \to {^LS^\mf{z}}$ are the canonical ones. The arrows $\xi_S$ and $\xi_S^\mf{z}$ are the unique ones extending $\xi_S^1$ and $\xi_S^\mf{e}$ by Lemma \ref{lem:lembex}. We would like to apply Corollary \ref{cor:lembcomp}, but unfortunately have no embedding $^LG^\mf{e} \to {^LG}$. Instead, following Fact \ref{fct:lembimg} we define
\[ \mc{S}^\mf{e} = \{x \in \mc{G}^\mf{e} | \forall t \in \hat S : \xi^\mf{e}(x) \xi_S(t \rtimes 1)\xi^\mf{e}(x)^{-1} = \xi_S(\sigma_x(t) \rtimes 1) \}, \]
which would be the image of $\xi_S^\mf{e}$ if we had an identification $^LG^\mf{e} \cong \mc{G}^\mf{e}$, which we do not. It is still an extension of $W_F$ by $\hat S^\mf{e}$ and Fact \ref{fct:lembimg} implies $\xi^\mf{z}(\mc{S}^\mf{e}) \subset \xi_S^\mf{z}({^LS^\mf{z}})$ and $\xi^\mf{e}(\mc{S}^\mf{e}) \subset \xi_S({^LS})$; in fact the latter statement can be strengthened to $\mc{S}^\mf{e} = (\xi^\mf{e})^{-1}(\xi_S({^LS}))$. Applying again the open mapping theorem we obtain $L$-embeddings $\alpha_0 : \mc{S}^\mf{e} \to {^LS^\mf{z}}$ and $\beta : \mc{S}^\mf{e} \to {^LS}$. Compose $\alpha_0$ with the $L$-automorphism of ${^LS^\mf{z}}$ given by inversion on $\hat S^\mf{z}$ to obtain $\alpha : \mc{S}^\mf{e} \to {^LS^\mf{z}}$, and consider $\alpha \times \beta : \mc{S}^\mf{e} \to {^L(S \times S^\mf{z})}$. Its composition with ${^L(S \times S^\mf{z})} \to {^LS^\mf{z}_1}$ kills $\hat S^\mf{e} \subset \mc{S}^\mf{e}$, thus descends to an $L$-homomorphism $\tilde a_S : W_F \to {^LS^\mf{z}_1}$, i.e. a 1-cocycle $a_S : W_F \to \hat S^\mf{z}_1$. As before one checks that $\tilde s^\mf{e} = \xi_S(s_S) \rtimes b$ and $(a_S^{-1},s_S) \in Z^1(W_F,(1-b_1) : \hat S^\mf{z}_1 \to \hat S)$, and we define $A_0$ to be the class of this element.

As in Subsection \ref{sub:pure_tf1} we define $\Delta_{III}^\tx{new}(\gamma^\mf{z},\tilde\delta)$ to be the pairing of $\tx{inv}(\gamma^\mf{z},(z,\delta))$ and $A_0$.

\subsection{Transfer of functions} \label{sub:trans}

In \S\ref{sub:pure_tf1} and \S\ref{sub:pure_tf2} we defined a factor $\Delta_{III}^\tx{new}$, which leads to the factor $\Delta_{KS}[\mf{w},\mf{e},\mf{z}]$ via \eqref{eq:pure_tf1}, which in turn leads to the factor $\Delta[\mf{w},\mf{e},\mf{z}]$ via \eqref{eq:pure_tf}.

Let $f \in \mc{C}^\infty_c(\tilde G_z(F))$. For any $\tilde\delta \in \tilde G_z(F)$ we can form the integral of $f$ over the $\tilde G_z(F)$-conjugacy class of $\tilde\delta$, after fixing an invariant measure on this conjugacy class. We will call this integral $O_{\tilde\delta}(f)$. Since the $\tilde G_z(F)$-conjugacy class of $\delta$ decomposes as a disjoint union of finitely many $G_z(F)$-conjugacy classes, $O_{\tilde\delta}(f)$ is a sum of finitely many twisted orbital integrals.

\begin{lem} \label{lem:trans} For any function $f \in \mc{C}^\infty_c(\tilde G_z(F))$ there exists a function $f^\mf{z} \in \mc{H}(G^\mf{z})$ such that for all strongly regular $\gamma^\mf{z} \in G^\mf{z}(F)$ we have
\[ SO_{\gamma^\mf{z}}(f^\mf{z}) = \sum_{\tilde\delta} \Delta[\mf{w},\mf{e},\mf{z}](\gamma^\mf{z},\tilde\delta) O_{\tilde\delta}(f) \]
where the sum runs over the set of strongly regular $\tilde G_z(F)$-conjugacy classes in $\tilde G_z(F)$. More precisely, if $f^{\mf{z},KS}$ is the function that satisfies
\[ SO_{\gamma^\mf{z}}(f^{\mf{z},KS}) = \sum_{\tilde\delta} \Delta_\tx{KS}[\mf{w},\mf{e},\mf{z}](\gamma^\mf{z},\tilde\delta) O_{\tilde\delta}(f), \]
where now $\tilde\delta$ runs over the strongly regular semi-simple elements in $[G\rtimes a]_z(F)$ modulo $G_z(F)$-conjugacy and $a^{-1} \in A$ is the image of $\tilde s^\mf{e}$, then $f^\mf{z}=f_0^{\mf{z},KS}$, where $f_0(\tilde \delta)=\sum_{c \in \tilde G_z(F)/G_z(F)}f(c^{-1}\tilde\delta c)$.
\end{lem}

\begin{proof}
This follows immediately from the deep results on geometric transfer in twisted endoscopy due to Shelstad \cite{She12} in the archimedean case and Ngo \cite{Ngo10} and Waldspurger \cite{Wal97}, \cite{Wal08} in the non-archimedean case. Indeed, in $\sum_{\tilde\delta} \Delta[\mf{w},\mf{e},\mf{z}](\gamma^\mf{z},\tilde\delta) O_{\tilde\delta}(f)$ we are summing over $\tilde G_z(F)$-conjugacy classes in $\tilde G_z(F)$, and then integrating over each such class. We may equally well sum over $G_z(F)$-conjugacy classes in $\tilde G_z(F)$, and then integrate over each such class. After this reparameterization, we plug in \eqref{eq:pure_tf} and use the fact that $\Delta_{KS}$ is invariant under $G_z(F)$-conjugation in the variable $\tilde\delta$ to switch the sums over $c$ and $\tilde\delta$. This brings the right hand side to
\[ \sum_{c \in \tilde G_z(F)/G_z(F)} \sum_{\tilde\delta \in \tilde G_z(F)/G_z(F)-conj}\!\!\!\!\!\!\!\!\!\!\!\!\Delta_{KS}[\mf{w},\mf{e},\mf{z}](\gamma^\mf{z},c\tilde\delta c^{-1})\int_{x \in G_z(F)/G_z(F)_{\tilde\delta}}\!\!\!\!\!\!\!\!f(x\tilde\delta x^{-1})dx. \]
Changing variables to replace $\tilde\delta$ and $x$ by $c^{-1}\tilde\delta c$ and $cxc^{-1}$ and moving the sum over $c$ to the right we obtain
\[ \sum_{\tilde\delta \in \tilde G_z(F)/G_z(F)-conj}\!\!\!\!\!\!\!\!\!\!\!\!\Delta_{KS}[\mf{w},\mf{e},\mf{z}](\gamma^\mf{z},\tilde\delta)\int_{x \in G_z(F)/G_z(F)_{\tilde\delta}}\sum_{c \in \tilde G_z(F)/G_z(F)}\!\!\!\!\!\!\!\!f(c^{-1}x\tilde\delta x^{-1}c)dx. \]
Let $a^{-1} \in A$ be the image of $\tilde s^\mf{e}$. Then $\Delta_{KS}=0$ unless $\tilde\delta \in [G \rtimes a]_z(F)$. Fix $\tilde\delta_0 \in [G \rtimes a]_z(F)$ and write $\theta=\tx{Ad}(\tilde\delta_0)$. Let $f_0(\delta)=\sum_{c \in \tilde G_z(F)/G_z(F)} f(c^{-1}\delta\tilde\delta_0 c)$.
Then we obtain
\[ \sum_{\delta \in G_z(F)/\theta-conj} \Delta_{KS}[\mf{w},\mf{e},\mf{z}](\gamma^\mf{z},\delta\tilde\delta_0)\int_{x \in G_z(F)/G_z(F)_{\delta\theta}}f_0(x\delta\theta(x^{-1})). \] 
By construction $\Delta_{KS}[\mf{w},\mf{e},\mf{z}](\gamma^\mf{z},\delta\tilde\delta_0)$ is a normalization of the Kottwitz-Shelstad transfer factor for the twisted group $(G_z,\theta)$ and its twisted endoscopic datum $\mf{e}$, evaluated at $(\gamma^\mf{z},\delta)$. The results of Shelstad, Ngo, and Waldspurger now imply the existence of a function $f^\mf{z}$ so that the above formula becomes equal to $SO_{\gamma^\mf{z}}(f^\mf{z})$.
\end{proof}

\subsection{Character identities} \label{sub:pure_charid}

Consider a parameter $\phi : L_F \to {^LG}$ and a semi-simple element $\tilde s \in \tilde S_\phi^{[z]}$. The pair $(\phi,\tilde s)$ leads to an endoscopic datum  $\mf{e}=(G^\mf{e},\mc{G}^\mf{e},\tilde s,\xi^\mf{e})$ by the spectral construction described in Subsection \ref{sub:endocnst}. Choose a $z$-pair $(G^\mf{z},\xi^\mf{z})$ and let $\phi^\mf{z}=\xi^\mf{z}\circ\phi$, a tempered parameter for $G^\mf{z}$. We assume the existence of an $L$-packet $\Pi_{\phi^\mf{z}}$ on $G^\mf{z}(F)$ and of its stable character $S\Theta_{\phi^\mf{z}}$. Let us write the injection $\tilde \Pi_{\phi,z} \to \tx{Irr}(\pi_0(\tilde S_\phi^{[z]}),{[z]})$ from Conjecture \ref{cnj:llc_pure_is} as $\tilde\pi \mapsto \rho_{\tilde\pi}$.

\begin{cnj} \label{cnj:llc_pure_ci}
For any pair of functions $f$ and $f^\mf{z}$ as in Lemma \ref{lem:trans} we have
\begin{equation} \label{eq:charid} S\Theta_{\phi^\mf{z}}(f^\mf{z}) = e(G_z)\sum_{\tilde\pi \in \tilde \Pi_{\phi,z}} \tx{tr}{\rho_{\tilde\pi}}(\tilde s) \cdot \Theta_{\tilde\pi_{\tilde\rho}}(f). \end{equation}
\end{cnj}

As we have already remarked in \S\ref{sub:pure_tf}, equation \eqref{eq:charid} applied to the connected case $\tilde G=G$ recovers equation \cite[(5.9),(5.11)]{KalRI}.

We will refer to Conjectures \ref{cnj:llc_pure_is} and \ref{cnj:llc_pure_ci} together as the \emph{refined local Langlands conjecture for pure inner forms of quasi-split disconnected groups}.

\begin{rem}
	The reader might wonder if $e(G_z)$ is the correct sign to use in \eqref{eq:charid}, given that we are now working with a disconnected group, and whether a generalization of this sign to the disconnected case is necessary. Indeed, certain new signs appear in the harmonic analysis of disconnected groups, but they can be decomposed into a product of the sign $e(G_z)$ associated to the identity component of $\tilde G_z$ and another sign associated to a particular coset of $G_z$ in $\tilde G_z$. We have chosen to absorb this second sign into the internal structure of $L$-packets for disconnected groups, so that it does not appear in \eqref{eq:charid}. Details about this second sign will appear in \cite{KM} and \cite{KalSteinberg}.
\end{rem}

\begin{rem}
Let $a \in A$ be the image of $\tilde s$. If the function $f$ is supported away from the $G_z(F)$-cosets in $\tilde G_z(F)$ which are $A$-conjugate to $a^{-1}$, then $f^\mf{z}=0$. Thus the conjecture contains the statement that the right hand side is also zero in this case.	
\end{rem}

\begin{rem} Let $\chi : A \to \C^\times$ be a character. Then $\tilde\pi_{\chi\otimes\tilde\rho}=\chi\otimes\tilde\pi_{\tilde\rho}$. If $a \in A$ is the image of $\tilde s$ and $f$ is a function on $\tilde G_z
(F)$ supported on the $G_z(F)$-coset of $b \in A$, then $\tx{tr}(\chi\otimes\tilde\rho)(\tilde s)=\chi(a)\tx{tr}(\tilde\rho)(\tilde s)$, while $\Theta_{\chi\otimes\tilde\pi_{\tilde\rho}}(f)=\chi(b)\Theta_{\tilde\pi_{\tilde\rho}}(f)$. From this it follows that the right hand side above is zero if $f$ is supported only on cosets for $b \in A$ such that $ab \neq 1$ in $A^\tx{ab}$.
\end{rem}

\section{The conjecture for rigid inner forms} \label{sec:rigid}

In the preceding section, we introduced a refined local Langlands conjecture for pure inner forms of quasi-split disconnected groups. Those are inner forms of quasi-split disconnected groups $G \rtimes A$ that arise from $H^1(\Gamma,G)$. A general inner form of $G \rtimes A$ arises from $H^1(\Gamma,G/Z(G)^A)$ and in this section we are going to extend the conjecture from pure inner forms to general inner forms. Just like in the connected setting, the notion of an inner form needs to be rigidified. For this we can use the cohomology set $H^1(u \to W,Z(G)^A \to G)$ defined in \cite{KalRI}, see also \cite{KalRIBG}. However, in order to normalize the transfer factors in the disconnected case, we shall need a generalization of this cohomology set to complexes of tori of length 2, as well as a Tate-Nakayama duality theorem for this generalization. This will be the concern in the first two subsections below. Thankfully, what is needed is little more than a combination of the arguments of \cite{KalRI} and \cite[App. A]{KS99}.

\subsection{Definitions of hyper(co)homology groups} \label{sub:coho}

Consider a complex $Z \to T \to U$, where $T$ and $U$ are tori, $Z$ is finite, and $Z \to T$ is injective. We write $f$ for the map $T \to U$, and leave the map $Z \to T$ unnamed. Let $\bar T$ be the quotient $T/Z$. The map $f$ induces a map $\bar f : \bar T \to U$.

We shall first define and study a cohomology group $H^1(u \to W,Z \to T \to U)$ that combines the group $H^1(u \to W,Z \to T)$ of \cite{KalRI} and the group $H^1(\Gamma,T \to U)$ of \cite[App. A]{KS99}. Define $Z^1(u \to W,Z \to T \to U)$ to consists of pairs $z \in Z^1(u \to W,Z \to T)$ and $c \in C^0(\Gamma,U)$ such that $\bar f(\bar z)=\partial c$, where $\bar z \in Z^1(\Gamma,\bar T)$ is the image of $z$. Define $H^1(u \to W,Z \to T \to U)$ to be the quotient of $Z^1(u \to W,Z \to T \to U)$ by the subgroup $B^1(\Gamma,T \to U)$ consisting of $\{(t^{-1}\sigma(t), f(t))|t \in T(\bar F)\}$.

This definition involves a particular choice of extension $1 \to u \to W \to \Gamma \to 1$ in the distinguished isomorphism class. Just like in the case of $H^1(u \to W,Z \to G)$, the cohomology set $H^1(u \to W,Z \to T \to U)$ is independent of that choice, in that there is a unique isomorphism between the two versions of it coming from two choices of extensions. The argument is as follows. It is enough to show that an automorphism of the extension $W$ acts trivially on $H^1(u \to W,Z \to T \to U)$. The vanishing of $H^1(\Gamma,u)$ asserted in \cite[Theorem 3.1]{KalRI} implies that such an automorphism is of the form $\tx{Ad}(x)$ for some $x \in u$. An element $(z,u) \in Z^1(u \to W,Z \to T \to U)$ is sent by $\tx{Ad}(x)$ to $(z',u)$ where $z'(w)=z(xwx^{-1})=z(x \cdot \sigma(x^{-1}))z(w)=z(x)\cdot \sigma(z(x)))^{-1}\cdot z(w)$, where $\sigma \in \Gamma$ is the image of $w$. So the difference between $(z,u)$ and $(z',u)$ is measured by $(z(x)\cdot \sigma(z(x))^{-1},1) \in B^1(\Gamma,T \to U)$. We are using here that $f|_Z=1$.

We have the following analog of \cite[(3.6)]{KalRI}:

\begin{equation} \label{eq:bfd2}
\resizebox{\linewidth}{!}{$
\xymatrix{
&\bar T(F)\ar[d]\ar@{=}[r]&\bar T(F)\ar[d]\\
1\ar[r]&H^1(\Gamma,Z)\ar[r]^-{\tx{Inf}}\ar[d]&H^1(u \rw W,Z \rw Z)\ar[r]^-{\tx{Res}}\ar[d]&\mathrm{Hom}(u,Z)^\Gamma\ar@{=}[d]\\
1\ar[r]&H^1(\Gamma,T \to U)\ar[r]^-{\tx{Inf}}\ar@{=}[d]&H^1(u\rw W,Z\to T \to U)\ar[r]^-{\tx{Res}}\ar[d]^a&\mathrm{Hom}(u,Z)^\Gamma\ar[d]\ar[r]&H^2(\Gamma,T \to U)\ar@{=}[d]\\
&H^1(\Gamma,T \to U)\ar[r]&H^1(\Gamma,\bar T \to U)\ar[r]\ar[d]&H^2(\Gamma,Z)\ar[d]\ar[r]&H^2(\Gamma,T \to U)\\
&&1&1
}$} \end{equation}

We also have the following analog of the long exact sequence \cite[(A.1.1)]{KS99}
\[ \begin{aligned}
0&\to H^0(\Gamma,Z)\to H^0(\Gamma,T)\to H^0(\Gamma,U)\to\\
&\to H^1(u \to W,Z \to T \to U)\to H^1(u \to W,Z \to T)\to H^1(\Gamma,U)\to\\
&\to H^2(\Gamma,\bar T \to U)\to H^2(\Gamma,\bar T)\to H^2(\Gamma,U)\to\\
&\to\dots
\end{aligned}
\]
Note that the kernel of $H^1(u \to W,Z \to T \to U)\to H^1(u \to W,Z \to T)$ lies in the subgroup $H^1(\Gamma,T \to U)$ of $H^1(u \to W,Z \to T \to U)$ and the map $H^1(u \to W,Z \to T)\to H^1(\Gamma,U)$ factors through the surjection $H^1(u \to W,Z \to T)\to H^1(\Gamma,\bar T)$. The difference between \cite[(A.1.1)]{KS99} and \eqref{eq:bfd2} is that we have replaced $H^1(\Gamma,T \to U)$ by $H^1(u \to W,Z \to T \to U)$, $H^i(\Gamma,T \to U)$ by $H^i(\Gamma,\bar T \to U)$, and $H^i(\Gamma,T)$ by $H^i(\Gamma,\bar T)$, for $i>1$.

Finally let $K$ and $C$ be the kernel and cokernel of $f$, respectively, so that we have an exact sequence $1 \to K \to T \to U \to C \to 1$ of diagonalizable groups. By assumption $Z \subset K$ and thus we also have $1 \to \bar K \to \bar T \to U \to C$, where $\bar K=K/Z$. We have the commutative diagram with exact rows
\[ 
\resizebox{\linewidth}{!}{$
\xymatrix{
0\ar[r]&H^1(\Gamma,K)\ar[r]\ar@{^(->}[d]&H^1(\Gamma,T \to U)\ar[r]\ar@{^(->}[d]&H^0(\Gamma,C)\ar[r]\ar@{=}[d]&H^2(\Gamma,K)\ar[d]\\
0\ar[r]&H^1(u \to W,Z \to K)\ar[r]\ar@{->>}[d]&H^1(u \to W,Z \to T \to U)\ar[r]\ar@{->>}[d]&H^0(\Gamma,C)\ar[r]\ar@{=}[d]&H^2(\Gamma,\bar K)\ar@{=}[d]\\
0\ar[r]&H^1(\Gamma,\bar K)\ar[r]&H^1(\Gamma,\bar T \to U)\ar[r]&H^0(\Gamma,C)\ar[r]&H^2(\Gamma,\bar K)
}$}
\]

Next we shall define a functor $\bar Y_{+,\tx{tor}}(Z \to T \to U)$ that combines the functor $\bar Y_{+,\tx{tor}}(Z \to T)$ of \cite{KalRI} and the homology groups $H_0(W_{K/F},X_*(T) \to X_*(U))_0$ of \cite[App. A]{KS99}. Consider the homomorphism $f_* : X_*(T) \to X_*(U)$. The assumption $Z \subset \tx{ker}(f)$ implies that this homomorphism extends (necessarily uniquely) to $f_* : X_*(\bar T) \to X_*(U)$. We consider this as a complex placed in degrees $0$ and $1$. For every finite Galois extension $K/F$ splitting $T$ and $U$ we have the hyperhomology groups $H_0(W_{K/F},X_*(T) \to X_*(U))$ and $H_0(W_{K/F},X_*(\bar T) \to X_*(U))$, as well as their subgroups $H_0(-)_0$ defined in \cite[App. A.3]{KS99}. Let us recall some details. The group of inhomogeneous $n$-chains $C_n(W_{K/F},X_*(T))$ consists of all set-theoretic maps $W_{K/F}^n \to X_*(T)$ with finite support. If $y$ is such a map, its differential $\partial y : W_{K/F}^{n-1} \to X_*(T)$ is given by
\begin{eqnarray*}
\partial y (w_1,\dots,w_{i-1})&=&\sum_x x^{-1}y(x,w_1,\dots,w_{n-1})\\
&+&\sum_{i=1}^{n-1}(-1)^i\sum_x y(w_1,\dots,w_{i-1},w_ix^{-1},x,w_{i+1},\dots,w_{n-1})\\
&+&(-1)^n\sum_x y(w_1,\dots,w_{n-1},x),
\end{eqnarray*}
where $x$ runs over $W_{K/F}$. The group $Z_0(W_{K/F},X_*(T) \to X_*(U))$ has the explicit description as the set of pairs $\{(\lambda,\mu_1)| \lambda \in C_0(W_{K/F},X_*(T)),\mu_1 \in C_1(W_{K/F},X_*(U)),f_*(\lambda)=\partial \mu_1)$, while the group $B_0(W_{K/F},X_*(T) \to X_*(U))$ is given by $\{(\partial \lambda_1,f_*(\lambda_1)-\partial\mu_2)|\lambda_1 \in C_1(W_{K/F},X_*(T)),\mu_2 \in C_2(W_{K/F},X_*(U))\}$. Then $H_0=Z_0/B_0$. The subgroup $Z_0(-)_0$ consists of those $(\lambda,\mu_1)$ satisfying in addition $N_{K/F}\lambda=0$, and $H_0(-)_0=Z_0(-)_0/B_0$. Note that $H_0(W_{K/F},X_*(T))_0=H^{-1}_\tx{Tate}(\Gamma_{K/F},X_*(T))=[X_*(T)/IX_*(T)]_\tx{tor}$, where $I$ is the augmentation ideal in $\Gamma_{K/F}$, or equivalently in $\Gamma$. In \cite{KalRI} we used the notation $Y_\tx{tor}(T)$ for this finite abelian group.

\begin{fct} \label{fct:tn++esy1}
We have the exact sequence
\[ \begin{aligned}
&H_1(W_{K/F},X_*(T))\to H_1(W_{K/F},X_*(U))\to\\
\to&H_0(W_{K/F},X_*(T)\to X_*(U))_0 \to Y_\tx{tor}(T) \to Y_\tx{tor}(U).	
\end{aligned}
 \]
\end{fct}
\begin{proof}
Left to the reader.
\end{proof}

We define $H_0(W_{K/F},X_*(T) \to X_*(\bar T) \to X_*(U))_0=Z_0(W_{K/F},X_*(\bar T) \to X_*(U))_0/B_0(W_{K/F},X_*(T) \to X_*(U))$.

\begin{fct} \label{fct:tn++esy2}
We have the exact sequence
\[ \begin{aligned}
&H_1(W_{K/F},X_*(T))\to H_1(W_{K/F},X_*(U)) \to \\
\to&H_0(W_{K/F},X_*(T) \to X_*(\bar T) \to X_*(U))_0
\to \bar Y_{+,\tx{tor}}(Z \to T) \to Y_\tx{tor}(U).	
\end{aligned}
 \]
\end{fct}
\begin{proof}
Left to the reader.
\end{proof}

There is a coinflation map $C_n(W_{L/F},X_*(T)) \to C_n(W_{K/F},X_*(T))$ for a tower $L/K/F$ defined by
\[ \tx{coinf}y(w_1,\dots,w_n) = \sum_{\dot w_i \in p^{-1}(w_i)}y(\dot w_1,\dots,\dot w_n), \]
where $p : W_{L/F} \to W_{K/F}$ is the natural projection. This map respects differentials and induces a corresponding map 
\[ H_n(W_{L/F},X_*(T) \to X_*(U)) \to H_n(W_{K/F},X_*(T) \to X_*(U)). \] 
It maps $H_0(W_{L/F},X_*(T) \to X_*(U))_0$ to $H_0(W_{K/F},X_*(T) \to X_*(U))_0$, this relies on the torsion-freeness of $X_*(T)$.

\begin{fct} \label{fct:ks1} Consider a tower of finite Galois extensions $L/K/F$ and assume $K$ splits $T$ and $U$. Then the following diagram commutes
\[ \xymatrix{
H_0(W_{L/F},X_*(T) \to X_*(U))_0\ar[r]\ar[d]&H^1(K/F,T(K) \to U(K))\\
H_0(W_{K/F},X_*(T) \to X_*(U))_0\ar[r]&H^1(L/F,T(L) \to U(L))\ar[u]\\
}
\]
where the left map is coinflation, the right map is inflation, and the horizontal maps are the isomorphisms \cite[(A.3.4)]{KS99}. Both vertical maps are isomorphisms.
\end{fct}
\begin{proof}
This is diagram \cite[(A.3.11)]{KS99}, and its commutativity is proved there. The fact that inflation is an isomorphisms follows from the 5-lemma applied to the exact sequence
\[\resizebox{0.98\textwidth}{!}{$T(F) \to U(F) \to H^1(K/F,T(K) \to U(K)) \to H^1(K/F,T(K)) \to H^1(K/F,U(K))$} \]
and its $L/F$-analog. The fact that coinflation is an isomorphism follows from the commutativity of the above diagram.
\end{proof}

The coinflation map induces a map
\[ \resizebox{0.98\textwidth}{!}{$H_0(W_{L/F},X_*(T) \to X_*(\bar T) \to X_*(U))_0 \to H_0(W_{K/F},X_*(T) \to X_*(\bar T) \to X_*(U))_0$}. \]

\begin{fct} This is an isomorphism.
\end{fct}
\begin{proof} We apply the 5-lemma to the exact sequence
\[ \resizebox{0.98\textwidth}{!}{$H_1(X_*(T)) \to H_1(X_*(U)) \to H_0(X_*(T) \to X_*(\bar T) \to X_*(U))_0 \to \bar Y_{+,\tx{tor}}(T) \to Y_\tx{tor}(U)$}, \]
where we take homology of $W_{L/F}$, and then map it, via the coinflation map, to the same exact sequence but for $W_{K/F}$. For the last two terms coinflation induces the identity. For the first two terms, it is an isomorphism due to Fact \ref{fct:ks1} applied to the complexes $1 \to T$ and $1 \to U$.
\end{proof}

We define $\bar Y_{+,\tx{tor}}(Z \to T \to U)$ as the inverse limit of $H_0(W_{K/F},X_*(T) \to X_*(\bar T) \to X_*(U))_0$ with respect to coinflation.

\begin{fct} \label{fct:tn++esy3}
Let $H_1(X_*(T))$ denote the inverse limit of $H_1(W_{K/F},X_*(T))$ with respect to coinflation. We have the exact sequence
\[ H_1(X_*(T)) \to H_1(X_*(U)) \to \bar Y_\tx{tor}(Z \to T \to U) \to \bar Y_{+,\tx{tor}}(T) \to Y_\tx{tor}(U). \]
\end{fct}

Finally, we consider the dual homomorphism $\hat f : \hat U \to \hat T$. It lifts (uniquely) to a homomorphism $\hat{\bar f} : \hat U \to \hat{\bar T}$. Let $\hat Z$ be the kernel of the isogeny $\hat{\bar T} \to \hat T$, and let $\hat K$ and $\hat C$ be the kernel and cokernel of $\hat f$. Then $[\hat{\bar f}]^{-1}(\hat Z)=\hat K$.

We define the group $Z^1_\tx{cts}(W_F,\hat Z \to \hat{\bar T} \from \hat U)$ to consist of the pairs $(z,\dot c)$, where $z \in Z^1_\tx{cts}(W_F,\hat U)$ and $\dot c \in \hat{\bar T}$ satisfying $\partial c=\hat f(z)$, where $c \in \hat T$ is the image of $\dot c$. We define $B^1(W_F,\hat Z \to \hat{\bar T} \from \hat U)$ to consist of $(\partial u,\hat{\bar f}(u))$ for $u \in \hat U$, and $H^1=Z^1/B^1$. This group fits into the exact sequence
\[ H^1_\tx{cts}(W_F,\hat T) \lw H^1_\tx{cts}(W_F,\hat U) \lw H^1_\tx{cts}(W_F,\hat Z \to \hat{\bar T}\from \hat U) \lw [\hat{\bar T}]^+ \lw \hat U^\Gamma. \]
Define $H^1_\tx{cts}(W_F,\hat Z \to \hat{\bar T}\from \hat U)_\tx{red}$ to be the quotient of $H^1_\tx{cts}(W_F,\hat Z \to \hat{\bar T}\from \hat U)$ by the image of $[\hat{\bar T}]^{+,\circ}$. Then we obtain the exact sequence
\begin{equation} \label{eq:tn++esd1}
H^1_\tx{cts}(W_F,\hat T) \lw H^1_\tx{cts}(W_F,\hat U) \lw H^1_\tx{cts}(W_F,\hat Z \to \hat{\bar T}\from \hat U)_\tx{red} \lw \pi_0([\hat{\bar T}]^+) \lw \pi_0(\hat U^\Gamma). \end{equation}

We introduce on $H^1(u \to W,Z \to T \to U)$ the unique topology that makes the homomorphism $U(F) \to H^1(u \to W,Z \to T \to U)$ continuous and open. Analogously, we introduce on $\bar Y_\tx{tor}(Z \to T \to U)$ the unique topology that makes the homomorphisms $H_1(W_{K/F},X_*(U)) \to \bar Y_\tx{tor}(Z \to T \to U)$ continuous and open. Here $H_1(W_{K/F},X_*(U))$ is topologized to make the Langlands isomorphism a homeomorphism.

\subsection{Generalized Tate-Nakayama duality} \label{sub:tnd++}

We shall now define a perfect pairing
\begin{equation} \label{eq:tnd++}
H^1(u \to W,Z \to T \to U) \otimes H^1_\tx{cts}(W_F,\hat Z \to \hat{\bar T} \from \hat U)_\tx{red} \to \C^\times
\end{equation}
that generalizes the pairing \cite[(A.3.12),(A.3.16)]{KS99}, which can be seen as the special case $Z=1$. We do this in two steps -- first introducing a pairing of elementary nature between $H^1_\tx{cts}(W_F,\hat Z \to \hat{\bar T} \from \hat U)$ and $\bar Y_{+,\tx{tor}}(Z \to T \to U)$, and then an isomorphism of arithmetic nature $\bar Y_{+,\tx{tor}}(Z \to T \to U) \to H^1(u \to W,Z \to T \to U)$.

Given $(z,\dot c) \in Z^1_\tx{cts}(W_{K/F},\hat Z \to \hat{\bar T} \to \hat U)$ and $(\bar\lambda,\mu_1) \in Z_0(W_{K/F},X_*(\bar T) \to X_*(U))_0$ define $\<(z,\dot c),(\bar\lambda,\mu_1)\>_K \in \C^\times$ as
\[ \<\dot c,\bar\lambda\>_{\bar T} \cdot \prod_{w \in W_{K/F}} \<z(w),\mu_1(w)\>_U^{-1}, \]
where $\<-,-\>_{\bar T}$ is the pairing $\hat{\bar T} \times X_*(\bar T) \to \C^\times$ and $\<-,-\>_U$ is the analogous pairing for $U$. It is immediate that if $L/K/F$ is a tower of Galois extensions and $z$ is inflated from $W_{K/F}$ we have $\<(z,\dot c),\tx{coinf}(\bar\lambda,\mu_1)\>_K=\<(z,\dot c),(\bar\lambda,\mu_1)\>_L$. It is immediately checked that this pairing annihilates the (co)boundaries on both sides, as well as the image of $[\hat{\bar T}]^{+,\circ}$, and therefore induces a pairing
\begin{equation} \label{eq:elempair} H^1_\tx{cts}(W_F,\hat Z \to \hat{\bar T} \to \hat U)_\tx{red} \otimes \bar Y_\tx{tor}(Z \to T \to U) \to \C^\times \end{equation}
functorial in $Z \to T \to U$.

Recall the pairing $H^1_\tx{cts}(W_F,\hat U) \otimes H_1(W_F,X_*(U)) \to \C^\times$ that underlies the Langlands isomorphism $H^1_\tx{cts}(W_F,\hat U) \to \tx{Hom}_\tx{cts}(U(F),\C^\times)$ and the pairing $\pi_0(\hat T^\Gamma) \otimes Y_\tx{tor}(T) \to \C^\times$. The latter was generalized to $\pi_0([\hat{\bar T}]^+) \otimes \bar Y_{+,\tx{tor}}(Z \to T) \to \C^\times$ in \cite[Prop. 5.3]{KalRI}.

\begin{fct} \label{fct:tn++d1}
The pairing \eqref{eq:elempair} is compatible with the pairing $\pi_0([\hat{\bar T}]^+) \otimes \bar Y_{+,\tx{tor}}(Z \to T) \to \C^\times$, as well as the \emph{negative} of the pairing $H^1(W_F,\hat U) \otimes H_1(W_F,X_*(U)) \to \C^\times$, and induces an isomorphism
\[ H^1_\tx{cts}(W_F,\hat Z \to \hat{\bar T} \from \hat U)_\tx{red} \to \tx{Hom}_\tx{cts}(\bar Y_\tx{tor}(Z \to T \to U),\C^\times). \]
\end{fct}
\begin{proof}
The compatibility of the three pairings is immediate from the explicit formula defining \eqref{eq:elempair}. The compatibility with the negative Langlands pairing together with the definition of the topology on $\bar Y_\tx{tor}(Z \to T \to U)$ implies that the image of the resulting homomorphism 
\[ H^1_\tx{cts}(W_F,\hat Z \to \hat{\bar T} \to \hat U) \to \tx{Hom}(\bar Y_\tx{tor}(Z \to T \to U),\C^\times) \] 
lies in $\tx{Hom}_\tx{cts}(...)$. Applying the functor $\tx{Hom}_\tx{cts}(-,\C^\times)$ to the exact sequence of Fact \ref{fct:tn++esy3} produces an exact sequence: for $\tx{Hom}(-,\C^\times)$ this is because $\C^\times$ is an injective abelian group, and passing from abstract to continuous homomorphisms doesn't ruin exactness due to the definition of the topology on $\bar Y_\tx{tor}(Z \to T \to U)$. This exact sequence maps to the exact sequence \eqref{eq:tn++esd1}, with the first two maps being the negative Langlands pairing, the middle map being \eqref{eq:elempair}, and the fourth and fifth map coming from \cite[Proposition 5.3]{KalRI}. All maps except for the middle one are known to be isomorphisms, and the 5-lemma applies.
\end{proof}

We now turn to the isomorphism $\bar Y_{+,\tx{tor}}(Z \to T \to U) \to H^1(u \to W,Z \to T \to U)$. We fix as in \cite[\S4.4ff]{KalRI} an exhaustive tower $E_k/F$ of finite Galois extensions, compatible sections $s_k : \Gamma_{E_k/F} \to W_{E_k/F}$ and $\zeta_k : \Gamma_{E_k/F} \to \Gamma_{E_{k+1}/F}$, a co-final sequence $n_k$ of natural numbers, a compatible sequence $l_k : \bar F^\times \to \bar F^\times$ of $n_k$-roots. Define $c_k(\sigma,\tau)=\tx{rec}_k^{-1}(s_k(\sigma)s_k(\tau)s_k(\sigma\tau)^{-1})$. Then $\xi_k = dl_kc_k \sqcup_{E_k/F} \delta_e \in Z^2(\Gamma,u_k)$ gives rise to the extension $W_k = u_k\boxtimes_{\xi_k}\Gamma$ of $\Gamma$ by $u_k$. We have the 1-cochain $\alpha_k \in C^1(\Gamma,u_k)$ of \cite[(4.8)]{KalRI} leading to the surjective group homomorphism $f_k : W_{k+1} \to W_k$ defined by $f_k(x \boxtimes \sigma)=p(x)\alpha_k(\sigma)\boxtimes\sigma$, where $p : u_{k+1} \to u_k$ is the surjective group homomorphism of \cite[(3.2)]{KalRI}. Then $W=\varprojlim_k W_k$ is an extension of $\Gamma$ by $u$ in the distinguished isomorphism class.

We are now going to construct the isomorphism by refining and merging together the constructions of \cite[\S A.3]{KS99} and \cite[\S4.6]{KalRI}. More precisely, a central role in the constructions of \cite[\S A.3]{KS99} is played by two maps $\phi = \phi_{T,k} : C_1(W_{E_k/F},X_*(T)) \to T(E_k)$ and $\psi=\psi_{T,k} : C_0(W_{E_k/F},X_*(T))_0 \to Z^1(\Gamma_{E_k/F},T(E_k))$, where $C_0(W_{E_{k}/F},X_*(T))_0$ is simply the kernel of the norm map for the action of $\Gamma_{E_{k}/F}$ on $X_*(T)$. They are functorial in $T$ and satisfy $\phi\circ\partial=0$ and $\partial\circ\phi=\psi\circ\partial$. We shall now recall these maps and give a refinement $\dot\psi$ of $\psi$ using some material from \cite[\S4.6]{KalRI}.

Fix $k$ such that $E_k$ splits both $T$ and $U$ and $\tx{ord}(Z)$ divides $n_k$. Consider $\bar\lambda \in X_*(\bar T)$ and $\mu_1 : W_{E_k/F} \to X_*(U)$ such that  $(\bar\lambda,\mu_1) \in Z_0(W_{E_k/F},X_*(\bar T) \to X_*(U))_0$. As in \cite[\S A.3]{KS99} define $\phi_U(\mu_1) \in U(E_k)$ by
\[ \phi_{U,k}(\mu_1) = \prod_{\sigma,\tau,a} \sigma(\mu_1(as(\tau)))(c_k(\sigma,\tau)^{-1}\sigma(a)^{-1}), \]
the product running over $\Gamma_{E_k/F} \times \Gamma_{E_k/F} \times E_k^\times$. As explained there, this is an explicit formula for the restriction map of 1-chains $C_1(W_{E_k/F},X_*(U)) \to C_1(E_k^\times,X_*(U))$ composed with the isomorphism $C_1(E_k^\times,X_*(U)) \to X_*(U)\otimes_\Z E_k^\times=U(E_k)$. Furthermore, we define $\dot\psi_T(\bar\lambda) \in Z^1(u \to W,Z \to T)$ as the inflation along $W \to W_k=u_k \boxtimes_{\xi_k} \Gamma$ of the element $z_{\bar\lambda,k}$ of \cite[Lemma 4.7]{KalRI}, which we recall is defined as
\[ x \boxtimes \rho \mapsto \phi_{\bar\lambda,k}(x) \cdot (l_kc_k \sqcup_{E_k/F} n_k\bar\lambda)(\rho) = \phi_{\bar\lambda,k}(x)\cdot\prod_{\sigma \in \Gamma_{K/F}} \rho\sigma(n_k\bar\lambda)(l_kc_{\rho,\sigma}). \]
The image $\bar z_{\bar\lambda,k} \in Z^1(\Gamma,\bar T)$ of $z_{\bar\lambda,k}$ is given by $c_k \cup \bar\lambda=\psi_{\bar T}(\bar\lambda)$ and hence satisfies the equation $f(\bar z_{\bar\lambda,k})-\partial \phi_U(\mu_1)=f(\psi_{\bar T}(\bar\lambda))-\partial\phi_U(\mu_1)=\psi_U(f_*(\bar\lambda))-\psi_U(\partial\mu_1)=0$, due to the functoriality of $\psi$. We conclude that $(z_{\bar\lambda,k},\phi_U(\mu_1)) \in Z^1(u \to W,Z \to T \to U)$.

Now consider $(\partial\lambda_1,f_*(\lambda_1)-\partial\mu_2) \in B_0(W_{E_k/F},X_*(T) \to X_*(U))$. Then we have $\dot\psi_T(\partial\lambda_1)=\psi_T(\partial\lambda_1)=\partial \phi_T(\lambda_1)$, and hence $(\dot\psi_T(\partial\lambda_1),\phi_U(f_*(\lambda_1)-\partial\mu_2))=(\partial\phi_T(\lambda_1),f_*(\phi_T(\lambda_1)))$ is a coboundary.

We conclude that we have defined a group homomorphism
\[ H_0(W_{E_k/F},X_*(T) \to X_*(\bar T) \to X_*(U))_0 \to H^1(u \to W,Z \to T \to U). \]
Next, we consider the composition of this homomorphism with the coinflation map
\[ \resizebox{0.98\textwidth}{!}{$H_0(W_{E_{k+1}/F},X_*(T) \to X_*(\bar T) \to X_*(U))_0 \to H_0(W_{E_k/F},X_*(T) \to X_*(\bar T) \to X_*(U))_0$}. \]

In \cite[\S A.3]{KS99} a homomorphism 
\[ c : C_0(W_{E_{k+1}/F},X_*(T))_0 \to C^0(E_{k+1}/F,T(E_{k+1}))  \]
is defined, and it is shown that
\[ \tx{inf}\circ\phi_k\circ\tx{coinf} = \phi_{k+1} + c\partial, \]
\[ \tx{inf}\circ\psi_k\circ\tx{coinf} = \psi_{k+1} + \partial c. \]
The homomorphism $c$ is defined by the formula
\[ c(\lambda) = \prod_{\sigma \in \Gamma_{E_k/F}}(\sigma\lambda)\left(\prod_{\nu \in \Gamma_{E_{k+1}/E_k}} c_{k+1}(v,\zeta_k(\sigma)) \right). \]
The compatibility of the chosen sections $s_k$ and $s_{k+1}$ implies, via \cite[Lemma 4.4]{KalRI}, that this homomorphism is trivial, because the inner product is equal to $c_k(1,\sigma)=1$. It follows that for $\mu_1' : W_{E_{k+1}/F} \to X_*(U)$ the element $\phi_{U,k+1}(\mu_1') \in U(E_{k+1})$ is equal to the image of $\phi_{U,k}(\tx{coinf}(\mu_1')) \in U(E_k)$ under the natural inclusion $U(E_k) \to U(E_{k+1})$. On the other hand, the inflation of $z_{\bar\lambda,k}$ to $W_{k+1}$ equals $z_{\bar\lambda,k+1}$ according to \cite[Lemma 4.7]{KalRI}, which in our notation here means $\dot\psi_{T,k}(\tx{coinf}(\bar\lambda))=\dot\psi_{T,k+1}(\bar\lambda)$. This gives a commutative diagram
\[ \xymatrix{
	Z_0(W_{E_{k+1}/F},X_*(\bar T) \to X_*(U))_0\ar[dd]^{\tx{coinf}}\ar[rd]^{\dot\psi_{T,k+1},\phi_{U,k+1}}\\
	&Z^1(u \to W,Z \to T \to U)\\
	Z_0(W_{E_{k}/F},X_*(\bar T) \to X_*(U))_0\ar[ru]^{\dot\psi_{T,k},\phi_{U,k}}\\
}
\]
already on the level of (co)cycles, and it in turn induces a commutative diagram on the level of (co)homology, leading to a homomorphism
\begin{equation} \label{eq:arithiso} \bar Y_\tx{tor}(Z \to T \to U) \to H^1(u \to W,Z \to T \to U).\end{equation}

\begin{pro} \label{fct:tn++d2}
The homomorphism \eqref{eq:arithiso} is a functorial isomorphism. It is independent of the choices made in its construction.
\end{pro}
\begin{proof}
It is immediate from the construction that this homomorphism is functorial. The fact that it is an isomorphism follows from the 5-lemma, applied to the exact sequence just below diagram \eqref{eq:bfd2} and the corresponding exact sequence of Fact \ref{fct:tn++esy2}. The maps between the first two terms of these exact sequences are the Langlands isomorphism $H_1(X_*(T)) \to T(F)$ and its analog for $U$, the map between the third terms is \eqref{eq:arithiso}, between the fourth terms it is the isomorphism $\bar Y_{+,\tx{tor}}(Z \to T) \to H^1(u \to W,Z \to T)$ of \cite[\S4]{KalRI}, and between the fifth terms it is the Tate-Nakayama isomorphism $Y_\tx{tor}(U) \to H^1(\Gamma,U)$.

We next argue that this homomorphism is independent of the choices of sections $s_k$ (and also $\zeta_k$) and root maps $l_k$. For this, let $\zeta_k'$, $s_k'$, and $l_k'$ be other choices. We obtain $c_k' \in Z^2(\Gamma_{E_k/F},E_k^\times)$, $\xi_k' \in Z^2(\Gamma,u_k)$, $W_k'=u_k\boxtimes_{\xi_k'}\Gamma$. Let $W'=\varprojlim W_k'$. The construction above gives a group homomorphism
\[ \bar Y_\tx{tor}(Z \to T \to U) \to H^1(u \to W',Z \to T \to U). \]
Every isomorphism $W' \to W$ of extensions induces the same isomorphism $H^1(u \to W,Z \to T \to U) \to H^1(u \to W',Z \to T \to U)$ and we need to show that the triangle
\[ \xymatrix{
&H^1(u \to W,Z \to T \to U)\ar[dd]\\
\bar Y_\tx{tor}(Z \to T \to U)\ar[ru]\ar[rd]\\
&H^1(u \to W',Z \to T \to U)
}
\]
commutes. Define $\eta_k : \Gamma_{E_k/F} \to E_k^\times$ by $s_k'(\sigma)=\eta_k(\sigma)s_k(\sigma)$. Define $\alpha_{k',k} \in C^1(\Gamma,u_k)$ by
\[ \alpha_{k',k}(\sigma) = (l_k'c_k' \cdot (l_kc_k)^{-1} \cdot (dl_k\eta_k)^{-1}) \sqcup_{E_k/F} \delta_e. \]
\begin{lem} \label{lem:tn++i1} The assignment $x \boxtimes \sigma \mapsto x\alpha_{k',k}(\sigma) \boxtimes \sigma$ defines an isomorphism of extensions $\bar g_k : W'_k \to W_k$ that satisfies  $z_{\bar\lambda,k} \circ \bar g_k = z_{\bar\lambda,k}' \cdot d(l_k\eta_k \sqcup_{E_k/F} n_k\bar\lambda)^{-1}$.
\end{lem}
\begin{proof}
This is a direct computation, using \cite[Fact 4.3]{KalRI}.
\end{proof}

Consider the diagram
\[ \xymatrix{
	W_{k+1}'\ar[r]^{\bar g_{k+1}}\ar[d]^{f_k'}&W_{k+1}\ar[d]^{f_k}\\
	W_k'\ar[r]^{\bar g_k}&W_k
}
\]
This diagram does not commute. Define $\beta_k : \Gamma_{E_k/F} \to \bar F^\times$ by
\[ \beta_k(a) = l_k\eta_k(a)^{-1}\prod_{\substack{b \in \Gamma_{E_{k+1}/F}\\b \mapsto a}} l_k\eta_{k+1}(b). \]
\begin{lem} \label{lem:tn++i2}
\begin{enumerate}
	\item $\beta_k(\sigma)^{n_k}=1$ and hence $\beta_k \in u_k$;
	\item $f_k\circ \bar g_{k+1} = \tx{Ad}(\beta_k^{-1}) \circ \bar g_k \circ f_k'$.
\end{enumerate}
\end{lem}
\begin{proof}
We begin with the second point. From the definitions of $f_k$ and $\bar g_k$ we have
\begin{eqnarray*}
f_k(\bar g_{k+1}(x \boxtimes \sigma))&=&p(dl_{k+1}\eta_{k+1}\sqcup_{E_{k+1}/F}\delta_e)^{-1}(dl_k\eta_k\sqcup_{E_k/F}\delta_e)\cdot \bar g_k(f_k'(x \boxtimes \sigma))\\
&=&d[p(l_{k+1}\eta_{k+1}\sqcup_{E_{k+1}/F}\delta_e)^{-1}(l_k\eta_k\sqcup_{E_k/F}\delta_e)]\cdot \bar g_k(f_k'(x \boxtimes \sigma)).
 \end{eqnarray*}
Recall the torus $S_k$ defined as the quotient of $\tx{Res}_{E_k/F}\mb{G}_m$ by the diagonal copy of $\mb{G}_m$. Its subgroup $S_k[n_k]$ of $n_k$-torsion points is precisely $u_k$. We can compute $l_k\eta_k\sqcup_{E_k/F}\delta_e \in S_k$ explicitly and see that it is represented by the map $\Gamma_{E_k/F} \to \bar F^\times$ sending $a$ to $l_k\eta_k(a)$. The analogous formula holds for $l_{k+1}\eta_{k+1}\sqcup_{E_{k+1}/F}\delta_e \in S_{k+1}$, whose image under $p$ then sends $a$ to $\prod_b l_k\eta_{k+1}(b)$, where $b$ runs over the elements of $\Gamma_{E_{k+1}/F}$ mapping to $a$. Thus the argument of $d$ is $\beta_k^{-1}$ as claimed.

We come to the first point and need to prove that the function $\Gamma_{E_k/F} \to \bar F^\times$ defined by $ a\mapsto \eta_k(a)^{-1} \prod_{b\mapsto a} \eta_{k+1}(b)$ represents the trivial element of $S_k$. For this we recall that the sections $s_k$ and $s_{k+1}$ were chosen to satisfy
\[ s_{k+1}(y\zeta_k(x)) = s_{k+1}(y)s_{k+1}(\zeta_k(x))\quad\tx{and}\quad s_k(x)=\pi^W_k(s_{k+1}(\zeta_k(x))), \]
for $y \in \Gamma_{E_{k+1}/E_k}$ and $x \in \Gamma_{E_k/F}$, where $\pi_k^W$ is the natural projection $W_{E_{k+1}/F} \to W_{E_k/F}$. From these we obtain via direct calculation the following identities
\[ \eta_{k+1}(v\zeta_k'(a))=\eta_{k+1}(v)\cdot {^v\eta_{k+1}(\zeta_k'(a))}\quad\tx{and}\quad \eta_k(a)=\prod_{v\in \Gamma_{E_{k+1}}/E_k} {^v\eta_{k+1}(\zeta_k'(a))}, \]
which imply $\eta_k(a)^{-1}\prod_{b\mapsto a} \eta_{k+1}(b)=\prod_v\eta_{k+1}(v)$. This is a constant function in $a$, hence represents the trivial element of $S_k$.
\end{proof}

Choose $\dot\beta_k \in u$ mapping to $\beta_k \in u_k$. Define $\dot\beta_{<k}=\prod_{i=1}^{k-1} \dot\beta_i$. Define $g_k : W_k' \to W_k$ as $\tx{Ad}(\dot\beta_{<k})\circ\bar g_k$. Then $(g_k)_k$ commutes with the transition maps $f_k$ and $f_k'$ and induces an isomorphism $g : W' \to W$. We transport $z_{\bar\lambda}$ via $g$ and obtain an element $z_{\bar\lambda}'' \in Z^1(u \to W',Z \to S)$ that we want to compare with $z_{\bar\lambda}'$. Lemma \ref{lem:tn++i1} implies
\begin{eqnarray*}
z_{\bar\lambda,k}''(x \boxtimes \sigma)&=&z_{\bar\lambda,k}'(x \boxtimes\sigma) \cdot \phi_{\bar\lambda,k}(\dot\beta_{<k} \cdot {^\sigma\dot\beta_{<k}}^{-1})\cdot d(l_k\eta_k \sqcup_{E_k/F} n_k\bar\lambda)^{-1}\\
&=&z_{\bar\lambda,k}'(x \boxtimes\sigma) \cdot d(\phi_{\bar\lambda,k}(\dot\beta_{<k}) \cdot l_k\eta_k \sqcup_{E_k/F} n_k\bar\lambda)^{-1}
\end{eqnarray*}
On the other hand, the identity $\phi_{U,k}(\mu_1)=\phi_{U,k}'(\mu_1)-\eta_k \cup f_*(\bar\lambda)$ was verified in \cite[\S A.3]{KS99}. Since $f|_Z=1$ we have $f(\phi_{\bar\lambda,k}(\dot\beta_{<k}) \cdot l_k\eta_k \sqcup_{E_k/F} n_k\bar\lambda)=\bar f(\eta_k \cup \bar\lambda)= \eta_k \cup f_*(\bar\lambda)$. It follows that $(\dot\psi_{T,k}(\bar\lambda),\phi_{U,k}(\mu_1))$ is cohomologous to $(\dot\psi_{T,k}'(\bar\lambda),\phi_{U,k}'(\mu_1))$, and so are their inflations.

Finally we argue that the homomorphism is independent of the choices of sequences $n_k$ and $E_k$. If $n_k'$ is another sequence, we may reduce to the special case $n_k|n_k'$ by comparing both $n_k$ and $n_k'$ to $n_k''=n_kn_k'$. In the special case $n_k|n_k'$ choose a compatible system $l_k'$ with $l_{k+1}'^{n_{k+1}'/n_k'}=l_k'$ and define $l_k=l_k'^{n_k'/n_k}$. It is immediate to check that we have equality of cocycles $\xi_k=\xi_k'$ and $z_{\bar\lambda,k}=z_{\bar\lambda,k}'$ (we have of course chosen $\zeta_k=\zeta_k'$ and $s_k=s_k'$). This shows independence of the choice of $n_k$. For the choice of $E_k$, note first that passing to a co-final subsequence has no effect. If $E_k'$ is another sequence, we may pass to co-final subsequences of both $E_k$ and $E_k'$ to arrange $E_k \subset E_k' \subset E_{k+1} \subset E_{k+1}'$. Define $E_k''$ by $E_{2k}''=E_k$ and $E_{2k+1}''=E_k'$. Then $E_k''$ is again an exhaustive sequence, of which both $E_k$ and $E_k'$ are co-final subsequences. This shows independence of the choice of $E_k$.
\end{proof}

\begin{lem} \label{lem:tn++d3}
The isomorphism \ref{eq:arithiso} satisfies the following compatibilities.
\begin{enumerate}
	\item The maps $H_1(W_F,X_*(U)) \to \bar Y_{+,\tx{tor}}(Z \to T \to U)$ and $H^1(u \to W,Z \to T \to U) \to H^0(\Gamma,U)$ translate the isomorphism \ref{eq:arithiso} to the \emph{negative} of the Langlands isomorphism $H_1(W_F,X_*(U)) \to H^0(\Gamma,U)$.
	\item The maps $\bar Y_{+,\tx{tor}}(Z \to T \to U) \to \bar Y_{+,\tx{tor}}(Z \to T)$ and $H^1(u \to W,Z \to T \to U) \to H^1(u \to W,Z \to T)$ translate the isomorphism \eqref{eq:arithiso} to the isomorphism constructed in \cite[\S4]{KalRI}.
\end{enumerate}
\end{lem}
\begin{proof}
This follows by inspecting the construction of \eqref{eq:arithiso}. Indeed, the definition of $\dot\psi_T(\bar\lambda)$ as the inflation of $z_{\bar\lambda,k}$ from $W_k$ to $W$ is the same as the construction in \cite[\S4.6]{KalRI}. On the other hand, the definition of $\phi_U$ used here is the same as the one in \cite[\S A.3]{KS99}. The fact that it yields the negative of the Langlands isomorphism comes from the inverse in the formula $\prod_{a \in K^\times} x_a(a^{-1})$ appearing in the middle of page 131 in loc. cit.
\end{proof}

\begin{cor} \label{cor:tn++d4}
The pairing \eqref{eq:tnd++} satisfies the following compatibilities.
\begin{enumerate}
	\item The maps $H^0(\Gamma,U) \to H^1(u \to W,Z \to T \to U)$ and $H^1(W_F,\hat Z \to \hat T \from \hat U) \to H^1(W_F,\hat U)$ translate the pairing \eqref{eq:tnd++} the the Langlands pairing.
	\item The maps $H^1(u \to W,Z \to T \to U) \to H^1(u \to W,Z \to T)$ and $\pi_0([\hat{\bar T}]^+) \to H^1(W_F,\hat Z \to \hat T \from \hat U)_\tx{red}$ translate the pairing \eqref{eq:tnd++} to the pairing \cite[Corollary 5.4]{KalRI}.
\end{enumerate}
\end{cor}
\begin{proof}
This follows directly from Fact \ref{fct:tn++d1} and Lemma \ref{lem:tn++d3}. Note that in the case of the Langlands pairing both the Fact and the Lemma contain a negation, and the two cancel out.
\end{proof}

\subsection{Rational classes and invariants for rigid inner forms} \label{sub:rigid_rat}

With the cohomological preliminaries out of the way, we can now extend the considerations of Section \ref{sec:pure} to the case of general inner forms. In this subsection we extend the concepts of rational classes and their invariants.

We begin again with a quasi-split disconnected group $\tilde G = G \rtimes A$. More precisely, let $G$ be a connected reductive group, defined and quasi-split over $F$. Let $(T,B,\{X_\alpha\})$ be an $F$-pinning of $G$ and let $A$ be a finite group that acts on $G$ by pinned automorphisms. Assume given an action of $\Gamma$ on $A$ so that for $\sigma \in \Gamma$ we have $\sigma(a(g)) = \sigma(a)(\sigma(g))$. As we argued in Subsection \ref{sub:pure} we may replace $A$ by $A^\Gamma$ and therefore assume that $\Gamma$ acts trivially on $A$.

A given $\bar z \in Z^1(\Gamma,G/Z(G)^A)$ leads to the inner form $\tilde G_{\bar z}$ of $G \rtimes A$, where $\Gamma$ acts on $\tilde G_{\bar z}(\bar F)$ via the twisted action $\sigma \mapsto \tx{Ad}(\bar z(\sigma)) \rtimes \sigma$. The elements of $\tilde G_{\bar z}(F)$ are those $\tilde\delta \in (G \rtimes A)(\bar F)$ that commute with $\bar z(\sigma) \rtimes \sigma$. Given a norm $(S,\gamma)$ of $\tilde\delta=\delta\rtimes a$ we would like to define a cohomological invariant measuring the relative position of $(S,\gamma)$ and $\tilde\delta$. If we mimic the constructions of Subsection \ref{sub:pure_inv} we would arrive at an element $\tx{inv}(\gamma,(\bar z,\delta))$ of $H^1(\Gamma,S/Z(G)^A \stackrel{1-a}{\lrw} S)$, but that would be too crude for our purposes.

In order to define the right invariant, we need to work with $z \in Z^1(u \to W,Z(G)^A \to G)$ instead of $\bar z \in Z^1(\Gamma,G/Z(G)^A)$. Thus we consider the set of pairs $(z,\tilde\delta)$, where $z \in Z^1(u \to W,Z(G)^A \to G)$, $\tilde\delta \in (G \rtimes A)(\bar F)$, and $\tilde\delta$ commutes with $\bar z(\sigma)\rtimes\sigma$, where now $\bar z \in Z^1(\Gamma,G/Z(G)^A)$ is the image of $z$ modulo $Z(G)^A$. This is the set of rational elements of rigid inner forms of $G \rtimes A$. The surjectivity of $Z^1(u \to W,Z(G)^A \to G) \to Z^1(\Gamma,G/Z(G)^A)$ asserted in \cite[Proposition 3.6]{KalRI} implies that this set surjects onto the set of rational elements of inner forms considered above. Furthermore, the set of rational elements of pure inner forms of $G \rtimes A$ injects into the set of rational elements of rigid inner forms of $G \rtimes A$. The group $G$ acts on the latter set by the same formula as in the case of pure inner forms, and the orbits of that action are the set of rational conjugacy classes of rational elements of rigid inner forms.

We can extend the cohomological notation of Subsection \ref{sub:pure_rat} as follows. Given two homomorphisms $(a,b) : G \rightrightarrows G$ and a central subgroup $Z \subset G$ that equalizes them, we consider the set $Z^1_{b,a}(u \to W,Z(G)^A \to G \rightrightarrows G)$ of pairs $(z,\delta)$, where $z \in Z^1(u \to W,Z \to G)$ and $\delta \in G$ satisfying $a(z(w))=\delta^{-1}b(z(w))\sigma_w(\delta)$, where $\sigma_w \in \Gamma$ is the image of $w \in W$. In our applications we will take $b=\tx{id}$ and abbreviate $Z^1_{b,a}$ to $Z^1_a$. As before, $(z,\delta)$ lies in $Z^1_a$ if and only if $\tilde\delta=\delta \rtimes a$ commutes with $\tilde z(w)=z(w) \rtimes \sigma_w$, and we write $\tilde Z^1_a$ for the set of commuting pairs $(\tilde z,\tilde\delta)$. The group $G$ acts by conjugation on the set $\tilde Z^1_a$, or equivalently by $(g^{-1}z(w)\sigma_w(g),g^{-1}\delta a(g))$ on the set $Z^1_a$, and the sets of orbits under this action are denoted by $\tilde H^1_a$ respectively $H^1_a$. The set of rational elements of rigid inner forms of $G \rtimes A$ is $\bigcup_{a \in A} \tilde Z^1_a$, and the set of rational conjugacy classes of rational elements is the set $\bigcup_{a \in A} \tilde H^1_a$.

As in the case of pure inner forms, given a rational element $(\tilde z,\tilde\delta)$ and a norm $(S,\gamma)$ for the $G$-conjugacy class of $\tilde\delta$, we choose a representative $(\tilde z^*,\tilde\delta^*)$ of the $G$-orbit of $(\tilde z,\tilde\delta)$ as in Lemma \ref{lem:inv} and the same argument implies that $(\tilde z^*,\tilde\delta^*) \in \tilde Z^1_a(u \to W;Z(G)^A \to S \rightrightarrows S)$ and its cohomology class is independent of the choice of $(\tilde z^*,\tilde \delta^*)$. Moreover, $(z^{*,-1},\delta^*)$ lies in the set $H^1(u \to W,Z(G)^A \to S \stackrel{1-a}{\lrw} S)$ defined in Subsection \ref{sub:coho}. We shall denote either of these classes by $\tx{inv}(\gamma,(z,\delta))$ or $\tx{inv}(\gamma,(\tilde z,\tilde\delta))$. The image of this invariant in $H^1(\Gamma,S/Z(G)^A \stackrel{1-a}{\lrw} S)$ is equal to the cruder invariant $\tx{inv}(\gamma,(\bar z,\delta))$ mentioned above.

\subsection{Refined endoscopic data} \label{sub:ref_endo}

As in the case of connected groups, rigid inner forms require a refinement of the notion of endoscopic datum. The necessary refinement is directly analogous to that in the connected case. Namely, let $Z \subset Z(G)^A$ be finite, $\bar G=G/Z$, and $\hat{\bar G} \to G$ the isogeny dual to $G \to \bar G$. Given an endoscopic datum $\mf{e}=(G^\mf{e},\mc{G}^\mf{e},\tilde s^\mf{e},\xi^\mf{e})$ in the sense of Subsection \ref{sub:endo}, a refinement consists of choosing a preimage $\dot s^\mf{e} \in \hat{\bar G} \rtimes A$ of $\tilde s^\mf{e}$. The refined endoscopic datum is then $\mf{\dot e}=(G^\mf{e},\mc{G}^\mf{e},\dot s^\mf{e},\xi^\mf{e})$. An isomorphism $\mf{\dot e} \to \mf{\dot e}'$ of two such data is given by $g \in \hat G$ satisfying $\xi^{\mf{e}'}=\tx{Ad}(g)\circ\xi^\mf{e}$ and $\dot s^{\mf{e}'}=\tx{Ad}(g)\dot s^\mf{e}$ modulo $Z(\hat{\bar G})^\circ$.

\subsection{Normalized transfer factors} \label{sub:rigid_tf}

Given a refined endoscopic datum $\mf{\dot e}$ and a $z$-pair $\mf{z}$ for $\mf{e}$ we shall now define a normalized transfer factor: a function $\Delta[\mf{w},\mf{\dot e},\mf{z}]$ that assigns complex numbers to pairs $(\gamma^\mf{z},\tilde\delta)$ of strongly regular semi-simple elements $\gamma^\mf{z} \in G^\mf{z}(F)$ and $\tilde\delta \in \tilde G_z(F)$. This factor is given by the same formula \eqref{eq:pure_tf} as in the case of pure inner forms, but with a different construction of $\Delta_{KS}[\mf{w},\mf{\dot e},\mf{z}]$, which depends on the refinement $\mf{\dot e}$ of $\mf{e}$. That in turn is given by the same formula \eqref{eq:pure_tf1}, but we have to specify what $\Delta_{III}^\tx{new}$ is. We shall now give this construction in the general case involving a $z$-pair.

The considerations are rather analogous to those of Subsection \ref{sub:pure_tf2}. We follow the notation there. Thus we have $\gamma^\mf{z} \in S^\mf{z}(F)$, $(\tilde z,\tilde\delta) \in \tilde Z^1_{b^{-1}}(u \to W,Z \to G \rightrightarrows G)$, a norm $(S,\gamma)$ for $\tilde\delta$, and a representative $(\tilde z^*,\tilde\delta^*)$ of the $G$-conjugacy class of $(\tilde z,\tilde\delta)$ with $\delta^* \in S(\bar F)$ mapping to $\gamma \in S_b(F)$. The element $\delta^\mf{z}=(\delta^*,\gamma^\mf{z})$ lies in the fiber product $S_1^\mf{z}$ of $S \to S_b \cong S^\mf{e} \from S^\mf{z}$. Under the homomorphism $(b_1^{-1}-1) : S \to S_1^\mf{z}$, the 1-cocycle $z^* \in Z^1(u \to W,Z \to S)$ maps to a 1-cocycle $(b_1^{-1}-1)z^* \in Z^1(\Gamma,S_1^\mf{z})$ that satisfies $(b_1^{-1}-1)z^*(\sigma)=(\delta^\mf{z})^{-1}\sigma(\delta^\mf{z})$, and so $(z^{*,-1},\delta^\mf{z})$ belongs to $Z^1(u \to W,Z \to S \stackrel{1-b_1^{-1}}{\lrw} S_1^\mf{z})$. The class $\tx{inv}(\gamma^\mf{z},(z,\delta))$ of this element is independent of the choice of $(z^*,\delta^*)$. Its image in $Z^1(u \to W,Z \to S \stackrel{1-b^{-1}}{\lrw} S)$ equals the class $\tx{inv}(\gamma,(z,\delta))$ defined in Subsection \ref{sub:rigid_rat}.

Next we define a class $\dot A_0 \in H^1(W_F,\hat Z \to \hat S \from \hat S_1^\mf{z})$ refining the class $A_0 \in H^1(W_F,\hat S \to \hat S)$ of Subsection \ref{sub:pure_tf2}. Following the definition of $A_0$ we have the element $(a_S^{-1},s_S) \in Z^1(W_F,(1-b_1):\hat S_1^\mf{z} \to \hat S)$. In addition to $\tilde s^\mf{e}=\xi_S(s_S) \rtimes b$, we now also have $\dot s^\mf{e}=\xi_S(\dot s_S) \rtimes b$, where we form $\bar S=S/Z$ and use the unique extension of $\xi_S$ to $^L{\bar S} \to {^L{\bar G}}$ to define $\dot s_S \in \hat{\bar S}$. Then $(a_S^{-1},\dot s_S) \in Z^1(W_F,\hat Z \to \hat S \from \hat S_1^\mf{z})$ and its class is $\dot A_0$.

We now define $\Delta_{III}^\tx{new}(\gamma^\mf{z},(z,\delta))$ to be the value of the pairing constructed in Subsection \ref{sub:tnd++} at the classes $\tx{inv}(\gamma^\mf{z},(z,\delta))$ and $\dot A_0$.

\subsection{The local correspondence and character identities} \label{sub:llc_rigid}

Let $\phi : L_F \to {^LG}$ be a tempered Langlands parameter. In subsection \ref{sub:pure_llc} we introduced the group of $\tilde G$-equivalences $\tilde S_\phi=\tx{Cent}(\phi,\hat G \rtimes A)$. It was part of an exact sequence
\[ 1 \to S_\phi \to \tilde S_\phi \to A^{[\phi]} \to 1, \]
where $A^{[\phi]}$ is the stabilizer in $A$ of the $G$-equivalence class of $\phi$. For a finite subgroup $Z \subset Z(G)^A$ we have the isogenies $G \to \bar G=G/Z$ and $\hat{\bar G} \to \hat G$ and we define $\tilde S_\phi^+$ to be the preimage in $\hat{\bar G} \rtimes A$ of $\tilde S_\phi$. This is analogous to the definition of $S_\phi^+$ as the preimage in $\hat{\bar G}$ of $S_\phi$ given in \cite[\S5.4]{KalRI}. We have again the exact sequence
\[ 1 \to S_\phi^+ \to \tilde S_\phi^+ \to A^{[\phi]} \to 1. \]
We are now interested in the rigid inner form $\tilde G_z$ for some $z \in Z^1(u \to W,Z \to G)$. Let $A^{[z]}$ be the stabilizer of the class of $z$, and $A^{[\phi],[z]}=A^{[\phi]} \cap A^{[z]}$. Pulling back the above exact sequence along the inclusion $A^{[\phi],[z]} \to A^{[\phi]}$ we obtain the exact sequence
\[ 1 \to S_\phi^+ \to \tilde S_\phi^{+,[z]} \to A^{[\phi],[z]} \to 1. \]
In the case $F=\R$ set $^K\tilde G_z$ to be the associated $K$-group, i.e. the disjoint union of $\tilde G_{z'}$ for all $z'$ in the image of $H^1(\R,G_{z,\tx{sc}}) \to H^1(\R,G_z) \to H^1(u \to W,Z(G)^A \to G_z) \to H^1(u \to W,Z(G)^A \to G)$.

\begin{cnj} \label{cnj:llc_rigid}
\begin{enumerate}
	\item The choice of an $A$-admissible Whittaker datum $\mf{w}$ on $G$ determines a bijection between the set of irreducible admissible $G$-tempered representations of $\tilde G_z(F)$ when $F/\Q_p$, or of $^K\tilde G_z(F)$ when $F=\R$, and the set of $(\hat G \rtimes A^{[z]})$-conjugacy classes of pairs $(\phi,\tilde\rho)$, where $\phi : L_F \to {^LG}$ is a tempered Langlands parameter, and $\tilde\rho \in \tx{Irr}(\pi_0(\tilde S_\phi^{+,[z]}),{[z]})$. When $z=1$ the representation corresponding to $(\phi,\tilde\rho)$ is $\mf{w}$-generic if and only if $\tilde\rho=1$.
	\item This bijection satisfies the character identity \eqref{eq:charid} for a pair of functions $f$ and $f^\mf{z}$ as in Lemma \ref{lem:trans}, where now the transfer factor is the one constructed in Subsection \ref{sub:rigid_tf}.
\end{enumerate}
\end{cnj}

\section{Change of Whittaker data} \label{sec:change_whit}

In \cite{KalGen} we studied how the bijection $\tx{Irr}(\pi_0(S_\varphi^+)) \to \Pi_\varphi$ of the refined local Langlands conjecture depends on the Whittaker datum $\mf{w}$, in the case of a connected reductive group. Strictly speaking loc. cit. considered only pure and extended pure inner twists, but not rigid inner twists, which were unavailable at the time. In this section we shall extend these considerations to the case of rigid inner forms of quasi-split groups and may be connected or disconnected. 

Consider first the connected case, which will serve primarily to recall notation from \cite{KalGen}. Let $G$ be a quasi-split connected reductive group defined over $F$ and let $\mf{w}_1,\mf{w}_2$ be two Whittaker data. There is a unique element of $\tx{cok}(G(F) \to G_\tx{ad}(F))$ conjugating $\mf{w}_1$ to $\mf{w}_2$, which we denote by $(\mf{w}_1,\mf{w}_1)$. Recall from \cite[Lemma 4.1]{KalGen} that there is a natural injection
\[ \tx{cok}\Big(G(F) \to G_\tx{ad}(F)\Big) \to \tx{ker}\Big(H^1(W_F,Z(\hat G_\tx{sc})) \to H^1(W_F,Z(\hat G))\Big)^D. \]
It essentially comes from Poitou--Tate duality 
\[ H^1(\Gamma,Z(G_\tx{sc})) \otimes H^1(\Gamma,X^*(Z(G_\tx{sc}))) \to H^2(\Gamma,\mb{G}_m) \to \Q/\Z \] 
and the identification $X^*(Z(G_\tx{sc}))=Z(\hat G_\tx{sc})$ via the exponential map $\exp : X_*(\hat T_\tx{sc})\otimes_\Z\C \to \hat T_\tx{sc}$ with kernel $X_*(\hat T_\tx{sc})$. Given a tempered Langlands parameter $\phi : L_F \to {^LG}$ we endow the exact sequences
\[ 1 \to Z(\hat G_\tx{sc}) \to \hat G_\tx{sc} \to \hat G_\tx{ad} \to 1,\qquad 1 \to Z(\hat G) \to \hat G \to \hat G_\tx{ad} \to 1 \]
with $L_F$-action via $\tx{Ad}(\phi(-))$. The actions on $Z(\hat G_\tx{sc})$ and $Z(\hat G)$ are of course simply the $\Gamma$-action inflated to $L_F$ and $H^1(L_F,-)=H^1(W_F,-)$ for these two groups. The connecting homomorphism $H^0(L_F,\hat G_\tx{ad}) \to H^1(\Gamma,Z(\hat G_\tx{sc}))$ is continuous and thus factors through the component group of the complex algebraic group $H^0(L_F,\hat G_\tx{ad})$. We have $S_\phi=H^0(L_F,\hat G)$ and its image under that connecting homomorphism lands in $\tx{ker}(H^1(W_F,Z(\hat G_\tx{sc})) \to H^1(W_F,Z(\hat G)))$. Therefore $(\mf{w}_1,\mf{w}_2)$ induces a character of $\pi_0(S_\phi/Z(\hat G)^\Gamma)=\pi_0(S_\phi^+/Z(\hat{\bar G})^+)$. If
\[ \iota_i : \tx{Irr}(S_\phi^{+}) \to \Pi_\phi(G) \]
are the two bijections of the refined local Langlands correspondence, where we are using compound $L$-packets encompassing all rigid inner forms, then according to \cite[(1.1)]{KalGen} we have
\[ \iota_2(\rho)=\iota_1(\rho \otimes(\mf{w}_1,\mf{w}_2)).\]

We now turn to the disconnected case. Thus let $\tilde G=G \rtimes A$ be a quasi-split, (possibly) disconnected, reductive group, and let $\mf{w}_1,\mf{w}_2$ be $A$-admissible Whittaker data for $G$. Let $z \in Z^1(u \to W,Z(G)^A \to G)$. We denote by 
\[ \iota_i : \tx{Irr}(\pi_0(\tilde S_\phi^{+,[z]}),{[z]}) \to \Pi_\phi(\tilde G_z) \]
the bijections of Conjecture \ref{cnj:llc_rigid} with respect to the Whittaker data $\mf{w}_i$. The approach to comparing these is the same as in the connected case. We use the exact sequences
\[ 1 \to Z(\hat G_\tx{sc}) \to \hat G_\tx{sc} \rtimes A \to \hat G_\tx{ad} \rtimes A \to 1,\qquad 1 \to Z(\hat G) \to \hat G \rtimes A \to \hat G_\tx{ad} \rtimes A \to 1 \]
to obtain the connecting map $H^0(L_F,\hat G_\tx{ad} \rtimes A) \to H^1(W_F,Z(\hat G_\tx{sc}))$. This map is no longer a homomorphism, but rather a twisted homomorphism (i.e. a 1-cocycle) for the action for $\hat G_\tx{ad} \rtimes A$ on $Z(\hat G_\tx{sc})$ given by the projection to $A$ and the natural action of $A$ on $Z(\hat G_\tx{sc})$. Nonetheless, this map factors through $\pi_0(H^0(L_F,\hat G_\tx{ad} \rtimes A))$ and sends $\tilde S_\phi=H^0(L_F,\hat G \rtimes A)$ to $\tx{ker}(H^1(W_F,Z(\hat G_\tx{sc})) \to H^1(W_F,Z(\hat G)))$. In this way we obtain a twisted homomorphism 
\[ \pi_0(\tilde S_\phi/Z(\hat G)^\Gamma) \to \tx{ker}(H^1(W_F,Z(\hat G_\tx{sc})) \to H^1(W_F,Z(\hat G))). \] 
Since the character $(\mf{w}_1,\mf{w}_2)$ of $\tx{ker}(H^1(W_F,Z(\hat G_\tx{sc})) \to H^1(W_F,Z(\hat G)))$ is $A$-invariant, its pull back under this twisted homomorphism is a character of $\pi_0(\tilde S_\phi/Z(\hat G)^\Gamma)=\pi_0(\tilde S_\phi^+/Z(\hat{\bar G})^+)$, which we may pull back further to $\pi_0(\tilde S_\phi^{+,[z]})$.

\begin{pro} 
\[ \iota_2(\rho) = \iota_1(\rho\otimes(\mf{w}_1,\mf{w}_2)). \]
\end{pro}
\begin{proof}
Let $\tilde s = s \rtimes a \in \tilde S_\phi$. As in the proof of \cite[Theorem 4.3]{KalGen} it is enough to prove the identity
\[ \Delta_{KS}[\mf{w}_2] = \<(\mf{w}_1,\mf{w}_2),s\>^{-1} \cdot \Delta_{KS}[\mf{w}_1]. \] 
To prove this we choose a non-trivial unitary character $\psi_F$ of $F$ and $A$-invariant $F$-pinnings $\tx{spl}_i$ giving rise to $\mf{w}_i$ according to Corollary \ref{cor:adm}. We write $\Delta_{KS}$ according to \eqref{eq:pure_tf1} and note that only $\Delta_I$ depends on the pinnings. Recall from Proposition \ref{pro:stein}(2,3) that $G_\tx{sc}^A$ and $G_\tx{ad}^A$ are connected, and the latter is the adjoint group of $G^1$. Therefore we have the exact sequence
\[ 1 \to Z(G_\tx{sc})^A \to G_\tx{sc}^A \to G_\tx{ad}^A \to 1. \]
Furthermore, by Proposition \ref{pro:stein}(9) there exists an $F$-rational point of the adjoint group of $G^1$ conjugating $\tx{spl}_1$ to $\tx{spl}_2$. We let $g \in G_\tx{sc}^A$ lift this $F$-rational point and let $x \in H^1(F,Z(G_\tx{sc})^A)$ be the image of $g$ under the connecting homomorphism for the above exact sequence.

The argument of \cite[(2.3.1)]{LS87} shows that the twisted splitting invariant in $H^1(F,T_\tx{sc}^a)$ with respect to $\tx{spl}_2$ is the product of the twisted splitting invariant with respect to $\tx{spl}_1$ with $x$ (note that loc. cit. uses conjugation on the right, so their $g$ is our $g^{-1}$). Therefore $\Delta_I[\mf{w}_2]=\<x,s\>^{-1}\Delta_I[\mf{w}_1]$, as claimed.
\end{proof}

\section{Change of component group} \label{sec:change_comp}

\subsection{Restriction} \label{sub:comp_rest}

Assume now given a map of finite groups $B \to A$. We can consider the disconnected groups $G^A = G \rtimes A$ and $G^B = G \rtimes B$, where $B$ acts on $G$ via its map to $A$. We can consider restriction of representations along the map $G^B_z(F) \to G^A_z(F)$. Dually, the map $\hat G \rtimes B \to \hat G \rtimes A$ induces for each tempered Langlands parameter $\phi$ a map $\pi_0(S_\phi^{B,+,[z]}) \to \pi_0(S_\phi^{A,+,[z]})$ and we can consider restriction of representations along this map as well.

Let the $G$-tempered representations $\pi^A$ of $G^A_z(F)$ and $\pi^B$ of $G^B_z(F)$ correspond under Conjecture \ref{cnj:llc_rigid} to the pairs $(\phi,\rho^A)$ and $(\phi,\rho^B)$, respectively, where $\rho^A \in \tx{Irr}(\pi_0(S_\phi^{A,+,[z]}),{[z]})$ and $\rho^B \in \tx{Irr}(\pi_0(S_{\phi'}^{B,+,[z]}),{[z]})$.

\begin{cnj} \label{cnj:llc_rest}
	The multiplicity of $\pi^B$ in $\tx{Res}\,\pi^A$ is equal to the multiplicity of $\rho^B$ in $\tx{Res}\,\rho^A$.
\end{cnj}
	
\begin{rem}
	Let $\tilde G=G \rtimes A$ and $B=\{1\}$. Applying Conjecture \ref{cnj:llc_rest} we obtain a complete description of the set $\tilde\Pi_{\phi,z}$ in terms of the $L$-packet $\Pi_{\phi,z}$ of the connected group $G_z(F)$: An irreducible $G$-tempered representation $\tilde\pi$ of $\tilde G_z(F)$ belongs to $\tilde\Pi_{\phi,z}$ if and only if its restriction to $G_z(F)$ intersects $\Pi_{\phi,z}$. Equivalently, the set $\tilde\Pi_{\phi,z}$ consists precisely of the irreducible constituents of the inductions to $\tilde G_z(F)$ of the elements of $\Pi_{\phi,z}$. Hence the content of Conjecture \ref{cnj:llc_rigid} is in the internal structure and character identities with normalized transfer factors. Note that, just like in the connected case, the packets $\tilde\Pi_{\phi,z}$ are disjoint and exhaust the set of isomorphism classes of irreducible admissible $G$-tempered representations, assuming this is the case for the packets $\Pi_{\phi,z}$.
\end{rem}
	
As another immediate consequence we obtain the following seemingly more general statement, which could be more convenient for some applications.

Let the $G$-tempered representation $\pi^A$ of $G^A_z(F)$ correspond under Conjecture \ref{cnj:llc_rigid} to the pair $(\phi,\rho^A)$ with $\phi : L_F \to {^LG}$ and $\rho^A \in \tx{Irr}(\pi_0(S_\phi^{A,+,[z]}),{[z]})$. Let the $G$-tempered representation $\pi^B$ of $G^B_z(F)$ correspond under Conjecture \ref{cnj:llc_rigid} to the pair $(\phi',\rho^B)$ with $\phi' : L_F \to {^LG}$ and $\rho^B \in \tx{Irr}(\pi_0(S_{\phi'}^{B,+,[z]}),{[z]})$. 

\begin{cor} \label{cor:llc_rest}
The multiplicity of $\pi^B$ in $\tx{Res}\,\pi^A$ is zero unless $\phi$ and $\phi'$ are conjugate under $\hat G \rtimes A^{[z]}$. Assuming that and conjugating $(\phi,\rho^A)$ under $\hat G \rtimes A$ to arrange $\phi=\phi'$, this multiplicity is equal to the multiplicity of $\rho^B$ in $\tx{Res}\,\rho^A$.
\end{cor}
\begin{proof}
	Since $\pi^A$ depends only on the $\hat G \rtimes A^{[z]}$-conjugacy class of $\phi$, we are free to replace $\phi$ with another member of this conjugacy class. If $\phi'$ lies in it, we arrange $\phi=\phi'$ and apply Conjecture \ref{cnj:llc_rest}. 

	If $\phi'$ does not lie in the $\hat G \rtimes A^{[z]}$-conjugacy class of $\phi$, then the $\hat G \rtimes A^{[z]}$-conjugacy class of $\phi$ and the $\hat G \rtimes B^{[z]}$-conjugacy class of $\phi'$ are disjoint. But 
	\[ \tx{Res}^{G^A_z(F)}_{G_z(F)}\Pi_\phi(G^A_z) = \cup_{\phi_1} \Pi_{\phi_1}(G_z) \]
	where $\phi_1$ runs over the set of $\hat G$-conjugacy classes inside of the  $\hat G \rtimes A^{[z]}$-conjugacy class of $\phi$, and analogously 
	\[ \tx{Res}^{G^B_z(F)}_{G_z(F)}\Pi_\phi(G^B_z) = \cup_{\phi_1'} \Pi_{\phi_1'}(G_z) \]
	where $\phi_1'$ runs over the set of $\hat G$-conjugacy classes inside of the  $\hat G \rtimes B^{[z]}$-conjugacy class of $\phi'$. Since no pair $(\phi_1,\phi_1')$ has equal members, and all the tempered packets on $G_z(F)$ are disjoint, the multiplicity of $\pi^B$ in $\tx{Res}\,\pi^A$ is zero. 
\end{proof}

\begin{rem}
	We can factor the map $B \to A$ as $B \to \bar B \to A$, where $\bar B$ is the image of $B$ in $A$. Then $G^B_z(F) \to G^{\bar B}_z(F)$ is surjective, hence the pull-back of an irreducible representation is irreducible. Therefore, for the purposes of discussing multiplicities in restriction, we may assume that $B \to A$ is injective.
\end{rem}

\subsection{Slicing by cosets} \label{sub:comp_coset}

In this section we assume the validity of the refined local Langlands correspondence for connected groups, as well as its functoriality as expressed in Conjecture \ref{cnj:func}. Given a tempered parameter $\phi : L_F \to {^LG}$ we then have the $L$-packet $\Pi_\phi(G_z)$. If $\pi \in \Pi_\phi(G_z)$ and $\rho \in \tx{Irr}(\pi_0(S_\phi^+))$ corresponding to each other then Conjecture \ref{cnj:func} implies $A^{[z]}_\pi=A^{[\phi]}_\rho$. We have the elements $\alpha_\pi \in H^2(A^{[z]}_\pi,\C^\times)$ and $\alpha_\rho \in H^2(A^{[\phi]}_\rho,\C^\times)$ corresponding to the projective extension of $\pi$ to $\tilde G_z(F)_\pi$ and of $\rho$ to $\pi_0(\tilde S_{\phi,\rho}^{+,[z]})$, respectively. The elements $\alpha_\pi$ and $\alpha_\rho$ are equal if and only if the representation $\pi \boxtimes \rho^\vee$ of $G_z(F) \times \pi_0(S_\phi^+)$ has an extension to $\tilde G_z(F)_\pi \times_{A^{[z]}_\pi} \pi_0(\tilde S_{\phi,\rho}^{+,[z]})$. Such an extension is then well-defined up to a character of $A_{\pi}^{[z]}$.

\begin{cnj} \label{cnj:coset}
Let $\phi : L_F \to {^LG}$ be a tempered parameter. Let $\pi \in \Pi_\phi(G_z)$ and $\rho \in \tx{Irr}(\pi_0(S_\phi^+))$ correspond to each other. The representation $\pi\boxtimes\rho^\vee$ of $G_z(F) \times \pi_0(S_\phi^+)$ has an extension $(\pi\boxtimes\rho^\vee)^\tx{can}$ to $\tilde G_z(F)_\pi \times_{A^{[z]}_\pi} \pi_0(\tilde S_{\phi,\rho}^{+,[z]})$ such that for $a \in A$, $\tilde f \in \mc{C}^\infty_c([G \rtimes a]_z(F))$ and $\tilde s \in \tilde S_\phi$ mapping to $a^{-1}$ we have 
\[ S\Theta_{\phi^\mf{z}}(f^\mf{z})=e(G_z)\sum_{\substack{\pi \in \Pi_\phi(G_z)\\\pi\circ a\cong\pi}} \tx{tr}\,(\pi\boxtimes\rho^\vee)^\tx{can}(\tilde f,\tilde s^{-1}),\]
where $\mf{e}$ is the endoscopic datum corresponding to the pair $(\tilde s,\phi)$ by the spectral construction of \S\ref{sub:endocnst}, $\mf{z}$ is a $z$-pair for it, and
$f^\mf{z} \in \mc{C}^\infty_c(G^\mf{z}(F))$ satisfies 
\begin{equation} \label{eq:matchslice}
	SO_\gamma(f^\mf{z})\quad =\qquad \ssum{\tilde\delta \in [G \rtimes a]_z(F)/G_z(F)-conj} \Delta_{KS}[\mf{w},\mf{e},\mf{z}](\gamma^\mf{z},\tilde\delta)\int_{x \in G_z(F)/G_z(F)_{\tilde\delta}}\tilde f(x\tilde\delta x^{-1}).
\end{equation}
\end{cnj}

\begin{rem} \label{rem:canext}
We note that the extension $(\pi\boxtimes\rho^\vee)^\tx{can}$ is unique if it exists, due to the character identities it is supposed to satisfy. Furthermore, these character identities imply that for any $b \in A$ the isomorphism 
\[ \tilde G_z(F)_\pi \times_{A^{[z]}_\pi} \pi_0(\tilde S_{\phi,\rho}^{+,[z]}) \to \tilde G_z(F)_{b\pi} \times_{A^{[z]}_{b\pi}} \pi_0(\tilde S_{b\phi,b\rho}^{+,[z]}) \]
(given up to an inner automorphism) of conjugation by $b$ identifies the extension $(\pi \boxtimes \rho^\vee)^\tx{can}$ with $(b\pi \boxtimes b\rho^\vee)^\tx{can}$.
\end{rem}

\begin{pro} \label{pro:slice}
Conjecture 
\ref{cnj:coset} is equivalent to Conjectures \ref{cnj:llc_rigid} and \ref{cnj:llc_rest}. 
\end{pro}

As a preparation for the proof we need the following elementary discussion. Consider an exact sequence of locally pro-finite groups
\[ 1 \rw H \rw \tilde H \rw A \rw 1 \]
with $A$ finite and an $A$-invariant subset $X$ of the set of isomorphism classes of irreducible smooth representations of $H$. The group
\[ \tilde H \times_A \tilde H = \{(\tilde h_1,\tilde h_2) \in \tilde H \times \tilde H| \tilde h_1 \in \tilde h_2 H \} \] fits into the exact sequence
\[ 1 \to H \times H \to \tilde H \times_A \tilde H \to A \to 1. \]
\begin{lem} \label{lem:ex2}
\begin{enumerate}
	\item If $X=\{x\}$ there exists a projective representation $\tilde x$ of $\tilde H$ extending $x$ and satisfying $\tilde x(h \tilde h)=x(h) \circ \tilde x(\tilde h)$ for $h \in H$ and $\tilde h \in \tilde H$. The external tensor product $\tilde x \boxtimes \tilde x^\vee$ is a linear representation of $\tilde H \times_A \tilde H$ depending only on $x$, but not on $\tilde x$.
	\item The isomorphism class of the representation
	\[ \tilde X = \bigoplus_x \tx{Ind}_{\tilde H_{x} \times_{A_x} \tilde H_{x}}^{\tilde H \times_A \tilde H} \tilde x\boxtimes\tilde x^\vee, \]
	where $x$ runs over a set of representatives for the $A$-orbits in $X$, is independent of that set and is an extension of $\bigoplus_{x \in X}x\boxtimes x^\vee$.
	\item We have
	\[ \tx{Ind}_{\tilde H \times_A \tilde H}^{\tilde H \times \tilde H}\tilde X = \bigoplus (\xi \boxtimes\xi^\vee), \]
	where $\xi$ runs over the set of irreducible representations of $\tilde H$ lying over $X$.
	\item Given a diagram of extensions
	\[ \xymatrix{
	1\ar[r]&H\ar[r]\ar@{=}[d]&\tilde H_1\ar[r]\ar[d]&A_1\ar[r]\ar[d]&1\\
	1\ar[r]&H\ar[r]&\tilde H_2\ar[r]&A_2\ar[r]&1 }\]
	and an $A_2$-invariant set $X_2$, let $X_1$ be the set $X_2$ with the action of $A_1$ restricted from that of $A_2$. The representation $\tilde X_1$ of $\tilde H_1 \times_{A_1} \tilde H_1$ is the pull-back of the representation $\tilde X_2$ of $\tilde H_2 \times_{A_2} \tilde H_2$.
\end{enumerate}
\end{lem}

\begin{lem} \label{lem:ex1}
We are given two extensions $1 \to H_i \to \tilde H_i \to A \to 1$ of locally profinite groups with $A$ finite. Let $X$ be an $A$-set equipped with $A$-equivariant injections $X \to \tx{Irr}(H_i)$. For $x \in X$ we write $x_i$ for its image in $\tx{Irr}(H_i)$ and $\alpha_{x_i} \in H^2(A_x,\C^\times)$ for the associated class. We assume $\alpha_{x_1}=\alpha_{x_2}$. Assume given an extension to $\tilde H_{1,x} \times_{A_x} \tilde H_{2,x}$ of the representation $x_1 \boxtimes x_2^\vee$ of $H_1 \times H_2$, and call this extension $\tilde x$. Assume that for $a \in A$ and $y=ax$ we have $\tilde y = \tilde x \circ \tx{Ad}(a^{-1})$. Then
\begin{enumerate}
	\item If $X$ is a single $A$-orbit the isomorphism class of the representation 
	\[ \tilde X := \tx{Ind}_{\tilde H_{1,x} \times_{A_x} \tilde H_{2,x}}^{\tilde H_1 \times_A \tilde H_2} \tilde x \]
	is independent of the choice of $x \in X$. For general $X$ set
	\[ \tilde X := \bigoplus_{X' \in X/A} \tilde X'. \]
	\item The representation $\tilde X$ is an extension to $\tilde H_1 \times_A \tilde H_2$ of $\bigoplus_{x \in X} x_1 \boxtimes x_2^\vee$.
	\item Let 
	\[ \tx{Ind}_{\tilde H_1 \times_A \tilde H_2}^{\tilde H_1 \times \tilde H_2} \tilde X = \bigoplus (\xi_1 \boxtimes \xi_2^\vee)^{m(\xi_1,\xi_2)} \]
	be the decomposition into irreducible pieces. Then $m(\xi_1,\xi_2) \leq 1$ and the correspondence $\tx{Irr}(\tilde H_1) \leftrightarrow \tx{Irr}(\tilde H_2)$ afforded by $m$ is a bijection between the sets of irreducible representations of $\tilde H_i$ lying over the sets $X_i$.
	\item Let $A' \subset A$ be a subgroup and write $\tilde H_i' \subset \tilde H_i$ for the preimage of $A'$. All previous points can be applied to $\tilde H_i'$ in place of $\tilde H_i$. Let $\xi_i \in \tx{Irr}(\tilde H_i)$ and $\xi_i' \in \tx{Irr}(\tilde H_i')$ be such that $\xi_1 \leftrightarrow \xi_2$ under the bijection of point 3, and $\xi_1' \leftrightarrow \xi_2'$ under the analogous bijection. Then the multiplicity of $\xi_1'$ in $\xi_1|_{\tilde H_1'}$ equals the multiplicity of $\xi_2'$ in $\xi_2|_{\tilde H_2'}$.
\end{enumerate}
\end{lem}

\begin{proof}[Proof of Lemma \ref{lem:ex1}]
Since all statements break up according to the orbits of $A$ in $X$ we assume for the rest of the proof that $X$ is a single $A$-orbit. The independence of $\tilde X$ from $\tilde x$ follows from the assumption $\tilde y=\tilde x\circ\tx{Ad}(a^{-1})$ for $y=ax$ and the fact that $X$ is a transitive $A$-set. The fact that $\tilde X$ is an extension of $\bigoplus x_1 \boxtimes x_2^\vee$ follows from the induction-restriction formula. 

For the third claim we perform induction in stages
\[ \tx{Ind}_{\tilde H_{1,x} \times \tilde H_{2,x}}^{\tilde H_1 \times \tilde H_2} \tx{Ind}_{\tilde H_{1,x} \times_{A_x} \tilde H_{2,x}}^{\tilde H_{1,x} \times \tilde H_{2,x}} \tilde x\]
and consider first the inner induction. By assumption there exist projective representations $\tilde x_i$ of $\tilde H_{i,x}$ such that we have an equality $\alpha_{\tilde x_1}=\alpha_{\tilde x_2}$ of elements of $Z^2(A_x,\C^\times)$ and such that the restriction of $\tilde x_1 \boxtimes \tilde x_2^\vee$ from $\tilde H_{1,x} \times \tilde H_{2,x}$ to $\tilde H_{1,x} \times_{A_x} \tilde H_{2,x}$ is equal to $\tilde x$. Write $\alpha_{\tilde x}$ for the common value of $\alpha_{\tilde x_i}$. Then
\[ \tx{Ind}_{\tilde H_{1,x} \times_{A_x} \tilde H_{2,x}}^{\tilde H_{1,x} \times \tilde H_{2,x}} \tilde x = \left(\tx{Ind}_{\tilde H_{1,x} \times_{A_x} \tilde H_{2,x}}^{\tilde H_{1,x} \times \tilde H_{2,x}} \tb{1}\right) \otimes (\tilde x_1 \boxtimes \tilde x_2^\vee),\]
where on the right we are performing twisted induction with 2-cocycle $(\alpha_{\tilde x}^{-1},\alpha_{\tilde x})$ in the left factor, and then tensoring with the $(\alpha_{\tilde x},\alpha_{\tilde x}^{-1})$-projective representation $\tilde x_1 \boxtimes \tilde x_2^\vee$ to obtain a linear representation of $\tilde H_{1,x} \times \tilde H_{2,x}$. The representation $\tx{Ind}_{\tilde H_{1,x} \times_{A_x} \tilde H_{2,x}}^{\tilde H_{1,x} \times \tilde H_{2,x}} \tb{1}$ of $\tilde H_{1,x} \times \tilde H_{2,x}$ is the inflation of the representation $\tx{Ind}_{A_x}^{A_x \times A_x} \tb{1}$, where $A_x$ is embedded diagonally into $A_x \times A_x$. The latter is isomorphic to the twisted group algebra $\C[A_x]_{\alpha_{\tilde x}}$ seen as a left-right-bimodule over itself, and as such decomposes as the direct sum $\bigoplus_\tau \tau\boxtimes\tau^\vee$, where $\tau$ runs over the set of isomorphism classes of irreducible $\alpha_{\tilde x}$-projective representations of $A_x$. This shows that
\[ \tx{Ind}_{\tilde H_{1,x} \times_{A_x} \tilde H_{2,x}}^{\tilde H_{1,x} \times \tilde H_{2,x}} \tilde x = \bigoplus_\tau (\tau\otimes \tilde x_1)\boxtimes(\tau\otimes\tilde x_2)^\vee. \]
As $\tau$ runs over the set of isomorphism classes of $\alpha_{\tilde x}$-projective representations of $A_x$, $\tau\otimes\tilde x_i$, runs over the set of irreducible linear representations of $\tilde H_{i,x}$ whose restriction to $H_i$ contains $x_i$, and $\tx{Ind}_{\tilde H_i,x}^{\tilde H_i}$ runs over the set of irreducible linear representations of $\tilde H_i$ whose restriction to $H_i$ contains $x_i$.

For the fourth claim we write $\xi_i=\tx{Ind}_{\tilde H_{i,x}}^{\tilde H_i} \tilde x_i \otimes \tau$, where $\tilde x_i$ is an extension of $x_i$ to a projective representation of $\tilde H_{i,x}$ with 2-cocycle $\alpha_{\tilde x} \in Z^2(A_x,\C^\times)$ and $\tau$ is a projective representation of $A_x$ with 2-cocycle $\alpha_{\tilde x}^{-1}$. Write correspondingly $\xi_i'=\tx{Ind}_{\tilde H_{i,x}'}^{\tilde H_i'}\tilde x_i \otimes \tau'$, where we take the restriction of $\tilde x_i$ to $\tilde H_{i,x}'$ and $\tau'$ is a projective representation of $A'_x$ with 2-cocycle given by the restriction of $\alpha_{\tilde x}^{-1}$. The Mackey formula shows that
\[ \tx{Res}_{\tilde H_i'}^{\tilde H_i}\tx{Ind}_{\tilde H_{i,x}}^{\tilde H_i}\tilde x_i \otimes\tau = \bigoplus_{c \in A' \lmod A/A_x} \tx{Ind}_{\tilde H_{i,cx}'}^{\tilde H_i'}\tx{Res}_{\tilde H_{i,cx}'}^{\tilde H_{i,cx}}c(\tilde x_i\otimes\tau). \]
The summation index parameterizes the $A'$-orbits in the $A$-orbit of $x$. Since $\xi_i'$ lies over the $A'$-orbit of $x$, the multiplicity of $\xi_i'$ in the above restriction is zero for all summands except possibly the one indexed by $c=1$. This summand decomposes into the irreducible representations as
\[ \bigoplus_{\tau''}\tx{Ind}_{\tilde H_{i,x}'}^{\tilde H_i'}(\tilde x_i\otimes\tau'')^{m(\tau,\tau'')}, \] 
where $\tau''$ runs over the irreducible projective representations of $A_x'$ and $m(\tau,\tau'')$ is the multiplicity of $\tau''$ in the restriction of $\tau$. We conclude that $m(\xi_i,\xi'_i)=m(\tau,\tau')$.
\end{proof}

\begin{proof}[Proof of Lemma \ref{lem:ex2}]
For the first point we may fix a set $\dot A \subset \tilde H$ of representatives for $A=\tilde H/H$, an isomorphism $\tilde x(\dot a) : V_x \to V_x$ of complex vector spaces with the intertwining property $\tilde x(\dot a)\circ x(\dot a^{-1}h\dot a) = x(h)\circ\tilde x(\dot a)$ for all $h \in H$, and define $\tilde x(h\dot a)=x(h)\circ\tilde x(\dot a)$ for all $h \in H$ and $a \in A$. This has the required properties. We define the automorphism $\tilde x^\vee(\tilde h) := \tilde x(\tilde h)^{*,-1}$ of $V_x^\vee$. The linearity of $\tilde x \boxtimes \tilde x^\vee$ follows from the fact that the 2-cocycle of $\tilde x$ is inflated from $A$ and the 2-cocycle of $\tilde x^\vee$ is its inverse. Keeping $\dot A$ fixed, the independence from $\tilde x$ follows because another choice is of the form $c \cdot \tilde x$ for $c\in C^1(A,\C^\times)$. The independence of the choice of $\dot A$ now follows readily.

If $y=x\circ\tx{Ad}(a)$ choose $\tilde h_a \in \tilde H$ mapping to $a$ and define $\tilde y := \tilde x \circ\tx{Ad}(\tilde h_a)$. Then $\tilde y$ is a projective extension of $y$ and satisfies the conditions of point 1. We have $\tilde y \boxtimes \tilde y^\vee = (\tilde x \boxtimes \tilde x^\vee)\circ\tx{Ad}(\tilde h_a)$ and the second point follows from Lemma \ref{lem:ex1}.

The third point follows from the proof of the previous lemma, for we see from the argument given there that the right hand side is 
\[ \bigoplus_\tau \tx{Ind}_{\tilde H_x \times \tilde H_x}^{\tilde H \times \tilde H}(\tilde x\otimes\tau)\boxtimes(\tilde x\otimes\tau)^\vee, \] 
where $\tau$ runs over the irreducible $\alpha_{\tilde x}$-projective representations of $A_x$.

For the fourth point we note that the right square in the diagram is automatically cartesian. Since pull-back is transitive it is enough to treat the extreme cases when $A_1 \to A_2$ is injective respectively surjective, a property that is then inherited by the maps $\tilde H_1 \to \tilde H_2$ and $\tilde H_1 \times_{A_1} \tilde H_1 \to \tilde H_2 \times_{A_2} \tilde H_2$. We may assume that $X_2$ is a single $A_2$-orbit. Choose $x \in X_2$. In the injective case we apply point 2 and the Mackey formula to see that
\[ \tx{Res}_{\tilde H_1 \times_{A_1} \tilde H_1}^{\tilde H_2 \times_{A_2} \tilde H_2}\tx{Ind}_{\tilde H_{2,x} \times_{A_{2,x}} \tilde H_{2,x}}^{\tilde H_2 \times_{A_2} \tilde H_2} \tilde x \boxtimes \tilde x^\vee \]
is given by
\[ \bigoplus_{a \in A_1 \lmod A_2/A_{2,x}} \tx{Ind}_{\tilde H_{1,ax} \times_{A_{1,ax}} \tilde H_{1,ax}}^{\tilde H_1 \times_{A_1} \tilde H_1}\tx{Res}_{\tilde H_{1,ax} \times_{A_{1,ax}} \tilde H_{1,ax}}^{\tilde H_{2,ax} \times_{A_{2,ax}} \tilde H_{2,ax}} a(\tilde x \boxtimes \tilde x^\vee).\]
We have already argued that $a\tilde x$ is a projective extension of $ax$ to $\tilde H_{2,ax}$ with the required property for point 1, and it is clear that its restriction to $\tilde H_{1,ax}$ is such as well. Since $A_1 \lmod A_2/A_{2,x}$ parameterizes the $A_1$-orbits in $X$, the claim follows again from point $2$.

In the surjective case the kernel $N$ of $A_1 \to A_2$ is also the kernel of the surjective maps $\tilde H_1 \to \tilde H_2$, $\tilde H_{1,x} \to \tilde H_{2,x}$, $\tilde H_1 \times_{A_1} \tilde H_1 \to \tilde H_2 \times_{A_2} \tilde H_2$, and $\tilde H_{1,x} \times_{A_{1,x}} \tilde H_{1,x} \to \tilde H_{2,x} \times_{A_{2,x}} \tilde H_{2,x}$. The set $X_1$ is a single $A_1$-orbit. We apply again point 2. If $\tilde x$ is a projective extension of $x$ to $\tilde H_{2,x}$ satisfying the condition of point 1, then its pull-back to $\tilde H_{1,x}$ is a projective extension of $x$ that satisfies the same condition. We have
\[ \tx{Inf}_{\tilde H_1 \times_{A_1} \tilde H_1}^{\tilde H_2 \times_{A_2} \tilde H_2}\tx{Ind}_{\tilde H_{2,x} \times_{A_{2,x}} \tilde H_{2,x}}^{\tilde H_2 \times_{A_2} \tilde H_2} \tilde x \boxtimes \tilde x^\vee = \tx{Ind}_{\tilde H_{1,x} \times_{A_{1,x}} \tilde H_{1,x}}^{\tilde H_1 \times_{A_1} \tilde H_1} \tx{Inf}_{\tilde H_{1,x} \times_{A_{1,x}} \tilde H_{1,x}}^{\tilde H_{2,x} \times_{A_{2,x}} \tilde H_{2,x}} \tilde x \boxtimes \tilde x^\vee,\]
where $\tx{Inf}$ stands for inflation, i.e. pull-back.
\end{proof}

\begin{proof}[Proof of Proposition \ref{pro:slice}]
Assume conjecture \ref{cnj:coset}. We consider the extensions $1 \to G_z(F) \to \tilde G_z(F)^{[\phi]} \to A^{[\phi],[z]} \to 1$ and $1 \to \pi_0(S_\phi^+) \to \pi_0(\tilde S_\phi^{+,[z]}) \to A^{[\phi],[z]} \to 1$. Let $X_1=\Pi_\phi(G_z)$ and $X_2=\tx{Irr}(\pi_0(S_\phi^+),[z])$. By Conjecture \ref{cnj:func} these sets are in an $A^{[\phi],[z]}$-equivariant bijection. Define $\Pi_\phi(\tilde G_z^{[\phi]})$ to be the set of irreducible representations of $\tilde G_z(F)^{[\phi]}$ whose restriction to $G_z(F)$ meets $X_1$. Conjecture \ref{cnj:coset} fulfills the assumptions of Lemma \ref{lem:ex1}, see also Remark \ref{rem:canext}. Point 3 of that Lemma provides a bijection between $\Pi_\phi(\tilde G_z^{[\phi]})$ and $\tx{Irr}(\pi_0(S_\phi^{+,[z]},[z]))$, and point 4 asserts that this bijection preserves multiplicities upon restriction along a map $B \to A$. By Conjecture \ref{cnj:func} the stabilizer in $\tilde G_z(F)$ of any element of $\Pi_\phi(\tilde G_z^{[\phi]})$ is contained in $\tilde G_z(F)^{[\phi]}$. Therefore induction gives a bijection from $\Pi_\phi(\tilde G_z^{[\phi]})$ to the set $\Pi_\phi(\tilde G_z)$ of irreducible representations of $\tilde G_z(F)$ whose restriction to $G_z(F)$ meets $\Pi_\phi(G_z)$, hence the first point of Conjecture \ref{cnj:llc_rigid} holds. Multiplicities are still preserved, hence Conjecture \ref{cnj:llc_rest} holds.

Conversely, assume Conjecture \ref{cnj:llc_rest} and the first point of Conjecture \ref{cnj:llc_rigid}. Let $\pi \in \Pi_\phi(G_z)$ correspond to $\rho \in \tx{Irr}(\pi_0(S_\phi^+))$. Consider the group $\tilde G_z(F)_\pi \times_{A^{[z]}_\pi} \pi_0(\tilde S_{\phi,\rho}^{+,[z]})$. We claim that this group arises by taking $F$-points of a disconnected algebraic group that fits in the framework discussed in this paper. Indeed, let $\tilde G_\pi$ denote the preimage in $\tilde G$ of $A_\pi^{[z]}$. We have the isomorphism of algebraic groups
\[ G \rtimes_A \pi_0(\tilde S_{\phi,\rho}^{+,[z]}) \to \tilde G_\pi \times_{A^{[z]}_\pi} \pi_0(\tilde S_{\phi,\rho}^{+,[z]}),\qquad (g \rtimes \tilde s) \mapsto (g \rtimes a_{\tilde s}) \times \tilde s, \]
where on the left the subscript $A$ indicates that the semi-direct product is taken for the action of $\pi_0(\tilde S_{\phi,\rho}^{+,[z]})$ on $G$ via the projection $\pi_0(\tilde S_\phi^{+,[z]}) \to A_\pi^{[z]}$, while on the right the group $\pi_0(\tilde S_\phi^{+,[z]})$ acts trivially on $\tilde G$. The above isomorphism is equivariant for the natural embedding of $G$ into both sides and hence induces an isomorphism between the rational forms over $F$ determined by the element $z \in Z^1(u \to W,Z \to G)$.

According to Conjectures \ref{cnj:llc_rigid} and \ref{cnj:llc_rest} the extensions of the representation $\pi\boxtimes\rho^\vee$ of $G_z(F) \times \pi_0(S_\phi^+)$ to a representation of $\tilde G_z(F)_\pi \times_{A^{[z]}_\pi} \pi_0(\tilde S_{\phi,\rho}^{+,[z]})$ are in natural bijection with the extensions of the representation $\rho\boxtimes\rho^\vee$ of 
\[ \pi_0(\tx{Cent}(\phi,\hat G \times \pi_0(S_\phi^+))^+)=\pi_0(S_\phi^+) \times \pi_0(S_\phi^+) \] 
to the group
\[ \pi_0(\tx{Cent}(\phi,\hat G \rtimes_A \pi_0(\tilde S_{\phi,\rho}^{+,[z]}))^{+,[z]}). \]
To compute this group we use the isomorphism
\[ \hat{\bar G} \rtimes_A \tilde S_{\phi,\rho}^{+,[z]} \to \hat{\bar G} \rtimes_c \tilde S_{\phi,\rho}^{+,[z]},\qquad (g,\tilde s) \mapsto (g s^{-1},\tilde s), \]
where the subscript $c$ on the right indicates that we are taking the semi-direct product with respect to the natural conjugation action of $\tilde S_{\phi,\rho}^{+,[z]} \subset \hat{\bar G} \rtimes A$ on $\hat{\bar G}$, and $\tilde s = s \rtimes a$. This isomorphism restricts to an isomorphism $\hat{\bar G} \times S_\phi^+ \to \hat{\bar G} \rtimes_c S_\phi^+$. We now apply $\tx{Cent}(\phi,-)^{+,[z]}$ to both sides of the inclusion $\hat{\bar G} \rtimes_c S_\phi^+ \to \hat{\bar G} \rtimes_c \tilde S_{\phi,\rho}^{+,[z]}$ and obtain the natural inclusion
\[ S_\phi^+ \rtimes_c S_\phi^+ \to S_\phi^+ \rtimes_c \tilde S_{\phi,\rho}^{+,[z]}, \]
which under the isomorphism $(s,\tilde s) \mapsto (s\tilde s,\tilde s)$ becomes the natural inclusion
\[ S_\phi^+ \times S_\phi^+ \to \tilde S_{\phi,\rho}^{+,[z]} \times_{A^{[\phi],[z]}_\rho} \tilde S_{\phi,\rho}^{+,[z]}. \]
Tracing through all identifications we see that we are looking for a natural extension of the representation $\rho\boxtimes\rho^\vee$ of the source of this inclusion to a representation of the target. But Lemma \ref{lem:ex2} provides just such an extension.

We come now to the character identities. The right hand side of the character identities in Conjecture \ref{cnj:llc_rigid} is
\[ \sum_{\tilde\pi \in \Pi_\phi(\tilde G_z)} \tx{tr}\, \tilde\pi\boxtimes\tilde\rho_{\tilde \pi}^\vee(\tilde f \times \tilde s^{-1}).\]
This is the character of the representation $\bigoplus_{\tilde\pi} \tilde\pi\boxtimes\tilde\rho_{\tilde\pi}^\vee$ evaluated at the function $\tilde f \otimes \delta_{\tilde s^{-1}}$. By the preceding discussion that representation is equal to
\[ \bigoplus_{\pi \in \Pi_\phi(G_z)/A^{[\phi],[z]}} \tx{Ind}_{\tilde G_z(F)_\pi \times_{A^{[z]}_\pi} \pi_0(\tilde S_{\phi,\rho}^{+,[z]})}^{\tilde G_z(F) \times \pi_0(\tilde S_{\phi}^{+,[z]})} \tilde\pi^\tx{can}. 
\]
Let $a \in A$ be the image of $\tilde s$. If $a \notin A_\pi$ the character of the corresponding induced representation is zero at $\tilde f \otimes \delta_{\tilde s^{-1}}$. Therefore we may restrict the sum by the condition $a \in A^{[z]}_\pi$, equivalently $a\pi \cong \pi$. Applying the Frobenius character formula we obtain 
\[ \sum_{\substack{\pi \in \Pi_\phi(G_z)/A^{[\phi],[z]}\\ a\pi \cong \pi}} \tx{tr}\,\tilde\pi^\tx{can}\left( \sum_{x}(\tilde f\otimes\delta_{\tilde s^{-1}})^x \right),\]
where $x$ runs over the coset space
\[ (\tilde G_z(F) \times  \pi_0(\tilde S_{\phi}^{+,[z]}))/(\tilde G_z(F)_\pi \times_{A^{[z]}_\pi} \pi_0(\tilde S_{\phi,\rho}^{+,[z]})) \cong A^{[z]} \times A^{[z],[\phi]}/A^{[z]}_\rho. \]
The compatibility of $\tilde\pi^\tx{can}$ with conjugation under $A$ implies that the above sum becomes
\[ \sum_{\substack{\pi \in \Pi_\phi(G_z)\\ a\pi \cong \pi}} \tx{tr}\,\tilde\pi^\tx{can}\left( \sum_{c \in A^{[z]}}\tilde f^c\otimes\delta_{\tilde s^{-1}} \right). \]
We conclude that
\[ \sum_{\tilde\pi \in \Pi_\phi(\tilde G_z)} \tx{tr}\, \tilde\pi\boxtimes\tilde\rho^\vee(\tilde f \times \tilde s^{-1}) = \sum_{\substack{\pi \in \Pi_\phi(G_z)\\ a\pi \cong \pi}} \tx{tr}\,\tilde\pi^\tx{can}\left( \tilde f_0\times\tilde s^{-1} \right), \]
where $\tilde f_0=\sum_{c \in A^{[z]}} \tilde f^c$. Recalling from Lemma \ref{lem:trans} that $\tilde f^\mf{z}=\tilde f_0^{\mf{z},KS}$ we see that the character identities in Conjectures \ref{cnj:llc_rigid} and \ref{cnj:coset} are equivalent.
\end{proof}

\begin{rem} Proposition \ref{pro:slice} reduces the proof of the endoscopic character identities to the case of a cyclic $A$. It does not completely reduce the internal structure of $L$-packets to the case of cyclic $A$, because in the case when $A_\pi$ is not cyclic one still needs to show the existence of the extension $\tilde\pi^\tx{can}$.
\end{rem}

\subsection{The cyclic case} \label{sub:comp_cyc}

In this subsection we revisit the classical setting where we have a connected reductive group equipped with an automorphism. We begin with a quasi-split connected reductive group $G$ equipped with an automorphism $\theta$ fixing an $F$-pinning. We further assume $\theta$ is of finite order. Set $A=\<\theta\>$ and $\tilde G = G \rtimes \<\theta\>$. Let $z \in Z^1(u \to W,Z(G)^A \to G)$. The map $\tilde G_z(F) \to A$ is surjective if and only if the class $[z]$ is fixed by $\theta$, which we assume from now on, for otherwise we can pass to a power of $\theta$ without changing $\tilde G_z(F)$. Fix an arbitrary $\tilde \delta_z \in \tilde G_z(F)$ mapping to $\theta$ and set $\theta_z =\tx{Ad}(\tilde\delta_z)$. The twisted group we are interested in is $G_z$ with automorphism $\theta_z$.

Let $\phi : L_F \to {^LG}$ be such that its $\hat G$-conjugacy class is fixed by $\theta$. Then we have
\begin{equation} \label{eq:sphicyc} 
1 \to \pi_0(S_\phi) \to \pi_0(\tilde S_\phi) \to A \to 1 
\end{equation}
and $\tilde S_\phi^{[z]}=\tilde S_\phi$. This isomorphism class of a representation $\pi \in \Pi_\phi(G_z)$ is $\theta_z$-fixed if and only if the isomorphism class of the corresponding $\rho \in \tx{Irr}(S_\phi^+,[z])$ is $\theta$-fixed. Assuming that this is the case, there is a natural extension of the representation $\pi\boxtimes\rho^\vee$ of $G_z(F) \times \pi_0(S_\phi^+)$ to a representation of $\tilde G_z(F) \times_A \pi_0(\tilde S_\phi^+)$ given as follows: Since $A$ is cyclic $\rho$ extends to a representation $\tilde\rho$ of $\pi_0(\tilde S_\phi^+)$. By Conjecture \ref{cnj:llc_rigid} there is a corresponding extension $\tilde\pi$ of $\pi$ to $\tilde G_z(F)$. Another extension of $\rho$ is of the form $\tilde\rho\otimes\chi$ for some character $\chi$ of $A$. The representation of $\tilde G_z(F)$ corresponding to $\tilde\rho\otimes\chi$ is then $\tilde\pi\otimes\chi$. Therefore the representation $\tilde\pi \boxtimes\tilde\rho^\vee$ of $\tilde G_z(F) \times \pi_0(\tilde S_\phi^+)$, when pulled back to $\tilde G_z(F) \times_A \pi_0(\tilde S_\phi^+)$, is independent of the choice of $\tilde\rho$. This is $(\pi\boxtimes\rho^\vee)^\tx{can}$ of Conjecture \ref{cnj:coset}.

\subsection{Passing from $A$ to $A^{[z],[\phi]}$} \label{sub:comp_rest1}

Proposition \ref{pro:slice} shows that, once $[\phi]$ and $[z]$ have been fixed and Conjecture \ref{cnj:func} has been assumed, Conjecture \ref{cnj:llc_rigid} for the group $G \rtimes A$ reduces to the same conjecture for the group $G \rtimes A^{[z],[\phi]}$. More explicitly, let $B=A^{[z],[\phi]}$ and write $G^A=G \rtimes A$ and $G^B=G \rtimes B$. Let $\pi^B \in \Pi_\phi(G^B_z)$. All members of $\tx{Res}^{G^B_z(F)}_{G_z(F)}\pi^B$ belong to the packet $\Pi_\phi(G_z)$. An element of $G^A_z(F)$ that normalizes $G^B_z(F)$ and intertwines $\pi^B$ must therefore lie in $G^B_z(F)$. Thus $\pi^A:=\tx{Ind}_{G^B_z(F)}^{G^A_z(F)}\pi^B$ is irreducible, and in this way one obtains a bijection $\Pi_\phi(G^B_z) \to \Pi_\phi(G^A_z)$. In the same way one obtain a bijection $\tx{Irr}(\pi_0(\tilde S_\phi(G^B)^{+,[z]}),[z]) \to \tx{Irr}(\pi_0(\tilde S_\phi(G^A)^{+,[z]}),[z])$.

\subsection{Induction} \label{sub:ind}

Let $\tilde G=G \rtimes A$ be a quasi-split disconnected reductive group and $A \to B$ an embedding. Set $H=\tx{Ind}_A^BG = \{h : B \to G|h(ab)=a(h(b))\}$ with pointwise multiplication and $\tilde H = H \rtimes B$. The purpose of this subsection is to show that Conjecture \ref{cnj:coset} for $G \rtimes A$ implies this same conjecture for $H \rtimes B$. This will be accomplished in Proposition \ref{pro:ind}. The discussion required in the proof is elementary, but unfortunately quite technical and cumbersome.

Let $a \in A$. An element $z_G \in Z^1(u \to W,Z(G)^A \to G)$ has $a$-invariant cohomology class if and only if there exists $g_a \in G$ such that 
\begin{equation}
a(z_G(w))=g_a^{-1}z_G(w)\sigma_w(g_a). \label{eq:ind1a}
\end{equation}
This is equivalent to $g_a \rtimes a \in \tilde G_{z_G}(F)$. Assuming that, a representation $\pi_G$ of $G_{z_G}(F)$ has an $a$-invariant isomorphism class if and only if there exists a vector space isomorphism $\tilde\pi_A(g_a \rtimes a) : V_{\pi_G} \to V_{\pi_G}$ satisfying
\begin{equation}
\pi_G(g_a \cdot a(g) \cdot g_a^{-1})\circ\tilde\pi_G(g_a \rtimes a)=\tilde\pi_G(g_a \rtimes a)\circ \pi_G(g). \label{eq:ind1b}
\end{equation}
Note that $g_a\cdot a(g) \cdot g_a^{-1}=(g_a \rtimes a)\cdot g \cdot (g_a \rtimes a)^{-1}$, and further that the existence of $\tilde\pi_G(g_a \rtimes a)$ is independent of the choice of $g_a$, for any other choice will be of the form $g'_ag_a$ with $g'_a \in G_{z_G}(F)$ and we can take $\tilde\pi_G(g'_ag_a \rtimes a)=\pi_G(g'_a)\circ\tilde\pi_G(g_a \rtimes a)$.

Let $z_H \in Z^1(u \to W,Z(H)^B \to H)$. Thus $z_H(w_1w_2,b)=z_H(w_1,b) \cdot w_1z_H(w_2,b)$ and $z_H(w,ab)=a(z_H(w,b))$. Let $b \in B$. The cohomology class of $z_H$ is $b$-invariant if and only if there exists $h_b \in H$ satisfying the analog of Equation \eqref{eq:ind1a}. Again this is equivalent to $\tilde h = h_b \rtimes b \in \tilde H$ lying in $\tilde H_{z_H}(F)$ and in terms of the function $h_b : B \to G$ means
\begin{equation} \label{eq:ind1c}
z_H(w,b'b)=h^{-1}_b(b')z_H(w,b')\sigma_w(h_b(b')),\qquad\forall b' \in B. \end{equation}

A representation $(\pi_H,V_{\pi_H})$ of $H_{z_H}(F)$ can be represented as a collection of vector spaces $\{V_c|c \in A \lmod B\}$ and on each $V_c$ a family of representations $\pi_H^{\dot c} : G_{z_H(-,\dot c)}(F) \to \tx{Aut}_\C(V_c)$ indexed by $\dot c \in c$ satisfying the compatibility relation
\[ \pi_H^{a\dot c}(a(g))=\pi_H^{\dot c}(g) \]
for all $a \in A$ and $g \in G_{z_H(-,\dot c)}(F)$. Then we have $V_{\pi_H}=\otimes_c V_c$ and $\pi_H(h)=\otimes_{c \in A \lmod B} \pi_H^c(h(c))$ and each factor is well-defined. We shall write $\pi_H=\boxtimes_c \pi_H^c$.

Assuming the existence of $h_b \rtimes b \in H_{z_H}(F)$, a  representation $\pi_H$ of $H_{z_H}(F)$ has a $b$-invariant isomorphism class if and only if there exists a vector space isomorphism $\tilde\pi_H(h_b \rtimes b) : V_{\pi_H} \to V_{\pi_H}$ satisfying the analog of Equation \eqref{eq:ind1b}, which in terms of the data $\{V_c\}$ and $\{\pi_H^{\dot c}\}$ can be expressed as 
\[ \tilde\pi_H(h_b \rtimes b)(\otimes_c v_c)=\otimes_c \tilde\pi_H(h_b \rtimes b)_c(v_{cb}), \]
where 
\[ \tilde\pi_H(h_b \rtimes b)_c : V_{cb} \to V_c \]
is an isomorphism of vector spaces satisfying
\begin{equation} \label{eq:ind1d}
\pi_H^{\dot c}(g)\circ\tilde\pi_H(h_b \rtimes b)_c = \tilde\pi_H(h_b \rtimes b)_c \circ \pi_H^{\dot cb}(h_b(\dot c)^{-1}g h_b(\dot c))\end{equation}
for one, hence any, lift $\dot c \in B$ of $c \in A \lmod B$.

For $g \in G$ and $b \in B$ we define $g^{\delta_b} \in H$ to be the function $B \to G$ supported on $Ab$ and sending $ab \in Ab$ to $a(g) \in G$. Given a section $s : A \lmod B \to B$ of the natural projection (which we may view as a map $B \to B$ invariant under left multiplication by $A$) we define the retraction $r : B \to A$ by $b=r(b)s(b)$. We have $r(ab)=ar(b)$.

\begin{lem} \label{lem:ind1}
\begin{enumerate}
	\item Given $z_H \in Z^1(u \to W,Z(H)^B \to H)$ define $z_G(w) = z_H(w,1)$. Then $[z_H] \mapsto [z_G]$ establishes a bijection between 
	\[ H^1(u \to W,Z(H)^B \to H)^B \to H^1(u \to W,Z(G)^A \to G)^A. \]
	\item Assume $z_G(w)=z_H(w,1)$. Given a representation $\pi_H$ of $H_{z_H}(F)$, present it as the collection $(\pi^{\dot c}_H)$ of representations as above, and define a representation $\pi_G$ of $G_{z_G}(F)$ by $\pi_G(g)=\pi_H^1(g)$. 
	Then $[\pi_H] \mapsto [\pi_G]$ establishes a bijection between the set of $B$-fixed isomorphism classes of irreducible representations of $H_{z_H}(F)$ and the set of $A$-fixed isomorphism classes of irreducible representations of $G_{z_G}(F)$.
	\item Fix a section $s : A \lmod B \to B$. Given $z_G \in Z^1(u \to W,Z(G)^A \to G)$ define $z_H \in Z^1(u \to W,Z(H)^B \to H)$ by $z_H(w,as(c))=a(z_G(w))$. Given a representation $\pi_G$ of $G_{z_G}(F)$ define a representation of $H_{z_H}(F)$ by $\boxtimes_c \pi_H^c$, where $\pi_H^{s(c)}=\pi_G$. These assignments are inverses of the above bijections.
\end{enumerate}
\end{lem}
\begin{proof}
Let $z_H \in Z^1(u \to W,Z(H)^B \to H)$ and $\pi_H \in \tx{Irr}(H_{z_H}(F))$ have $B$-fixed classes. Then Equation \eqref{eq:ind1c} for $b'=1$ and $b=a\in A$ shows that the class of $z_G(w)=z_H(w,1)$ is $a$-fixed, while Equation \eqref{eq:ind1d} with $\dot c=1$ and $b=a \in A$ shows that the class of $\pi_G(g)=\pi_H^1(g)$ is $a$-fixed. It is clear that the classes of $z_G$ and $\pi_G$ depend only on those of $z_H$ and $\pi_H$.

Consider conversely $z_G \in Z^1(u \to W,Z(G)^A \to G)$ and $\pi_G \in \tx{Irr}(G_{z_G}(F))$ whose classes are $A$-fixed and let $g_a$ and $\tilde\pi_G(g_a \rtimes a)$ be chosen to satisfy Equations \eqref{eq:ind1a} and \eqref{eq:ind1b}. Fix a section $s : A \lmod B \to B$ and let $r : B \to A$ be defined by $b=r(b)s(b)$ for all $b \in B$. Set $h_b(as(c)) = a(g_{r(s(c)b)})$. Then $h_b \in H$ and an easy calculation shows that Equation \eqref{eq:ind1a} implies Equation \eqref{eq:ind1c} for $z_H(w,as(c)):=a(z_G(w))$, and further that if $h \in H_{z_H}(F)$ then $h(s(c)) \in G_{z_G}(F)$ for all $c \in A \lmod B$. This allows us to define a representation $\pi_H$ of $H_{z_H}(F)$ acting on the vector space $(V_{\pi_G})^{\otimes_c}$ by $\pi_H(h)=\otimes_c \pi_G(h(s(c)))$. In other words, the representation $\pi_H$ is given by the constant collection of vector spaces $\{V_c=V_{\pi_G}|c \in A \lmod B\}$ and for each $c \in A \lmod B$ we have $\pi_H^{as(c)}(g)=\pi_G(a^{-1}(g))$. Define $\tilde\pi_H(h_b \rtimes b)_c : V_{cb} \to V_c$ to be given by $\tilde\pi_G(g_a \rtimes a) : V_{\pi_G} \to V_{\pi_G}$ for $a=r(s(c)b)$. Then Equation \eqref{eq:ind1d} for $\dot c=s(c)$ follows from Equation \eqref{eq:ind1b}. It is clear that the classes of $z_H$ and $\pi_H$ depend only on those of $z_G$ and $\pi_G$. 

We have thus established the desired maps in both directions and must now check that they are mutually inverse. Starting with $z_G$ and $\pi_G$ and constructing $z_H$ and $\pi_H$ it is immediate that $z_H(w,1)=z_G(w)$ and $\pi_H^1(g)=\pi_G(g)$. Conversely start with $z_H$ and $\pi_H$ and define $z_G(w)=z_H(w,1)$ and $\pi_G(g)=\pi_H^1(g)$. Let now $z_H^0(w,as(c))=a(z_G(w))$ and $\pi_H^0(h)=\otimes_c \pi_G(h(s(c)))$. We need to show that the classes of $z_H$ and $z_H^0$ are equal, and the classes of $\pi_H$ and $\pi_H^0$ are equal.

For $z_H$ and $z_H^0$ we need to show the existence of $h \in H$ such that for all $a \in A$ and $c \in A \lmod B$ we have
\[ z_H^0(w,as(c))=h(as(c))^{-1}z_H(w,as(c))\sigma(h(as(c))), \]
which due to the $A$-equivariance of all terms and the definition of $z_H^0$ reduces to
\[ z_H(w,1)=h(s(c))^{-1}z_H(w,s(c))\sigma_w(h(s(c))), \]
which follows from Equation \eqref{eq:ind1c} with $h=h_b$, $b={s(c)^{-1}}$ and $b'=s(c)$. Thus the element $h$ we are looking for is given by $h(as(c)) := a(h_{s(c)^{-1}}(s(c)))$.

Before we can compare the classes of $\pi_H$ and $\pi_H^0$ we note that the former is a representation of $H_{z_H}(F)$, while the latter is a representation of $H_{z_H^0}(F)$. We must therefore precompose $\pi_H$ with the isomorphism $\tx{Ad}(h) : H_{z_H^0}(F) \to H_{z_H}(F)$. Thus we need to show the existence of a $(\pi_H\circ\tx{Ad}(h),\pi_H^0)$-equivariant vector space isomorphism $V_{\pi_H} \to V_{\pi_H^0}$. This reduces to finding for each $c \in A \lmod B$ a vector space isomorphism $V_c \to V_1$ translating the action of $G_{z_G}(F)$ on $V_c$ given by $\pi_H^{s(c)}\circ\tx{Ad}(h_{s(c)^{-1}}(s(c)))$ to the action of on $V_1$ given by $\pi_H^1$. According to Equation \eqref{eq:ind1d} such an isomorphism is given by $\tilde\pi_H(h_b \rtimes b)_c^{-1}$ for $b=s(c)^{-1}$.
\end{proof}

\begin{lem} \label{lem:ind2}
Under the bijection $\pi_G \leftrightarrow \pi_H$ of Lemma \ref{lem:ind1} the element of $H^2(B,\C^\times)$ corresponding to $\pi_H$ is the corestriction of the element of $H^2(A,\C^\times)$ corresponding to $\pi_G$.

More precisely, let $z_G \in Z^1(u \to W,Z(G)^A \to G)$ and $\pi_G \in \tx{Irr}(G_{z_G}(F))$ have $A$-fixed classes. For each $a \in A$ fix $g_a \in G$ and $\tilde\pi_G(g_a \rtimes a)$ satisfying Equations \eqref{eq:ind1a} and \eqref{eq:ind1b}, so that we have the element
\[ \alpha(a_1,a_2) = \tilde\pi_G( g_{a_1} \rtimes a_1)\circ\tilde\pi_G( g_{a_2} \rtimes a_2) \circ \tilde\pi_G( g_{a_1} \rtimes a_1 \cdot g_{a_2} \rtimes a_2)^{-1} \]
of $Z^2(A,\C^\times)$ representing the class associated to $\pi_G$, where the third term is defined via the rule $\tilde\pi_G(gg_a \rtimes a)=\pi_G(g)\tilde\pi_G(g_a \rtimes a)$ for $g \in G_{z_G}(F)$. Define $z_H(w,as(c))=a(z_G(w))$. Define the representation $\pi_H$ of $H_{z_H}(F)$ as $\pi_H=\boxtimes_c \pi_H^c$, $\pi_H^{s(c)}=\pi_G$. For each $b \in B$ define the element $h_b \in H$ by $h_b(as(c))=a(g_{r(s(c)b)})$ and the isomorphism $\tilde\pi_H(h_b \rtimes b) : \otimes_c V_{\pi_G} \to \otimes_c V_{\pi_G}$ by $\tilde\pi_H(h_b \rtimes b)(\otimes_c v_c)=\otimes_c \tilde\pi_G(g_{r(s(c)b)} \rtimes r(s(c)b))(v_{cb})$. Then $h_b$ and $\tilde\pi_H(h_b \rtimes b)$ satisfy Equations \eqref{eq:ind1c} and \eqref{eq:ind1d} and the associated element
\[ \beta(b_1,b_2) = \tilde\pi_H( h_{b_1} \rtimes b_1)\circ\tilde\pi_H( h_{b_2} \rtimes b_2) \circ \tilde\pi_H( h_{b_1} \rtimes b_1 \cdot h_{b_2} \rtimes b_2)^{-1} \]
of $Z^2(B,\C^\times)$ is obtained from $\alpha$ by applying the cochain formula for corestriction with respect to the section $s$.
\end{lem}
\begin{proof}
That Equations \eqref{eq:ind1c} and \eqref{eq:ind1d} are satisfied was already discussed in the proof of Lemma \ref{lem:ind1}. It remains to prove the corestriction claim, which is the following identity:
\[ \beta(b_1,b_2) = \prod_c \alpha(r(s(c)b_1),r(s(cb_1)b_2)). \]
The (scalar) endomorphism
\[ \tilde\pi_H( h_{b_1} \rtimes b_1)\circ\tilde\pi_H( h_{b_2} \rtimes b_2) \circ \tilde\pi_H( h_{b_1} \rtimes b_1 \cdot h_{b_2} \rtimes b_2)^{-1} \]
of the vector space $\otimes_c V_{\pi_G}$ is by definition a tensor product of endomorphisms of $V_{\pi_G}$ and it is enough to show that the endomorphism of the tensor factor indexed by $c$ is given by multiplication by the scalar $\alpha(r(s(c)b_1),r(s(cb_1)b_2))$. By definition this endomorphism is given by
\[ \tilde\pi_H(h_{b_1} \rtimes b_1)_c \circ \tilde\pi_H(h_{b_2} \rtimes b_2)_{cb_1} \circ \tilde\pi_H(h_{b_1}\rtimes b_1 \cdot h_{b_2} \rtimes b_2)_c^{-1}, \]
where the subscript notation is as in the proof of Lemma \ref{lem:ind1}. We compute
\begin{eqnarray*}
&&\tilde\pi_H(h_{b_1}\rtimes b_1 \cdot h_{b_2} \rtimes b_2)_c\\
&=&\pi_H(h_{b_1}\cdot b_1h_{b_2}\cdot h_{b_1b_2}^{-1})\circ\tilde\pi_H(h_{b_1b_2} \rtimes b_1b_2)_c\\
&=&\pi_G(h_{b_1}(s(c))\cdot h_{b_2}(s(c)b_1)h_{b_1b_2}(s(c))^{-1})\circ\tilde\pi_G(g_{r(s(c)b_1b_2)}\rtimes r(s(c)b_1b_2))\\
&=&\tilde \pi_G(g_{r(s(c)b_1)}\cdot r(s(c)b_1)g_{r(s(cb_1)b_2)}\cdot g_{r(s(c)b_1b_2)}^{-1}\cdot g_{r(s(c)b_1b_2)}\rtimes r(s(c)b_1b_2))\\
&=&\tilde\pi_G(g_{r(s(c)b_1)} \rtimes r(s(c)b_1) \cdot g_{r(s(cb_1)b_2)} \rtimes r(s(cb_1)b_2)). 
\end{eqnarray*}
With this we see that the endomorphism of the tensor factor indexed by $c$ is given by 
\begin{align*}
\tilde\pi_G(g_{r(s(c)b_1)} \rtimes r(s(c)b_1))&\circ\tilde\pi_G(g_{r(s(cb_1)b_2)} \rtimes r(s(cb_1)b_2))\\
&\circ\tilde\pi_G(g_{r(s(c)b_1)} \rtimes r(s(c)b_1) \cdot g_{r(s(cb_1)b_2)} \rtimes r(s(cb_1)b_2)),
\end{align*}
which is precisely $\alpha(r(s(c)b_1),r(s(cb_1)b_2))$.
\end{proof}

We consider the group homomorphism
\[ \tx{ev}_1 : H \to G,\qquad h \mapsto h(1). \]
It is $A$-equivariant, hence extends to a group homomorphism
\[ \tx{ev}_1 : H \rtimes A \to G \rtimes A,\qquad h \rtimes a \mapsto \tx{ev}_1(h) \rtimes a. \]
It also respects rational structures under the convention $z_G(w)=z_H(w,1)$ that has been used so far.

More generally we consider $b \in B$ and a section $l : A \lmod B/\<b\> \to B$. For every $d \in A \lmod B/\<b\>$ let $n_d$ be the size of the orbit of the element $Al(d)$ of $A \lmod B$ for the action of $b$ on $A \lmod B$ by right multiplication. Equivalently, $n_d$ is the smallest non-negative number $n$ satisfying $b^n \in l(d)^{-1}Al(d)$. Write $A_d:=l(d)^{-1}Al(d) \subset B$ so that $b^{n_d} \in A_d$, and write $a_d=l(d)b^{n_d}l(d)^{-1} \in A$. We obtain a section $s : A \lmod B \to B$ by
\begin{equation} \label{eq:sec1}
s(Al(d)b^i) = l(d)b^i,\qquad i=0,\dots,n_d-1.	
\end{equation}
The group homomorphism
\[ \tx{ev}_{l(d)} : H \to G,\qquad h \mapsto h(l(d)) \]
satisfies $\tx{ev}_{l(d)}(l(d)^{-1}al(d)h)=a\tx{ev}_{l(d)}(h)$ for $a \in A$ and $h \in H$ and therefore extends to a group homomorphism
\[ \tx{ev}_{l(d)} : H \rtimes A_d \to G \rtimes A, \qquad h \rtimes l(d)^{-1}al(d) \mapsto h(l(d)) \rtimes a,\]
which is defined over $F$ under the assumption $z_H(w,l(d))=z_G(w)$, which is implied by the assumption $z_H(w,s(c))=z_G(w)$ for all $c \in A \lmod B$.

\begin{lem} \label{lem:indprod0}
Let $z_G^d \in Z^1(u \to W,Z(G)^A \to G)$. Define $z_H \in Z^1(u \to W,Z(H)^B \to B)$ by $z_H(w,al(d)b^i)=a(z_G^d(w))$. The map
\[ H \to \prod_d \prod_{i=0}^{n_d-1} G,\qquad h \mapsto \prod_d \prod_{i=0}^{n_d-1} h(l(d)b^i) \]
is an isomorphism of algebraic groups. It respects the quasi-split rational structures on both sides, as well as their twists by $z_H$ and $(z_G^d)_d$ respectively. It translates the action by $b$ to the action by $(\Theta_d)_d$, where $\Theta_d(g_{d,0},\dots,g_{d,n_d-1})=(g_{d,1},\dots,g_{n_d-1},a_d(g_0))$.
\end{lem}
\begin{proof}
This is an immediate computation.
\end{proof}

Note that $h(l(d)b^i)=\tx{ev}_{l(d)}(b^ih)=\tx{ev}_{l(d)}(\tx{Ad}(1 \rtimes b)^ih)$. Since the action of $b$ on $H$, as well as the action of $a_d$ on $G$, need not respect the rational structures given by $z_H$ and $z_G^d$, respectively, the following slight variation of the above isomorphism will also be useful.

\begin{lem} \label{lem:indprod}
Let $\tilde h \in [H \rtimes b]_{z_H}(F)$. The map
\[ H \to \prod_d \prod_{i=0}^{n_d-1} G,\qquad h \mapsto \prod_d \prod_{i=0}^{n_d-1} \tx{ev}_{l(d)}(\tx{Ad}(\tilde h)^ih) \]
is an isomorphism of algebraic groups that respects the twists of the quasi-split rational structures by $z_H$ and $(z_G^d)_d$, respectively. It translates the action of conjugation by $\tilde h$ to the action sending $(g_{d,0},\dots,g_{d,n_d-1})_d$ to $(g_{d,1},\dots,g_{d,n_d-1},\tx{Ad}(\tilde g_d) g_0)_d$, where $\tilde g_d=\tx{ev}_{l(d)}(\tilde h^{n_d}) \in [G \rtimes a_d]_{z_G}(F)$.
\end{lem}
\begin{proof}
This is an immediate computation.
\end{proof}

For a moment we consider the following situation that encapsulates each factor in the first product of above lemma. 

\begin{lem} \label{lem:basictwist}
Let $J$ be a locally profinite group with an automorphism $\theta$ and consider the locally profinite group $I=J \times J \times \dots \times J$ with the automorphism $\Theta(j_0,\dots,j_{n-1})=(j_1,\dots,j_{n-1},\theta(j_0))$. Consider the maps
\[ m,p_0 : I \to J,\quad m(j_0,\dots,j_{n-1})=j_0\dots j_{n-1},\quad p_0(j_0,\dots,j_{n-1})=j_0\]
as well as
\[ \Delta,i_0 : J \to I,\quad \Delta(j)=(j,\dots,j),\quad i_0(j)=(j,1,\dots,1). \]
\begin{enumerate}
	\item We have $m(g_I^{-1}\cdot \delta_I\cdot \Theta(g_I))=p_0(g_I)^{-1}\cdot m(\delta_I)\cdot\theta(p_0(g_I))$ for $g_I,\delta_I \in I$.
	\item The map $m$ induces a bijection from the set of $\Theta$-twisted conjugacy classes in $I$ to the set of $\theta$-twisted conjugacy classes in $J$ with inverse given by $i_0$.
	\item The map $p_0$ induces an isomorphism of groups $\tx{Cent}_\Theta(\delta_I,I) \to \tx{Cent}_\theta(m(\delta_I),J)$ whose inverse sends $s$ to $\tx{Ad}(g_I)\Delta(s)$, where $g_I=(g_0,\dots,g_{n-1})$ and $g_i=(\delta_0\dots \delta_{i-1})^{-1}$.
	\item If $f_0,\dots,f_{n-1} \in \mc{C}^\infty_c(J)$, $f_I=f_0\otimes\dots\otimes f_{n-1}$, and $\delta_I \in I$, then
	\[ TO^{I,\Theta}_{\delta_I}(f_I)=TO^{J,\theta}_{m(\delta_I)}(f_0 * f_1 * \dots * f_{n-1}), \]
	where the convolution $f_0*\dots*f_{n-1} \in \mc{C}_c^\infty(J)$ is defined by $f_0*\dots*f_{n-1}(x)=\int f_0(h_1)f_1(h_1^{-1}h_2)\dots f_{n-2}(h_{n-2}^{-1}h_{n-1}) f_{n-1}(h_{n-1}^{-1}x)dh_1\dots dh_{n-1}$.
	\item Let $\pi$ be an admissible representation of $J$ and let $\tilde\pi : \pi\circ\theta^{-1} \to \pi$ be an isomorphism. Then $\pi_I = \pi\boxtimes\dots\boxtimes\pi$ is an admissible representation of $I$ and $\tilde\pi_I(v_0\otimes\dots\otimes v_{n-1})=v_1\otimes\dots\otimes v_{n-1}\otimes\tilde\pi(v_0)$ is an isomorphism $\pi_I \circ \Theta^{-1} \to \pi_I$. We have
\[ \tx{tr}(\pi_I(f_I)\circ\tilde\pi_I)=\tx{tr}(\pi(f_0 * \dots * f_{n-1})\circ\tilde\pi). \]
\end{enumerate}	
\end{lem}
\begin{proof}
The first point is an immediate computation. It follows that $m$ induces a map between the sets of twisted conjugacy classes, and $p_0$ induces a map between the twisted centralizers. The fact that $i_0$ respects twisted conjugacy follows from $i_0(g^{-1}\delta\theta(g))=g_I^{-1} i_0(\delta) \Theta(g_I)$ for $g_I=(g,\theta(g),\dots,\theta(g))$. The fact that $i_0$ is inverse to $m$ as maps between twisted conjugacy classes follows from the trivial relation $m(i_0(\delta))=\delta$ and the relation $i_0(m(\delta_I))=g_I^{-1}\delta_I\Theta(g_I)$ for $\delta_I=(\delta_0,\dots,\delta_{n-1})$ and $g_I=(g_0,\dots,g_{n-1})$ with $g_0=1$ and $g_i=\delta_i \dots \delta_{n-1}$ for $i>0$. The fact that $p_0$ has the given inverse as maps between twisted centralizers is immediate.

For the equality of twisted orbital integrals we take $\delta_I=(\delta_0,\dots,\delta_{n-1})$ and write out the left-hand side as
\[ \int_{\tx{Cent}_\Theta(\delta_I,I) \lmod I} f_0(g_0^{-1}\delta_0g_1) f_1(g_1^{-1}\delta_1g_2)\dots f_{n-1}(g_{n-1}^{-1}\delta_{n-1}\theta(g_0))dg_0\dots dg_{n-1}, \]
where the integration variable is $g_I=(g_0,\dots,g_{n-1})$, 
and use the substitution $h_0=g_0$, $h_i=g_0^{-1}\delta_0\dots \delta_{i-1}g_i$ for $i>0$.

The fact that $\tilde\pi_I$ is an isomorphism $\pi_I \circ \Theta^{-1} \to \pi_I$ is immediate. To verify the equality of traces write $\phi_i=\pi(f_i) \in \tx{End}_\C(V_\pi)$. Then $\pi_I(f_I)$ is the operator $\phi_0 \otimes \dots \otimes \phi_{n-1}$ of $V_{\pi_I}=V_\pi \otimes \dots \otimes V_\pi$, while $\pi(f_0* \dots *f_{n-1})$ is the operator $\phi_0\circ\dots\circ\phi_{n-1}$ of $V_\pi$. Therefore the claimed equality is
\[ \tx{tr}(\phi_0\otimes\dots\otimes\phi_{n-1}\circ\tilde\pi_I|V_{\pi_I}) = \tx{tr}(\phi_0\circ\dots\circ\phi_{n-1}\circ\tilde\pi|V_\pi). \]
Both sides are continuous and $n$-linear in the $\phi_i$, which allows us to reduce the proof first to the case that $\phi_i$ has finite rank, and then to the case that it has rank $1$, i.e. $\phi_i=\lambda_i \otimes w_i$ for $\lambda_i \in V_\pi^*$ and $w_i \in V_\pi$. The operator on the left has rank $1$ and is given by $((\lambda_0\otimes\dots\otimes\lambda_{n-1})\circ\tilde\pi_I) \otimes (w_0\otimes\dots\otimes w_{n-1})$. Its trace is therefore $\lambda_0(w_1)\lambda_1(w_2)\dots\lambda_{n-1}(\tilde\pi(w_0))$. The operator on the right also has rank $1$ and is given by $(\lambda_0(w_1)\lambda_1(w_2)\dots\lambda_{n-2}(w_{n-1})) \cdot w_0 \otimes \lambda_{n-1}\circ\tilde\pi$. Its trace is therefore given by $\lambda_0(w_1)\dots \lambda_{n-2}(w_{n-1})\lambda_{n-1}(\tilde\pi(w_0))$. Thus the two traces are equal.
\end{proof}

We now return the group $H \rtimes B$. We fix $\tilde h \in [H \rtimes b]_{z_H}(F)$ and let $\tilde g_d = \tx{ev}_{l(d)}(\tilde h^{n_d}) \in [G \rtimes a_d]_{z_G}(F)$. Consider given functions $f_{d,i} \in \mc{C}^\infty_c(G_{z_G}(F))$ for $d \in A \lmod B/\<b\>$ and $i=0,\dots,n_d-1$. The tensor product $\otimes_{d,i} f_{d,i}$ becomes, via the isomorphism of Lemma \ref{lem:indprod}, a function $f_H \in \mc{C}^\infty_c(H_{z_H}(F))$. Write $\tilde f_H = R_{\tilde h}^{-1}f_H \in \mc{C}^\infty_c([H \rtimes b]_{z_H}(F))$ for the function $\tilde f_H(h \cdot \tilde h)=f_H(h)$. Analogously we obtain for each $d$ the function $R_{\tilde g_d}^{-1}(f_{d,0}* \dots *f_{d,n_d-1}) \in [G \rtimes a_d]_{z_G}(F)$.

Fix a collection $(\pi_d)_{d \in A \lmod B/\<b\>}$ of representations of $G_{z_G}(F)$ and isomorphisms $\tilde\pi_d : \pi_d \circ\tx{Ad}(\tilde g_d)^{-1} \to \pi_d$. Via the isomorphism of Lemma \ref{lem:indprod} we can transport $\boxtimes_d \pi_d^{\boxtimes n_d}$ to a representation $\pi_H$ of $H_{z_H}(F)$ and $\tilde\pi_H(\otimes_d (v_{d,0} \otimes \dots \otimes v_{d,n_d-1}))=\otimes_d (v_{d,1}\otimes \dots \otimes v_{d,n_d-1}\otimes\tilde\pi_d(v_{d,0}))$ to an isomorphism $\pi_H \circ\tx{Ad}(\tilde h)^{-1} \to \pi_H$. 

\begin{cor} \label{cor:indorb}
\begin{enumerate}
\item The map $[H \rtimes b]_{z_H}(F) \to \prod_d [G \rtimes a_d]_{z_G}(F)$ sending $\tilde h'$ to $(\tilde g'_d)_d$ defined as $\tilde g'_d = \tx{ev}_{l(d)}(\tilde h'^{n_d})$, induces a bijection between the set of $H_{z_H}(F)$-conjugacy classes in $[H \rtimes b]_{z_H}(F)$ and the set of $G_{z_G}(F)^d$-conjugacy classes in $\prod_d [G \rtimes a_d]_{z_G}(F)$.
\item If $\tilde h' \in [H \rtimes b]_{z_H}(F)$ maps to $(\tilde g'_d)_d \in \prod_d [G \rtimes a_d]_{z_G}(F)$ then
\[ O^H_{\tilde h'}(R_{\tilde h}^{-1}f_H) = \prod_d O^{G^d}_{\tilde g'_d}(R_{\tilde g_d}^{-1}(f_{d,0}*\dots*f_{d,n_d-1})). \]
\item We have
\[ \tx{tr}(\pi_H(f_H)\circ\tilde\pi_H) = \prod_d \tx{tr}(\pi_d(f_{d,0}*\dots*f_{d,n_d-1})\circ\tilde\pi_d). \]
\end{enumerate}
\end{cor}
\begin{proof}
We have the bijection $H_{z_H}(F) \to [H \rtimes b]_{z_H}(F)$ sending $h$ to $h \cdot \tilde h$. It translates the conjugation action of $H_{z_H}(F)$ on $[H \rtimes b]_{z_H}(F)$ to the twisted conjugation action of $H_{z_H}(F)$ on itself, with respect to the automorphism $\tx{Ad}(\tilde h)$. The isomorphism of Lemma \ref{lem:indprod} identifies the group $H_{z_H}(F)$ with the group $\prod_d \prod_i G_{z_G}(F)$ and the automorphism $\tx{Ad}(\tilde h)$ with the automorphism sending $(g_{d,0},\dots,g_{d,n_d-1})$ to $(g_{d,1},\dots,g_{d,n_d-1},\tx{Ad}(\tilde g_d)g_{d,0})$. According to Lemma \ref{lem:basictwist}, the map sending $\tilde h'=h' \cdot \tilde h \in [H \rtimes b]_{z_H}(F)$ to 
\[ (\tx{ev}_{l(d)}(h' \cdot (\tilde hh'\tilde h^{-1})\dots (\tilde h^{n_d-1}h'\tilde h^{1-n_d})))_d = (\tx{ev}_{l(d)}(\tilde h'^{n_d}) \cdot \tx{ev}_{l(d)}(\tilde h^{n_d})^{-1} \] 
is a bijection from the set of $H_{z_H}(F)$-conjugacy classes in $[H \rtimes b]_{z_H}(F)$ to the set of $G_{z_G}(F)^d$-twisted conjugacy classes in $G_{z_G}(F)^d$ with respect to the automorphism $(\tx{Ad}(\tilde g_d))_d$. Composing this bijection with the bijection $G_{z_G}(F) \to [G \rtimes a_d]_{z_G}(F)$ sending $g'$ to $g' \tilde g_d$ we obtain the first claim. With this translation set up, the other two claims follow readily from Lemma \ref{lem:basictwist}.
\end{proof}

We now turn to the dual side. We can take $\hat H=\tx{Ind}_A^B \hat G$. More precisely, if $T$ is a torus with $A$-action, we have the identification $\tx{Ind}_A^BX_*(T)=X_*(\tx{Ind}_A^BT)$ sending an element $\lambda^B : B \times \mb{G}_m \to T$ of $\tx{Ind}_A^BX_*(T)$ to the element $x \mapsto \lambda^B(-,x)$ of $X_*(\tx{Ind}_A^BT)$. The pairing 
\begin{equation} \label{eq:ind_duality}
\<\chi^B,\lambda^B\> = \prod_{c \in A \lmod B} \<\chi^B(c),\lambda^b(c)\>
\end{equation}
between $\tx{Ind}_A^BX^*(T)$ and $\tx{Ind}_A^BX_*(T)$ is perfect and equivariant for $\Gamma$ and $B$ and identifies $\tx{Ind}_A^BX^*(T)$ with $X^*(\tx{Ind}_A^BT)$ as $\Gamma$-modules with $B$-action. If $(T,C)$ and $(\hat T,\hat C)$ are $\Gamma$-invariant Borel pairs for $G$ and $\hat G$ respectively, then $(\tx{Ind}_A^BT,\tx{Ind}_A^BC)$ and $(\tx{Ind}_A^B\hat T,\tx{Ind}_A^B\hat C)$ are such pairs for $\tx{Ind}_A^BG$ and $\tx{Ind}_A^B\hat G$, respectively. The duality between $X_*(T)$ and $X_*(\hat T)$ that realizes the duality between $G$ and $\hat G$ induces, via the above pairing, a duality between $\tx{Ind}_A^BX_*(T)=X_*(\tx{Ind}_A^BT)$ and $\tx{Ind}_A^BX_*(\hat T)=X_*(\tx{Ind}_A^BX_*(\hat T))$, and therefore realizes the duality between $\tx{Ind}_A^BG$ and $\tx{Ind}_A^B\hat G$.

Let $a \in A$. A Langlands parameter $\phi_G : L_F \to {^LG}$, which we represent as $\phi_G(x)=\phi_{G,0}(x) \rtimes x$ with $\phi_{G,0} : L_G \to \hat G$, has $a$-invariant $\hat G$-conjugacy class if and only if there exists an element $\check g_a \in \hat G$ satisfying
\begin{equation}
a(\phi_{G,0}(x))=\check g_a^{-1}\phi_{G,0}(x)\sigma_x(\check g_a) \label{eq:ind1ad}.
\end{equation}
This is equivalent to $\check g_a \rtimes a \in \tilde S_{\phi_G}$. Assuming this, a representation $(\rho_G,V_G)$ of $S_{\phi_G}$ has an $a$-invariant isomorphism class if and only if there is a vector space isomorphism $\tilde\rho_G(\check g_a \rtimes a) : V_{\rho_G} \to V_{\rho_G}$ satisfying
\begin{equation}
\rho_G(\check g_a \cdot a(\check g) \cdot \check g_a^{-1})\circ\tilde\rho_G(\check g_a \rtimes a)=\tilde\rho_G(\check g_a \rtimes a)\circ \rho_G(\check g). \label{eq:ind1bd}
\end{equation}

Let $b \in B$. A Langlands parameter $\phi_H : L_F \to {^LH}$, which we again represent as $\phi_H(x)=\phi_{H,0}(x) \rtimes x$ with $\phi_{H,0} : L_F \to \hat H$, has a $b$-invariant $\hat H$-conjugacy class if and only if there exists an element $\check h_b \in \hat H$ satisfying
\begin{equation} \label{eq:ind1cd}
\phi_{H,0}(x,b'b)=\check h^{-1}_b(b')\phi_{H,0}(x,b')\sigma_x(\check h_b(b')),\qquad\forall b' \in B. \end{equation}
Again this means $\check h_b \rtimes b \in \tilde S_{\phi_H}$. Assuming this, a representation $(\rho_H,V_{\rho_H})$ of $S_{\phi_H}$, represented again by the collection $\{V_c|c \in A \lmod B\}$ of vector spaces and the collection $\rho_H^{\dot c}$ of representations $\rho_H^{\dot c} : S_{\phi_H(-,\dot c)} \to \tx{Aut}(V_c)$ for all $\dot c \in c$, subject to $\rho_H^{a\dot c}(as)=\rho_H^{\dot c}(s)$, has a $b$-invariant isomorphism class if and only if there exists a vector space isomorphism $\tilde\rho_H(\check h_b \rtimes b) : V_{\rho_H} \to V_{\rho_H}$ satisfying 
\[ \tilde\rho_H(\check h_b \rtimes b)(\otimes_c v_c)=\otimes_c \tilde\rho_H(\check h_b \rtimes b)_c(v_{cb}), \]
where 
\[ \tilde\rho_H(\check h_b \rtimes b)_c : V_{cb} \to V_c \]
is an isomorphism of vector spaces satisfying
\begin{equation} \label{eq:ind1dd}
\rho_H^{\dot c}(\check g)\circ\tilde\rho_H(\check h_b \rtimes b)_c = \tilde\rho_H(\check h_b \rtimes b)_c \circ \rho_H^{\dot cb}(\check h_b(\dot c)^{-1}\check g \check h_b(\dot c))\end{equation}
for one, hence any, lift $\dot c \in B$ of $c \in A \lmod B$.

\begin{lem} \label{lem:ind1d}
\begin{enumerate}
	\item The assignment $\phi_{G,0}(x)=\phi_{H,0}(x,1)$ establishes a bijection between the $A$-invariant $\hat G$-conjugacy classes of parameters $\phi_G(x)=\phi_{G,0}(x) \rtimes x$ and the $B$-invariant $\hat H$-conjugacy classes of parameter $\phi_H(x)=\phi_{H,0}(x) \rtimes x$.
	\item Assume $\phi_{G,0}(x)=\phi_{H,0}(x,1)$. The assignment $\rho_G(\check g)=\rho_H^1(\check g)$ establishes a bijection between the set of $B$-fixed isomorphism classes of irreducible representations of $S_{\phi_H}$ and the set of $A$-fixed isomorphism classes of irreducible representations of $S_{\phi_G}$.
	\item Fix a section $s : A \lmod B \to B$. The assignments $\phi_{H,0}(x,as(c))=a\phi_{G,0}(x)$ and $\rho_H(\check h)=\otimes_c \rho_G(\check h(s(c))) \in \tx{End}_\C(V_{\rho_G}^{\otimes_c})$ are inverses of the above bijections.
\end{enumerate}
\end{lem}
\begin{proof}
The proof is the same as for Lemma \ref{lem:ind1} so we will not repeat it. 
\end{proof}

In order to deal with rigid inner forms, and not just with pure inner forms, we need to work with the preimage of $S_{\phi_G}$ in finite covers of $\hat G$. It is more uniform to work with the projective limit $\hat{\bar G}$ of all finite covers of $\hat G$. It receives an action of $A$. Then $\hat{\bar H} := \tx{Ind}_A^B(\hat{\bar G})$ is the projective limit of all finite covers of $\hat H$. A class $[z_G] \in H^1(u \to W,Z(G)^A \to A)^A$ induces an $A$-invariant character of $Z(\hat{\bar G})^+$, hence a character of the coinvariants $(Z(\hat{\bar G})^+)_A$. One checks that the inclusion $Z(\hat{\bar G})^+ \to Z(\hat{\bar H})^+$ induces an isomorphism $(Z(\hat{\bar G})^+)_A \to (Z(\hat{\bar H})^+)_B$. Let $S_{\phi_G}^+$ be the preimage of $S_{\phi_G}$ in $\hat{\bar G}$ and let $S_{\phi_H}^+$ be the preimage of $S_{\phi_H}$ in $\hat{\bar H}$.

\begin{lem} \label{lem:ind1.5d}
	Let $[z_G] \in H^1(u \to W,Z(G)^A \to A)^A$ and  $[z_H] \in H^1(u \to W,Z(H)^B \to B)^B$ correspond under the bijection of Lemma \ref{lem:ind1}(1). Then the characters they induce by the Tate--Nakayama isomorphism correspond under the isomorphism
	\[ (Z(\hat{\bar G})^+)_A \to (Z(\hat{\bar H})^+)_B. \] 
	Moreover, the assignments of Lemma \ref{lem:ind1d}(2,3) extend to mutually inverse bijections 
	\[ \tx{Irr}(S_{\phi_G}^+,[z_G])^G \to \tx{Irr}(S_{\phi_H}^+,[z_H])^H.  \]
\end{lem}
\begin{proof}
	Left to the reader.
\end{proof}

\begin{lem} \label{lem:ind2d}
Under the bijection $\rho_G \leftrightarrow \rho_H$ of Lemma \ref{lem:ind1.5d} the element of $H^2(B,\C^\times)$ corresponding to $\rho_H$ is the corestriction of the element of $H^2(A,\C^\times)$ corresponding to $\rho_G$.

More precisely, let $\phi_G$ and $\rho_G \in \tx{Irr}(S_{\phi_G}^+,[z_G])$ have $A$-fixed classes. For each $a \in A$ fix $\check g_a \in \hat{\bar G}$ and $\tilde\rho_G(\check g_a \rtimes a)$ satisfying Equations \eqref{eq:ind1ad} and \eqref{eq:ind1bd}, so that we have the element
\[ \alpha(a_1,a_2) = \tilde\rho_G( \check g_{a_1} \rtimes a_1)\circ\tilde\rho_G( \check g_{a_2} \rtimes a_2) \circ \tilde\rho_G( \check g_{a_1} \rtimes a_1 \cdot \check g_{a_2} \rtimes a_2)^{-1} \]
of $Z^2(A,\C^\times)$ representing the class associated to $\pi_G$. Define $\phi_{H,0}(x,as(c))=a\phi_{G,0}(x)$ and the representation $\rho_H$ of $S_{\phi_H}^+$ on $\otimes_c V_{\rho_G}$ by $\rho_H(\check h)=\otimes_c \rho_G(\check h(s(c)))$. For each $b \in B$ define the element $\check h_b \in \hat{\bar H}$ by $\check h_b(as(c))=a(\check g_{r(s(c)b)})$ and the isomorphism $\tilde\rho_H(\check h_b \rtimes b) : \otimes_c V_{\rho_G} \to \otimes_c V_{\rho_G}$ by $\tilde\rho_H(\check h_b \rtimes b)(\otimes_c v_c)=\otimes_c \tilde\rho_G(\check g_{r(s(c)b)} \rtimes r(s(c)b))(v_{cb})$. Then $\check h_b$ and $\tilde\rho_H(\check h_b \rtimes b)$ satisfy Equations \eqref{eq:ind1cd} and \eqref{eq:ind1dd} and the associated element
\[ \beta(b_1,b_2) = \tilde\rho_H( \check h_{b_1} \rtimes b_1)\circ\tilde\rho_H( \check h_{b_2} \rtimes b_2) \circ \tilde\rho_H( \check h_{b_1} \rtimes b_1 \cdot \check h_{b_2} \rtimes b_2)^{-1} \]
of $Z^2(B,\C^\times)$ is obtained from $\alpha$ by applying the cochain formula for the corestriction with respect to the section $s$.
\end{lem}
\begin{proof}
The proof is the same as for Lemma \ref{lem:ind2}, so we will not repeat it.
\end{proof}

Consider again $b \in B$ and a section $l : A \lmod B/\<b\> \to B$. Let $n_d$ and $s : A \lmod B \to B$ be as in \eqref{eq:sec1}. Recall $a_d =l(d)b^{n_d}l(d)^{-1} \in A$. 

\begin{lem} \label{lem:indprodd0}
Let $\phi_{G,0}^d : L_F \to \hat G$. Define $\phi_{H,0} : L_F \to \hat H$ by $\phi_{H,0}(x,al(d)b^i)=a\phi_G^d(x)$. The map
\[ \hat H \to \prod_d \prod_{i=0}^{n_d-1} \hat G,\qquad \check h \mapsto \prod_d \prod_{i=0}^{n_d-1} \check h(l(d)b^i) \]
is a $\Gamma$-equivariant isomorphism of algebraic groups. It translates conjugation by $\phi_H(x)$ to conjugation by $(\phi_G^d(x))_d$. It translates the action by $b$ to the action by $(\Theta_d)_d$, where $\Theta_d(\check g_{d,0},\dots,\check g_{d,n_d-1})=(\check g_{d,1},\dots,\check g_{n_d-1},a_d(\check g_0))$.
\end{lem}
\begin{proof}
This is an immediate computation.
\end{proof}

\begin{lem}  \label{lem:indprodd}
Let $\check h_b \rtimes b \in \tilde S_{\phi_H}$. The map
\[ \hat H \to \prod_d \prod_{i=0}^{n_d-1}\hat G,\qquad \check h \mapsto \prod_d \prod_i \tx{ev}_{l(d)}(\tx{Ad}(\check h_b \rtimes b)^ih) \]
is an isomorphism of algebraic groups. It translates the action of $\tx{Ad}(\phi_H(w))$ to the diagonal action of $(\tx{Ad}(\phi_G^d(w)))_d$. It translates the action of conjugation by $\check h_b \rtimes b$ to the action sending $(\check g_{d,0},\dots,\check g_{d,n_d-1})_d$ to $(g_{d,1},\dots,g_{d,n_d-1},\tx{Ad}(\check g_d \rtimes a_d)\check g_{d,0})_d$, where $\check g_d \rtimes a_d=\tx{ev}_{l(d)}((\check h_b \rtimes b)^{n_d})$.
\end{lem}
\begin{proof}
Direct computation.
\end{proof}

\begin{lem} \label{lem:basictwistd}
Let $J$ be a quasi-split connected reductive group with a pinned automorphism $\theta$. Consider $I=J \times \dots \times J$ with the pinned automorphism $\Theta$ defined by $\Theta(g_0,\dots,g_{n-1})=(g_1,\dots,g_{n_d-1},\theta(g_0))$. Let $\hat\theta$ and $\hat\Theta$ denote the duals of $\theta$ and $\Theta$. We have $\hat I=\hat J \times \dots \times \hat J$ and $\hat\Theta(\check g_0,\dots,\check g_{n-1})=(\hat\theta(\check g_{n-1}),\check g_0,\dots,\check g_{n-2})$.
\begin{enumerate}
	\item If $(J^\mf{e},\mc{J}^\mf{e},\tilde s_J^\mf{e},\xi_J^\mf{e})$ is a refined endoscopic datum for $J \rtimes \theta$ and we write $\tilde s_J^\mf{e}=\check s_J^\mf{e} \rtimes \hat\theta$ and $\xi_J^\mf{e}(\jmath)=\xi_{J,0}^\mf{e}(\jmath) \rtimes w_\jmath$, and define $I^\mf{e}=J^\mf{e}$, $\mc{I}^\mf{e}=\mc{J}^\mf{e}$, $\tilde s_I^\mf{e}=\check s_I^\mf{e} \rtimes\Theta$, $\check s_I^\mf{e}=(\check s_J^\mf{e},1,\dots,1)$, $\xi_I^\mf{e}(\jmath)=(\xi_{J,0}^\mf{e}(\jmath),\dots,\xi_{J,0}^\mf{e}(\jmath)) \rtimes w_\jmath$. Then $(I^\mf{e},\mc{I}^\mf{e},\tilde s^\mf{e}_I,\xi^\mf{e}_I)$ is a refined endoscopic datum for $I \rtimes \Theta$.

	\item If $(I^\mf{e},\mc{I}^\mf{e},\tilde s^\mf{e}_I,\xi^\mf{e}_I)$ is a refined endoscopic datum for $I \rtimes \Theta$ and we write $\tilde s_I^\mf{e}=\check s_I^\mf{e} \rtimes \hat\Theta$, $\check s_I^\mf{e}=(\check s_0,\dots,\check s_{n-1})$, $\xi^\mf{e}_I(\imath)=(\xi^\mf{e}_0(\imath),\dots,\xi^\mf{e}_{n-1}(\imath))\rtimes w_\imath$, and define $J^\mf{e}=I^\mf{e}$, $\mc{J}^\mf{e}=\mc{I}^\mf{e}$, $\tilde s_J^\mf{e}=\check s_J^\mf{e} \rtimes\hat\theta$, $\check s_J^\mf{e}=\check s_{n-1} \dots \check s_{0}$, $\xi^\mf{e}_J(\imath)=\xi_{n-1}^\mf{e}(\imath)\rtimes w_\imath$, then $(J^\mf{e},\mc{J}^\mf{e},\tilde s_J^\mf{e},\xi_J^\mf{e})$ is a refined endoscopic datum for $J \rtimes \theta$.

	\item The above constructions given mutually inverse bijections between the sets of isomorphism classes of refined endoscopic data for $J \rtimes \theta$ and $I\rtimes\Theta$.
	
	\item Let $\phi_{J,0} : L_F \to \hat J$ and define $\phi_{I,0}=(\phi_{J,0},\dots,\phi_{J,0})$. In the above constructions we have $\tilde s_I \in \tilde S_{\phi_I}^+$ if and only if $\tilde s_J \in \tilde S_{\phi_J}^+$ and the isomorphism classes of the resulting refined endoscopic data correspond under the above bijections.

	\item Let $z_J \in Z^1(u \to W,Z(J)^\theta \to J)$ and define $z_I=(z_J,\dots,z_J) \in Z^1(u \to W,Z(I)^\Theta \to I)$. If $\mf{e}$ is a refined endoscopic datum, both for $I \rtimes \Theta$ and for $J \rtimes \theta$ via the above bijections, $\mf{z}$ is a $z$-pair for $\mf{e}$, $\mf{w}$ is a $\theta$-special Whittaker datum for $J$ and hence a $\Theta$-special Whittaker datum for $I$, $\delta_I = (\delta_0,\dots,\delta_{n-1}) \in I_{z_I}(F)$, and $\delta_J = \delta_0\dots \delta_{n-1}$ we have
	\[ \Delta_{KS}[\mf{w},\mf{e},\mf{z}](\gamma^{\mf{z}},\delta_I \rtimes \Theta)=\Delta_\tx{KS}[\mf{w},\mf{e},\mf{z}](\gamma^\mf{z},\delta_J \rtimes \theta). \]
\end{enumerate}
\end{lem}

\begin{proof}
For the first two points we only need to check parts \ref{item:e0} and \ref{item:e1} of the definition of endoscopic datum in \S\ref{sub:endo}. These verifications are immediate and left to the reader. For the third point, the assignment $(J^\mf{e},\mc{J}^\mf{e},\tilde s_J^\mf{e},\xi_J^\mf{e}) \mapsto (I^\mf{e},\mc{I}^\mf{e},\tilde s^\mf{e}_I,\xi^\mf{e}_I) \mapsto (J^\mf{e},\mc{J}^\mf{e},\tilde s_J^\mf{e},\xi_J^\mf{e})$ is the identity on data, even before taking isomorphism classes. On the other hand, the element $(\check g_0,\dots,\check g_{n-1}) \in \hat I$ with $\check g_i=\check s_{n-1}\dots \check s_{i+1}$ gives an isomorphism between the source and target of the assignment $(I^\mf{e},\mc{I}^\mf{e},\tilde s^\mf{e}_I,\xi^\mf{e}_I) \mapsto (J^\mf{e},\mc{J}^\mf{e},\tilde s_J^\mf{e},\xi_J^\mf{e}) \mapsto (I^\mf{e},\mc{I}^\mf{e},\tilde s^\mf{e}_I,\xi^\mf{e}_I)$. For the fourth point it is enough to start with $(\tilde s_J,\phi_J)$, let $\tilde s_I$ be as in the first point, produce from $(\tilde s_J,\phi_J)$ respectively $(\tilde s_I,\phi_I)$ endoscopic data $(J^\mf{e},\mc{J}^\mf{e},\tilde s_J^\mf{e},\xi_J^\mf{e})$ respectively $(I^\mf{e},\mc{I}^\mf{e},\tilde s^\mf{e}_I,\xi^\mf{e}_I)$ via the spectral construction of \S\ref{sub:endocnst}, and then verify that these two data are related by the construction of the first point. This is immediate and left to the reader.

The remainder of the proof will be concerned with the equality of transfer factors. We consider each individual term in the product \eqref{eq:pure_tf1}
\[ \Delta_{KS} = 
\epsilon_L(V,\psi) (\Delta_I^\tx{new})^{-1}\Delta_{II}(\Delta_{III}^\tx{new})^{-1}\Delta_{IV}. \]
These terms were recalled in \S\ref{sub:pure_tf}, except $\Delta_{III}^\tx{new}$, for which we follow the construction given in \S\ref{sub:rigid_tf}. These terms depend on various auxiliary data recalled in \S\ref{sub:pure_tf} and their comparison requires that we compare this auxiliary data for the group $I$ and the group $J$.

We fix a $\theta$-invariant $F$-pinning $(T_J,B_J,\{X_\alpha\})$ of $J$ and a non-trivial character $\psi : F \to \C^\times$. Taking the product of this pinning gives a $\Theta$-invariant $F$-pinning of $I$ and all $\Theta$-invariant pinnings of $I$ arise this way. In this way $\theta$-special Whittaker data for $J$ correspond to $\Theta$-special Whittaker data for $I$. We fix a norm $(S_I,\gamma)$ for $\delta_I \rtimes \Theta$. Here $S_I \subset I_{z_I}$ is a maximal torus defined over $F$, invariant under $\Theta$, and contained in a Borel subgroup $C_I \subset I$ defined over $\bar F$ and invariant under $\Theta$. Moreover $\gamma \in [S_I]_\Theta(F)$ and there exists $g_I \in I$ such that $g_I^{-1}(\delta_I \rtimes \Theta) g_I = \delta_I^* \rtimes \Theta$ with $\delta_I^* \in S_I$ whose image in $[S_I]_\Theta$ is $\gamma$. It is immediate that $S_I=S_J^n$ and $C_I=C_J^n$ for a $\theta$-invariant Borel pair $S_J \subset C_J \subset J$, with $S_J$ defined over $F$. Moreover, the product map $S_I \to S_J$ induces an isomorphism $[S_I]_\Theta \to [S_J]_\theta$. If we write $g_I=(g_0,\dots,g_{n-1})$ then $\delta_I^*=(g_0^{-1}\delta_0g_1,\dots,g_{n-2}^{-1}\delta_{n-2}g_{n-1},g_{n-1}^{-1}\delta_{n-1}\theta(g_0))$ and its image in $S_J$ is given by $\delta_J^*=g_0^{-1}\delta_0\dots \delta_{n-1}\theta(g_0)$. Therefore $(S_J,\gamma)$ is a norm for $\delta_J\rtimes\theta$.

The set of $\Theta$-orbits in $R(S_I,I)$ is in natural bijection with the set of $\theta$-orbits in $R(S_J,J)$. In this way $R_\tx{res}(S_I,I)=R_\tx{res}(S_J,J)$. We fix $a$-data and $\chi$-data for this set.

We can now compare the individual terms of $\Delta_{KS}$ for $I$ and $J$. The term $\epsilon_L(V,\psi)$ for $I$ is the root number of the virtual Galois representation $X^*(T_I)^\Theta_\C - X^*(T^\mf{e})_\C$. But $X^*(T_I)^\Theta=X^*(T_J)^\theta$ so this equals the term $\epsilon_L(V,\psi)$ for $J$. 

The term $\Delta_{II}$ is a fraction. The denominator for $I$ equals the denominator for $J$ by virtue of the identification $I^\mf{e}=J^\mf{e}$. The numerator for $I$ is a product over the $\Gamma$-orbits in $R_\tx{res}(S_I,I)$ and the factor corresponding to $\alpha_\tx{res}$ involves the quantity $N_\Theta\alpha_I(\delta_I^*)$, where $\alpha_I \in R(S_I,I)$ represents $\alpha_\tx{res}$. Now $R(S_I,I)=R(S_J,J) \cup \dots \cup R(S_J,J)$ and if $\alpha_J \in R(S_J,J)$ represents $\alpha_\tx{res}$ then $N_\Theta\alpha_I(\delta_I^*)=N_\theta\alpha_J(\delta_J^*)$. Therefore the numerators of $\Delta_{II}$ for $I$ matches the numerator of $\Delta_{II}$ for $J$. 

The term $\Delta_{IV}$ is discussed in the same way as the term $\Delta_{II}$.

The term $\Delta_I^\tx{new}$ for $J$ is defined as the Tate-Nakayama pairing applied to an element $t_J \in H^1(\Gamma,S^{\theta,\circ})$ with an element $\check s_{J,\theta} \in \pi_0([\hat S]_\theta^\Gamma)$. The diagonal inclusions $J \to I$ and $S_J \to S_I$ become isomorphisms $J^{\theta,\circ} \to I^{\Theta,\circ}$ and $S_J^{\theta,\circ} \to S_I^{\Theta,\circ}$. Tracing through the construction we see that under the isomorphism $H^1(\Gamma,S_J^{\theta,\circ}) \to H^1(\Gamma,S_I^{\Theta,\circ})$ the elements $t_J$ and $t_I$ are identified. 

The element $\check s_{I,\theta} \in \pi_0([\hat S_I]^\Gamma_\Theta)$ is obtained by recognizing that $\check s_I$ lies in the image of a certain embedding $\hat S_I \to \hat I$, so that it can be transported to $\hat S_I$ under that embedding and then mapped to $[\hat S_I]_\Theta$. Dual to the isomorphism $S_J^{\theta,\circ} \to S_I^{\Theta,\circ}$ is the isomorphism $[\hat S_I]_\Theta \to [\hat S_J]_\theta$ induced by the product map $\hat S_I=\hat S_J \times \dots \times \hat S_J \to \hat S_J$. Since the image of $\check s_I \in \hat S_I$ under the product map produces the element $\check s_J \in \hat S_J$, we see that the term $\Delta_I^\tx{new}$ for $I$ equals the term $\Delta_I^\tx{new}$ for $J$.

We come to the term $\Delta_{III}^\tx{new}$. We shall give the proof in the special case of pure inner forms and no $z$-pair. The proof in the general case is the same, but with more cumbersome notation that obscures the main point. This term for $J$ is given by the Tate-Nakayama pairing of the element $\tx{inv}(\gamma,(z_J,\delta_J)) \in H^1(\Gamma,S_J \stackrel{1-\theta}{\lrw} S_J)$ with the element $A_{0,J} \in H^1(W_F,\hat S_J \stackrel{1-\hat\theta}{\lrw} \hat S_J)$.

We consider the two dual commutative diagrams
\[ \xymatrix{
	S_I\ar[r]^{p_0}\ar[d]_{1-\Theta}&S_J\ar[d]^{1-\theta}&&\hat S_I&\hat S_J\ar[l]_{i_0}\\
	S_I\ar[r]_m&S_J&&\hat S_I\ar[u]^{1-\hat\Theta}&\hat S_J\ar[u]_{1-\hat\theta}\ar[l]^\Delta
}
\]
Where $m$ is the multiplication map, $\Delta$ is the diagonal inclusion, $i_0$ is the inclusion into the first coordinate, and $p_0$ is the projection onto the first coordinate. These diagrams can be seen as morphisms of complexes of tori of length 2, the complexes being the vertical arrows and the morphisms being the horizontal arrows. It is immediate to check that these morphisms are quasi-isomorphisms. Therefore they induce isomorphisms $H^1(\Gamma,S_I \stackrel{1-\Theta}{\lrw} S_I) \to H^1(\Gamma,S_J \stackrel{1-\theta}{\lrw} S_J)$ and $H^1(W_F,\hat S_J \stackrel{1-\hat\theta}{\lrw} \hat S_J) \to H^1(W_F,\hat S_I \stackrel{1-\hat\Theta}{\lrw} \hat S_I)$. The element $\tx{inv}(\gamma,(z_I,\delta_I\rtimes\Theta))$ is the pair $((g_I^{-1}z_I(\sigma)\sigma(g_I))^{-1},\delta_I^*) \in Z^1(\Gamma,S_I \stackrel{1-\Theta}{\lrw} S_I)$. We had already noted that the image of $\delta_I^*$ under the multiplication map is $\delta_J^*$. At the same time, the image of $g_I^{-1}z_I(\sigma)\sigma(g_I)$ under $p_0$ is $g_0^{-1}z_J(\sigma)\sigma(g_0)$. Thus the image of $\tx{inv}(\gamma,(z_I,\delta_I\rtimes\Theta))$ under the first of these isomorphisms is indeed $\tx{inv}(\gamma,(z_J,\delta_J\rtimes\Theta))$.

To compare $A_{0,J}$ and $A_{0,I}$ we consider the commutative diagram
\[ \xymatrix{
&^LI&\\
^LJ^\mf{e}\ar[ur]^{\xi^\mf{e}_I}\ar[dr]_{\xi^\mf{e}_J}&&^LJ^1\ar[ul]_\Delta\ar[dl]^{\tx{nat}}\\
&^LJ\ar[uu]_\Delta&\\
^LS^\mf{e}\ar[uu]^{\xi_S^\mf{e}}&&^LS_\theta\ar[ll]^{^L\varphi_{\gamma^\mf{e},\gamma}}\ar[uu]_{\xi_S^1}
}
\]
The element $A_{0,J}$ is the class of $(a_{S_J}^{-1},\check s_J)$, where $a_{S_J} : W_F \to \hat S_J$ is the 1-cocycle measuring the difference between $\xi^\mf{e}_J\circ \xi_S^\mf{e}\circ {^L\varphi_{\gamma^\mf{e},\gamma}}$ and $\tx{nat}\circ\xi_S^1$. The element $A_{0,I}$ is the class of $(a_{S_I}^{-1},\check s_I)$, where $a_{S_I} : W_F \to \hat S_I$ is the 1-cocycle measuring the difference between $\xi^\mf{e}_I\circ \xi_S^\mf{e}\circ {^L\varphi_{\gamma^\mf{e},\gamma}}$ and $\Delta\circ\circ\xi_S^1$. The commutativity of the diagram shows that $a_{S_I}=\Delta(a_{S_J})$. On the other hand $\check s_I=i_0(\check s_J)$ and we conclude that the image of $A_{0,J}$ under the isomorphism $H^1(W_F,\hat S_J \stackrel{1-\hat\theta}{\lrw} \hat S_J) \to H^1(W_F,\hat S_I \stackrel{1-\hat\Theta}{\lrw} \hat S_I)$ equals $A_{0,I}$.
\end{proof}

\begin{cor} \label{cor:indprodd}
Let $\phi_H : L_F \to {^LH}$ and $\tilde s_H \in \tilde S_{\phi_H}$. Write $\tilde s_H = \check s_H \rtimes b$. Fix a section $l : A \lmod B/\<b\> \to B$. For each $d \in A \lmod B/\<b\>$ define $\tilde s_d=\tx{ev}_{l(d)}(\tilde s_H^{n_d})$ and $\phi_{d,0}=\tx{ev}_{l(d)} \circ \phi_{H,0}$. Then $\phi_d : L_F \to {^LG}$ and $\tilde s_d \in \tilde S_{\phi_d}$.

\begin{enumerate}
	\item Let $\mf{e}_H$ and $\mf{e}_d$ be the endoscopic data associated to $(\tilde s_H,\phi_H)$ and $(\tilde s_d,\phi_d)$. There is a natural identification
	\[ \mf{e}_H = \prod_d \mf{e}_d. \]
	\item Fix a $z$-pair $\mf{z}_H$ for $\mf{e}_H$ and express it as $\prod_d \mf{z}_d$. If $\gamma^{\mf{z}_H}$ corresponds to $(\gamma^{\mf{z}_d})_d$ under this identification, and the stable class of $\tilde h' \in [H \rtimes b^{-1}]_{z_H}(F)$ corresponds to the stable class of $(\tilde g'_d)_d$ under the bijection of Corollary \ref{cor:indorb}, then
	\[ \Delta_\tx{KS}(\gamma^{\mf{z}_H},\tilde h') = \prod_d \Delta_\tx{KS}(\gamma^{\mf{z}_d},\tilde g'_d). \]
	\item The transfer of the function $R_{\tilde h}^{-1}f_H \in \mc{C}^\infty_c(\tilde H_{z_H}(F))$ to $\mf{z}_H$ equals the tensor product over $d$ of the transfers of the functions $R_{\tilde g_d}^{-1}(f_{d,0}*\dots*f_{d,n_d-1}) \in \mc{C}^\infty_c(\tilde G_{z_G}(F))$ to $\mf{z}_d$.
\end{enumerate}
\end{cor}
\begin{proof}
We consider the isomorphism 
\[ H \to \prod_d \prod_{i=0}^{n_d-1} G,\qquad h \mapsto \prod_d \prod_{i=0}^{n_d-1} h(l(d)b^{-i}) \]
of Lemma \ref{lem:indprod0} and the isomorphism
\[ \hat H \to \prod_d \prod_{i=0}^{n_d-1} \hat G,\qquad \check h \mapsto \prod_d \prod_{i=0}^{n_d-1} \check h(l(d)b^{-i}) \]
of Lemma \ref{lem:indprodd0}, both applied to the element $b^{-1}$. They are dual to each other (cf. \eqref{eq:ind_duality}). The first of them translates $b^{-1}$ to $(\Theta_d)_d$ and hence $a_d^{-1}$ to $\theta_d$, while the second translates $b$ to $(\hat\Theta_d)_d$ and $a_d$ to $\hat\theta_d$.

Since the $\hat H$-conjugacy class of $\phi_H$ is $b$-invariant, we may conjugate $(\phi_H,\tilde s_H)$ by $\hat H$ to arrange that the image of $\phi_H^0$ under the second isomorphism is of the form $(\phi_d^0,\dots,\phi_d^0)_d$ for parameters $\phi_d : L_F \to {^LG}$. Let $(\check s_{d,i})$ be the image of $\check s_H$. Then Lemma \ref{lem:basictwistd} gives the identification $\mf{e}_H=\prod_d \mf{e}_d'$, where $\mf{e}_d'$ is the endoscopic datum associated to $(\tilde s_d,\phi_d)$ with $\tilde s_d'=\check s_d' \rtimes a_d$ and 
\[ \check s_d'=\check s_{d,n_d-1}\dots \check s_{d,0}=\check s_H(l(d)b^{1-n_d})\cdot\check s_H(l(d)b^{2-n_d})\dots \check s_H(l(d)). \]
The element $(\check s_{d,n_d-1}\dots \check s_{d,1})^{-1}$ conjugates $\tilde s_d'$ to $\tilde s_d$ and hence provides an isomorphism $\mf{e}_d' \to \mf{e}_d$.

Let $\tilde h' = h' \rtimes b^{-1}$ and let $(g_{d,i})$ be the image of $h'$. It is immediate that $g_{d,0}\dots g_{d,n_d-1}\rtimes a_d^{-1}=\tx{ev}_{l(d)}(\tilde h')=\tilde g_d'$, therefore Lemma \ref{lem:basictwistd} 
implies
\[ \Delta_{KS}(\gamma^\mf{z},\tilde h')=\prod_d \Delta_{KS}(\gamma^{\mf{z}_d},\tilde g'_d).\]

The identification of transfers of functions follows from the equation
\begin{eqnarray*}
SO_{\gamma^{\mf{z}_H}}((R_{\tilde h}^{-1}f_H)^{\mf{e}_H})&=&\sum_{\tilde h'} \Delta_\tx{KS}(\gamma^{\mf{z}_H},\tilde h') O^H_{\tilde h'}(R_{\tilde h}^{-1}f_H)\\
&=&\prod_d\sum_{\tilde g'_d} \Delta_\tx{KS}(\gamma^{\mf{z}_d},\tilde g'_d) O^{G}_{\tilde g'_d}(R_{\tilde g_d}^{-1}(f_{d,0}*\dots*f_{d,n_d-1}))\\
&=&\prod_d SO_{\gamma^{\mf{z}_d}}((R_{\tilde g_d}^{-1}(f_{d,0}*\dots*f_{d,n_d-1}))^{\mf{e}_d}).
\end{eqnarray*}
Here $\tilde h'$ runs over the set of $H_{z_H}(F))$-classes in $[H \rtimes b^{-1}]_{z_H}(F)$. We have used Corollary \ref{cor:indorb} for $b^{-1}$ to identify this set with the set $(\tilde g_d')_d$ of $G_{z_G}(F)^d$-conjugacy classes in $\prod_d [G \rtimes a_d^{-1}]_{z_G}(F)$, and to related the corresponding orbital integrals.
\end{proof}

We continue with $z_G$ and $\phi_G$ whose equivalence classes are $A$-fixed and consider $\pi_G \in \Pi_{\phi_G}$ corresponding to $\rho_G \in \tx{Irr}(\pi_0(S_{\phi_G}^+))$ whose equivalence classes are also $A$-fixed (provided these exist). Recall that $z_G$ and $\phi_G$ determine $B$-fixed equivalence classes of $z_H$ and $\phi_H$, and that $\pi_G$ and $\rho_G$ determine $B$-fixed isomorphism classes of representations $\pi_H$ of $H_{z_H}(F)$ and $\rho_H$ of $\pi_0(\tilde S_{\phi_H}^+)$. Assume given an extension of $\pi_G \boxtimes \rho_G^\vee$ from $G_{z_G}(F) \times \pi_0(S_{\phi_G}^+)$ to $\tilde G_{z_G}(F) \times_A \pi_0(\tilde S_{\phi_G}^+)$. We claim that it determines an extension of $\pi_H \boxtimes \rho_H^\vee$ from $H_{z_H}(F) \times \pi_0(S_{\phi_H}^+)$ to $\tilde H_{z_H}(F) \times_B \pi_0(\tilde S_{\phi_H}^+)$. To see this, fix $g_a \in G$ and $\tilde\pi_G(g_a \rtimes a) : V_{\pi_G} \to V_{\pi_G}$ satisfying Equations \eqref{eq:ind1a} and \eqref{eq:ind1b}, and fix analogously $\check g_a \in \hat G$ and $\tilde\rho_G(\check g_a \rtimes a) : V_{\rho_G} \to V_{\rho_G}$ satisfying Equations \eqref{eq:ind1ad} and \eqref{eq:ind1bd}. We demand that these choices are made in such a way that the restriction to $\tilde G_{z_G}(F) \times_A \pi_0(\tilde S_{\phi_G}^+)$ of the exterior tensor product $\tilde\pi_G \boxtimes \tilde\rho_G^\vee$ is the given extension of $\pi_G \boxtimes \rho_G^\vee$. We fix a section $s : A \lmod B \to B$ and according to Lemma \ref{lem:ind1} we can take $z_H(w,as(c))=a(z_G(w))$ and $\pi_H(h)=\otimes_c \pi_G(h(s(c)))$ acting on $V_{\pi_G}^{\otimes_c}$, and according to Lemma \ref{lem:ind2d} we can take $\phi_{H,0}(x,as(c))=a\phi_{G,0}(x)$ and $\rho_H(\check h)=\otimes_c \rho_G(\check h(s(c)))$. We then define for each $b \in B$ the element $h_b \in H$ and the isomorphism $\tilde\pi_H(h_b \rtimes b)$ as in Lemma \ref{lem:ind2} and the element $\check h_b \in \hat H$ and the isomorphism $\tilde\rho_H(\check h_b \rtimes b)$ as in Lemma \ref{lem:ind2d} and consider $\tilde\pi_H \boxtimes \tilde\rho_H^\vee$ restricted to $\tilde H_{z_H}(F) \times_B \pi_0(\tilde S_{\phi_H}^+)$. According to Lemmas \ref{lem:ind2} and \ref{lem:ind2d} this is a linear representation and extends $\pi_H \boxtimes \rho_H^\vee$. 

\begin{lem} \label{lem:ind4}
The restriction of $\tilde\pi_H \boxtimes \tilde\rho_H^\vee$ to $\tilde H_{z_H}(F) \times_B \pi_0(\tilde S_{\phi_H}^+)$ is independent of the choices of $g_a$, $\tilde\pi_G(g_a \rtimes a)$, $\check g_a$, and $\tilde\rho_G(\check g_a \rtimes a)$.
\end{lem}
\begin{proof}
Keeping $\{g_a\}$ and $\{\check g_a\}$ fixed, for any $a$ the isomorphism $\tilde\pi_G(g_a \rtimes a)$ can only be changed to $z_a\tilde\pi_G(g_a \rtimes a)$ for some $z_a \in \C^\times$. Since the restriction of $\tilde\pi_G \boxtimes \tilde\rho_G$ to $\tilde G_{z_G}(F) \times_A \pi_0(\tilde S_{\phi_G}^+)$ is fixed, this means that $\tilde\rho_G(\check g_a \rtimes a)$ must be changed to $z_a^{-1}\tilde\rho_G(\check g_a \rtimes a)$. Now $\tilde\pi_H(h_b \rtimes b)$ is multiplied by $z_a$ raised to the power of the cardinality of $\{c|a=r(s(c)b)\}$, while $\tilde\rho_H(\check h_b \rtimes b)$ is multiplied by $z_a^{-1}$ raised to the same power, so we see that the restriction of $\tilde\pi_H \boxtimes \rho_H^\vee$ to $\tilde H_{z_H}(F) \times_B \pi_0(\tilde S_{\phi_H}^+)$ remains unchanged.

Replace now $g_a$ by $g'_a = g^0_ag_a$ and $\check g_a$ by $\check g'_a=\check g^0_a\check g_a$ for $g^0_a \in G_{z_G}(F)$ and $\check g^0_a \in S_{\phi_G}^+$. By the previous argument the choices of $\tilde\pi_G(g'_a \rtimes a)$ and $\tilde\rho_G(\check g'_a \rtimes a)$ will not influence the construction. We choose them to be $\pi_G(g^0_a)\circ\tilde\pi_G(g_a \rtimes a)$ and $\rho_G(\check g^0_a) \circ \tilde\rho_G(\check g_a \rtimes a)$ respectively. We claim that $\tilde\pi_H$ and $\tilde\rho_H$ are both unchanged. Indeed, $h_b$ is now replaced by $h'_b$ defined by $h'_b(as(c))=a(g'_{r(s(c)b)})=a(g^0_{r(s(c)b)})\cdot a(g_{r(s(c)b)})$. Define $h^0 \in H$ by $h^0(as(c))=a(g^0_{r(s(c)b)})$. Then $h^0 \in H_{z_H}(F)$ and $h'_b =h^0h_b$. The new choices now stipulate 
\[ \tilde\pi_H(h'_b \rtimes b)_c=\tilde\pi_G(g'_{r(s(c)b)} \rtimes r(s(c)b)) = \pi_G(g^0_{r(s(c)b)}) \circ \tilde\pi_G(g_{r(s(c)b)} \rtimes r(s(c)b)), \] 
while according to the old choices we have 
\[ \tilde\pi_H(h'_b \rtimes b)=\pi_H(h^0)\circ\tilde\pi_H(h_b \rtimes b) \]
and we see that both of these values for $\tilde\pi_H(h'_b \rtimes b)_c$ agree. The argument for $\tilde\rho_H$ is analogous.
\end{proof}

Thus far we have focused on the setting in which the objects $z_H$, $\pi_H$, $\phi_H$, $\rho_H$ have equivalence classes fixed by $B$. In general this will not be the case, but one can reduce to that case by a simple application of Mackey theory, at the expense of introducing rather cumbersome notation. This is what we turn to next.

Let $B' \subset B$ be a subgroup. Fix a section $l : A \lmod B/B' \to B$ of the natural projection. For each $d \in A \lmod B/B'$ let $A_d=l(d)^{-1}Al(d)$, $A'_d=A_d \cap B'$, and write $G^d$ for the group $G$ with action of $A_d$ defined by $(l(d)^{-1}al(d)) \cdot_d g=ag$. Fix a section $s_d : A'_d \lmod B' \to B'$ of the natural projection. Then each element of $B$ has a unique expression of the form $b=al(d)s_d(c'_d)$ for $d \in A \lmod B/B'$ and $c'_d \in A'_d \lmod B'$. This gives a section $A \lmod B \to B$.

\begin{cor} \label{cor:ind6}

\begin{enumerate}
	\item Let $z_H \in Z^1(u \to W,Z(H)^{B'} \to H)$ have a $B'$-fixed class. For each $d \in A \lmod B/B'$ we obtain an element of $Z^1(u \to W,Z(G^d)^{A'_d} \to G^d)$ by $z_{G^d}(w)=z_H(w,l(d))$, and the map $z_H \mapsto (z_{G^d})_d$ is a bijection between $H^1(u \to W,Z(H)^{B'} \to H)^{B'}$ and $\prod_d H^1(u \to W,Z(G^d)^{A'_d} \to G^d)^{A'_d}$.

	\item Let $\pi_H$ be an irreducible representation of $H_{z_H}(F)$ whose class is $B'$-fixed. For each $d \in A \lmod B/B'$ we obtain an irreducible representation $\pi_{G^d}$ of $G^d_{z_{G^d}}(F)$ on $V_{l(d)}$ by $\pi_{G^d}(g)=\pi_H^{l(d)}(g)$. The map $\pi_H \mapsto (\pi_{G^d})_d$ is a bijection between $\tx{Irr}(H_{z_H}(F))^{B'}$ and $\prod_d \tx{Irr}(G^d_{z_{G^d}}(F))^{A'_d}$.

	\item Let $\phi_H(x)=\phi_{H,0}(x) \rtimes x$ be a Langlands parameter for $H$ whose class is $B'$-fixed. For each $d \in A \lmod B/B'$ we obtain a Langlands parameter $\phi_{G^d,0}$ for $G^d$ by $\phi_{G^d,0}(x)=\phi_{H,0}(x,l(d))$. The map $\phi_H \mapsto (\phi_{G^d})_d$ is a bijection between $\Phi(H)^{B'}$ and $\prod_d \Phi(G^d)^{A'_d}$.

	\item Let $\rho_H$ be an irreducible representation of $\pi_0(S_{\phi_H}^+)$. For each $d \in A \lmod B/B'$ we obtain an irreducible representation of $\pi_0(S_{\phi_{G^d}}^+)$ by $\rho_{G^d}(\check g)=\rho_H^{l(d)}(\check g)$. The map $\rho_H \mapsto (\rho_{G^d})_d$ is a bijection between the sets $\tx{Irr}(\pi_0(S_{\phi_H}^+))^{B'}$ and $\prod_d \tx{Irr}(\pi_0(S_{\phi_{G^d}}^+))^{A'_d}$.

	\item The inverse of the above bijections are given by $z_H(w,al(d)s_d(c_d'))=a(z_{G,d}(w))$, $\phi_H(x,al(d)s_d(c_d'))=a\phi_{G^d}(x)$, $\pi_H(h)=\otimes_d \otimes_{c_d'} \pi_{G^d}(h(l(d)s_d(c_d')))$, and $\rho_H(\check h)=\otimes_d \otimes_{c_d'} \rho_{G^d}(\check h(l(d)s_d(c_d')))$.

	\item For each $d \in A \lmod B /B'$ and $a'_d \in A'_d$ choose $g_{a'_d} \in G^d$ and $\tilde\pi_{G^d}(g_{a'_d} \rtimes a'_d) : V_{\pi_{G^d}} \to V_{\pi_{G^d}}$ satisfying Equations \eqref{eq:ind1a} and \eqref{eq:ind1b} for $G^d \rtimes A'_d$. For each $b' \in B'$ define $h_{b'} \in H$ by $h_{b'}(al(d)s_d(c'_d))=ag_{r_d(s_d(c'_d)b')}$ and 
	\[ \tilde\pi_H(h_{b'} \rtimes b') : \otimes_d \otimes_{c'_d} V_{d,c'_d} \to \otimes_d \otimes_{c'_d} V_{d,c'_d} \] by 
	\[\tilde\pi_H(h_{b'} \rtimes b')(\otimes_d \otimes_{c'_d} v_{d,c'_d})=\otimes_d \otimes_{c'_d} \tilde\pi_{G^d}(g_{r_d(s_d(c'_d)b')} \rtimes r_d(s_d(c'_d)b'))(v_{d,c'_d \cdot b'}). \]
	Then these satisfy Equations \eqref{eq:ind1c} and \eqref{eq:ind1d}. Moreover, the 2-cocycle $\beta \in Z^2(B',\C^\times)$ for $\tilde\pi_H$ is the product over $d \in A \lmod B/B'$ of the co-restrictions (computed with respect to the sections $s_d$) of the 2-cocycles $\alpha_d \in Z^2(A'_d,\C^\times)$ for $\tilde\pi_{G^d}$.

	\item For each $d \in A \lmod B /B'$ and $a'_d \in A'_d$ choose $\check g_{a'_d} \in \hat G^d$ and $\tilde\rho_{G^d}(\check g_{a'_d} \rtimes a'_d) : V_{\rho_{G^d}} \to V_{\rho_{G^d}}$ satisfying Equations \eqref{eq:ind1ad} and \eqref{eq:ind1bd} for $\hat G^d \rtimes A'_d$. For each $b' \in B'$ define $\check h_{b'} \in \hat H$ by $\check h_{b'}(al(d)s_d(c'_d))=a\check g_{r_d(s_d(c'_d)b')}$ and 
	\[ \tilde\rho_H(\check h_{b'} \rtimes b') : \otimes_d \otimes_{c'_d} V_{d,c'_d} \to \otimes_d \otimes_{c'_d} V_{d,c'_d} \] by 
	\[\tilde\rho_H(h_{b'} \rtimes b')(\otimes_d \otimes_{c'_d} v_{d,c'_d})=\otimes_d \otimes_{c'_d} \tilde\rho_{G^d}(g_{r_d(s_d(c'_d)b')} \rtimes r_d(s_d(c'_d)b'))(v_{d,c'_d \cdot b'}). \]
	Then these satisfy Equations \eqref{eq:ind1cd} and \eqref{eq:ind1dd}. Moreover, the 2-cocycle $\beta \in Z^2(B',\C^\times)$ for $\tilde\rho_H$ is the product over $d \in A \lmod B/B'$ of the co-restrictions (computed with respect to the sections $s_d$) of the 2-cocycles $\alpha_d \in Z^2(A'_d,\C^\times)$ for $\tilde\rho_{G^d}$.

	\item Let $b_1 \in B$ and set $B_1'=b_1B'b_1^{-1}$. Define $l_1 : A \lmod B/B_1' \to B$ by $l_1(d_1)=l(d)b_1^{-1}$ for $d_1 \in A \lmod B/B_1'$ and $d=d_1b_1 \in A \lmod B/B'$. Define $s_{d_1}: A_{d_1}' \lmod B_1' \to B_1'$ by $s_{d_1}(c_{d_1}')=b_1s_d(c'_d)b_1^{-1}$ for $c_d' \in A_d' \lmod B'$ and $c_{d_1}'=b_1c_d'b_1^{-1} \in A_{d_1}'\lmod B_1'$.

	If $(z_H,\pi_H)$ corresponds to $(z_{G^d},\pi_{G^d})_d$ based on the choices of $l$ and $(s_d)$, then $b(z_H,\pi_H)$ corresponds to $(z_{G^{d_1}},\pi_{G^{d_1}})$ via the choices of $l_1$ and $(s_{d_1})$, where $z_{G^{d_1}}(w)=z_{G^d}(w)$ and $\pi_{G^{d_1}}=\pi_{G^d}$.

\end{enumerate}

\end{cor}
\begin{proof}
We have the Mackey isomorphism
\[ \tx{Res}^B_{B'}H \to \prod_{d \in A \lmod B /B'} \tx{Ind}_{A'_d}^{B'} \tx{Res}_{A'_d}^{A_d}G^d. \]
It sends $h \in H$ to the collection $(h_d)_d$ given by $h_d(b')=h(l(d)b')$. Write $H_d=\tx{Ind}_{A'_d}^{B'} \tx{Res}_{A'_d}^{A_d}G^d$, so that $\tx{Res}^B_{B'}H=\prod H_d$.

The element $z_H$ is mapped to the collection $(z_{H_d})_d$, where $z_{H_d} \in Z^1(u \to W,Z(H_d)^{B'} \to H_d)$. The class of each $z_{H_d}$ is $B'$-invariant, as one checks by sending \eqref{eq:ind1c} through the Mackey isomorphism. In turn, $z_{H_d}$ corresponds by Lemma \ref{lem:ind1} to $z_{G,d} \in Z^1(u \to W,Z(G^d)^{A'_d} \to G^d)$. Explicitly, we have $z_H(w,al(d)s_d(c_d'))=a(z_{H_d}(w,s_d(c_d')))=a(z_{G^d}(w))$.

According to the product $H=\prod_d H_d$, the representation $\pi_H$ is given by $\otimes\pi_{H_d}$, where $\pi_{H_d}$ is a representation of $H_{d,z_{H_d}}(F)$ on a vector space $V_d$. Thus $\pi_H$ acts on $\otimes_d V_{d}$ as $\pi_H(h)=\otimes_d \pi_{H_d}(h_d)$. Therefore each $\pi_{H_d}$ acts on $\otimes_{c_d'} V_{d,c_d'}$ as $\pi_{H_d}(h_d)=\otimes_{c_d'}\pi_{H_d}^{c_d'}(h_d(c_d'))$. Thus we have $\pi_H^{l(d)s_d(c_d')}=\pi_{H_d}^{s_d(c_d')}$ acting on $V_{d,c_d'}$. The class of each $\pi_{H_d}$ is $B'$-invariant, so $\pi_{H_d}$ corresponds to the representation $\pi_{G^d}$ of $G^d$ on the vector space $V_{d,1}$ given by $\pi_{G^d}(g)=\pi_{H_d}^1(g)=\pi_H^{l(d)}(g)$.

The statements concerning $z_H$ and $\pi_H$ now follow immediately from Lemma \ref{lem:ind1} by taking products over $d$. The statement about $\tilde\pi_H$ follows from Lemma \ref{lem:ind2}. The argument for the dual side is analogous, using Lemmas \ref{lem:ind1d} and \ref{lem:ind2d} instead.
\end{proof}

We are almost ready to prove the main result, Proposition \ref{pro:ind}. The final bit that is required is the following elementary discussion of signs.

\begin{lem} \label{lem:ind_eps}
	Let $f : A \lmod B \to \{\pm1\}$ be a function. Define the subset $B^f \subset G$ and the function $\epsilon_f : B^f \to \{\pm1\}$  by 
	\begin{eqnarray*}
		B^f&=&\{b \in B|\forall b' \in B:f(b'b)=f(b')\}\\
		\epsilon(b)&=&\prod_{d \in A\lmod B/\<b\>} f(d)^{n_d-1},
	\end{eqnarray*}
	where for $d \in A \lmod B$ the number $n_d$ is the size of the orbit through $d$ for the action of $\<b\>$ on $A \lmod B$ by multiplication on the right. Then $B^f$ is a subgroup of $B$ and $\epsilon$ is a group homomorphism.
\end{lem}
\begin{proof}
	The action of $B$ on $A \lmod B$ by right multiplication induces an action of $B$ on all functions $A \lmod B \to \{\pm1\}$. The subset $B^f$ is the stabilizer of $f$ for this action, and is thus a subgroup.

	Let $b \in B^f$. The value $\epsilon(b)$ is a product over the $\<b\>$-orbits in $A \lmod B$.Write $(A \lmod B) = (A \lmod B)^+ \cup (A \lmod B)^-$ according to the value of $f$ being $+1$ or $-1$. The action of $B^f$ preserves this decomposition. A $\<b\>$-orbit lying in $(A \lmod B)^+$ contributes a factor of $+1$ to the value of $\epsilon(b)$. A $\<b\>$-orbit lying in $(A \lmod B)^-$ contributes a factor of $(-1)^{n_d-1}$ to the value of $\epsilon(b)$. But this factor is exactly the signature of the permutation that $b$ induces on this orbit. We conclude that $\epsilon(b)=\tx{sgn}(b|(A \lmod B)^-)$.
\end{proof}

\begin{pro} \label{pro:ind}
Assume that Conjectures \ref{cnj:func} and \ref{cnj:coset} hold for $G \rtimes A$. Then they also hold for $H \rtimes B$.
\end{pro}
\begin{proof}
Let $\phi_H : L_F \to {^LH}$ be a tempered Langlands parameter, $z_H \in Z^1(u \to W,Z(H)^B \to B)$ and $\rho_H \in \tx{Irr}(S_{\phi_H}^+,[z_H])$. We are assuming the validity of the refined local Langlands correspondence, so there is a corresponding $\pi_H \in \Pi_{\phi_H}(H_{z_H})$. 

Consider any subgroup $B' \subset B$ fixing the equivalence classes of $z_H$, $\pi_H$, $\phi_H$, $\rho_H$. Choose a sections $l : A \lmod B/B' \to B$,  as well as a section $s_d : A'_d \lmod B' \to B'$ for each $d \in A \lmod B/B'$, as in the discussion before the statement of Corollary \ref{cor:ind6}. That Corollary provides collections $(z_{G^d})$,$(\pi_{G^d})$,$(\phi_{G^d})$,$(\rho_{G^d})$ indexed by $d \in A \lmod B/B'$, where $\phi_{G^d} \in \Phi(G^d)$, $\rho_{G^d} \in \tx{Irr}(\pi_0(S_{\phi_{G^d}}^+))$, $z_{G^d} \in Z^1(u \to W,Z(G^d)^{A'_d} \to G^d)$, and $\pi_{G^d}\in\tx{Irr}(G^d_{z_{G^d}}(F))$. Since the refined local Langlands correspondence is compatible with products of reductive groups we see that for each $d$, $(\phi_{G^d},\rho_{G^d})$ corresponds to $(z_{G^d},\pi_{G^d})$. The part of Corollary \ref{cor:ind6} that describes the compatibility of forming these collections with the action of $B$, applied to the case $B'=\{1\}$, shows that for any $b_1 \in B$ the pair $b_1(\phi_H,\rho_H)$ corresponds to the pair $b_1(z_H,\pi_H)$. That is, Conjecture \ref{cnj:func} holds for $H \rtimes B$. This uses the assumed validity of Conjecture \ref{cnj:func} for $G \rtimes A$, because the formulas describing this action are in terms of the retractions $r : B \to A$ and the action of $A$ on $G$.

In particular we see $B^{[\phi_H]}_{\rho_H}=B^{[z_H]}_{\pi_H}$. Take $B'$ to be this group and apply the above discussion to obtain the collections $(z_{G^d})$,$(\pi_{G^d})$,$(\phi_{G^d})$,$(\rho_{G^d})$. Let $\tilde\pi_{G_d}\boxtimes\tilde\rho_{G^d}^\vee$ be the extension of $\pi_{G^d}\boxtimes\rho_{G^d}^\vee$ to $\tilde G^d_{z_{G^d}}(F)_{\pi_{G^d}} \times_{A'_d} \pi_0(\tilde S_{\phi_{G^d},\rho_{G^d}}^{+,[z_{G^d}]})$ that Conjecture \ref{cnj:coset} for $G^d \rtimes A'_d$ provides. Taking the tensor product over $d$ of the extensions provided by the construction prior to Lemma \ref{lem:ind4} gives an extension $\tilde\pi_H \boxtimes\tilde\rho_H^\vee$ of $\pi_H \boxtimes\rho_H^\vee$ to $\tilde H_{z_H}(F)_{\pi_H} \times_{B'} \pi_0(\tilde S_{\phi_H,\rho_H}^{+,[z_H]})$. Well-definedness of this extension and bijectivity of these constructions follows from Lemma \ref{lem:ind4} and Corollary \ref{cor:ind6}.

This extension is however \emph{not} the correct one, as we will see soon when discussing the character identity. It needs to be twisted as follows. Consider the function 
\[ f : B \to \{\pm1\}, \qquad f(b) :=  e(G_{z_H(-,b)}).\]
This function is left-invariant under $A$ and right invariant under $B^{[z_H]}$. Therefore Lemma \ref{lem:ind_eps} provides a sign character 
\[ \epsilon_{z_H} : B^{[z_H]}  \to \{\pm 1\}. \]
One checks that the following identity holds
\begin{equation} \label{eq:ind_eps}
	e(H_{z_H}) = \epsilon_{z_H}(b) \cdot \prod_d e(G^d_{z_{G^d}}).
\end{equation}

We inflate this character to $[H \rtimes B]_{z_H}(F)$ and define
\begin{equation} \label{eq:ind_hcan}
(\pi_H \boxtimes \rho_H^\vee)^\tx{can} := (\tilde\pi_H \boxtimes \tilde\rho_H^\vee) \otimes \epsilon_{z_H}.
\end{equation}

We now come to the character identity. Thus we fix $\check h_b \rtimes b \in \tilde S_{\phi_H}$, $h_b \rtimes b^{-1} \in \tilde H_{z_H}(F)$, a function $f_H \in \mc{C}^\infty_c(H_{z_H}(F))$, and $\check t \in S_{\phi_H}$, and consider
\begin{equation} \label{eq:charidind1}
e(H_{z_H})\sum_{\substack{\pi_H \in \Pi_{\phi_H}\\ \pi_H \circ b \cong \pi_H}} \tx{tr}((\pi_H \boxtimes\rho_H^\vee)^\tx{can})(R_{h_b \rtimes b^{-1}}^{-1}f_H \times (\check t\check h_b \rtimes b)).	
\end{equation}

To relate the above to the character identity for $G$ we explicate the construction we just employed, with $B'=\<b\>$. We apply Lemmas \ref{lem:indprod} and \ref{lem:indprodd} to represent $\pi_H$ as $\boxtimes_d \pi_d^{\boxtimes n_d}$ and $\rho_H$ as $\boxtimes_d \rho_d^{\boxtimes n_d}$ where, for each $d \in A\lmod B/\<b\>$,  $\pi_d$ is a representations of $G_{z_G}(F)$ invariant under $g_d \rtimes a_d^{-1}=\tilde g_d=\tx{ev}_{l(d)}((h_b \rtimes b^{-1})^{n_d})$ and $\rho_d$ is a representation of $S_{\phi_G}^+$ invariant under $\check g_d \rtimes a_d=\tx{ev}_{l(d)}((\check h_b \rtimes b)^{n_d})$. We fix isomorphisms $\tilde\pi_d : \pi_d \circ \tx{Ad}(\tilde g_d)^{-1} \to \pi_d$ and $\tilde\rho_d : \rho_d \circ \tx{Ad}(\check g_d \rtimes a_d)^{-1} \to \rho_d$, making the choices so that $\tilde\pi_d \boxtimes\tilde\rho_d^\vee$ is the canonical extension of $\pi_d \boxtimes \rho_d^\vee$. This means that the elements of $Z^2(A,\C^\times)$ coming from $\tilde\pi_d$ and from $\tilde\rho_d$ are equal. We obtain a the projective representation $\boxtimes_d (\tilde\pi_d \boxtimes \tilde\rho_d^\vee)^{\boxtimes n_d}$ of $\prod_d \prod_i [G \rtimes \<a_d\>]_{z_G} \times_{\<a_d\>} [\hat G \rtimes \<a_d\>b]_{\phi_G}$, which actually is a linear representation due to Lemmas \ref{lem:ind2} and \ref{lem:ind2d}. Under Lemmas \ref{lem:indprod} and \ref{lem:indprodd} this representation becomes identified with the canonical extension $\tilde\pi_H \boxtimes \tilde\rho_H^\vee$ of $\pi_H \boxtimes \rho_H^\vee$ to $[H \rtimes \<b\>]_{z_H}(F) \times_{\<b\>} [\hat H \rtimes \<b\>]_{\phi_H}$. 

We write $f_H=\otimes_d \otimes_{i=0}^{n_d-1} f_{d,i}$ and $\check t = \prod_d \prod_i \check s_{d,i}$ and then Lemma \ref{lem:basictwist} implies that $\tx{tr}(\tilde\pi_H \boxtimes\tilde\rho_H^\vee)(R_{h_b \rtimes b^{-1}}^{-1}f_H \times (\check t\check h_b \rtimes b))$ equals
\[ \prod_d \tx{tr}(\tilde\pi_d \boxtimes\tilde\rho_d^\vee)(R_{g_d \rtimes a_d^{-1}}^{-1}f_{d,0}*\dots*f_{d,n_d-1},\check s_{d,0}\dots \check s_{d,n_d-1}\check g_d \rtimes a_d). \]
The set $\{\pi_H \in \Pi_{\phi_H}|\pi_H \circ b \cong \pi_H\}$ is translated to the set $\{\otimes_d \pi_d \in \otimes_d \Pi_{\phi_d}|\pi_d\circ a_d \cong \pi_d\}$. Therefore \eqref{eq:charidind1} becomes the product of $e(H_{z_H}) \cdot \epsilon_{z_H}(b)$ with 
\[ \prod_d \sum_{\substack{\pi_d \in \Pi_{\phi_d}\\\pi_d\circ a_d \cong \pi_d}} \tx{tr}(\tilde\pi_d \boxtimes\tilde\rho_d^\vee)(R_{g_d \rtimes a_d}^{-1}f_{d,0}*\dots*f_{d,n_d-1},\check s_{d,0}\dots \check s_{d,n_d-1}\check g_d \rtimes a_d).\]
The parameter $\phi_d$ and the element $\check s_{d,0}\dots \check s_{d,n_d-1}\check g_d \rtimes a_d \in \hat G \rtimes A$ lead to an endoscopic datum $\mf{e}_d$ and parameter $\phi_{\mf{e}_d}$. The character identities for $\tilde G$ imply that the above equals the product of $e(H_{z_H}) \cdot \epsilon_{z_H}(b)$ with 
\[ \prod_d e(G^d_{z_{G^d}}) S\Theta_{\phi_{\mf{e}_d}}(f^{\mf{e}_d}_\tx{KS}), \]
where $f^{\mf{e}_d}_{KS}$ is the transfer of $R_{g_d \rtimes a_d}^{-1}f_{d,0}*\dots*f_{d,n_d-1}$ with respect to $\Delta_{KS}$.
By Corollary \ref{cor:indprodd} the endoscopic datum for $H$ and $\check t \check h_b \rtimes b$ and $\phi_H$ is $\prod_d \mf{e}_d$, and the function $\otimes f^{\mf{e}_d}_{KS}$ has $KS$-matching orbital integrals with $f_H$. Noting \eqref{eq:ind_eps}, the proof is complete.

\end{proof}

\section{The case of tori} \label{sec:tori}
In this section we are going to sketch the proof of Conjecture \ref{cnj:llc_rigid} in the case where the reductive group $G$ is a torus. We will write $T$ instead of $G$ to emphasize this. Note that, while a torus $T$ is tautologically quasi-split, an inner form of $T \rtimes A$ need not be quasi-split. This is the main source of complications we will have to deal with.

\subsection{Initial considerations} \label{sub:init}

Let $\phi : W_F \to \hat T$ and let $[\phi]$ denote both the equivalence class of $\phi$ and the corresponding character $T(F) \rw \C^\times$. Let $Z \subset T$ be finite and defined over $F$, and $z \in Z^1(u \to W,Z \to T)$. Then
\[ \tilde T_z(F) = (T(\bar F) \rtimes A)^{\tilde{\bar z}(\Gamma)} = \{ \tilde\delta \in T(\bar F) \rtimes A| \tx{Ad}(\tilde{\bar z}(\sigma))\tilde\delta=\tilde\delta\ \forall \sigma \in \Gamma\}. \]
The group $\tilde T_z(F)$ is an extension
\[ 1 \rw T(F) \rw \tilde T_z(F) \rw A^{[z]} \rw 1. \]
The set $\tilde \Pi_{\phi,z}$ consists of those irreducible admissible representations of $\tilde T_z(F)$ whose restriction to $T(F)$ contains the character $[\phi]$. All these representations are finite-dimensional. They can be described as follows. We have the pull-back and push-out diagram
\[ \xymatrix{
1\ar[r]&T(F)\ar[r]\ar@{=}[d]&(T(\bar F) \rtimes A)^{\tilde{\bar z}(\Gamma)}\ar[r]&A^{[z]}\ar[r]&1\\
1\ar[r]&T(F)\ar[r]\ar[d]^{[\phi]}&(T(\bar F) \rtimes A^{[\phi]})^{\tilde{\bar z}(\Gamma)}\ar[r]\ar[u]\ar[d]&A^{[z],[\phi]}\ar[r]\ar@{=}[d]\ar@{^(->}[u]&1\\
1\ar[r]&\C^\times\ar[r]&\mc{E}^z_{[\phi]}\ar[r]&A^{[z],[\phi]}\ar[r]&1
} \]
The bottom extension is central. If we let $\tx{Irr}(\mc{E}^z_{[\phi]},\tx{id})$ denote the set of irreducible representations of $\mc{E}^z_{[\phi]}$ whose central character restricts to the identity on $\C^\times$, then inflating an element of $\tx{Irr}(\mc{E}^z_{[\phi]},\tx{id})$ to $(T(\bar F) \rtimes A^{[\phi]})^{\tilde{\bar z}(\Gamma)}$ and then inducing it to $(T(\bar F) \rtimes A)^{\tilde{\bar z}(\Gamma)}$ provides a canonical bijection
\[ \tx{Irr}(\mc{E}^z_{[\phi]},\tx{id}) \rw \tilde \Pi_{\phi,z}. \]

Dually, we have $\tilde S_\phi^{[z]}=(\hat T \rtimes A^{[z]})^{\phi(W_F)}$. Its preimage $\tilde S_\phi^{+,[z]}$ in $\hat{\bar T} \rtimes A$ fits in the following push-out diagram
\[ \xymatrix{
1\ar[r]&\pi_0([\hat{\bar T}]^+)\ar[r]\ar[d]^{[z]}&\pi_0(\tilde S_\phi^{+,[z]})\ar[r]\ar[d]&A^{[z],[\phi]}\ar[r]\ar@{=}[d]&1\\
1\ar[r]&\C^\times\ar[r]&\mc{E}^\phi_{[z]}\ar[r]&A^{[z],[\phi]}\ar[r]&1
} \]
and again there is a canonical bijection $\tx{Irr}(\mc{E}^\phi_{[z]},\tx{id}) \rw \tx{Irr}(\tilde S_{\phi,[z]},{[z]})$ given simply by inflation.

While the extensions $\mc{E}^\phi_{[z]}$ and $\mc{E}^z_{[\phi]}$ are constructed from essentially the same data, their constructions are in some sense dual to each other. In Subsection \ref{sub:eiso} we will construct a natural isomorphism of extensions $\mc{E}^z_{[\phi]} \cong \mc{E}^\phi_{[z]}$. The resulting bijections
\begin{equation} \label{eq:cnj1tori} \tx{Irr}(\tilde S_{\phi,[z]},{[z]}) \rw \tx{Irr}(\mc{E}^\phi_{[z]},\tx{id}) \cong \tx{Irr}(\mc{E}^z_{[\phi]},\tx{id}) \rw \tilde \Pi_{\phi,z} \end{equation}
will imply Conjecture \ref{cnj:llc_pure_is}. We will then go on to verify the identity claimed in Conjecture \ref{cnj:llc_pure_ci}.

\subsection{The isomorphism $\mc{E}^z_{[\phi]} \cong \mc{E}^\phi_{[z]}$} \label{sub:eiso}

We will first realize the extensions $\mc{E}^z_{[\phi]}$ and $\mc{E}^\phi_{[z]}$ explicitly as twisted products of $\C^\times$ with $A^{[\phi],[z]}$.

For each element $a \in A^{[\phi],[z]}$ we choose $t_a \in T(\bar F)$ and $s_a \in \hat T$ such that
\[ (z,t_a) \in Z^1_a(u \to W,Z \to T \rrw T) \qquad\tx{and}\qquad (\phi,s_a\rtimes a) \in Z^1_a(W_F,\hat T \rrw \hat T). \]
More explicitly, if we write $\phi(w)=\phi_0(w)\rtimes w$, then the above can be reformulated as
\[ z(w)^{-1} \cdot a(z(w)) = t_a^{-1}\cdot\sigma_w(t_a)\qquad\tx{and}\qquad \phi_0(w)^{-1} \cdot a(\phi_0(w))=s_a^{-1} \cdot \sigma_w(s_a), \]
where $\sigma_w \in \Gamma$ is the image of $w \in W$ in the first case, and of $w \in W_F$ in the second.
The choices of $t_\bullet$ and $s_\bullet$ give us sections $a \mapsto t_a \rtimes a$ and $a \mapsto s_a \rtimes a$ of the two extensions

\[ 1 \rw T(F) \rw (T(\bar F) \rtimes A^{[\phi]})^z \rw A^{[\phi],[z]} \rw 1 \]
and
\[ 1 \rw \hat T^\Gamma \rw (\hat T \rtimes A^{[z]})^\phi \rw A^{[\phi],[z]} \rw 1. \]
Choose a lift $\dot s_a \in \hat{\bar T}$ for each $s_a$. Then $a \mapsto \dot s_a \rtimes a$ is a section of the extension
\[ 1 \rw \pi_0([\hat{\bar T}]^+) \rw \tilde S_\phi^{+,[z]} \rw A^{[\phi],[z]} \rw 1. \]
The 2-cocycles corresponding to the sections $a \mapsto t_a \rtimes a$ and $a \mapsto \dot s_a \rtimes a$ are
\[ \alpha(a,b) = t_a \cdot {^at_b} \cdot t_{ab}^{-1}\qquad\tx{and}\qquad \beta(a,b) = \dot s_a \cdot {^a\dot s_b} \cdot \dot s_{ab}^{-1}. \]
Let $\bar\alpha = [\phi] \circ\alpha$ and $\bar\beta = [z] \circ\beta$. Then we have
\[ \mc{E}^z_{[\phi]} = \C^\times \boxtimes_{\bar\alpha} A^{[\phi],[z]}\qquad\tx{and}\qquad\mc{E}^\phi_{[z]} = \C^\times \boxtimes_{\bar\beta} A^{[\phi],[z]}. \]
By construction, for each $a \in A^{[\phi],[z]}$, we have $(z_0^{-1},t_{a^{-1}})\in Z^1(u \to W,Z \to T \stackrel{1-a^{-1}}{\lrw} T)$ and $(\phi_0^{-1},\dot s_a) \in Z^1(W_F,Z \to \hat T \stackrel{1-a}{\llw} \hat T)$. We put
\[ h(a) := \bar\alpha(a^{-1},a)\cdot \left\< (z^{-1},t_{a^{-1}}),(\phi_0^{-1},\dot s_a)\right\>, \]
where the pairing $\<-,-\>$ is \eqref{eq:tnd++}.

\begin{pro} \label{pro:isoh}
The map $x\boxtimes a \mapsto xh(a) \boxtimes a$ is an isomorphism $\mc{E}^z_{[\phi]} \rw \mc{E}^\phi_{[z]}$. It is independent of the choices of $t_a$ and $\dot s_a$.
\end{pro}
\begin{proof}
It is obvious that the map is bijective, but we need to show that it is a homomorphism. This amounts to the equation
\begin{equation} \label{eq:h} h(a)h(b)h(ab)^{-1} = \bar\alpha(a,b)\bar\beta(a,b)^{-1}. \end{equation}
We choose for each $a \in A$ an element $(\bar\lambda_a,\mu_a) \in Z_0(W_{K/F},X_*(\bar T) \stackrel{1-a}{\lrw}X_*(T))_0$ whose image in $H^1(u \to W,Z \to T \stackrel{1-a}{\lrw} T)$ under the isomorphism \eqref{eq:arithiso} equals $(z^{-1},t_a)$. Here $K/F$ is a suitably large Galois extension. Note that all $\bar\lambda_a \in Z_0(W_{K/F},X_*(\bar T))_0=X_*(\bar T)^{N_{K/F}}$ have the same image in $X_*(\bar T)^N/IX_*(T)$, since their images under the isomorphism $X_*(\bar T)^N/IX_*(T) \to H^1(u \to W,Z \to T)$ all equal $z^{-1}$. Thus we may choose a single $\bar\lambda \in Z_0(W_{K/F},X_*(T))_0$ and arrange, by modifying $(\bar\lambda_a,\mu_a)$ by a coboundary, that $\bar\lambda_a=\bar\lambda$ for all $a$. Then the pairing $\<(z^{-1},t_{a^{-1}}),(\phi_0^{-1},\dot s_a)\>$ is equal to
\[ \<\bar\lambda,\dot s_a\> \cdot \prod_{w \in W_{K/F}} \<\mu_{a^{-1}}(w),\phi_0(w)\>, \]
according to the definition of \eqref{eq:tnd++} as the composition of \eqref{eq:elempair} and \eqref{eq:arithiso}. Here the angle brackets denote the canonical pairing $X_*(T) \otimes \hat T \to \C^\times$ and its analog for $\bar T$.

With this, we can now compute $h(a)h(b)h(ab)^{-1}$. For $h(b)$ and $h(ab)$ we simply plug in this formula. For $h(a)$, we shall replace $(\phi_0^{-1},\dot s_a)$ by the element $({^b\phi}_0^{-1},\dot s_a \cdot {^a \dot s_b}\cdot \dot s_b^{-1})$, which is easily seen to be cohomologous using the fact that  $(\phi_0^{-1},\dot s_b) \in Z^1(W_F,\hat Z \to \hat T \stackrel{1-b}{\llw} \hat T)$. All together we obtain
\begin{eqnarray*}
h(a)h(b)h(ab)^{-1}&=&\bar\alpha(a^{-1},a)\bar\alpha(b^{-1},b)\bar\alpha((ab)^{-1},ab)^{-1}\cdot\\
&&\<\bar\lambda,\dot s_a\cdot {^a\dot s_b} \cdot \dot s_{ab}^{-1}\> \cdot \prod_w \< {^{b^{-1}}}\mu_{a^{-1}}+\mu_{b^{-1}}-\mu_{(ab)^{-1}},\phi_0(w)\>.	
\end{eqnarray*}

Using 
\[ (\phi_0^{-1},\dot s_a) \in Z^1(W_F,\hat Z \to \hat T \stackrel{1-a}{\llw} \hat U),\quad  (\bar\lambda,\mu_a) \in Z_0(W_{K/F},X_*(\bar T) \stackrel{1-a^{-1}}{\lrw} X_*(T)) \] 
one checks that 
\[ \dot s_a\cdot {^a \dot s_b} \cdot \dot s_{ab}^{-1} \in [\hat {\bar T}]^+,\quad {^{b^{-1}}}\mu_{a^{-1}}+\mu_{b^{-1}}-\mu_{(ab)^{-1}} \in Z_1(W_{K/F},X_*(T)). \] 
The functoriality of the maps $\dot\psi$ and $\phi$ that make up the isomorphism \eqref{eq:arithiso} implies that we have
\begin{eqnarray*}
h(a)h(b)h(ab)^{-1}&=&\bar\alpha(a^{-1},a)\bar\alpha(b^{-1},b)\bar\alpha((ab)^{-1},ab)^{-1}\cdot\\
&&\<(z^{-1},\dot s_a\cdot {^a\dot s_b} \cdot \dot s_{ab}^{-1}),(\phi_0^{-1},t_{b^{-1}} \cdot {^{b^{-1}}t_{a^{-1}}}\cdot t_{(ab)^{-1}}^{-1})\>.
\end{eqnarray*}
where the pairing is now between $H^1(u \to W,Z \to T \stackrel{0}{\lrw} T)$ and $H^1(W_F,\hat Z \to \hat T \stackrel{0}{\llw} \hat T)$. Using Corollary \ref{cor:tn++d4} we see
\[
h(a)h(b)h(ab)^{-1}=\bar\alpha(a^{-1},a)\bar\alpha(b^{-1},b)\bar\alpha((ab)^{-1},ab)^{-1}\bar\alpha(b^{-1},a^{-1})^{-1}\cdot\bar\beta(a,b)^{-1}.
\]
Finally, an elementary computation using the fact that $\bar\alpha$ is a cocycle shows that all terms involving $\bar\alpha$ combine to $\bar\alpha(a,b)$.

It remains to show that the isomorphism $\mc{E}^z_{[\phi]} \rw \mc{E}^\phi_{[z]}$ we have just constructed is independent of the choices involved in its construction, that is of the choices of elements $t_a \in T(\bar F)$ and $\dot s_a \in \hat{\bar T}$.
For this we need to check that if we replace $t_a$ by $x_at_a$ with $x_a \in T(F)$, then $h(a)$ is replaced by $\<[\phi],x_a\>h(a)$, and if we replace $\dot s_a$ by $y_a\dot s_a$ with $y_a \in [\hat{\bar T}]^+$, then $h(a)$ is replaced by $\<[z],y_a\>^{-1}h(a)$. Both of these verifications are immediate.
\end{proof}

\subsection{Remarks and generalizations}

Before we continue with the proof of Conjecture \ref{cnj:llc_rigid} for tori, we would like to point out a beautiful symmetry between $\tilde T_z(F)$ and $\tilde S_\phi^{[z]}$ that may have become covered under the debris of generality. To see it more clearly, let us consider the special case where the Langlands parameter $\phi$ extends to the Galois group (thus it corresponds to a character of $T(F)$ whose composition with the norm map $N_{K/F} : T(K) \to T(F)$ is trivial for some finite extension $K/F$) and the inner form of $T \rtimes A$ we are considering is pure. As above we shall write $\phi : \Gamma \to \hat T \rtimes \Gamma$ for the Langlands parameter, and $\phi_0 : \Gamma \to \hat T$ for the corresponding cocycle, so that $\phi(\sigma) = \phi_0(\sigma) \rtimes \sigma$. We shall use the analogous notation $z : \Gamma \to T \rtimes \Gamma$ and $z_0 : \Gamma \to T$ for the pure inner form, slightly deviating from the notation of the rest of the paper, where we used $z$ and $\tilde z$ instead. We are writing $T$ for $T(\bar F)$, in the same way we are writing $\hat T$ for $\hat T(\C)$.

Now $\tilde T_z(F)=(T \rtimes A)^{z(\Gamma)}$ and $S_\phi=(\hat T \rtimes A)^{\phi(\Gamma)}$. These fit into the extensions
\[ 1 \to T^{z(\Gamma)} \to (T \rtimes A)^{z(\Gamma)} \to A^{[z]} \to 1 \]
and
\[ 1 \to \hat T^{\phi(\Gamma)} \to (T \rtimes A)^{\phi(\Gamma)} \to A^{[\phi]} \to 1.\]
We have written $T^{z(\Gamma)}$ for $T^\Gamma=T(F)$ and $\hat T^{\phi(\Gamma)}=\hat T^\Gamma$ to emphasize the symmetry. Now $[\phi]$ is a character of $T^{z(\Gamma)}$ and $[z]$ is a character of $T^{\phi(\Gamma)}$. We pull back the above extensions along the inclusions of $A^{[z],[\phi]}$ into $A^{[z]}$ and $A^{[\phi]}$ and obtain the push-out diagrams
\[ \xymatrix{
1\ar[r]&T^{z(\Gamma)}\ar[r]\ar[d]^{[\phi]}&(T \rtimes A)^{z(\Gamma),[\phi]}\ar[r]&A^{[z],[\phi]}\ar[r]&1\\
&\C^\times
}
\]
and
\[ \xymatrix{
1\ar[r]&\hat T^{\phi(\Gamma)}\ar[r]\ar[d]^{[z]}&(\hat T \rtimes A)^{\phi(\Gamma),[z]}\ar[r]&A^{[z],[\phi]}\ar[r]&1\\
&\C^\times
}
\]
Which lead to the extensions $\mc{E}^z_{[\phi]}$ and $\mc{E}^\phi_{[z]}$ of $A^{[z],[\phi]}$ by $\C^\times$. The symmetry of the situation now makes it rather natural to expect that these two extensions are closely related.

Moving towards the opposite end on the spectrum of clarity, we are now going to formulate a situation a bit more general then the one considered in Subsection \ref{sub:eiso}. We will not need this generalization in the present paper, but will need it in a forthcoming paper in a rather different set-up.

Let $T$ be an algebraic torus $T$ defined over $F$, and $A$ a finite group acting on $T$ by $F$-automorphisms. Let $Z \subset T$ be a finite subgroup defined over $F$ and fixed pointwise by $A$. Let $\phi : W_F \to \hat T \rtimes W_F$ and $z \in Z^1(u \to W,Z \to T)$. Write $\hat{\bar T}$ for the complex dual group of $\bar T=T/Z$.

Instead of considering the split extensions $T \rtimes A$ and $\hat{\bar T} \rtimes A$, we now assume given extensions $1 \to T \to \tilde T \to A \to 1$ and $1 \to \hat{\bar T} \to \mc{\bar T} \to A \to 1$ that may or may not be split. Dividing out by $\hat Z$ we obtain an extension $1 \to \hat T \to \mc{T} \to A \to 1$. We emphasize that no relation is assumed between $\tilde T$ and $\mc{T}$. We assume that after taking $F$-points the sequence $1 \to T(F) \to \tilde T(F) \to A^\Gamma \to 1$ is still exact, and after taking $\Gamma$-invariants the sequence $1 \to \hat T^\Gamma \to \mc{T}^\Gamma \to A^\Gamma \to 1$ remains exact. Let $[\mc{\bar T}]^+$ be the preimage of $\mc{T}^\Gamma$ in $\mc{\bar T}$.

Let $1 \to \C^\times \to \mc{E}_{[\phi]}^0 \to A^{[\phi]} \to 1$ be the push-out of $1 \to T(F) \to \tilde T(F)^{[\phi]} \to A^{[\phi]} \to 1$ along $[\phi] : T(F) \to \C^\times$. Let $1 \to \C^\times \to \mc{E}_{[z]}^0 \to A^{[z]} \to 1$ be the push-out of $1 \to [\hat{\bar  T}]^+ \to [\mc{\bar T}]^{+,[z]} \to A^{[z]} \to 1$ along $[z] : [\hat{\bar  T}]^+ \to \C^\times$.

We now consider the inner form $\tilde T_z$. We have $\tilde T_z(F)=\{\tilde t \in \tilde T(\bar F)|\forall \sigma \in \Gamma:\tilde t= \tx{Ad}(\bar z(\sigma))\sigma(\tilde t)\}$, where again $\bar z \in Z^1(\Gamma,\bar T)$ is the image of $z$, and we are using that the conjugation action of $T$ on $\tilde T$ factors through $\bar T$ because $Z$ is pointwise fixed by $A$. The assumption that $\tilde T(F) \to A^\Gamma$ is surjective implies that $\tilde T_z(F) \to A^{[z]}$ is surjective, where again $A^{[z]}$ is the stabilizer in $A$ of the class $[z] \in H^1(u \to W,Z \to T)$. Thus we have the extension
\[ 1 \to T(F) \to \tilde T_z(F) \to A^{[z]} \to 1. \]
We pull back along the inclusion $A^{[z],[\phi]} \to A^{[z]}$ and push out along $[\phi] : T(F) \to \C^\times$ to obtain an extension $1 \to \C^\times \to \mc{E}_{[\phi]}^z \to A^{[z],[\phi]} \to 1$.

Dually we consider the centralizer $S_\phi = \mc{T}^{\phi(W_F)}$ of $\phi$ in $\mc{T}$. Again the assumption that $\mc{T}^\Gamma \to A^\Gamma$ is surjective implies that $S_\phi \to A^{[\phi]}$ is surjective. Let $S_\phi^+$ be the preimage of $S_\phi$ in $\mc{\bar T}$, so that we have the extension
\[ 1 \to [\hat{\bar T}]^+ \to S_\phi^+ \to A^{[\phi]} \to 1. \]
We pull back along the inclusion $A^{[z],[\phi]} \to A^{[\phi]}$ and push out along $[z] : [\hat{\bar T}]^+ \to \C^\times$ to obtain an extension $1 \to \C^\times \to \mc{E}_{[z]}^\phi \to A^{[z],[\phi]} \to 1$.

\begin{pro} Let $\mc{E}_{[\phi]}^{0,[z]}$ and $\mc{E}_{[z]}^{0,[\phi]}$ be the pull-backs of $\mc{E}_{[\phi]}^0$ and $\mc{E}_{[z]}^0$ along the inclusions of $A^{[z],[\phi]}$ into $A^{[\phi]}$ and $A^{[z]}$. An isomorphism of extensions $\zeta : \mc{E}_{[\phi]}^{0,[z]} \to \mc{E}_{[z]}^{0,[\phi]}$ determines an isomorphism of extensions $\xi : \mc{E}_{[\phi]}^z \to \mc{E}_{[z]}^\phi$. If $\zeta$ is multiplied by a character $A^{[z],[\phi]} \to \C^\times$, then $\xi$ is multiplied by the same character.
\end{pro}
\begin{proof}
To lighten notation, we replace $A$ by its subgroup $A^{[z],[\phi]}$. For each $a \in A$ choose lifts $\theta_a \in \tilde T(F)$ and $\tau_a \in [\mc{\bar T}]^+$, as well as elements $t_a \in T(\bar F)$ such that $t_a\theta_a \in \tilde T_z(F)$ and $\dot s_a \in \hat{\bar T}$ such that $\dot s_a\tau_a \in S_\phi^+$.

The section $a \mapsto t_a\theta_a$ realizes the extension $\mc{E}_{[\phi]}^z$ as the twisted product $\C^\times \boxtimes_{\bar\alpha'}A$, where $\alpha' \in Z^2(A,T(F))$ is defined as $\alpha'(a,b)=t_a\theta_at_b\theta_b(t_{ab}\theta_{ab})^{-1}$ and $\bar\alpha' = [\phi]\circ\alpha' \in Z^2(A,\C^\times)$. The section $a \mapsto \dot s_a\tau_a$ realizes the extension $\mc{E}_{[z]}^\phi$ as the twisted product $\C^\times\boxtimes_{\bar\beta'}A$, where $\beta' \in Z^2(A,[\hat{\bar T}]^+)$ is defined as $\beta'(a,b)=\dot s_a\tau_a \dot s_b\tau_b (\dot s_{ab}\tau_{ab})^{-1}$ and $\bar\beta'=[z]\circ\beta' \in Z^2(A,\C^\times)$.

We have $\alpha'(a,b)=\alpha(a,b)\cdot \alpha_0(a,b)$ with $\alpha(a,b)=t_a \cdot {^at_b} \cdot t_{ab}^{-1}$ and $\alpha_0(a,b)=\theta_a\theta_b\theta_{ab}^{-1}$. The element $\alpha_0 \in Z^2(A,T(F))$ is the 2-cocycle corresponding to the section $a \mapsto \theta_a$, which then identifies $\mc{E}_{[\phi]}^{0,[z]}$ with $\C^\times\boxtimes_{\bar\alpha_0}A$. Analogously we have $\beta'(a,b)=\beta(a,b)\beta_0(a,b)$ with $\beta(a,b)=\dot s_a {^a\dot s_b} \dot s_{ab}^{-1}$ and $\beta_0(a,b)=\tau_a\tau_b\tau_{ab}^{-1}$. The element $\beta_0 \in Z^2(A,\hat{\bar T})$ is the 2-cocycle corresponding to the section $a \mapsto \tau_a$, which then identifies $\mc{E}_{[z]}^{0,[\phi]}$ with $\C^\times \boxtimes_{\bar\beta_0}A$.

Let $\zeta : \mc{E}_{[\phi]}^{0,[z]} \to \mc{E}_{[z]}^{0,[\phi]}$ be an isomorphism of extensions. The composition $\C^\times \boxtimes_{\bar\alpha_0} A \to \mc{E}_{[\phi]}^{0,[z]} \to \mc{E}_{[z]}^{0,[\phi]} \to \C^\times\boxtimes_{\bar\beta_0}A$ is given by $x \boxtimes a \mapsto x\zeta_0(a)\boxtimes a$, where $\zeta_0 : A \to \C^\times$ is defined as $\zeta_0(a)=\zeta(\bar\theta_a)\bar\tau_a^{-1}$ and satisfies $\zeta_0(a)\zeta_0(b)\zeta_0(ab)^{-1}=\bar\alpha_0(a,b)^{-1}\bar\beta_0(a,b)$. Here $\bar\theta_a \in \mc{E}_{[\phi]}^{0,[z]}$ and $\bar\tau_a \in \mc{E}_{[z]}^{0,[\phi]}$ are the images of $\theta_a$ and $\tau_a$ respectively.

Let $h : A \to \C^\times$ be defined as in Subsection \ref{sub:eiso}. We claim that
\[ \C^\times\boxtimes_{\bar\alpha\bar\alpha_0} A \to \C^\times\boxtimes_{\bar\beta\bar\beta_0}A,\qquad x\boxtimes a \mapsto h(a)\zeta_0(a) \boxtimes a \]
is an isomorphism of extensions and the composition
\[ \xi : \mc{E}_{[\phi]}^z \to \C^\times\boxtimes_{\bar\alpha\bar\alpha_0} A \to \C^\times\boxtimes_{\bar\beta\bar\beta_0}A \to \mc{E}_{[z]}^\phi \]
depends only on $\zeta$, and not on the choices of $\theta_a$, $\tau_a$, $t_a$, or $\dot s_a$.

The first part of the claim is equivalent to $h(a)h(b)h(ab)^{-1}=\bar\alpha(a,b)^{-1}\bar\beta(a,b)$, which was the content of the proof of Proposition \ref{pro:isoh}. This proof remains valid verbatim in the current situation. For the second claim, the independence of the choices of $t_a$ and $\dot s_a$ was already addressed in the proof of Proposition \ref{pro:isoh}. Now say we replace $\theta_a$ by $x_a\theta_a$ and $\tau_a$ by $\dot y_a\tau_a$, for $x_a \in T(F)$ and $\dot y_a \in [\hat{\bar T}]^+$. Since we already have independence of $t_a$ and $\dot s_a$, we are free to replace $t_a$ by $x_a^{-1}t_a$ and $\dot s_a$ by $\dot y_a^{-1}\dot s_a$. This has the effect of keeping $\bar\alpha\bar\alpha_0$ and $\bar\beta\bar\beta_0$, as well as the first and third arrows in the last displayed sequence, unchanged. At the same time, $h(a)$ is replaced by $h(a)\<[\phi],x_a^{-1}\>\<[z],\dot y_a\>$, while $\zeta_0(a)$ is replaced by $\zeta_0(a)\<[\phi],x_a\>\<[z],\dot y_a^{-1}\>$, so the middle arrow is unchanged as well.

Finally, if $\zeta$ is replaced by $\delta\zeta$, then $\zeta_0$ is replaced by $(\delta\zeta)_0$ specified by $(\delta\zeta)_0(a)=(\delta\zeta)(\bar\theta_0)\bar\tau_a^{-1}=\delta(a)\zeta(\bar\theta_0)\bar\tau_a^{-1}=\delta(a)\zeta_0(a)$. If follows that the isomorphism $\C^\times\boxtimes_{\bar\alpha\bar\alpha_0} A \to \C^\times\boxtimes_{\bar\beta\bar\beta_0}A$ is multiplied by $\delta$, and the same is then true for $\xi$.
\end{proof}

\subsection{Computing the right-hand side of \eqref{eq:charid}}
In this section we will compute the virtual character
\[ \Theta^{\tilde s}_\phi := \sum_\rho \tx{tr}\rho(\tilde s) \cdot \Theta_{\pi_\rho} \]
where $\rho$ runs over the set $\tx{Irr}(\pi_0(\tilde S_\phi^{[z]}),{[z]})$. We recall from Subsections \ref{sub:init} and \ref{sub:eiso} that we have the following diagram
\[ \xymatrix{
1\ar[r]&T(F)\ar[r]\ar@{=}[d]&(T(\bar F) \rtimes A)^{\tilde{\bar z}(\Gamma)}\ar[r]&A^{[z]}\ar[r]&1\\
1\ar[r]&T(F)\ar[r]\ar[d]^{[\phi]}&(T(\bar F) \rtimes A^{[\phi]})^{\tilde{\bar z}(\Gamma)}\ar[r]\ar[u]\ar[d]^F&A^{[z],[\phi]}\ar[r]\ar@{=}[d]\ar@{^(->}[u]&1\\
1\ar[r]&\C^\times\ar[r]\ar@{=}[d]&\mc{E}^z_{[\phi]}\ar[r]\ar[d]^H&A^{[z],[\phi]}\ar[r]\ar@{=}[d]&1 \\
1\ar[r]&\C^\times\ar[r]&\mc{E}^\phi_{[z]}\ar[r]&A^{[z],[\phi]}\ar[r]&1\\
1\ar[r]&\pi_0([\hat {\bar T}]^+)\ar[r]\ar[u]^{[z]}&\pi_0(\tilde S_\phi^+)\ar[r]\ar[u]^G&A^{[z],[\phi]}\ar[r]\ar@{=}[u]&1\\
} \]
We can be more explicit about the maps $F$, $G$, and $H$. Recall from Subsection \ref{sub:eiso} that we have fixed elements $t_a \in T(\bar F)$ and $\dot s_a \in \hat{\bar T}$ such that $a \mapsto t_a \rtimes a$ and $a \mapsto \dot s_a \rtimes a$ are sections of the top and bottom extensions.
The corresponding 2-cocycles are
\[ \alpha(a,b) = t_a \cdot {^at_b} \cdot t_{ab}^{-1}\qquad\tx{and}\qquad \beta(a,b) = \dot s_a \cdot {^a\dot s_b} \cdot \dot s_{ab}^{-1}. \]
and setting $\bar\alpha = [\phi] \circ\alpha$ and $\bar\beta = [z] \circ\beta$ allows us to make the identifications
\[ \mc{E}^z_{[\phi]} = \C^\times \boxtimes_{\bar\alpha} A^{[\phi],[z]}\qquad\tx{and}\qquad\mc{E}^\phi_{[z]} = \C^\times \boxtimes_{\bar\beta} A^{[\phi],[z]}. \]
The maps $F$, $G$, and $H$ are now explicitly given by
\[ F(t\cdot t_a \rtimes a) = [\phi](t) \boxtimes a,\qquad G(\dot s \cdot \dot s_b \rtimes b) = [z](s)\boxtimes b,\qquad H(x \boxtimes a) = xh(a) \boxtimes a. \]
We now fix $t \cdot t_a \rtimes a \in \tilde T_z(F)$ and $\dot s \cdot \dot s_b \rtimes b \in \tilde S_\phi^{[z]}$ and set out to compute
\[ \Theta_\phi^{s\cdot s_b \rtimes b}(t\cdot t_a \rtimes a). \]
Since this is a virtual character of $(T(\bar F) \rtimes A)^z$ which is induced from a virtual character of $(T(\bar F) \rtimes A^{[\phi]})^z$, there is no loss of generality if we assume $t \cdot t_a \rtimes a \in (T(\bar F) \rtimes A^{[\phi]})^z$, which simply means $a \in A^{[\phi],[z]}$. Then, by construction, we have
\begin{eqnarray}
\label{eq:os1} O_\phi^{s\cdot s_b \rtimes b}(t\cdot t_a \rtimes a)&=&\!\!\!\!\!\!\!\!\!\!\!\!\!\!\sum_{\tau \in \tx{Irr}(\C^\times\boxtimes_{\bar\beta}A^{[\phi],[z]},\tx{id})}\!\!\!\!\!\!\!\!\!\!\!\!\!\!\chi_\tau(G(\dot s\cdot \dot s_b\rtimes b))\\
&\cdot&|A^{[\phi],[z]}|^{-1}\!\!\!\!\!\!\!\!\!\sum_{\substack{c \in A^{[z]}\\ cac^{-1}\in A^{[\phi],[z]}}}\!\!\!\!\!\!\!\!\!\chi_\tau(HF((t_c \rtimes c)(t\cdot t_a \rtimes a)(t_c \rtimes c)^{-1})) \nonumber
\end{eqnarray}
where $\chi_\tau$ denotes the character of the finite dimensional representation $\tau$ and the second line is the Frobenius formula for the character of the representation on $\tilde T_z(F)$ induced from $\tau\circ HF$.

We compute
\[ (t_c \rtimes c)(t\cdot t_a \rtimes a)(t_c \rtimes c)^{-1} = {^ct}\cdot \zeta(c,a)\cdot t_{cac^{-1}} \rtimes cac^{-1}, \]
where
\[ \zeta(c,a) := {^ct_a}t_c{^{(cac)^{-1}}t_c^{-1}}t_{cac^{-1}}^{-1} \in T(F). \]
With this we have
\[ \Theta_\phi^{\dot s\cdot \dot s_b \rtimes b}(t\cdot t_a \rtimes a) = [z](\dot s)\sum_c[\phi]({^ct}\zeta(c,a)) h(cac^{-1}) |A^{[\phi],[z]}|^{-1}\sum_\tau \chi_{\bar\tau}(b)\chi_{\bar\tau}(cac^{-1}). \]
We now apply Lemma \ref{lem:orth} to the sum over $\tau$ and conclude that if $cac^{-1}$ is not conjugate to $b^{-1}$ then the corresponding summand is zero. It is more convenient to apply this information not to the expression we just obtained, but to the original expression we started with, namely \eqref{eq:os1}. This allows us rewrite that expression as
\begin{eqnarray*}
&&|A^{[\phi],[z]}|^{-1}||Z_{A^{[\phi],[z]}}(b^{-1})|^{-1}\sum_\tau\chi_\tau(G(\dot s\cdot \dot s_b\rtimes b))
\\
&&\sum_{\substack{y \in A^{[\phi],[z]}\\ c \in A^{[z]}\\ cac^{-1} = yb^{-1}y^{-1} }}\chi_\tau(HF((t_c \rtimes c)(t\cdot t_a \rtimes a)(t_c \rtimes c)^{-1})) \nonumber
\end{eqnarray*}
and making the substitution $c \mapsto yc$ this equals
\begin{eqnarray*}
&&|A^{[\phi],[z]}|^{-1}||Z_{A^{[\phi],[z]}}(b^{-1})|^{-1}\sum_\tau\chi_\tau(G(\dot s\cdot \dot s_b\rtimes b))
\\
&&\sum_{\substack{y \in A^{[\phi],[z]}\\ c \in A^{[z]}\\ cac^{-1} = b^{-1} }}\chi_\tau(HF((t_{yc} \rtimes yc)(t\cdot t_a \rtimes a)(t_{yc} \rtimes yc)^{-1})) \nonumber
\end{eqnarray*}
Since the images of $t_c \rtimes c \in \tilde T_z(F)$ and $t_{yc} \rtimes yc \in \tilde T_z(F)$ in the quotient
\[ (T(\bar F) \rtimes A^{[\phi]})^z \lmod (T(\bar F) \rtimes A)^z \cong A^{[\phi],[z]}\lmod A^{[z]} \]
are equal, the character of $\tau\circ HF$ will remain unchanged if we replace $yc$ by $y$. Doing this leads to the expression
\begin{eqnarray*}
\Theta_\phi^{\dot s\cdot \dot s_b \rtimes b}(t\cdot t_a \rtimes a)&=&
|Z_{A^{[\phi],[z]}}(b^{-1})|^{-1}\sum_\tau\chi_\tau(G(\dot s\cdot \dot s_b\rtimes b))
\\
&&\sum_{\substack{c \in A^{[z]}\\ cac^{-1} = b^{-1} }}\chi_\tau(HF((t_c \rtimes c)(t\cdot t_a \rtimes a)(t_c \rtimes c)^{-1})).
\end{eqnarray*}
The same analysis as for \eqref{eq:os1} now leads to
\[ \resizebox{0.98\textwidth}{!}{$\Theta_\phi^{\dot s\cdot \dot s_b \rtimes b}(t\cdot t_a \rtimes a) = [z](\dot s)\sum_c[\phi]({^ct}\zeta(c,a)) h(cac^{-1}) |Z_{A^{[\phi],[z]}}(b^{-1})|^{-1}\sum_\tau \chi_{\bar\tau}(b)\chi_{\bar\tau}(cac^{-1})$}. \]
Since now $cac^{-1}=b^{-1}$ we can apply Lemma \ref{lem:orth} and, recalling the definition of $h(b)$ from Subsection \ref{sub:eiso}, we obtain
\begin{eqnarray*}
\Theta_\phi^{\dot s\cdot \dot s_b \rtimes b}(t\cdot t_a \rtimes a)&=&[z](\dot s)h(b^{-1})\bar\beta(b,b^{-1})\sum_c[\phi]({^ct}\zeta(c,a))\\
&=&[z](\dot s)h(b)^{-1}\bar\alpha(b,b^{-1})\sum_c[\phi]({^ct}\zeta(c,a))\\
&=&[z](\dot s)\left\< (z^{-1},t_{b^{-1}}),(\phi^{-1},\dot s_b)\right\>^{-1}\sum_c[\phi]({^ct}\zeta(c,a))
\end{eqnarray*}

\subsection{Computing the left-hand side of \eqref{eq:charid}}

We will now compute the lift to $\tilde T_z(F)$ of the virtual character $S\Theta_{\phi^\mf{e}}$. Recall that we have fixed an element $\dot{\tilde s} = \dot s \cdot \dot s_b \rtimes b \in \tilde S_\phi^{[z]}$ and $\mf{e}$ is the endoscopic triple for the twisted group $(G,b^{-1})$ corresponding to $\tilde s$ and $\phi$ and augmented by a choice of an $L$-embedding $\xi^\mf{e} : {^LG^\mf{e}} \rw {^LG}$ whose image contains the image of $\phi$. In our special case of $G=T$, we have $G^\mf{e}=T_{b^{-1}}$ and we can choose $\xi^\mf{e} : [\hat T]^{b,\circ} \rtimes W_F \rw \hat T \rtimes W_F$ to be given by $(t,w) \mapsto t \cdot \phi(w)$. With this choice, $\phi^\mf{e}(w)=1 \rtimes w$ and hence $S\Theta_{\phi^\mf{e}}$ is the trivial character of $T_{b^{-1}}(F)$. On the other hand, the function $f^\mf{\dot e}$ from Lemma \ref{lem:trans} is given by
\begin{eqnarray*}
f^\mf{\dot e}(\gamma^\mf{e})&=&\sum_{\tilde\delta \in \tilde T_z(F)/\tilde T_z(F)-conj} \Delta(\gamma^\mf{e},\tilde\delta) \int_{\tilde x \in \tilde T_z(F)/\tilde T_z(F)_{\tilde\delta}} f(\tilde x\tilde\delta \tilde x^{-1})d\tilde x\\
&=&\sum_{\tilde\delta \in \tilde T_z(F)/T_z(F)-conj} \Delta(\gamma^\mf{e},\tilde\delta) \int_{x \in T_z(F)/ T_z(F)_{\tilde\delta}} f(\tilde x\tilde \delta \tilde x^{-1})d\tilde x,\\
\end{eqnarray*}
where $\Delta$ is the transfer factor determined by $\mf{\dot e}$ and the fixed $L$-embedding (there is no Whittaker datum since we are dealing with tori). Thus the lift of $S\Theta_{\phi^\mf{e}}$ evaluated at $f$ is equal to
\begin{eqnarray*}
&&\int_{\gamma \in T_{b^{-1}}(F)} f^\mf{\dot e}(\gamma)d\gamma\\
&=&\int_\gamma \sum_{\tilde\delta \in \tilde T_z(F)/T_z(F)-conj} \Delta(\gamma,\tilde\delta) \int_{\tilde x \in T_z(F)/ T_z(F)_\delta} f(\tilde x\tilde \delta \tilde x^{-1})d\tilde x\\
&=&\int_\gamma \sum_{\tilde \delta \in \tilde T_z(F)/T_z(F)-conj} \sum_{c \in \tilde T_z(F)/T_z(F)} \Delta_{KS}(\gamma,c\tilde \delta c^{-1}) \int_{\tilde x \in T_z(F)/ T_z(F)_{\tilde \delta}} f(\tilde x\tilde \delta \tilde x^{-1})d\tilde x\\
\end{eqnarray*}
Recall that $\Delta_{KS}$ is supported in the variable $\tilde \delta$ on the coset $(G \rtimes b^{-1})_z(F)$. We obtain
\begin{eqnarray*}
&=&\int_\gamma \sum_{a \in A^{[z]}}\sum_{\tilde \delta \in [T \rtimes a]_z(F)/T_z(F)-conj} \sum_{\substack{c \in \tilde T_z(F)/T_z(F)\\ cac^{-1}=b^{-1}}} \Delta_{KS}(\gamma,c\tilde \delta c^{-1}) \int_{\tilde x \in T_z(F)/ T_z(F)_{\tilde \delta}} f(\tilde x\tilde \delta \tilde x^{-1})d\tilde x\\
\end{eqnarray*}
We interchange the integral over $\gamma$ with the sums over $a$ and $c$. Moreover, as $\gamma$ runs over $T_{b^{-1}}(F)$, $c^{-1}\gamma c$ runs over $T_a(F)$. We make the substitution $\gamma \mapsto c^{-1}\gamma c$ and arrive at
\begin{eqnarray*}
&=&\sum_{a \in A^{[z]}}\int_{\gamma \in T_a(F)} \sum_{\tilde \delta \in [T \rtimes a]_z(F)/T_z(F)-conj} \sum_{\substack{c \in \tilde T_z(F)/T_z(F)\\ cac^{-1}=b^{-1}}} \Delta_{KS}(c\gamma c^{-1},c\tilde \delta c^{-1}) \int_{\tilde x \in T_z(F)/ T_z(F)_{\tilde \delta}} f(\tilde x\tilde \delta \tilde x^{-1})d\tilde x\\
\end{eqnarray*}
Now $\Delta_{KS}(\gamma,\tilde \delta)$, in our special case of tori, is non-zero if and only if $\tilde \delta = \delta \rtimes b^{-1}$ and the image of $\delta$ in $T_{b}(F)$ equals $\gamma$. Thus the function
\[ \Phi(\tilde \delta)=\sum_{\substack{c \in \tilde T_z(F)/T_z(F)\\ cac^{-1}=b^{-1}}} \Delta_{KS}(c\gamma c^{-1},c\tilde \delta c^{-1}) \]
depends only on $\tilde \delta$, as $\gamma$ can be recovered from $\tilde \delta$. We arrive at the formula
\[
\sum_{a \in A^{[z]}}\int_{\gamma \in T_a(F)} \sum_{\substack{\tilde \delta \in [T \rtimes a]_z(F)/T_z(F)-conj\\ \delta \mapsto \gamma}} \Phi(\tilde \delta) \int_{\tilde x \in T_z(F)/ T_z(F)_\delta} f(\tilde x\delta \tilde x^{-1})d\tilde x
\]
Since $\Phi(\tilde \delta)$ is conjugation-invariant under $T_z(F)$, we obtain
\[
\sum_{a \in A^{[z]}}\int_{\gamma \in T_a(F)} \sum_{\substack{\tilde \delta \in [T \rtimes a]_z(F)/T_z(F)-conj\\ \delta \mapsto \gamma}} \int_{\tilde x \in T_z(F)/ T_z(F)_\delta} \Phi(\tilde x\tilde \delta \tilde x^{-1})  f(\tilde x\tilde \delta \tilde x^{-1})d\tilde x
\]
A simple integration formula now shows that this is equal to
\[ \int_{\tilde \delta \in \tilde T_z(F)} \Phi(\tilde \delta)f(\tilde \delta). \]
We conclude that the lift of $S\Theta_{\phi^\mf{e}}$ to $\tilde T_z(F)$ is represented by the function $\Phi$. We have
\begin{eqnarray*}
\Phi(t \cdot t_a \rtimes a)&=&\sum_{\substack{c \in \tilde T_z(F)/T_z(F)\\ cac^{-1}=b^{-1}}} \Delta_{KS}(c\gamma c^{-1},c(t \cdot t_a \rtimes a) c^{-1})\\
&=&\sum_{\substack{c \in \tilde T_z(F)/T_z(F)\\ cac^{-1}=b^{-1}}} \Delta_{KS}(c\gamma c^{-1},(t_c \rtimes c)(t \cdot t_a \rtimes a) (t_c \rtimes c)^{-1})\\
&=&\sum_{\substack{c \in \tilde T_z(F)/T_z(F)\\ cac^{-1}=b^{-1}}} \Delta_{KS}(c\gamma c^{-1},{^ct}\zeta(c,a)t_{cac^{-1}} \rtimes cac^{-1})\\
&=&\sum_{\substack{c \in A^{[z]} \\ cac^{-1}=b^{-1}}} \<(z^{-1},{^ct}\zeta(c,a)t_{b^{-1}}),(\phi_0^{-1},\dot s\dot s_b)\>^{-1}\\
&=&\<(z^{-1},t_{b^{-1}}),(\phi_0^{-1},\dot s\dot s_b)\>^{-1}\sum_{\substack{c \in A^{[z]} \\ cac^{-1}=b^{-1}}} [\phi]({^ct}\zeta(c,a)).\\
\end{eqnarray*}
The final expression is equal to the formula for $\Theta_\phi^{\dot s\cdot \dot s_b \rtimes b}(t\cdot t_a \rtimes a)$ obtained in the previous section. The proof of Conjecture \ref{cnj:llc_pure_ci} in the case of tori is now complete.

\begin{appendices}

\addtocontents{toc}{
\setlength{\cftbeforesecskip}{\cftbeforesubsecskip}
\setlength{\cftsecindent}{\cftsubsecindent}
\protect\renewcommand{\cftsecfont}{\cftsubsecfont}
\protect\renewcommand{\protect\cftsecdotsep}{\cftsubsecdotsep}
}

\section{Functoriality of the local correspondence for connected groups} \label{app:func}

Let $\phi : L_F \to {^LG}$ be a tempered Langlands parameter. As in \cite[\S5.4]{KalRI} and \cite[\S4.1]{KalRIBG} we expect to have a compound $L$-packet $\Pi_\phi$ and a commutative diagram
\[ \xymatrix{
	\Pi_\phi\ar[r]\ar[d]&\tx{Irr}(S_\phi^+)\ar[d]\\
	H^1(u \to W,Z(G) \to G)\ar[r]&\pi_0(Z(\hat{\bar G})^+)^*
}
\]
Recall here that $\Pi_\phi$ is a subset of the set $\Pi_\tx{temp}$ of tempered representations of rigid inner twists, that consists of tuples $(G_z,\xi,z,\pi)$, where $\xi : G \to G_z$ is an inner twist, $z \in Z^1(u \to W,Z(G) \to G)$ is such that $\xi^{-1}\sigma(\xi) = \tx{Ad}(\bar z(\sigma))$, where $\bar z \in Z^1(\Gamma,G_\tx{ad})$ is the image of $z$ under the natural projection $G \to G_\tx{ad}$, and $\pi$ is an irreducible tempered representation of $G_z(F)$.

The group $A$ acts on $Z^1(u \to W,Z(G) \to G)$ by $a(z)(w) = a(z(w))$. Given rigid inner twists $(\xi_i,z_i) : G \to G_i$  for $i=1,2$ and $a \in A$ such that $z_2=a(z_1)$ one checks that the isomorphism $b := \xi_2 \circ a \circ \xi_1^{-1} : G_1 \to G_2$ is defined over $F$. More generally, if $a(z_1)$ and $z_2$ are cohomologous and one chooses $h \in G$ with $z(w)=h^{-1}a(z(w))\sigma_w(h)$, then $b := \xi_2 \circ \tx{Ad}(h) \circ a \circ \xi^{-1}$ is defined over $F$. A different choice of $h$ will change $b$ only by an inner automorphism coming from $G_1(F)$.

Seen from a slightly different perspective, this can be formulated as an action of $A$ on the category of rigid inner twists of $G$, namely $a(\xi,z) = (\xi\circ a^{-1},a(z))$. This action can be upgraded to an action of $A$ on the set $\Pi_\tx{temp}$ by $a(G_z,\xi,z,\pi)=(G_z,\xi\circ a^{-1},a(z),\pi)$.

Consider now the dual side. Given a tempered Langlands parameter $\phi : L_F \to {^LG}$ and $\rho \in \tx{Irr}(S_\phi^+)$ we obtain $a\phi := a\circ \phi : L_F \to {^LG}$ and $a\rho := \rho \circ a^{-1} \in \tx{Irr}(S_{a\phi}^+)$. Thus $A$ acts on the space of refined Langlands parameters.

It is reasonable to expect that the above commutative diagram is natural with respect to this action. More precisely:

\begin{cnj} \label{cnj:func}
If $\dot \pi \in \Pi_\tx{temp}$ corresponds to $(\phi,\rho)$, then $a\dot\pi$ corresponds to $(a\phi,a\rho)$. Formulated equivalently, if $(G_1,\xi_1,z_1,\pi_1)$ and $(G_2,\xi_2,z_2,\pi_2)$ correspond to $(\phi,\rho)$ and $(a\phi,a\rho)$ respectively, then the isomorphism $b : G_1 \to G_2$ constructed above identifies $\pi_1$ with $\pi_2$.	
\end{cnj}

In the special case of a rigid inner twist $(G_z,\xi,z)$ for which the cohomology class of $z$ is fixed by $a$, in particular in the case $z=1$ where $G_z=G$, this amounts to a compatibility with automorphisms of the refined local Langlands correspondence for the group $G_z$. However, the above statement applies even to inner forms of $G$ which do not admit $a$ as an automorphism defined over $F$.

\section{Automorphisms of Weil-restricted groups} \label{app:weil}

Let $E/F$ be a finite extension, $\Delta=\tx{Gal}(\bar F/E) \subset \Gamma = \tx{Gal}(\bar F/F)$. Let $G$ be an absolutely simple connected reductive $E$-group. Let $a$ be an automorphism of $H=\tx{Res}_{E/F} G$. Recall the natural identification $H(\bar F)=\tx{Ind}_\Delta^\Gamma G(\bar F)$. For every $\sigma \in \Gamma$ let $^\sigma E$ be the subfield $\sigma(E)$ of $\bar F$ and let $G^\sigma$ be the $^\sigma E$-group obtained by twisting the rational structure, i.e. $G^\sigma = G \times_{\tx{Spec}(E)} \tx{Spec}(^\sigma E)$, where we have used the map $\sigma : E \to {^\sigma E}$. 

\begin{lem} \label{lem:weilauto}
There exists $\sigma_0 \in N_\Gamma(\Delta)$ and an isomorphism $a' : G \to G^{\sigma_0}$ such that 
\[ a(f)(\sigma) = a'(f(\sigma_0^{-1} \sigma)),\qquad \forall f \in H(\bar F)=\tx{Ind}_\Delta^\Gamma G(\bar F),\ \ \forall \sigma \in \Gamma.\]
The $\Delta$-coset of $\sigma_0$ is unique and $a'$ is uniquely determined by the choice of $\sigma_0$ within its $\Delta$-coset. If $\sigma_0$ is replaced by $\tau\sigma_0$ with $\tau \in \Delta$ then $a'$ is replaced by $\tau\circ a'$.
\end{lem}
\begin{proof}
Choose a set of representatives $\sigma_1,\dots,\sigma_n$ for $\Delta \lmod \Gamma$ and arrange $\sigma_1=1$. Then $f \mapsto (f(\sigma_1),\dots,f(\sigma_n))$ is an isomorphism $H(\bar F) \to \prod_{i=1}^n G(\bar F)$ of algebraic groups. It translates the automorphism $a$ to an automorphism of $\prod_{i=1}^n G(\bar F)$. Such an automorphism must map each factor in the product to another factor. In this way we obtain a permutation $p$ of the set $\Delta \lmod \Gamma$ which has the property that if $f \in H(\bar F)$ is a function supported on the coset $\Delta \sigma$, then $a(f)$ is a function supported on the coset $\Delta p(\sigma)$. Since $a$ is an $F$-automorphism, the permutation $p$ is $\Gamma$-equivariant, i.e. $p(\sigma\gamma)=p(\sigma)\gamma$. It follows that there exists $\sigma_0 \in N_\Gamma(\Delta)$ such that $p(\gamma)=\sigma_0 \gamma$, and the $\Delta$-coset of $\sigma_0$ is unique.

Given $g \in G(\bar F)$ and $\sigma \in \Gamma$ let $g^{\delta_\sigma} \in H(\bar F)$ be the unique function supported on $\Delta\sigma$ and with value $g$ at $\sigma$. Define the $\bar F$-automorphism $a'$ of $G$ by $a'(g)=a(g^{\delta_1})(\sigma_0)$. One checks immediately that $a'(\tau g)=\sigma_0 \tau \sigma_0^{-1} a'(g)$, so that $a'$ is in fact an isomorphism of $E$-groups $G \to G^{\sigma_0}$. The equality $a(f)(\sigma)=a'(f(\sigma_0^{-1}\sigma))$ can be checked on functions $f$ of the form $g^{\delta_\gamma}$ for arbitrary $g \in G(\bar F)$ and $\gamma \in \Gamma$. We compute that $a(g^{\delta_\gamma})(\sigma)$ equals
\[ \resizebox{0.98\textwidth}{!}{$ a(\gamma^{-1}(g^{\delta_1}))(\sigma)=a(g^{\delta_1})(\sigma\gamma^{-1})=\sigma\gamma^{-1}\sigma_0^{-1}a'(g)=a'(\sigma_0^{-1}\gamma^{-1}\sigma g)=a'(g^{\delta_\gamma}(\sigma_0^{-1}\sigma))$}, \]
provided $\sigma_0^{-1}\gamma^{-1}\sigma \in \Delta$, and that $a(g^{\delta_\gamma})=1=a'(g^{\delta_\gamma}(\sigma_0^{-1}\sigma))$ otherwise.
\end{proof}

\section{Orthogonality relations for projective characters} \label{app:projchar}

We have now constructed the bijection \eqref{eq:cnj1tori}. Our next goal is to show that with this bijection the character identities \eqref{eq:charid} hold. In this section we will prove a lemma that will be needed for the evaluation of the right hand side of \eqref{eq:charid}. It is a refinement of the orthogonality relation
\[ \sum_{\tau\in \tx{Irr}(A)} \overline{\chi_\tau(a)}\chi_\tau(b) = \begin{cases} |Z_A(a)|,& b \in C_A(a)\\ 0,&\tx{else} \end{cases} \]
for the characters of the irreducible representations of a finite group $A$. Here $Z_A(a)$ and $C_A(a)$ are the centralizer and the conjugacy class of $a$ in $A$.

The refinement we need is the following. Consider a central extension
\[ 1 \rw Z \rw E \rw A \rw 1. \]
We assume that $A$ is finite. For $e \in E$ we will write $\bar e$ for its image in $A$. Let $\psi : Z \rw \C^\times$ be a character, and let $\tx{Irr}(E,\psi)$ be the set of all irreducible representations of $E$ whose central character restricted to $Z$ equals $\psi$. This is a finite set and each element of it is finite dimensional. For each $\tau \in \tx{Irr}(E,\psi)$ its character $\chi_\tau$ is a class function on $E$ that satisfies $\chi_\tau(ze)=\psi(z)\chi_\tau(e)$. This implies that if the images of $e,e' \in E$ in $A$ commute, so that $ee'e^{-1}e'^{-1} \in Z$, but $\psi(ee'e^{-1}e'^{-1}) \neq 1$, then $\chi_\tau(e')=0$. We will say that $e' \in E$ is $\psi$-centralizing, if for all $e \in E$ such that $ee'e^{-1}e'^{-1} \in Z$ we have $\psi(ee'e^{-1}e'^{-1})=1$. Note that the notion if being $\psi$-centralizing is invariant under conjugation as well as under translation by $Z$. The refinement of the orthogonality relations we need is the following.
\begin{lem} \label{lem:orth} Assume that $e \in E$ is $\psi$-centralizing. Then
\[ \sum_{\tau\in \tx{Irr}(E,\psi)} \chi_\tau(e)\chi_\tau(e') = \begin{cases} |Z_A(e)|\psi(ee'),& ee' \in Z\\ 0,&\bar e^{-1} \notin C_A(\bar e') \end{cases} \]
\end{lem}
\begin{proof}
We will make use of the character theory of projective representations of finite groups, for an exposition of which we refer the reader to \cite{Che13}. We first form the push-out
\[ \xymatrix{
1\ar[r]&Z\ar[r]\ar[d]^{\psi}&E\ar[r]\ar[d]&A\ar[r]\ar@{=}[d]&1\\
1\ar[r]&\C^\times\ar[r]&E_\psi\ar[r]&A\ar[r]&1
} \]
Inflation provides a bijection $\tx{Irr}(E_\psi,\tx{id}) \cong \tx{Irr}(E,\psi)$ that preserves characters. Moreover, $e' \in E$ is $\psi$-centralizing if and only if its image in $E_\psi$ is $\tx{id}$\-centralizing. This reduces the problem to $Z=\C^\times$ and $\psi=\tx{id}$. Next we fix a set-theoretic splitting $s : A \rw E$ such that all values of the corresponding $2$-cocycle $\alpha(a,b)=s(a)s(b)s(ab)^{-1}$ are complex roots of unity, see \cite[Lemma 3.1]{Che13}. For each $\tau \in \tx{Irr}(E_\psi,\tx{id})$ set $\bar\tau = \tau \circ s$. Then $\bar\tau$ is a projective representation of $A$ with cocycle $\alpha$ and the map $\tau \mapsto \tau\circ s$ is a bijection between $\tx{Irr}(E_\psi,\tx{id})$ and the isomorphism classes of projective representations of $A$ with cocycle $\alpha$. Let $f$ be the $\alpha$-class function \cite[Definition 3.13]{Che13} on $A$ supported on the $A$-conjugacy class of $\bar e^{-1}$ and having the property with $f(\bar e^{-1})=1$. This class function exists because $\bar e^{-1}$ is an $\alpha$-element. According to \cite[Theorem 3.15]{Che13}, we have
\[ f = \sum_{\bar\tau} \<f,\chi_{\bar\tau}\>\chi_{\bar\tau}. \]
Since both $f$ and $\chi_{\bar\tau}$ are $\alpha$-class functions and $\alpha$ is unitary, the product $f \cdot \ol{\chi_{\bar\tau}}$ is a 1-class function (i.e. an honest class function) and one sees
\[ \<f,\chi_{\bar\tau}\>=|A|^{-1}\sum_{a \in A}f(a)\ol{\chi_{\bar\tau}(a)} = |Z_A(\bar e^{-1})|^{-1}\ol{\chi_{\bar\tau}(\bar e^{-1})}. \]
We thus obtain
\[ |Z_A(\bar e)|f(\bar e')=\sum_\tau \ol{\chi_{\bar\tau}(\bar e^{-1})} \chi_{\bar\tau}(\bar e') \]
and then further
\begin{eqnarray*}
\sum_\tau \chi_\tau(e) \chi_\tau(e')&=&\sum_\tau \tx{tr}(\tau(e^{-1})^{-1}) \tx{tr}(\tau(e'))\\
&=&\sum_\tau \tx{tr}((z_{e^{-1}}\bar\tau(\bar e^{-1}))^{-1}) \tx{tr}(\tau(e'))\\
&=&z_{e^{-1}}^{-1}z_{e'}\sum_\tau \ol{\chi_{\bar\tau}(\bar e^{-1})}\chi_{\bar\tau}(\bar e')\\
&=&ee'\sum_\tau\ol{\chi_{\bar\tau}(\bar e^{-1})}\chi_{\bar\tau}(\bar e')\\
&=&ee'|Z_A(\bar e)|f(\bar e')
\end{eqnarray*}
We have used in this computation that the projective representation $\bar\tau$ is unitarizable, which is a consequence of our choice of $s$. The lemma follows.
\end{proof}

\section{Folding of root systems} \label{app:fold}

\begin{lem} \label{lem:rootsupp}
	Let $\Phi$ be a simply laced irreducible root system with base $\Delta$. For each positive root $\alpha$, let $f_\alpha : \Delta \to \Z_{\geq 0}$ be determined by $\alpha=\sum_{\gamma \in \Delta} f_\alpha(\gamma) \cdot \gamma$. 
	\begin{enumerate}
		\item The map 
		\[ \Phi^+ \to \{f : \Delta \to \Z_{\geq 0}\},\qquad \alpha \mapsto f_\alpha, \] 
		is injective and equivariant for all automorphisms of the pair $(\Phi,\Delta)$.
		\item For every $f : \Delta \to \Z_{\geq 0}$ define $\tx{supp}(f)=\{\alpha \in \Delta|f(\alpha) \neq 0\}$. Then $\tx{supp}(f_\alpha)$ is a connected subdiagram of the Dynkin diagram of $(\Phi,\Delta)$.
		\item Every $f : \Delta \to \{0,1\}$ with connected support is equal to $f_\alpha$ for some $\alpha \in \Phi^+$.
	\end{enumerate}
\end{lem}
\begin{proof}
	(1) is clear and (2),(3) are \cite[Lemma 2.6.14]{BTBOOK}.
\end{proof}

\begin{rem} \label{rem:rootsupp}
	In order to fully describe the image of the map $\alpha \mapsto f_\alpha$, it remains to describe those $f : \Delta \to \Z_{\geq 0}$ in the image that take at least one value greater than $1$. For the root systems of type $E$ these functions are tabulated in the Plates of \cite[Chap. VI]{BourLie4-6}. From the formulas for positive roots in terms of simple roots in those plates we also infer that there are no such maps in type $A_n$, and in type $D_n$ they are given by
	\[ \begin{dynkinDiagram}[affine mark=*,edge length=.7cm]{D}{*.**.***}
		\node[below=0cm] at (root 1) {\tiny $1$}; 
		\node[below=0cm] at (root 2) {\tiny $1$}; 
		\node[below=0cm] at (root 3) {\tiny $2$}; 
		\node[below=0cm] at (root 4) {\tiny $2$}; 
		\node[below=0cm] at (root 5) {\tiny $1$}; 		
		\node[below=0cm] at (root 6) {\tiny $1$}; 		
		\end{dynkinDiagram}\\
	\]
	where $|\{\gamma:f(\gamma)=1\}|>2$ and $|\{\gamma:f(\gamma)=2\}|>0$.
\end{rem}

The following elementary lemma is well-known. Various versions have been proved in \cite[\S1]{Ste68end} and \cite[\S3]{Haines18}. We record here a statement and proof that is convenient for our purposes.

\begin{pro} \label{pro:fold1}
	Let $V$ be a finite-dimensional real vector space, $\Phi \subset V$ a root system generating $V$, $\Phi^\vee \subset V^*$ the dual root system, $\Delta \subset \Phi$ a base, $\Delta^\vee \subset \Phi^\vee$ its dual, $\Theta$ a finite group acting on $V$ by automorphisms that preserve $\Delta$, hence also $\Phi$.
	\begin{enumerate}
		\item For $\alpha \in \Phi$ let $\alpha_\Theta$ denote the image of $\alpha$ in the space $V_\Theta$ of $\Theta$-coinvariants. Then $\Phi_\Theta=\{\alpha_\Theta|\alpha \in \Phi\}$ is a (possibly non-reduced) root system in $V_\Theta$, and $\Delta_\Theta=\{\alpha_\Theta|\alpha \in \Delta\}$ is a base.
		\item For $\alpha^\vee \in \Phi^\vee$ let $(\alpha^\vee)^\Theta$ denote the sum of the elements of the $\Theta$-orbit of $\alpha^\vee$. Then $(\Phi^\vee)^\Theta=\{(\alpha^\vee)^\Theta| \alpha^\vee \in \Theta\}$ is a reduced root system in $(V^*)^\Theta$ and $(\Delta^\vee)^\Theta=\{(\alpha^\vee)^\Theta|\alpha^\vee \in \Delta^\vee\}$ is a base.
		\item Under the duality between $V_\Theta$ and $(V^*)^\Theta$, the root system $(\Phi^\vee)^\Theta$ is dual to the subsystem of non-multipliable roots in $\Phi_\Theta$. 
  		\item More precisely, let $\alpha \in \Phi$ with image $\alpha_\Theta \in \Phi_\Theta$, and let $(\alpha_\Theta)^\vee \in (V^*)^\Theta$ be the unique element of Axiom $\tx{RS}_\tx{II}$ in \cite[Chap. VI, \S1, no. 1, Definition 1]{BourLie4-6}. Then one of the following mutually exclusive cases holds.
		  \begin{enumerate}
			  \item No two roots in the $\Theta$-orbit of $\alpha$ sum up to a root, $\alpha_\Theta$ is non-multipliable, and $(\alpha_\Theta)^\vee=(\alpha^\vee)^\Theta$.
     			\item The $\Theta$-orbit of $\alpha$ contains a root $\alpha' \neq \alpha$ such that $\alpha'+\alpha \in \Phi$, $\alpha_\Theta$ is multipliable, $2\alpha_\Theta=(\alpha+\alpha')_\Theta$, and $(\alpha_\Theta)^\vee=2(\alpha^\vee)^\Theta$. The root $\alpha'$ is unique with these properties.
		  \end{enumerate} 
  		\item The set of irreducible components of $\Phi_\Theta$ is in bijection with the set of orbits for the action of $\Theta$ on the set of irreducible components of $\Phi$. 
		\item A component of $\Phi_\Theta$ is non-reduced if and only if the corresponding $\Theta$-orbit of components of $\Phi$ consists of components of type $A_{2n}$, and the stabilizer of one (hence any) such component acts non-trivially on it.
  		\item The fibers of the projection $\Phi \to \Phi_\Theta$ are the orbits of $\Theta$.
        \item The action of $W^\Theta$ on $V_\Theta$ and $(V^*)^\Theta$ is faithful and identifies $W^\Theta$ with the Weyl group of the root systems $\Phi_\Theta$ and $(\Phi^\vee)^\Theta$.
        \item Let $\alpha \in \Phi$ with image $\alpha_\Theta \in \Phi_\Theta$ and let $s_{\alpha_\Theta} \in W^\Theta$ be the reflection corresponding to $\alpha_\Theta$. When $\alpha_\Theta$ is not multipliable, the $s_{\alpha_\Theta}$ is equal to the product of the pairwise commuting reflections $s_\beta$, where $\beta$ runs over the $\Theta$-orbit of $\alpha$. When $\alpha_\Theta$ is multipliable, $s_{\alpha_\Theta}$ is equal to $s_\beta$, where $\beta=\alpha+\alpha'$ with $\alpha'$ as in (4b).
        \item The Dynkin diagram of $\Phi_\Theta$ with respect to $\Delta_\Theta$ is given by the following ``folding'' procedure. Two elements $\alpha_\Theta,\beta_\Theta \in \Delta_\Theta$ are connected by a bond if and only if there exist lifts $\alpha,\beta \in \Delta$ that are connected by a bond. That bond is multiple in the following two situations:
        \begin{enumerate}
			\item Up to switching $\alpha_\Theta,\beta_\Theta$, there exists $\alpha \mapsto \alpha_\Theta$ that is connected to more than one $\beta \mapsto \beta_\Theta$ in the Dynkin diagram of $\Phi$. The strength of the bond is the number of $\beta \mapsto \beta_\Theta$ connected to $\alpha$, and an arrow is placed pointing towards $\beta_\Theta$.
   			\item Up to switching $\alpha_\Theta,\beta_\Theta$, there exist $\alpha_1,\alpha_2 \mapsto \alpha_\Theta$ that are connected to each other in the Dynkin diagram of $\Phi$. Then $\alpha_\Theta$ is multipliable, and a double bond with no arrow is placed between $\alpha_\Theta,\beta_\Theta$ to indicate (by convention) a component of type $BC_n$.
		\end{enumerate}
		\item Every irreducible root system arises as $\Phi_\Theta$ from some irreducible $\Phi$.
	\end{enumerate}
\end{pro}
\begin{proof}
	Decompose $\Phi$ into irreducible components. These are permuted by the action of $\Theta$. All statements are compatible with products, so we may assume that the action of $\Theta$ on the set of irreducible components of $\Phi$ is transitive. Then all components are of the same Dynkin type. Let $\Phi'$ be one such component, $\Delta'=\Phi' \cap \Delta$,  and $\Theta'$ its stabilizer in $\Theta$. Let $V'$ be the subspace of $V$ generated by $\Theta'$. The inclusion $V' \to V$ induces an isomorphism $(V')_{\Theta'} \to V_\Theta$, which identifies $(\Delta')_{\Theta'}$ with $\Delta_\Theta$, $(\Phi')_{\Theta'}$ with $\Phi_\Theta$, and $(W')^{\Theta'}$ with $W^\Theta$. Since $V'$ is a direct factor of $V$ we have the inclusion $(V')^* \to V^*$, which identifies $(V^*)^\Theta$ with $((V')^*)^{\Theta'}$, $(\Phi^\vee)^\Theta$ with $((\Phi')^\vee)^{\Theta'}$, and $(\Delta^\vee)^\Theta$ with $((\Delta')^\vee)^{\Theta'}$. This reduces the proof to the case that $\Phi$ is irreducible.

	If $\Phi$ is not simply laced, then $\Theta$ acts trivially on $\Phi$ and all statements are trivial. From now on we assume that $\Phi$ is simply laced and irreducible. We are also free to replace $\Theta$ by the quotient through which it acts effectively on $V$. Then $\Theta$ is cyclic, except when $\Phi$ is of type $D_4$, when $\Theta$ could by the symmetric group on $3$ letters. In the latter case, the unique subgroup of $\Theta$ of index $2$ has the same set of orbits in $\Phi$ and $\Phi^\vee$ as $\Theta$, so we may replace $\Theta$ by $\Theta'$ without changing the statements.

	This reduces the proof to the following core case: $\Phi$ is irreducible and simply laced, in particular reduced, and $\Theta$ is cyclic. Let $\theta \in \Theta$ be a generator.

	Since $\Delta$ forms a basis of the vector space $V$ that is stable under $\Theta$, the image $\Delta_\Theta$ of $\Delta$ in $V_\Theta$ is a basis for $V_\Theta$ and the projection map $\Delta \to \Delta_\Theta$ has fibers given by the orbits of $\Theta$ in $\Delta$. In particular, no element of $\Delta$ maps to zero in $V_\Theta$. By considering positive and negative roots in $\Phi$, we conclude that no element of $\Phi$ maps to zero in $V_\Theta$. Thus $\Phi_\Theta$ does not contain zero. It generates $V_\Theta$, because $\Phi$ generates $V$. This proves Axiom $\tx{RS}_{\tx{I}}$ of \cite[Chap. VI, \S1, no. 1, Definition 1]{BourLie4-6}. 

	Consider $\alpha \in \Phi$ and assume that its $\Theta$-orbit consists of pairwise orthogonal roots. We compute 
	\[ \<\alpha_\Theta,(\alpha^\vee)^\Theta\> = \sum_i \<\alpha,\alpha_i^\vee\> = \<\alpha,\alpha^\vee\> = 2, \]
	where the sum runs over the $\Theta$-orbit $\{\alpha_1,\dots,\alpha_n\}$ of $\alpha$. Write $s$ for the reflection corresponding to the pair $(\alpha_\Theta,(\alpha^\vee)^\Theta)$, i.e. the map
	\[ V_\Theta \to V_\Theta, v_\Theta \mapsto v_\Theta - \<v_\Theta,(\alpha^\vee)^\Theta\>\alpha_\Theta. \] 
	For any $\beta \in \Phi$ we have 
	\begin{eqnarray*}
	s(\beta_\Theta)&=&\Big[\beta-\sum_{i=1}^n \<\beta,\alpha_i^\vee\>\alpha_i\Big]_\Theta\\
	&=&\Big[ \beta-\sum_{i=1}^n \<s_{\alpha_1}\circ \dots \circ s_{\alpha_{i-1}}(\beta),\alpha_i^\vee\>\alpha_i \Big]_\Theta\\
	&=&[s_{\alpha_1}\circ\dots\circ s_{\alpha_n}(\beta)]_\Theta,
	\end{eqnarray*}
	where we have denoted by $[-]_\Theta$ the projection $V \to V_\Theta$ and have used $\<s_{\alpha_1}\circ \dots \circ s_{\alpha_{i-1}}(\beta),\alpha_i^\vee\>=\<\beta,\alpha_i^\vee\>$ due to the pairwise orthogonality of $\{\alpha_1,\dots,\alpha_n\}$. This shows that $s(\beta_\Theta) \in \Phi_\Theta$, which is Axiom $\tx{RS}_{\tx{II}}$, and identifies $(\alpha_\Theta)^\vee=(\alpha^\vee)^\Theta$. Moreover
	\[ \<\beta_\Theta,(\alpha^\vee)^\Theta\> = \sum_i \<\beta,\alpha_i^\vee\> \in \Z, \]
	hence $\tx{RS}_{\tx{III}}$.
	
	Consider now $\alpha \in \Phi$ whose $\Theta$-orbit has $\alpha' \neq \alpha$ so that $\alpha+\alpha' \in \Phi$. We claim that $\Phi$ is of type $A_{2n}$ and $\{\alpha,\alpha'\}$ is the $\theta$-orbit of $\alpha$. For this we use Lemma \ref{lem:rootsupp} and Remark \ref{rem:rootsupp}. Since $f_{\alpha+\alpha'}=f_{\alpha}+\theta(f_{\alpha})$, we see by inspecting the $\theta$-stable maps that the supports of $f_{\alpha}$ and $\theta(f_\alpha)$ must be disjoint, yet their union must be connected, and this is only possible when $\Phi$ is of type $A_{2n}$ and $\theta$ acts non-trivially. We now compute
	\[ \<\alpha_\Theta,(\alpha^\vee)^\Theta\> = \<\alpha,\alpha^\vee\>+\<\alpha,\alpha'^\vee\>=2-1=1. \]
	Let $s$ be the reflection corresponding to the pair $(\alpha_\Theta,2(\alpha^\vee)^\Theta)$. For any $\beta \in \Phi$ we have 
	\[ s(\beta_\Theta)=[\beta-2\<\beta,\alpha^\vee+\alpha'^\vee\>\alpha]_\Theta = [\beta-\<\beta,\alpha^\vee+\alpha'^\vee\>(\alpha+\alpha')]_\Theta = [s_{\alpha+\alpha'}(\beta)]_\Theta.\]
	Moreover, 
	\[ \<\beta_\Theta,(\alpha^\vee)^\Theta\> = 2\<\beta,\alpha^\vee\>+2\<\beta,\alpha'^\vee\> \in \Z, \]
	hence $\tx{RS}_\tx{III}$.

	We have thus proved the axioms of a root system for $\Phi_\Theta$. We have also argued that $\Delta_\Theta$ is a basis of $V_\Theta$. Since every element of $\Phi_\Theta$ is an integer combination of elements of $\Delta_\Theta$ with the same sign, (1) has been proved.

	Next, consider (7). Let $\alpha_1,\alpha_2 \in \Phi$ map to the same element $\bar\alpha \in \Phi_\Theta$. Then they are either both positive or both negative, assume wlog positive. We already know that the fibers of $\Delta \to \Delta_\Theta$ are precisely the $\Theta$-orbits. In particular, $\alpha_1,\alpha_2$ have the same height, which also equals the height of $\bar\alpha$. Therefore, if $\bar\alpha \in \Delta_\Theta$, then $\alpha_1,\alpha_2 \in \Delta$  hence lie in the same $\theta$-orbit. In our computations above we showed that all root reflections in $W(\Phi_\Theta)$ lie in $W^\Theta$, hence $W(\Phi_\Theta) \subset W^\Theta$. Let $w \in W(\Phi_\Theta)$ be such that $w\bar\alpha \in \Delta_\Theta \cup 2\Delta_\Theta$. If $w\bar\alpha \in \Delta_\Theta$, by the above argument $w\alpha_1,w\alpha_2$ lie in the same $\theta$-orbit, but then so do $\alpha_1,\alpha_2$ since $w \in W^\Theta$. If $w\bar\alpha \in 2\Delta_\Theta$, then $w\alpha_1$ (and also $w\alpha_2$) is of height $2$, more precisely the sum of two elements of $\Delta$ which map to the same element of $\Delta_\Theta$. A visual inspection of the Dynkin diagrams shows that this only happens in type $A_{2n}$ with $\theta$ acting non-trivially, in which case these two elements of $\Delta$ are unique (up to transposition). We conclude that $w\alpha_1$ is $\theta$-fixed and equals $w\alpha_2$. 
	
	This completes the proof of (7), but (6), as well as the part of (4) that identifies which members of $\Phi_\Theta$ are multipliable. The remainder of (4) was completed in the course of proving (1), hence (4) is proved, which in turn implies (2) and (3).

	In the course of proving (1) we also proved (9) and thus know that the Weyl group of $\Phi_\Theta$ is a subgroup of $W^\Theta$. On the other hand, the action of $W^\Theta$ on $V$ preserves induces an action on $V_\Theta$ which is faithful (because the action on $V^\Theta$ is faithful and the projection $V^\Theta \to V_\Theta$ is an isomorphism equivariant for this action). Since $W^\Theta$ preserves $\Phi$, it also preserves $\Phi_\Theta$. Thus $W^\Theta$ contains the Weyl group of $\Phi_\Theta$ and is contained in the automorphism group of $\Phi_\Theta$. Since the preimage of $\Delta_\Theta$ under $\Phi \to \Phi_\Theta$ is $\Delta$, any member of $W^\Theta$ that stabilizes $\Delta_\Theta$ also stabilizes $\Delta$ and is hence trivial. This proves (8).

	Consider (10). Let $\alpha_\Theta,\beta_\Theta \in \Delta_\Theta$ be distinct. According to \cite[Chap. VI, \S1, no. 3]{BourLie4-6}, the two numbers $\<\alpha_\Theta,(\beta_\Theta)^\vee\>$ and $\<\beta_\Theta,(\alpha_\Theta)^\vee\>$ are non-positive and, after possibly switching $\alpha,\beta$ so that $|\<\alpha_\Theta,(\beta_\Theta)^\vee\>| \geq |\<\beta_\Theta,(\alpha_\Theta)^\vee\>|$, the strength of the bond between $\alpha_\Theta$ and $\beta_\Theta$ equals $|\<\alpha_\Theta,(\beta_\Theta)^\vee\>|$, which is one of $0,1,2,3$, with $0$ indicating no bond. Using (4) we see
	\[ \<\alpha_\Theta,(\beta_\Theta)^\vee\> =n_\beta\sum_\beta \<\alpha,\beta^\vee\> \]
	for some $n_\beta \in \{1,2\}$, where the sum runs over all $\beta \mapsto \beta_\Theta$ while $\alpha \mapsto \alpha_\Theta$ is arbitrary. This number is non-zero precisely when there exist $\alpha,\beta$ with $\<\alpha,\beta^\vee\> \neq 0$, i.e. whose nodes in the Dynkin diagram of $\Phi$ are linked.

	Next we consider the strength of the bond. Recall that $\Phi$ is simply laced, because $\theta$ acts non-trivially. Therefore each $\<\alpha,\beta^\vee\>$  is either $0$ or $-1$, the latter precisely when $\alpha,\beta$ are linked in the Dynkin diagram of $\Phi$. Since $|\<\alpha_\Theta,(\beta_\Theta)^\vee\>|$ is constrained to be one of $0,1,2,3$, it is bigger then $1$ in the following two mutually exclusive cases:
	\begin{enumerate}
		\item $n_\beta=2$ and there is a unique $\beta \mapsto \beta_\Theta$ linked to any given $\alpha \mapsto \alpha_\Theta$. This is the case if and only if $\beta_\Theta$ is multipliable, according to (4b), in which case $\Phi$ is of type $A_{2n}$ by (6) and the strength of the bond is thus $2$.
  		\item $n_\beta=1$ and there are more than one $\beta \mapsto \beta_\Theta$ linked to any given $\alpha \mapsto \alpha_\Theta$. The strength of the bond is then the number of such $\beta$, and $\beta_\Theta$ is short, while $\alpha_\Theta$ is long.
	\end{enumerate}
	This proves (10). Examining the simply laced Dynkin diagrams we see that in each case, the folded diagram is connected, hence $\Phi_\Theta$ is irreducible. This proves (5), and (11) is obtained by inspection.
\end{proof}

\begin{rem}
	\begin{enumerate}
		\item The above Lemma can be applied with the roles of $\Phi$ and $\Phi^\vee$ switched, so we can consider the root system $\Phi_\Theta \subset V_\Theta$ and $\Phi^\Theta \subset V^\Theta$. We emphasize that $\Phi_\Theta$ may be non-reduced, while $\Phi^\Theta$ is always reduced. 
  		\item We further emphasize that $\Phi^\Theta$ is \emph{not} the root system obtained as the image of $\Phi$ under the trace map $\sum_{\theta \in \Theta}\theta : V \to V^\Theta$. Indeed, this trace map factors through the quotient $V_\Theta$ and produces an isomorphism $V_\Theta \to V^\Theta$, so the image of $\Phi$ under the trace map is a root system isomorphic to $\Phi_\Theta$, and hence generally non-reduced. Even when $\Phi_\Theta$ is reduced, its image under the trace map is in general distinct from $\Phi^\Theta$, and not even isomorphic to it. For example, when $\Phi$ is of type $A_{2n+1}$ and $\Theta$ acts faithfully and non-trivially, then $\Phi_\Theta$ is of type $C_n$, while $\Phi^\Theta$ is of type $B_n$, in accordance with the duality statement in the above lemma.
    	\item We can consider the subsystems $(\Phi_\Theta)^\tx{nd}$ and $(\Phi_\Theta)^\tx{nm}$ of non-divisible and non-multipliable elements in $\Phi_\Theta$. Then $(\Phi_\Theta)^\tx{nd}$ is dual to $((\Phi_\Theta)^\vee)^\tx{nm}$ and $(\Phi_\Theta)^\tx{nm}$ is dual to $((\Phi_\Theta)^\vee)^\tx{nd}$. The resulting two pairs of dual root systems are those considered in \cite[Definition 3.3]{Haines18}. More precisely, using the notation of that reference we have
     	\begin{eqnarray*}
		 (\Phi_\Theta)^\tx{nd}&=&\tx{res}_\Theta(\Phi),\\ (\Phi_\Theta)^\tx{nm}&=&\tx{res}'_\Theta(\Phi),\\ 
		 ((\Phi_\Theta)^\vee)^\tx{nd}&=&N_\Theta(\Phi),\\ 
		 ((\Phi_\Theta)^\vee)^\tx{nm}&=&N_\Theta'(\Phi).
	 	\end{eqnarray*}
		  Moreover $((\Phi_\Theta)^\vee)^\tx{nd}=\Phi^\Theta$.
    	\item Assume that $\Phi$ is irreducible and $\Theta$ acts faithfully. Let $\Phi^\tx{tr}$ be the image of $\Phi$ under the trace map $\sum_{\theta \in \Theta}\theta : V \to V^\Theta$. Then $\Phi^\Theta$ is the subsystem of non-divisible roots in $\Phi^\tx{tr}$. More precisely, $\Phi^\Theta=\Phi^\tx{tr}$ unless $\Phi$ is of type $A_{2n}$ and $\Theta$ is non-trivial, in which case $\Phi^\tx{tr}$ is of type $BC_n$ and $\Phi^\Theta$ is of type $C_n$.
	\end{enumerate}

\end{rem}

\section{The relationship between $G$ and $G^A$} \label{app:g-ga}

We continue with a connected reductive group $G$ defined and quasi-split over $F$, and $A$ a finite group of automorphisms that leaves invariant an $F$-pinning of $G$.

Let $G^1=G^{A,\circ}$ be the identity component of the $A$-fixed points in $G$. We collect here some results about the relationship between $G$ and $G^1$. They are well-known, and we are collecting them here for the convenience of the reader.

\begin{lem} \label{lem:linesigns}
	Let $(T,B,\{X_\alpha\})$ be an $A$-stable pinning of $G$. Let $\Phi \subset X^*(T)$ be the absolute root system of $G$ relative to $T$. If $a \in A$ fixes $\alpha \in \Phi$, then it acts on the root line $\tx{Lie}(G)_\alpha$ by the scalar $+1$ unless the component $\Phi'$ of $\Phi$ that contains $\alpha$ is of type $A_{2n}$ and $a$ acts non-trivially on it, in which case it acts by the scalar $-1$ on $\tx{Lie}(G)_\alpha$.
\end{lem}
\begin{proof}
	The question is insensitive to extending the ground field, so we may assume it is algebraically closed. We may also replace $G$ by its adjoint group. Then $G$ breaks up into a product of simple groups, permuted by $A$, according to the decomposition of $\Phi$ into irreducible factors.

	We may assume that $\Phi$ is irreducible by replacing $\Phi$ by $\Phi'$ and $A$ by the stabilizer of $\Phi'$ in $A$. We may also assume $\Phi$ is simply laced, otherwise $A$ acts trivially and the statement is trivial. We now use induction on the height of the root $\alpha$ and Lemma \ref{lem:rootsupp}.

	The case of height $1$ is that of a simple root, and follows by assumption on the existence of an $A$-stable pinning. Consider now a root of height $h>1$. If $\Phi$ is not of type $A_{2n}$, a visual inspection of the Dynkin diagram shows that $\alpha=\alpha_{-1}+\alpha_0+\alpha_1$, where $\alpha_{-1},\alpha_1$ are simple and switched by $a$, and $\alpha_0$ is of height $h-2$ and fixed by $a$. If $X_{-1},X_0,X_1$ are non-zero vectors belonging to the corresponding root lines, then $X=[X_{-1},[X_0,X_1]]$ is a non-zero root vector lying in the $\alpha$-root line. By induction $a(X_0)=X_0$, and the Jacobi identity implies $a(X)=X$.

	If $\Phi$ is of type $A_{2n}$, then a visual inspection of the Dynkin diagram shows that $\alpha=\alpha_{-1}+\alpha_1$ with $\alpha_{-1},\alpha_1$ of equal height $h/2$ and switched by $a$. If $X_{-1},X_1$ are non-zero root vectors in the corresponding root lines, then $X=[X_{-1},X_1]$ is a non-zero root vector lying in the $\alpha$-root line with $a(X)=-X$.
\end{proof}

The following proposition collects basic results of Steinberg from \cite{Ste68end} in a form similar to \cite[Theorem 1.1.A]{KS99}.

\begin{pro} \label{pro:stein}
In this proposition $F$ is an arbitrary field of characteristic zero.
	\begin{enumerate}
		\item The connected algebraic group $G^1=G^{A,\circ}$ is reductive and quasi-split.
		\item If $G$ is simply connected, then $G^A$ is connected, hence equal to $G^1$, but may not be simply connected.
		\item If $G$ is adjoint, then $G^A$ is connected, hence equal to $G^1$, and adjoint.
    	\item An $A$-stable Borel $F$-subgroup $B$ of $G$ contains an $A$-stable maximal $F$-torus $T$. 
    	\item If $(T,B)$ is an $A$-stable $F$-Borel pair, then $(T^1,B^1)$ is an $F$-Borel pair of $G^1$, where $T^1:=T \cap G^1$ and $B^1=B \cap G^1$. We have $T=\tx{Cent}(T^1,G)$ and $\Omega(T^1,G^1)=\Omega(T,G)^A$.
		\item The map $B \mapsto B^1=B \cap G^1$ is a bijection between the set of $A$-stable Borel $F$-subgroups of $G$ and the set of Borel $F$-subgroups of $G^1$.
     	\item The map $(T,B) \mapsto (T^1,B^1)$ is a bijection between the set of $A$-stable Borel $F$-pairs of $G$ and the set of Borel $F$-pairs of $G^1$. 
      	\item Assume $G$ is split. Let $(T,B)$ be a Borel $F$-pair of $G$ and $(T^1,B^1)$ the corresponding pair for $G^1$. Let $\Delta \subset \Phi$ and $\Delta^1 \subset \Phi^1$ be the associated root systems with sets of simple roots. The image of $\Phi$ under the restriction map $X^*(T) \to X^*(T^1)$ is a (possibly non-reduced) root system $\Phi_\tx{res}$, whose set of simple roots equals $\Delta^1$, and whose subsystem of indivisible roots equals $\Phi^1$. The fibers of the surjections $\Phi \to \Phi_\tx{res}$ and $\Delta \to \Delta^1$ induced by the restriction map are precisely the $A$-orbits.
  		\item If $(T,B,\{X_\alpha\})$ is an $A$-stable $F$-pinning of $G$, then $(T^1,B^1,\{X_{\alpha^1}\})$ is an $F$-pinning of $G^1$, where $X_{\alpha^1}=\sum_{\alpha \mapsto \alpha^1} X_\alpha$. This establishes a bijection between the set of $A$-stable $F$-pinnings of $G$ and the set of $F$-pinnings of $G^1$.
    	\item Any $A$-stable $F$-Borel pair can be extended to an $A$-stable $F$-pinning.
	\end{enumerate}
\end{pro}
\begin{proof}
	Steinberg's results in \cite{Ste68end} are formulated for a single automorphism $a$ of $G$ and over an algebraically closed field $F$. For the purposes of this proposition, we can reduce to this case as follows. We are free to replace $G$ by $G_\tx{sc}$ for all statements except (2), and by $G_\tx{ad}$ for all statements except (3), so we assume without loss of generality that $G$ is either simply connected or adjoint. 
	
	Let us first assume that $F$ is algebraically closed. Then $G$ decomposes into a product of simple factors permuted by $A$. All the claims reduce, by taking products, first to the case that the action of $A$ on the set of factors is transitive, and then to the case of a single factor. We may thus assume that $G$ is simple. We may also replace $A$ by the quotient through which it acts effectively. Then $A$ is cyclic, except possibly when $G$ is of type $D_4$, when $A$ could be the symmetric group on $3$ letters. 
	
	Consider first the case that $A$ is cyclic, generated by $a \in A$. We can now apply Steinberg's results to the automorphism $a$ of $G$. This automorphism is quasi-semisimple in the sense of \cite[\S9]{Ste68end}. 
	
	(2) When $G$ is simply connected, the group $G^a$ is connected by \cite[Corollary 9.7]{Ste68end}.
	
	(3) When $G$ is adjoint, the connectedness of $G$ is proved in \cite[Lemma 3.1]{Ree10}. Even though this reference assumes the field to be $\C$, this particular argument does not require this assumption. Indeed, the argument uses \cite[\S9.2,\S9.5]{Ste68end}, the first of which shows $\pi_0(G^a)=((1-a)G_\tx{sc} \cap Z(G_\tx{sc}))/(1-a)Z(G_\tx{sc})$, and the second gives an expression for $(1-a)G_\tx{sc} \cap Z(G_\tx{sc}))$. That expression involves an element $t' \in T_\tx{sc}$, where $T_\tx{sc}$ is a maximal torus of $G_\tx{sc}$ that is stable under $a$ and contained in a Borel subgroup also stable under $a$. We can choose this pair to be part of an $a$-stable pinning of $G$. The element $t'$ is described in part \cite[\S8.2(5)]{Ste68end} as having the property $\alpha(t')=c_{a(\alpha)}$ for every root $\alpha$, where $c_\alpha$ is defined in the beginning of \cite[\S8.2]{Ste68end}. Since $a$ preserves a pinning, all elements $c_\alpha$ are equal to $1$ and we can take $t'=1$. Then \cite[\S9.5]{Ste68end} shows $(1-a)G_\tx{sc} \cap Z(G_\tx{sc})=(1-a)T_\tx{sc} \cap Z(G_\tx{sc})$. Reeder then argues that the inclusion $Z(G_\tx{sc}) \to T_\tx{sc}$ induces an injective map $Z(G_\tx{sc})_a \to (T_\tx{sc})_a$, because the map $X^*(T_\tx{sc})^a \to X^*(Z(G_\tx{sc}))^a$ is surjective. Therefore, $(1-a)T_\tx{sc} \cap Z(G_\tx{sc})=(1-a)Z(G_\tx{sc})$, concluding the proof of connectedness of $G^a$ when $G$ is adjoint.

	We still need to prove that $G^a$ is adjoint. For this we will need to know (5) and (8), which do not rely on that claim. Assuming them, fix an $a$-stable Borel pair $(T,B)$. Then $T^a$ is a maximal torus of $G^a$ by (5), hence contains the center of $G^a$, which is the subgroup of $T^a$ killed by all elements of the set of simple roots $\Delta^1$. By (8) that set is the set of restrictions of the simple roots $\Delta$, so we conclude $Z(G^a)=T^a \cap Z(G)=\{1\}$. 

	(4) This is \cite[\S7.6]{Ste68end}.

	(5) We assume from now on that $G$ is simply connected and simple. Let $(T,B)$ be an $a$-stable Borel pair in $G$. Let $\Phi=\Phi(T,G)$ be the root system of $G$ relative to $T$, $\Phi^+=\Phi(T,B)$ the set of $B$-positive roots, and $\Delta \subset \Phi^+$ the set of simple roots. By \cite[\S8.2(3)]{Ste68end} the group $T^a$ is connected, hence a torus in $G^a$. We claim that no element of $\Phi$ restricts trivially to $T^a$, i.e. maps to zero in the module $X^*(T)_a$ of $a$-coinvariants in $X^*(T)$. It is enough to prove this for $\Phi^+$, hence for $\Delta$. Since $\otimes_\Z\Q$ is right-exact and $X^*(T)_a$ is torsion-free, we have $X^*(T)_a \subset (X^*(T) \otimes_\Z\Q)_a$, so it is enough to show that no element of $\Delta$ projects trivially to $(X^*(T) \otimes_\Z\Q)_a$. But $\Delta$ forms an $a$-stable basis of $X^*(T) \otimes_\Z\Q$, and the claim follows.

	The claim implies that $T=\tx{Cent}(T^a,G)$. In particular, $T^a$ is a maximal torus in $G^a$, for otherwise $\tx{Cent}(T^a,G^a)$ would be a Levi subgroup of $G^a$ with a non-trivial root system, hence non-commutative, contradicting its containment in $T$.

	For a moment assume that $(T,B)$ is part of an $a$-stable pinning of $G$. Then the constants $c_\alpha$ defined in the beginning of \cite[\S8.2]{Ste68end} have the property $c_\alpha=1$ when $\alpha$ is simple. Therefore \cite[\S8.2(5)]{Ste68end} shows that every element of $\Omega(T,G)^a$ is represented in $N(T,G)^a$. Now $T=\tx{Cent}(T^a,G)$ implies $\Omega(T^a,G^a)=N(T^a,G^a)/T^a=N(T,G)^a/T^a=\Omega(T^a,G^a)$.

	Returning to a general $a$-stable pair $(T,B)$, let $(T_0,B_0)$ be one that belongs to an $a$-stable pinning. Since both $T^a$ and $T_0^a$ are maximal tori of $G^a$, we may conjugate $(T,B)$ under $G^a$ to assume that $T^a=T_0^a$. Then $B$ and $B_0$ are two $a$-stable Borel subgroups that contain $T$, so $B=wB_0w^{-1}$ for a unique, hence $a$-fixed, element $w \in \Omega(T,G)$. But then $w \in \Omega(T^a,G^a)$, so we can further conjugate $(T,B)$ by $G^a$ to assume $T=T_0$ and $B=B_0$.

	Finally, we show that $B^a$ is a Borel subgroup of $G^a$. It is certainly solvable, and we have $B^a=T^aU^a$, where $U$ is the unipotent radical of $B$, so by \cite[\S8.2(2,3)]{Ste68end} we know $B^a$ is connected. It is therefore contained in a Borel subgroup $B_1'$ of $G^a$, which in turn is contained in an $a$-stable Borel subgroup $B_1$ of $G$ by \cite[7.4 Corollary]{Ste68end}. By \cite[\S7.6]{Ste68end} $B_1$ contains an $a$-stable maximal torus $T_1$. By what was argued so far, $B_1^a$ is a connected solvable subgroup of $G^a$ containing $B_1'$, hence equal to $B_1'$ by maximality of latter. But we have also argued that $(B_1,T_1)$ is conjugate by $G^a$ to $(B,T)$, and we conclude that $B^a$ is a Borel subgroup of $G^a$, as claimed.

	(6,7) From (5) we have a well-defined map $(T,B) \mapsto (T^1,B^1)$. Surjectivity is immediate, because the conjugation action of $G^a$ preserves the set of $a$-stable Borel pairs of $G$ and is transitive on the set of Borel pairs of $G^a$. Injectivity reduces to showing that two $a$-stable Borel subgroups $B_1,B_2$ that contain the same $a$-stable maximal torus $T$ and have $B_1^a=B_2^a$ are equal. But the element $w \in \Omega(T,G)$ conjugating $B_1$ to $B_2$ is then $a$-fixed, hence by (5) an element of $\Omega(T^a,G^a)$, which then preserves the Borel subgroup $B_1^a=B_2^a$ of $G^a$ and must therefore be trivial. This completes the proof of (7), and also the surjectivity of (6). For injectivity, it is enough to show that two $a$-stable Borel subgroups $B_1,B_2$ contain a common $a$-stable maximal torus. For that, apply \cite[\S7.6]{Ste68end} to $B_1 \cap B_2$ to conclude that this solvable group contains an $a$-stable torus $T$ that is maximal in $B_1 \cap B_2$. But it is well known that $B_1 \cap B_2$ contains a maximal torus of $G$, so $T$ must also be maximal in $G$.

	(8) According to (5), $T^a$ is a maximal torus of $G^a$ and $U^a$ is a maximal unipotent subgroup. In \cite[\S8.2(2)]{Ste68end}, Steinberg computes the decomposition of $U^a$ into root subgroups for $T^a$ and shows that the roots are certain elements of the image $\Phi_\tx{res}$ of $\Phi$ under the restriction map $X^*(T) \to X^*(T^1)$, and which these elements are is described in terms of the constants $c_\alpha$. Note that $X^*(T^1)\otimes_\Z\R$ is the space of $A$-coinvariants in $X^*(T)\otimes_\Z\R$, so $\Phi_\tx{res}$ is described in Proposition \ref{pro:fold1}.
	
	Lemma \ref{lem:linesigns} shows that the constants $c_\alpha$ are equal to $+1$ except when the image of $\alpha$ in $\Phi_\tx{res}$ is divisible, in which case these elements are equal to $-1$. This proves that $\Phi^1$ is the set of indivisible elements in $\Phi_\tx{res}$. That $\Delta^1$ is the image of $\Delta$ is part of Proposition \ref{pro:fold1}.

	(9) According to Proposition \ref{pro:fold1}, an element $\alpha^1 \in \Delta$ corresponds to an $A$-orbit in $\Delta$. Moreover, Steinberg describes that $\tx{Lie}(G^A)_{\alpha^1}=(\bigoplus_{\alpha \mapsto \alpha^1} \tx{Lie}(G)_\alpha)^A$. The element $X_{\alpha^1}=\sum_{\alpha \mapsto \alpha^1} X_\alpha$ lies in that line and is non-zero. This provides the desired pinning of $G^A$. Conversely, given a pinning $(T^1,B^1,\{X_{\alpha^1})$ of $G^A$, (7) provides a unique $A$-stable Borel pair $(T,B)$ of $G$ from which $(T^1,B^1)$ arises by intersection with $G^A$. Using the above direct sum decomposition we obtain $X_{\alpha^1}=\sum_{\alpha \mapsto \alpha^1} X_\alpha$ for unique $X_\alpha \in \tx{Lie}(G)_\alpha$ and the set $\{X_\alpha\}$ is $A$-stable.
	
	(10) This follows from the existence of one $A$-stable pinning and the fact that all $A$-stable Borel pairs are conjugate under $G^A$, according to (7).
	\vskip 5pt

	This completes the proof under the assumptions that $F$ is algebraically closed, $G$ is (without loss of generality) simple, and $A$ is cyclic. Keeping $F$ algebraically closed and $G$ simple, the only outstanding case is that $G$ is simple of type $D_4$ and $A$ is the symmetric group on 3 letters. We may also, as before, assume that $G$ is adjoint. Let $B \subset A$ be the unique subgroup of index $2$. Then $G^B$ is a connected adjoint group by (3), of type $G_2$. The $A$-stable pinning of $G$ endows $G^B$ with a pinning by (9), which is clearly $A/B$-stable. But a group of type $G_2$ doesn't have non-trivial pinned automorphisms, so we conclude that $A/B$ acts trivially on $G^B$, hence  $G^A=G^B$, and all results follow for that case.

	We are finally left with removing the assumption that $F$ is algebraically closed. This comes down to Galois equivariance. The bijections in (6), (7), and (9), are Galois-equivariant, hence descent to bijections on $F$-objects. The same reasoning applies to (5). Part (8) is unaffected by changing the base field due to the assumption that $G$ is split. Part (4) is obtained from the bijection (7) by choosing a maximal $F$-torus $T^1$ inside of $B^1$. Part (10) follows from the assumed existence of an $A$-stable $F$-pinning and the conjugacy of all $F$-rational $A$-Borel pairs of $G$ under $G^A(F)$, which follows from (7).
\end{proof}

\begin{cor} \label{cor:regnilp}
A nilpotent element of $\mf{g}^A$ is regular if and only if it is regular as an element of $\mf{g}$.
\end{cor}
\begin{proof}
	We begin by showing that regular nilpotent and $A$-fixed elements of $\mf{g}$ exist. For this, take an $A$-fixed pinning $(T,B,\{X_\alpha\})$ of $G$, which exists by assumption. Then $X=\sum X_\alpha$ is such an element.
	
	Now let $X \in \mf{g}$ be regular nilpotent and $A$-fixed. Then there is a unique, hence $A$-stable, Borel subgroup $B \subset G$ whose Lie algebra contains $X$. By Proposition \ref{pro:stein}(6), $B^1=B \cap G^1$ is a Borel subgroup of $G^1$ whose Lie algebra contains $X$ and is unique with this property, hence $X$ is regular nilpotent as an element of $\mf{g}^A$.
	
	Conversely let $X \in \mf{g}^A$ be regular nilpotent. Since all such elements are conjugate under $G^1$, we may assume that this $X$ comes from an $A$-fixed pinning of $G$ as in the first paragraph of the proof. But then $X$ is regular in $\mf{g}$ as well.
\end{proof}

\begin{cor} \label{cor:pinu}
	The map that associates to an $A$-stable $F$-pinning $(T,B,\{X_\alpha\})$ the element $\sum_\alpha X_{-\alpha} \in \mf{g}^A(F)$, where $X_{-\alpha} \in \mf{g}_\alpha$ is determined by $[X_\alpha,X_{-\alpha}]=H_\alpha$, is a bijection between the set of $G^1(F)$-conjugacy classes of $A$-stable $F$-pinnings of $G$ and the set of $G^1(F)$-conjugacy classes of regular semi-simple elements of $\mf{g}^A(F)$.
\end{cor}
\begin{proof}
	When $A=\{1\}$ this is well-known, see e.g. \cite[Lemma 3.1.1(3)]{DPR} or \cite[Lemma 5.1.A]{LS87}. For general $A$ use Proposition \ref{pro:stein}(9) and Corollary \ref{cor:regnilp}.
\end{proof}

\end{appendices}

\bibliographystyle{amsalpha}
\bibliography{../../TexMain/bibliography.bib}

\end{document}